\documentclass[a4paper, twoside,french]{book}

\usepackage[utf8]{inputenc}%caractères accentués
\usepackage[T1]{fontenc}
\usepackage[french, english]{babel}
\usepackage[babel]{csquotes}

\usepackage[backend=biber, style=trad-plain, doi=false, isbn=false, url=false,eprint=false]{biblatex}
\AtEveryBibitem{\clearlist{language}}%pas de langue dans la bibliographie
\usepackage[a4paper,left=3.7cm,right=5.2cm]{geometry}

\usepackage{amsmath}%substack, displaystyle
\usepackage{amsfonts}%mathbb, mathfrak, mathcal
\usepackage{amsthm}%newtheorem, proof
\usepackage{bbm}%fonction indicatrice
\usepackage{amssymb}%sphericalangle
\usepackage{float}%choisir l'emplacement des images
\usepackage{tikz}
\usetikzlibrary{shapes.misc}
\tikzset{cross/.style={cross out, draw, 
         minimum size=2*(#1-\pgflinewidth), 
         inner sep=0pt, outer sep=0pt}}
\usepackage{IEEEtrantools}%equations trop longues
\usepackage{url}%tilde dans les url
\usepackage{hyperref}

\newtheorem{thm}{Théorème}[chapter]
\newtheorem{prop}[thm]{Proposition}

\newtheorem{cor}[thm]{Corollaire}

\newtheorem{theorem}{Théorème}[section]
\newtheorem{proposition}[theorem]{Proposition}
\newtheorem{lemma}[theorem]{Lemme}
\newtheorem{corollary}[theorem]{Corollaire}

\theoremstyle{definition}
\newtheorem{definition}[theorem]{Définition}

\newtheorem*{example}{Exemple}
\newtheorem*{remark}{Remarque}
\newtheorem{exercise}{Exercice}

\newcommand{\comm}[1]{}
\newcommand{\details}[1]{#1}
\newcommand{\note}[1]{}

\newcommand{\Q}{\mathbb{Q}}
\newcommand{\QQ}{\overline{\mathbb{Q}}}
\newcommand{\Z}{\mathbb{Z}}
\newcommand{\R}{\mathbb{R}}
\newcommand{\C}{\mathbb{C}}
\newcommand{\N}{\mathbb{N}}
\newcommand{\PP}{\mathbb{P}}
\newcommand{\BS}{\mathbb{S}}

\newcommand{\ka}{\mathfrak{a}}
\newcommand{\kb}{\mathfrak{b}}
\newcommand{\kz}{\mathfrak{z}}

\newcommand{\g}{\mathfrak{g}}
\newcommand{\p}{\mathfrak{p}}
\newcommand{\uu}{\mathfrak{u}}

\newcommand{\FS}{\mathfrak{S}}

\newcommand{\cC}{\mathcal{C}}
\newcommand{\cF}{\mathcal{F}}

\newcommand{\cH}{\mathcal{H}}

\newcommand{\cS}{\mathcal{S}}
\newcommand{\cA}{\mathcal{A}}

\newcommand{\bv}{\mathbf{v}}
\newcommand{\bp}{\mathbf{p}}
\newcommand{\bu}{\mathbf{u}}
\newcommand{\bw}{\mathbf{w}}

\newcommand{\bcs}{\backslash}

\newcommand{\GL}{\mathrm{GL}}
\newcommand{\SL}{\mathrm{SL}}
\newcommand{\SO}{\mathrm{SO}}

\newcommand{\BA}{\mathrm{BA}}
\newcommand{\gl}{\mathfrak{gl}}

\newcommand{\bracket}[1]{\langle#1\rangle}
\newcommand{\abs}[1]{\lvert#1\rvert}
\newcommand{\Abs}[1]{\left\lvert#1\right\rvert}
\newcommand{\norm}[1]{\lVert#1\rVert}

\newcommand{\dd}{\mathrm{d}}

\newcommand{\tx}{\widetilde{X}}

\newcommand{\tphi}{\widetilde{\phi}}
\newcommand{\tpsi}{\widetilde{\psi}}
\newcommand{\tOmega}{\widetilde{\Omega}}
\newcommand{\eps}{\varepsilon}

\DeclareMathOperator{\Supp}{Supp}
\DeclareMathOperator{\rang}{rang}
\DeclareMathOperator{\End}{End}
\DeclareMathOperator{\diag}{diag}
\DeclareMathOperator{\Grass}{Grass}
\DeclareMathOperator{\card}{card}
\DeclareMathOperator{\Stab}{Stab}
\DeclareMathOperator{\Ad}{Ad}
\DeclareMathOperator{\ad}{ad}
\DeclareMathOperator{\Vect}{Vect}
\DeclareMathOperator{\covol}{covol}

\title{Groupes arithmétiques et approximation diophantienne}
\author{Nicolas de Saxcé}

\begin{document}

\selectlanguage{french}

%Page de garde : recto
\newgeometry{left=2.5cm,right=2.5cm}
\thispagestyle{empty}

\begin{flushleft}
\begin{tabular}{l}
\textsc{Habilitation à diriger des recherches}\\
\textsc{CNRS -- Université Sorbonne Paris Nord}\\
\textsc{Institut Galilée}
\end{tabular}
\end{flushleft}

\vspace{4.5cm}

\begin{center}
\rule{\textwidth}{.5pt}

\Huge
Groupes arithmétiques et approximation diophantienne

\rule{\textwidth}{.5pt}
%\maketitle

\normalsize
\vspace{2cm}
\textbf{Nicolas de Saxcé}

\end{center}

\vfill
Soutenue le 11 juin 2021 devant le jury composé de\\

\begin{tabular}{l}
M. \textbf{Julien Barral}, Professeur à l'université Sorbonne Paris Nord,\\
M. \textbf{Yves Benoist}, Directeur de recherches CNRS à l'université Paris-Saclay,\\
M. \textbf{Emmanuel Breuillard}, Professeur à l'université de Cambridge,\\
M. \textbf{Frédéric Paulin}, Professeur à l'université Paris-Saclay,\\
M. \textbf{Jean-François Quint}, Directeur de recherches CNRS à l'université de Bordeaux,\\
M. \textbf{Olivier Wittenberg}, Directeur de recherches CNRS à l'université Sorbonne Paris Nord,\\
\end{tabular}
\vspace{.5cm}

\noindent\emph{au vu des rapports de}
MM. \textbf{Julien Barral}, \textbf{Jean-François Quint} et \textbf{Manfred Einsiedler}, professeur à l'ETH Zürich.

\vspace{1cm}

%\maketitle

%résumé et abstract
\newpage
\newgeometry{left=3.7cm, right=4.7cm}

\begin{center}
\textsc{Groupes arithmétiques et approximation diophantienne}
\\[1ex]
{\bfseries \abstractname}
\end{center}
\par Nous développons une théorie de l'approximation diophantienne dans les variétés de drapeaux, obtenues comme quotient d'un groupe de Lie semi-simple défini sur $\Q$ par un sous-groupe parabolique.
En nous appuyant sur des résultats de la théorie des groupes arithmétiques, dûs entre autres à Borel et Harish-Chandra, et à Margulis et ses collaborateurs, nous démontrons dans ce cadre des généralisations des théorèmes classiques de l'approximation diophantienne.
\\[\bigskipamount]
\textbf{Mots-clefs} : groupes algébriques, dynamique homogène, approximation diophantienne
\par\noindent\hspace*{0.35\textwidth}\hrulefill\hspace*{0.35\textwidth}
\\[-\bigskipamount]

\selectlanguage{english}

\begin{center}
\textsc{Arithmetic groups and diophantine approximation}
\\[1ex]
{\bfseries \abstractname}
\end{center}
\par We develop a theory of diophantine approximation on generalized flag varieties, varieties that can be obtained as a quotient of a semisimple algebraic $\Q$-group by a parabolic $\Q$-subgroup.
Using methods from the theory of arithmetic groups, due in particular to Borel and Harish-Chandra, and to Margulis and his collaborators, we prove in this setting analogs of the classical theorems of diophantine approximation.
\\[\bigskipamount]

\selectlanguage{french}

%Remerciements
\newpage
\thispagestyle{empty}

\section*{Remerciements}

%\hspace*{0.3cm}

C'est il y a à peu près cinq ans, au cours de la rédaction d'un article sur l'approximation diophantienne dans les matrices et les groupes nilpotents, avec Menny Aka, Emmanuel Breuillard et Lior Rosenzweig, que m'est venue l'idée d'une théorie de l'approximation diophantienne pour les variétés de drapeaux, dont traite le présent mémoire.
Je suis reconnaissant à mes collaborateurs de m'avoir donné l'occasion de me plonger dans ce domaine de la théorie des nombres et de la dynamique homogène.
Par ailleurs, depuis le début de ce projet, j'ai pu faire part des problèmes qui m'intéressaient à plusieurs collègues, qui m'ont tous généreusement partagé leurs idées et leurs connaissances.
Je les en remercie vivement.
Ces discussions m'ont été une aide précieuse en de nombreux points, et je tiens à mentionner ici les plus importants.

\bigskip

Tout d'abord, des discussions avec Roland Casalis sur les fibrés en droites et l'homologie m'ont fait comprendre qu'il était plus naturel de ne pas privilégier de hauteur sur une variété de drapeaux, mais plutôt d'en fixer une par le choix d'une représentation irréductible arbitraire. Cela m'a guidé vers l'énoncé de la généralisation de la correspondance de Dani présentée au chapitre~\ref{chap:correspondance}.
Ensuite, pour l'énoncé du théorème de non divergence dans un groupe semi-simple arbitraire, je me suis beaucoup appuyé sur des échanges avec Barak Weiss qui m'a notamment expliqué la représentation d'une filtration de Harder-Narasimhan d'un réseau dans $\R^d$ à l'aide d'une fonction convexe sur le segment d'entiers $[0,d]$.
Ses explications m'ont donné une meilleure compréhension du phénomène de non divergence pour $\SL_d$, indispensable avant l'étude d'un cadre plus général.
Les commentaires d'Uri Shapira sur une démonstration partielle que je lui ai présentée lors d'un séjour au Technion m'ont aussi aidé à poursuivre mon travail dans cette direction.
Enfin, alors que ce mémoire était encore en cours de rédaction, Elon Lindenstrauss m'a signalé sa démonstration, avec Margulis, Mohammadi et Shah \cite{lmms}, d'un énoncé de non divergence quantitative voisin de celui qui m'intéressait.
Ses remarques judicieuses m'ont permis d'améliorer la présentation du chapitre~\ref{chap:nondivergence}, et d'en préciser certains énoncés.
Vers la fin de ce travail, j'ai été invité par Manfred Einsiedler à participer au trimestre \enquote{Dynamics: Topology and Numbers}, organisé avec Anke Pohl et Martin Möller à l'institut Hausdorff, à Bonn.
Ce séjour m'a fourni un cadre idéal pour achever la rédaction de ce mémoire.
J'y ai notamment profité des explications de Manfred Einsiedler sur le théorème de décroissance des coefficients dans $L^2(G/\Gamma)$ présenté dans son article \cite{emmv} en commun avec Margulis, Mohammadi et Venkatesh, et nécessaire à la démonstration du théorème de Khintchine, qui fait l'objet du chapitre~\ref{chap:khintchine}.
De loin en loin, les quelques fois où j'ai eu la chance de croiser Yves Benoist sur mon chemin m'ont aussi toujours fait avancer un peu plus rapidement. Je lui dois plusieurs éclaircissements sur la théorie des groupes algébriques et arithmétiques, et en particulier la démonstration de la proposition~\ref{strat}, utile à la définition de la distance de Carnot-Carathéodory sur une variété de drapeaux.
Je tiens aussi à remercier vivement Jean-François Quint pour sa relecture attentive d'une version préliminaire de ce mémoire et pour ses commentaires détaillés, qui m'ont permis de corriger plusieurs démonstrations dont la présentation laissait à désirer.
Enfin et surtout, je dois beaucoup à Emmanuel Breuillard:
% pour la plupart des problèmes dont il est question ici, ses conseils et ses remarques m'ont guidé dans mon travail.
 sans ses conseils et sans les nombreuses discussions nous avons eues depuis cinq ans au cours de notre collaboration, ce texte n'aurait pas vu le jour.

\bigskip

Pour conclure, je dois encore remercier Julien Barral, Manfred Einsiedler et Jean-François Quint pour leurs rapports sur ce mémoire, ainsi qu'Yves Benoist, Emmanuel Breuillard, Frédéric Paulin et Olivier Wittenberg qui ont bien voulu faire partie du jury de soutenance.

%\cleardoublepage

%Style et numérotation des pages (post-introduction)
%\newpage
%\cleardoublepage
%\pagestyle{fancy}

\chapter*{Avant-propos}

Ce mémoire sur l'approximation diophantienne a été rédigé pour soutenir une habilitation à diriger des recherches.
Contrairement à l'usage le plus répandu pour ce type de texte, la plupart des résultats présentés sont nouveaux; nous expliquons brièvement ici les raisons à cela.

\bigskip

Depuis mon entrée au CNRS en 2014, ma recherche s'est divisée en deux parties.
La première concerne les propriétés d'expansion dans les groupes de Lie simples.
Mes travaux sur ce sujet prolongent ce que j'avais commencé dans ma thèse de doctorat \emph{Sous-groupes boréliens des groupes de Lie}, sous la direction d'Emmanuel Breuillard, puis en post-doctorat avec Elon Lindenstrauss.
%Même si mes efforts n'ont pas toujours porté dessus, c'est la conjecture du trou spectral, selon laquelle toute marche aléatoire adaptée sur un groupe de Lie compact simple s'équidistribue à vitesse exponentielle, qui dirige mes recherches dans ce domaine.
%En effet, les techniques d'analyse et de combinatoire additive développées en particulier par Helfgott, Bourgain et Gamburd pour tenter de résoudre cette conjecture, y sont essentielles à tous mes résultats.
%: propriété du trou spectral \cite{benoistsaxce}, dimension de Hausdorff des sous-groupes boréliens \cite{}, équidistribution
%C'est le cas bien sûr dans pour notre étude de la dimension de Hausdorff des sous-groupes boréliens des groupes de Lie, dont traitent les articles \cite[], 
Je poursuis encore mon travail de recherche autour de ces questions, en particulier grâce à une collaboration avec Weikun He sur l'équidistribution des marches aléatoires linéaires sur le tore et ses applications au problème du trou spectral.
J'espère d'ailleurs donner l'année prochaine un cours sur ces questions, et rédiger à cette occasion un texte d'introduction au domaine.
Cela explique en partie le choix de ne pas décrire ces travaux dans ce mémoire, pour le consacrer entièrement à mon deuxième sujet d'étude, l'approximation diophantienne.

\bigskip

Mes premiers travaux sur le sujet, en collaboration avec Menny Aka, Emmanuel Breuillard et Lior Rosenzweig, datent de mon post-doctorat à l'université Hébraïque de Jérusalem, et portaient sur les propriétés diophantiennes des groupes de Lie nilpotents.
Ils ont été pour moi l'occasion de me familiariser avec les méthodes de dynamique homogène qui ont permis depuis un peu plus de vingt ans plusieurs avancées majeures en théorie des nombres.
Celle qui nous concerne plus particulièrement ici est la résolution par Kleinbock et Margulis de la conjecture de Sprindzuk sur l'extrémalité des sous-variétés non dégénérées. %, basée sur les observations de Margulis concernant le comportement des applications polynomiales à valeurs dans l'espace des réseaux.

Dans un projet avec Emmanuel Breuillard, nous avons observé que certains points de la démonstration de Kleinbock et Margulis pouvaient être mis en parallèle avec la démonstration de Schmidt du théorème du sous-espace sur l'approximation des nombres algébriques par des rationnels.
Ces démonstrations se composent de deux parties, une partie d'analyse, et une partie de géométrie.
La partie d'analyse est différente dans chacun des problèmes: pour le théorème du sous-espace, il faut étudier les points d'annulation des polynômes à plusieurs variables à coefficients entiers, tandis que pour la conjecture de Sprindzuk, il faut comprendre les petites valeurs des fonctions analytiques réelles.
Mais la partie de géométrie est essentiellement la même: il s'agit de comprendre la géométrie de l'espace des réseaux d'un espace euclidien.
Ces observations nous ont poussés à chercher une même approche pour ces deux types de résultats, soit dans la formulation d'énoncés valables dans les deux cadres, soit dans la rédaction des démonstrations.
%Nous renvoyons à l'article \cite{subspace} pour plus de détails sur le sujet.

Ensuite, grâce à Dmitry Kleinbock, je me suis intéressé à l'approximation diophantienne sur les quadriques, et notamment aux travaux récents de Kleinbock, Fishman, Merrill et Simmons sur le sujet.
On y voit que les propriétés diophantiennes dans la quadrique sont reliées au comportement des orbites diagonales dans un espace de réseaux associé au groupe orthogonal de la quadrique.
C'est cet exemple des quadriques qui m'a fait souhaiter un nouveau cadre pour l'approximation diophantienne, où le groupe linéaire $\SL_d$ serait remplacé par un groupe algébrique semi-simple général.

\bigskip

Comme il était temps pour moi d'écrire un mémoire d'habilitation à diriger des recherches, il m'a semblé approprié d'y exposer ce que j'avais compris en matière d'approximation diophantienne ces dernières années.
Un cadre convenable est celui des variétés de drapeaux, de la forme $P\bcs G$, où $G$ est un groupe algébrique semi-simple défini sur $\Q$, et $P$ un sous-groupe parabolique défini sur $\Q$.
Pour y établir les analogues des théorèmes de l'approximation diophantienne classique, on procède en deux étapes.

Tout d'abord, on établit une correspondance entre les propriétés diophantiennes des points de $X=P\bcs G$ et le comportement asymptotique des orbites diagonales dans l'espace de réseaux $G/\Gamma$, quotient de $G$ par un sous-groupe arithmétique.

Ensuite, pour pouvoir exploiter cette correspondance, certains résultats déjà connus pour l'espace des réseaux d'un espace euclidien doivent être disponibles plus généralement pour un espace $G/\Gamma$ obtenu comme quotient d'un groupe algébrique semi-simple défini sur $\Q$ par un sous-groupe arithmétique.
C'est le cas en particulier pour la non divergence de Kleinbock et Margulis, et pour les résultats obtenus avec Emmanuel Breuillard comme conséquences du théorème du sous-espace de Schmidt.
Grâce à la théorie de la réduction de Borel et Harish-Chandra, qui décrit la forme d'un domaine fondamental pour l'action de $\Gamma$ sur $G$, j'ai pu formuler et démontrer ces résultats dans ce nouveau cadre.
  
\bigskip

Ce mémoire est constitué d'une suite de chapitres liés les uns au autres.
À l'origine était prévu un chapitre pour chaque aspect de l'approximation diophantienne étudié: théorème de Khintchine, approximation des points algébriques et approximation sur des sous-variétés.
Mais je me suis éloigné de ce plan pour plusieurs raisons.
Tout d'abord, il m'a paru plus clair de dédier deux chapitres aux aspects géométriques du problème sur lesquels reposent les autres résultats: le chapitre~\ref{chap:correspondance} décrit donc la géométrie des variétés de drapeaux, et le chapitre~\ref{chap:reduction} celle des espaces de réseaux.
Ensuite, j'ai préféré présenter séparément dans les chapitres~\ref{chap:nondivergence} et \ref{chap:orbites} certains résultats généraux sur les espaces de réseaux, qui pourraient avoir d'autres applications.
Enfin, j'ai ajouté le chapitre~\ref{chap:exemples} pour décrire les résultats principaux du mémoire dans certains cas particuliers.
Ce chapitre mériterait d'ailleurs un traitement plus détaillé, puisque ce sont les exemples qui y sont donnés qui ont principalement motivé l'étude du cas général.

Si l'avenir le permet, les résultats présentés dans ce mémoire seront publiés. %, sous la forme de plusieurs articles sur le sujet ou sous la forme d'un livre.
En attendant le travail de relecture et de réécriture que cela nécessitera, nous prions le lecteur de corriger avec bienveillance les erreurs ou les incohérences qu'il trouvera
dans la rédaction.

\tableofcontents% \clearpage

\chapter{Introduction}

Ce mémoire a pour but d'exploiter les méthodes de la théorie des groupes arithmétiques pour étudier l'approximation diophantienne dans certaines variétés, obtenues comme quotient $X=P\bcs G$ d'un groupe algébrique semi-simple $G$ défini sur $\Q$, par un sous-groupe parabolique défini sur $\Q$.
Les variétés de cette forme sont communément appelées \emph{variétés de drapeaux}.
%On pourrait prendre $G$ réductif, ou même n'importe quel groupe algébrique, mais cela ne donnerait pas plus de variétés $P\bcs G$, car le sous-groupe parabolique $P$ contient toujours le centre de $G$ et son radical unipotent.
Les exemples les plus élémentaires de telles variétés sont les espaces projectifs $\PP^n$, $n\geq 1$, les variétés grassmanniennes $\Grass(\ell,d)$, $1\leq \ell\leq d$, les quadriques projectives, et la variété des drapeaux d'un espace vectoriel.

\bigskip

Dans la suite, une variété de drapeaux $X=P\bcs G$ sera toujours munie d'une distance de Carnot-Carathéodory naturelle $d$, dont la construction est détaillée au paragraphe~\ref{sec:cc}.
Les hauteurs que nous considérerons sur $X(\Q)$ seront obtenues par plongement de $X$ dans un espace projectif.
Rappelons que si $V$ est un espace vectoriel défini sur $\Q$, le choix d'une base rationnelle $(u_i)_{1\leq i\leq d}$ de $V$ permet de définir une hauteur sur les points rationnels de $\PP(V)$: ayant fixé une norme euclidienne sur $V$ pour laquelle la base $(u_i)$ est orthonormée, si $v\in\PP(V)(\Q)$ est représenté par un vecteur primitif $\bv$ dans le réseau $\oplus_i\Z u_i$, on pose 
\[
H(v) = \norm{\bv}.
\]
Toutes les hauteurs sur $\PP(V)$ construites de cette manière sont équivalentes, i.e. comparables à une constante multiplicative près.
Ensuite, si $V_\chi$ est une représentation irréductible rationnelle de $G$ engendrée par une unique droite $[e_\chi]$ de plus haut poids $\chi$ telle que $\Stab_G[e_\chi]=P$, on obtient une hauteur $H_\chi$ sur $X$ par restriction d'une hauteur sur $\PP(V_\chi)$ en identifiant $X=P\bcs G$ à l'orbite $G[e_\chi]$ de la droite de plus haut poids dans $\PP(V_\chi)$.
L'approximation diophantienne sur $X$ est l'étude de la qualité des approximations rationnelles d'un point $x\in X(\R)$: pour un grand paramètre $T\geq 0$, on cherche à évaluer en fonction de $T$ la distance minimale à $x$ d'un point rationnel $v$ de hauteur $H_\chi(v)\leq T$.

\bigskip

\emph{Dans toute la suite $X$ désigne une variété de drapeaux, obtenue comme quotient $X=P\bcs G$ d'un groupe algébrique semi-simple $G$ défini sur $\Q$ par un sous-groupe parabolique défini sur $\Q$.
On munit $X$ de la distance de Carnot-Carathéodory usuelle, et d'une hauteur $H_\chi$ associée à une représentation rationnelle irréductible de $G$ engendrée par une unique droite rationnelle de plus haut poids~$\chi$.}

\section{Résultats principaux}

Les résultats que nous démontrons dans ce mémoire se divisent en trois catégories, suivant la manière de choisir le point $x$ dans $X(\R)$.
Les premiers résultats concernent les propriétés diophantiennes d'un point $x$ choisi aléatoirement suivant la mesure de Lebesgue sur $X(\R)$, les seconds traitent d'un point $x\in X(\QQ)$, i.e. dont les coordonnées sont des nombres algébriques, et pour les derniers, le point $x$ sera choisi aléatoirement sur une sous-variété analytique de $X(\R)$.

\subsection*{Propriétés génériques des points de $X(\R)$}

Pour commencer, on définit l'\emph{exposant diophantien} $\beta_\chi(x)$ d'un point $x\in X(\R)$ par
\[ \beta_\chi(x) =
\inf \{\beta>0\ |\ \exists c>0:\,\forall v\in X(\Q),\, d(x,v)\geq cH_\chi(v)^{-\beta}\}.
\]
Le premier résultat général que nous démontrerons est que la fonction $\beta_\chi$ est constante presque partout sur $X(\R)$.

\begin{thm}[Valeur presque sûre de l'exposant]
\label{exppsi}
%Soit $X$ une variété de drapeaux, munie de la distance de Carnot-Carathéodory usuelle et d'une hauteur $H_\chi$ associée à un poids dominant $\chi$.
Il existe une constante explicite $\beta_\chi(X)\in\Q$ strictement positive telle que pour presque tout $x\in X(\R)$,
\[
\beta_\chi(x) = \beta_\chi(X).
\]
\end{thm}

Si l'on munit $X$ de la hauteur anti-canonique, ce qui correspond à choisir $\chi$ égal à la somme des racines apparaissant dans le radical unipotent de $P$, un résultat de Franke \cite{fmt} montre que le nombre de points rationnels $v\in X(\Q)$ de hauteur au plus $T$ satisfait, pour certaines constantes $c>0$ et $b\in\N^*$,
\begin{equation}
N_{X}(T) \sim c \cdot T (\log T)^{b-1}
\hspace{1cm} (T\to+\infty)
\end{equation}
L'exposant diophantien presque sûr dans $X$ prend alors la valeur naturelle
\[
\beta_\chi(X) = \frac{1}{\dim_{cc}X},
\]
où $\dim_{cc}X$ est la dimension de Carnot-Carathéodory de $X$, unique entier tel que le nombre de recouvrement de $X(\R)$ par des boules de rayon $\delta>0$ pour la métrique de Carnot-Carathéodory satisfasse $N(X,\delta)\asymp \delta^{-\dim_{cc}X}$ lorsque $\delta$ tend vers zéro.
Mohammadi et Salehi Golsefidy \cite[Theorem~4]{msg}, ont généralisé le résultat de Franke et montré que pour toute hauteur $H_\chi$, il existe des constantes $c,u_\chi>0$ et $v_\chi\in\N^*$ pour lesquelles
\begin{equation}\label{nxt}
N_{X}(T) \sim c \cdot T^{u_\chi} (\log T)^{v_\chi-1}
\hspace{1cm} (T\to+\infty).
\end{equation}
Une application facile du lemme de Borel-Cantelli permet de montrer qu'on a toujours l'inégalité $\beta_\chi(X)\leq \frac{u_\chi}{\dim_{cc}X}$, quoique l'inégalité puisse être stricte en général.

\smallskip

Au vu du théorème~\ref{exppsi} ci-dessus, il est naturel de s'intéresser à des propriétés diophantiennes plus fines.
Par analogie avec le célèbre théorème de Khintchine \cite{khintchine}, pour une fonction $\psi:\R^+\to\R^+$, nous considérons l'inégalité
\begin{equation}\label{khini}
d(x,v) \leq H_\chi(v)^{-\beta_\chi(X)} \psi(H_\chi(v))
\end{equation}
et montrons le théorème suivant.

\begin{thm}[Théorème de Khintchine pour une variété de drapeaux]
\label{khintchinei}
%Soit $X$ une variété de drapeaux, munie de la distance de Carnot-Carathéodory usuelle et d'une hauteur $H_\chi$ associée à un poids dominant $\chi$.
Il existe des constantes explicites $a_\chi>0$ et $b_\chi\in\N^*$ pour lesquelles l'énoncé suivant est vérifié.
Soit $\psi:\R^+\to\R^+$ une fonction décroissante.
\begin{itemize}
\item Si $\int_e^{+\infty} \psi(u)^{\frac{a_\chi}{\beta_\chi(X)}} (\log\log u)^{b_\chi-1} \frac{\dd u}{u}<+\infty$, alors, pour presque tout $x\in X(\R)$, l'inégalité \eqref{khini} n'admet qu'un nombre fini de solutions $v\in X(\Q)$.
\item Si $\int_e^{+\infty} \psi(u)^{\frac{a_\chi}{\beta_\chi(X)}} (\log\log u)^{b_\chi-1} \frac{\dd u}{u}=+\infty$, alors, pour presque tout $x\in X(\R)$, l'inégalité \eqref{khini} admet une infinité de solutions $v\in X(\Q)$.
\end{itemize}
\end{thm}

Dans le cas où $X=\PP^n$ est un espace projectif, on retrouve le théorème démontré par Khintchine \cite{khintchine} en 1926.
Le résultat était déjà connu aussi lorsque $X$ est une quadrique projective, depuis les travaux remarquables de Kleinbock et Merrill \cite{kleinbockmerrill} puis de Fishman, Kleinbock, Merrill et Simmons \cite{fkms} sur le sujet.
Les constantes $a_\chi$ et $b_\chi$ se calculent facilement à l'aide d'un système de racines associé à $G$, et proviennent d'un encadrement asymptotique de la mesure de certains voisinages de l'infini dans l'espace $G/\Gamma$, où $\Gamma$ est un sous-groupe arithmétique de $G$.
On renvoie à la proposition~\ref{equi} ci-dessous pour un énoncé plus précis.
Dans le cas particulier où $X$ est munie de la hauteur anti-canonique, $a_\chi=1$ et $b_\chi$ est égal au rang rationnel de la variété $X=P\bcs G$, i.e. $b_\chi=\rang G-\rang P$.

\subsection*{Approximation des points algébriques}

Nous noterons $\QQ$ le sous-corps de $\R$ constitué des éléments algébriques sur $\Q$.
Un résultat majeur concernant l'approximation des éléments de $\QQ$ par des rationnels est le théorème de Roth \cite{roth}, qui assure que si $x\in\QQ$ est irrationnel et $\eps>0$, alors l'inégalité
\[
\Abs{x-\frac{p}{q}} \leq \frac{1}{q^{2+\eps}}, \quad p\in\Z,\ q\in\N^*
\]
n'admet qu'un nombre fini de solutions $(p,q)$.
De manière équivalente, si $\PP^1$ est muni de la distance et de la hauteur usuelles, le théorème de Roth affirme que tout élément $x\in\PP^1(\QQ)\setminus\PP^1(\Q)$ vérifie $\beta(x)=2$.
Ce théorème a d'ailleurs été généralisé par Schmidt \cite{schmidt_simultaneous}: si $x\in\PP^n(\QQ)$ n'est inclus dans aucun sous-espace projectif rationnel propre, alors $\beta(x)=1+\frac{1}{n}$.
Nous montrons que ces résultats sont encore valables dans une variété de drapeaux $X$ arbitraire: hors de certaines contraintes rationnelles, tous les points algébriques de $X$ ont même exposant diophantien, égal à l'exposant presque sûr d'un point aléatoire de $X(\R)$.

Rappelons que la variété $X=P\bcs G$ se décompose en cellules de Schubert: si $B\subset G$ désigne un $\Q$-sous-groupe parabolique minimal, $W$ le groupe de Weyl associé, et $W_P=(W\cap P)\bcs W$, alors
\[
X = \bigsqcup_{w\in W_P} P w B.
\]
Notons $X_w=\overline{PwB}$ l'adhérence de la cellule de Schubert $PwB$.
Une \emph{variété de Schubert} dans $X$ est une variété de la forme $X_wg$, pour $w\in W_P$ et $g\in G$.
Si l'élément $g$ peut être choisi dans $G(\Q)$, nous dirons que la variété de Schubert est \emph{rationnelle}.

\begin{thm}[Points algébriques extrémaux]
\label{extalgi}
%Soit $X$ une variété de drapeaux, munie de la distance de Carnot-Carathéodory usuelle et d'une hauteur $H_\chi$ associée à un poids dominant $\chi$.
Si $x\in X(\QQ)$ n'appartient à aucune sous-variété de Schubert rationnelle%
\footnote{Il suffit en fait que $x$ ne soit dans aucune sous-variété de Schubert rationnelle \emph{instable}. On renvoie au paragraphe~\ref{sec:extalg} pour plus de détails sur ce sujet.}%
, alors $\beta_\chi(x)=\beta_\chi(X)$.
\end{thm}

Plus généralement, nous obtenons une formule pour l'exposant d'un point $x\in X(\QQ)$, en termes d'une certaine variété de Schubert rationnelle le contenant.
Mais l'énoncé précis de cette formule requiert l'introduction d'autres objets associés à $X$.
Nous noterons $T$ un $\Q$-tore déployé maximal de $G$ inclus dans $B$, $A=T^0(\R)$ la composante connexe des points réels de $T$, $\ka$ l'algèbre de Lie de $A$ et $\Pi\subset\ka^*$ la base du système de racines associé à $G$ et $T$ correspondant au parabolique minimal $B$.
Ces objets sont définis plus précisément dans Borel \cite[\S11]{borel_iga} et leur construction détaillée est donnée dans Borel et Tits \cite[\S5]{boreltits}.
Insistons sur le fait que l'on s'intéresse ici aux racines de $G$ par rapport au $\Q$-tore \emph{déployé} maximal $S$; on parle parfois du système des $\Q$-racines de $G$, dont la base $\Pi$ est de cardinal égal au $\Q$-rang de $G$.

Si $\theta\subset\Pi$ est l'ensemble de racines simples associé à $P$ -- tel que les racines négatives apparaissant dans $P$ se décomposent sur $\theta$ -- on définit un unique élément $Y$ dans $\ka$ par
\[
\alpha(Y) = \left\{
\begin{array}{ll}
0 & \mbox{si}\ \alpha\in\theta\\
-1 & \mbox{si}\ \alpha\in\Pi\setminus\theta
\end{array}
\right.
\]
L'espace $\ka$ étant muni d'une norme euclidienne invariante par l'action du groupe de Weyl, nous noterons $p_{\ka^-}:\ka\to\ka^-$ la projection au plus proche voisin sur la chambre de Weyl négative
\[
\ka^-=\{u\in\ka\ |\ \forall\alpha\in\Pi,\, \alpha(u)\leq 0\}.
\]
Enfin, l'action adjointe à droite du groupe de Weyl $W$ sur $\ka$ et $\ka^*$ sera notée en exposant: pour $w\in W$, $Y^w=(\Ad w)^{-1}Y$, et $\chi^w=\chi\circ(\Ad w)$.

\begin{thm}[Exposant d'un point algébrique]
\label{expalgi}
%Soit $X$ une variété de drapeaux, munie de la distance de Carnot-Carathéodory usuelle et d'une hauteur $H_\chi$ associée à un poids dominant $\chi$.
Pour chaque élément $x\in X(\QQ)$, il existe $w\in W_P$ et $\gamma\in G(\Q)$ tels que $x\in X_w\gamma$ et
\[
\beta_\chi(x) = \frac{1}{-\bracket{\chi,Y} + \bracket{\chi^w,p_{\ka^-}(Y^w)}}.
\]
\end{thm}

Dans le théorème ci-dessus, la variété de Schubert $X_w\gamma$ est en fait entièrement déterminée par $x$.
On renvoie le lecteur au paragraphe~\ref{sec:expalg} pour le détail de sa construction.
%%%%%%%%%%%%%%%%%%%%%%%%%%%%%%%
%VERSION PLUS PRÉCISE DU THÉORÈME
%%%%%%%%%%%%%%%%%%%%%%%%%%%%%%%
%Nous définissons aussi une relation d'ordre partiel $\prec$ sur $\ka$ donnée par
%\[
%Y_1\prec Y_2 \quad\Longleftrightarrow\quad \forall i\in[1,r],\ \omega_i(Y_1)\leq\omega_i(Y_2).
%\]
%Par exemple, si $(Y_j)_{j\in J}$ est une famille d'éléments de $\ka$, nous noterons $Y=\inf_{j\in J} Y_i$ l'élément défini par
%\[
%\forall i\in[1,r],\ \omega_i(Y) = \inf_{j\in J} \omega_j(Y).
%\]
%\begin{thm}[Exposant d'un point algébrique]
%\label{expalgi}
%Soit $X$ une variété de drapeaux, munie de la distance de Carnot-Carathéodory usuelle et d'une hauteur $H_\chi$ associée à un poids dominant $\chi$.
%Soit $x\in X(\QQ)$.
% et
%\[
%c_\infty = \inf\{ p_{\ka^-}(Y^w)\ |\ w\in W_P:\,\exists\gamma\in G(\Q):\, x\in X_w\gamma\}.
%\]
%Il existe une unique variété de Schubert rationnelle $PwB\gamma\ni x$ telle que
%\[
%p_{\ka^-}(Y^w) = c_\infty.
%\]
%L'exposant diophantien de $x$ est donné par la formule
%\[
%\beta_\chi(x) = \frac{1}{-\bracket{\chi,Y} + \bracket{\chi^w,p_{\ka^-}(Y^w)}}.
%\]
%\end{thm}
Comme corollaire remarquable du théorème ci-dessus, on peut montrer que, dans certains cas, la valeur minimale de l'exposant $\beta_\chi(x)$, $x\in X(\QQ)$, est égale à $\beta_\chi(X)$.

\begin{cor}[Minoration de l'exposant d'un point algébrique]
Si $X=P\bcs G$, avec $P$ un sous-groupe parabolique \emph{maximal} du $\Q$-groupe semi-simple $G$, alors 
\[
\min_{x\in X(\QQ)} \beta_\chi(x) = \beta_\chi(X).
\]
\end{cor}

Ce corollaire s'applique en particulier lorsque $X=\Grass(\ell,d)$ est la variété grassmannienne des sous-espaces de dimension $\ell$ dans $\R^d$, munie de la hauteur induite par le plongement de Plücker.
On obtient ainsi une réponse partielle à un problème de Schmidt \cite{schmidt_grass}:

\bigskip

\noindent\textbf{Problème ouvert:} Déterminer la quantité $\inf_{x\in \Grass(\ell,d)}\beta_\chi(x)$.

\bigskip

\noindent Plus de détails à ce sujet sont donnés au chapitre~\ref{chap:exemples}, où nous décrivons quelques exemples de variétés de drapeaux.
Nous verrons aussi dans ce chapitre que le corollaire ci-dessus est faux en général si le sous-groupe parabolique $P$ n'est pas maximal.
Il reste cependant valable dans un autre cas important: si $X=P\bcs G$ avec $G$ déployé sur $\Q$, $P$ un $\Q$-parabolique \emph{minimal}, et $H_\chi$ la hauteur anti-canonique sur $X$.

\subsection*{Approximation dans les sous-variétés}

Nous voulons maintenant décrire les propriétés diophantiennes presque sûres d'un point $x$ choisi aléatoirement sur une sous-variété analytique $M\subset X(\R)$.
Si $M$ est une variété analytique de dimension $m$, on considère une mesure $\lambda_M$ sur $M$ équivalente à la restriction à $M$ de la mesure de Hausdorff de dimension $m$.
Comme nos résultats ne dépendent de la mesure qu'à équivalence près, le choix précis de $\lambda_M$ n'aura pas d'importance dans la suite.

\bigskip

Le premier résultat que nous établissons est que l'exposant diophantien d'un point de $M$ est constant presque sûrement.
Ce fait remarquable généralise un résultat de Kleinbock \cite{kleinbock_dichotomy}, valable lorsque $X$ est un espace projectif.

\begin{thm}[Exposant d'une sous-variété analytique]
Soit $M$ une sous-variété analytique connexe de $X$.
Il existe une constante $\beta_\chi(M)$ telle que pour presque tout $x\in M$,
\(
\beta_\chi(x)=\beta_\chi(M).
\)
\end{thm}

Suivant la terminologie introduite par Sprindzuk \cite{sprindzuk}, nous dirons qu'une sous-variété analytique $M\subset X$ est \emph{extrémale} si $\beta_\chi(M)=\beta_\chi(X)$.
Dans le cas où $X=\PP^n$ est un espace projectif, une sous-variété analytique est dite \emph{non dégénérée} si elle n'est incluse dans aucun sous-espace projectif strict.
La conjecture de Sprindzuk, démontrée par Kleinbock et Margulis \cite{kleinbockmargulis} en 1998, stipulait que toute sous-variété analytique non dégénérée dans $\PP^n$ était extrémale.
Ici encore, ce résultat admet une généralisation naturelle au cadre des variétés de drapeaux.
Nous dirons qu'une sous-variété analytique $M\subset X$ est \emph{non dégénérée} si elle n'est incluse dans aucune sous-variété de Schubert propre.
Cette terminologie est bien sûr compatible avec celle déjà utilisée dans le cas ou $X$ est un espace projectif.

\begin{thm}[Extrémalité des sous-variétés non dégénérées]
Toute sous-variété analytique connexe $M$ non dégénérée%
\footnote{Ici encore il suffit de supposer que $M$ n'est incluse dans aucune sous-variété de Schubert instable, comme expliqué au paragraphe~\ref{sec:extan}.}
dans $X$ est extrémale.
\end{thm}

Notre dernier résultat mêle sous-variétés analytiques et nombres algébriques.
En prenant pour variété un singleton $M=\{x\}$, avec $x\in X(\QQ)$, on retrouvera d'ailleurs le théorème~\ref{expalgi} énoncé ci-dessus.
Nous reprenons les notations introduites juste avant le théorème~\ref{expalgi}.

\note{On pourrait améliorer un peu l'hypothèse en commençant par définir, grâce au lemme de sous-modularité peut-être, la variété de Schubert réelle $\cS(M)$ associée à $M$.
Ensuite, il suffit de supposer que $\cS(M)$ est définie sur $\QQ$.}
\note{De même, pour le théorème~\ref{expalgi}, il faudrait commencer par définir la variété de Schubert $\cS_{\Q}(x)$ associée à $x$.}

\begin{thm}[Exposant d'une sous-variété définie sur $\QQ$]
\label{expanalgi}
Soit $M$ une sous-variété analytique connexe de $X$ dont l'adhérence de Zariski est définie sur $\QQ$.
Il existe $w\in W_P$ et $\gamma\in G(\Q)$ tels que $M\subset X_w\gamma$ et pour presque tout $x\in M$,
\[
\beta_\chi(x) = \frac{1}{-\bracket{\chi,Y} + \bracket{\chi^w,p_{\ka^-}(Y^w)}}.
\]
\end{thm}

On observe que les énoncés concernant l'approximation diophantienne sur les sous-variétés sont très voisins de ceux concernant les points algébriques.
Cela s'explique par le fait que les contraintes géométriques qui apparaissent sont les mêmes dans les deux situations.
Pour les mettre en évidence, nous utilisons une correspondance entre les propriétés diophantiennes d'un point $x\in X(\R)$ et le comportement asymptotique de certaines orbites diagonales dans l'espace des réseaux d'un espace euclidien.

\section{Une correspondance}

Un \emph{réseau} $\Delta$ dans un espace euclidien $V$ est un sous-groupe discret de rang maximal.
En d'autres termes, pour une base $(v_i)_{1\leq i\leq d}$ de $V$, on peut écrire
\[
\Delta = \Z v_1 \oplus \dots \oplus \Z v_d.
\]
Le groupe linéaire $\GL(V)$ agit naturellement sur l'espace des réseaux de $V$.
Les travaux de Minkowski \cite{minkowski}
% Mahler \cite{mahler}
et Siegel \cite{siegel_gn} ont mis en évidence le fait que de nombreuses propriétés arithmétiques se comprennent simplement grâce à l'espace des réseaux.
Cette approche de l'arithmétique, souvent appelée \enquote{géométrie des nombres}, s'est aussi révélée particulièrement fructueuse pour l'approximation diophantienne.
Elle est en particulier à la base de la démonstration de Schmidt de son théorème du sous-espace \cite[Theorem~1F]{schmidt_da}.
Mais commençons par un exemple simple, qui illustre bien l'intérêt de l'espace des réseaux pour l'étude de l'approximation diophantienne.

\subsection*{L'espace projectif}

Nous avons défini plus haut l'exposant diophantien d'un point $x\in\PP^n(\R)$:
\[
\beta(x)=\inf\{\beta>0\ |\ \exists c>0:\,\forall v\in\PP^n(\Q),\, d(x,v)\geq cH(v)^{-\beta}\}.
\]
Si $x\in\PP^n(\R)$ s'écrit en coordonnées homogènes $x=[1:x_1:\dots:x_n]$, nous lui associons un réseau $\Delta_x$ dans $\R^d$ par la formule
\[
\Delta_x = s_x\Z^d,
\quad\mbox{où}\ 
s_x =
\begin{pmatrix}
1    &  0  &  \dots  &  0\\
-x_1 & 1   & \dots   &  0\\
\vdots & 0 & \ddots & 0 \\
-x_n & 0 & \dots & 1
\end{pmatrix}.
\]
La \emph{systole}, ou \emph{premier minimum}, d'un réseau $\Delta$ est la longueur minimale d'un vecteur non nul dans $\Delta$:
\[
\lambda_1(\Delta) = \inf\left\{ \lambda>0\ |\ \Delta\cap B(0,\lambda)\neq\{0\}\right\}.
\]
Nous considérons alors dans $\GL_d(\R)$ le sous-groupe
\[
a_t = \diag(e^{-t}, e^{t/n},\dots,e^{t/n}),
\]
et posons
\[
\gamma(x) = \limsup_{t\to+\infty} \frac{-1}{t} \log\lambda_1(a_t\Delta_x).
\]
La proposition élémentaire ci-dessous est tirée de Kleinbock et Margulis \cite[Theorem~8.5]{km_loglaws}, et généralise des résultats antérieurs de Dani \cite{dani_correspondence}.

\begin{prop}[Correspondance de Dani pour $\PP^n$]
Pour tout $x\in\PP^n(\R)$,
\[
\beta(x) = (1+\frac{1}{n})\frac{1}{1-\gamma(x)}.
\]
\end{prop}

Le point de départ de tous les résultats démontrés dans ce mémoire est une généralisation de cette proposition, dans laquelle l'espace projectif est remplacé par une variété de drapeaux $X$ arbitraire.

\subsection*{Variétés de drapeaux}

Rappelons que $X$ peut s'écrire comme l'espace quotient $X=P\bcs G$ d'un $\Q$-groupe semi-simple $G$ par un $\Q$-sous-groupe parabolique $P$, et que $X(\R)$ est munie de la distance de Carnot-Carathéodory naturelle définie au paragraphe~\ref{sec:cc}.
Nous fixons aussi une représentation rationnelle irréductible de $G$ sur un espace vectoriel $V_\chi$ engendré, comme $G$-module, par une unique droite rationnelle de plus haut poids $e_\chi$.
Nous supposons en outre que le stabilisateur dans $G$ de la droite engendrée par $e_\chi$ est égal à $P$.
Cela permet de plonger $X$ dans l'espace projectif $\PP(V_\chi)$; la hauteur sur $X(\Q)$ obtenue par restriction de la hauteur usuelle sur $\PP(V_\chi)$ est notée $H_\chi$.

\bigskip

Pour étudier l'approximation diophantienne dans $X$, munie de la hauteur $H_\chi$, nous utiliserons l'espace des réseaux de $V_\chi$.
Ayant fixé un réseau rationnel $V_\chi(\Z)$ dans $V_\chi$, %i.e. un sous-groupe abélien libre de rang maximal dans $V_\chi(\Q)$.
nous associons à chaque élément $x$ de $X(\R)$ un réseau
\[
\Delta_x=s_xV_\chi(\Z), \quad\mbox{où}\ s_x\in G\ \mbox{est tel que}\ x=Ps_x.
\]
Nous utiliserons le sous-groupe à un paramètre $(a_t)$ déjà mentionné ci-dessus:
\[
a_t = e^{tY},\quad\mbox{où}\ Y\in\ka\ \mbox{est défini par}\ 
\alpha(Y) = \left\{
\begin{array}{ll}
0 & \mbox{si}\ \alpha\in\theta\\
-1 & \mbox{si}\ \alpha\not\in\theta.
\end{array}
\right.
\]
La correspondance qui relie l'exposant diophantien $\beta_\chi(x)$ au comportement asymptotique de l'orbite $(a_t\Delta_x)_{t>0}$ met en jeu une quantité un peu plus subtile à définir que la systole d'un réseau.
Notons
\[
\tx = G\cdot e_\chi \subset V_\chi
\]
l'orbite de $e_\chi$ dans $V_\chi$ sous l'action de $G$.
L'ensemble $\tx$ est le cône dans $V_\chi$ contenant toutes les droites du plongement de $X$ dans $\PP(V_\chi)$, privé du point $0$.
L'espace $V_\chi$ se décompose en somme directe d'espaces de poids sous l'action du groupe abélien $(a_t)$, et nous notons
\[
\pi^+: V_\chi \to \R e_\chi
\]
la projection sur $e_\chi$ parallèlement à la somme de tous les autres espaces de poids.
Pour tout réseau $\Delta$ dans $V_\chi$, nous posons
\[
r_\chi(\Delta) = \inf\left\{ r>0\ |\ \exists v\in\Delta\cap B(0,r):\,
%\left\{
\begin{array}{c}
v \in \tx\\
\norm{\pi^+(v)} \geq \frac{\norm{v}}{2}
\end{array}
%\right.
\right\}.
\]
Remarquons qu'on peut avoir $r_\chi(\Delta)=+\infty$, par exemple si $\Delta$ ne rencontre pas le cône $\tx$.
Quoiqu'il en soit, on pose
\[
\gamma_\chi(x) = \limsup_{t\to+\infty}\frac{-1}{t}\log r_\chi(a_t\Delta_x).
\]
Il n'est pas difficile de se convaincre que cette quantité ne dépend pas du choix de l'élément $s_x$ tel que $x=Ps_x$; cela est expliqué en détail au paragraphe~\ref{sec:ed}.
On peut alors relier les quantités $\beta_\chi(x)$ et $\gamma_\chi(x)$ grâce à la proposition suivante.

\begin{prop}[Correspondance drapeau-réseau]
\label{dani}
Pour tout $x\in X(\R)$,
\[
\beta_{\chi}(x) = \frac{1}{-\chi(Y)-\gamma_{\chi}(x)}.
\]
\end{prop}

La définition de la fonction $r_\chi$, et notamment la contrainte $\norm{\pi^+(v)}\geq\frac{\norm{v}}{2}$ sur le petit vecteur $v$, rend cette proposition un peu moins simple à appliquer que son analogue pour l'espace projectif.
Elle nous sera néanmoins très utile.

\section{Espaces de réseaux}

Dans la correspondance décrite ci-dessus, les réseaux de $V_\chi$ qui interviennent sont de la forme $gV_\chi(\Z)$, avec $g\in G$.
Si on note $\Gamma$ le stabilisateur de $V_\chi(\Z)$ dans $G$, l'orbite $G\cdot V_\chi(\Z)$ s'identifie au quotient
\[
\Omega = G/\Gamma.
\]
Un tel quotient d'un groupe algébrique $G$ semi-simple défini sur $\Q$ par un sous-groupe arithmétique $\Gamma$ sera appelé un \emph{espace de réseaux}, car il s'identifie à une partie de l'espace des réseaux de l'espace euclidien $V_\chi$.
La géométrie de ces espaces est décrite avec assez de précision par la théorie de la réduction, dûe à Borel et Harish-Chandra, et très clairement exposée dans le livre de Borel \cite{borel_iga}.
Il y est notamment démontré que l'espace $\Omega$ est de volume fini pour la mesure de Haar, unique mesure de Radon invariante sous l'action de $G$, à constante multiplicative près.
Les résultats et les méthodes de \cite{borel_iga} seront d'ailleurs un outil essentiel dans tout ce mémoire.
À chacun des résultats d'approximation diophantienne dans la variété de drapeaux $X$ correspond un théorème sur l'espace de réseaux $\Omega$.
Nous décrivons maintenant ces résultats, qui, à notre avis, ont leur intérêt propre.

\subsection*{Mélange exponentiel}

Pour démontrer un théorème de Khintchine sur la variété de drapeaux $X$, nous suivons la méthode proposée par Kleinbock et Margulis \cite{km_loglaws} et utilisée aussi par Kleinbock et Merrill \cite{kleinbockmerrill} et Fishman, Kleinbock, Merrill et Simmons \cite{fkms} dans leur étude des quadriques.
Cette méthode est basée sur un théorème de décroissance exponentielle des coefficients dans la représentation $L^2(\Omega)$.
Nous admettrons ce résultat, dont la démonstration nous éloignerait trop de notre sujet d'étude, mais en donnons l'énoncé exact ici, puisqu'il sera essentiel dans la suite.
L'énoncé général ci-dessous découle d'un article récent de Einsiedler, Margulis, Mohammadi et Venkatesh \cite[\S\S 4.1 et 4.3]{emmv};
\note{Il faudrait en fait une version de l'inégalité (4.1) dans \cite{emmv} valable pour des fonctions $C^\infty$, et non seulement des vecteurs $K$-finis.
Avec les propriétés de décroissance de la fonction de Harish-Chandra (référence?), cela donnerait bien le résultat donné ici.
À compléter...}
on renvoie aussi à Kleinbock et Margulis \cite[\S3.4]{km_loglaws} pour une version antérieure.

\begin{thm}[Décroissance exponentielle des coefficients dans $L^2(\Omega)$]
\label{decay}
Soit $G$ un $\Q$-groupe semi-simple simplement connexe, $\Gamma$ un sous-groupe arithmétique, et $\Omega=G/\Gamma$.
Soit $T$ un $\Q$-tore déployé maximal de $G$, $A=T^0(\R)$ et $(a_t)$ un sous-groupe à un paramètre de $A$.
On suppose que
$\forall t>0,\ a_t=e^{tY}$, avec $Y\in\ka$ tel que pour toute projection $p_i:\g\to \g_i$ sur un facteur $\Q$-simple, $p_i(Y)\neq 0$.
Alors, il existe des constantes $\ell\in\N$ et $C,\tau>0$ telles que pour toutes fonctions $f_1,f_2\in C^\infty(\Omega)$ et tout $t>0$,
\[
\left\lvert\int_\Omega f_1(a_t\Delta)f_2(\Delta)\,m_\Omega(\dd\Delta) - m_\Omega(f_1)m_\Omega(f_2)\right\rvert
\leq C  e^{-\tau t} \norm{\Upsilon^\ell f_1}_2 \norm{\Upsilon^\ell f_2}_2,
\]
où $\Upsilon$ désigne l'opérateur différentiel $\Upsilon=1-\sum_i Y_i^2$, avec $(Y_i)$ une base orthonormée de l'algèbre de Lie $\mathfrak{k}$ d'un sous-groupe compact maximal de $G(\R)$.
\end{thm}

L'autre élément utile à la démonstration du théorème de Khintchine est un encadrement asymptotique du volume de l'ensemble $\Delta$ des éléments de $\Omega$ tels que $r_\chi(\Delta)\leq r$.
Nous vérifierons cet encadrement à l'aide de la théorie de la réduction, qui décrit un domaine fondamental pour l'action de $\Gamma$ par translation à droite sur~$G$. 

\note{L'ensemble considéré dans cette proposition n'est pas à proprement parler un voisinage de l'infini.}
\begin{prop}[Mesure d'un voisinage de l'infini]
\label{equi}
Il existe des constantes $a_\chi>0$ et $b_\chi\in\N^*$ telles que, à une constante multiplicative près dépendant de $G$ et $\Gamma$, pour tout $r<1$,
\[
m_\Omega(\{\Delta\in\Omega\ |\ r_\chi(\Delta)\leq r\}) \asymp r^{a_\chi} \abs{\log r}^{b_\chi -1}.
\]
\end{prop}

\subsection*{Position dans un espace de réseaux}

Si $\Omega=G/\Gamma$ est le quotient d'un $\Q$-groupe algébrique semi-simple par un sous-groupe arithmétique, nous définissons maintenant une fonction $c:\Omega\to\ka^-$ qui décrit la position d'un point dans $\Omega$, à une constante près.
C'est l'étude du comportement de cette fonction le long des orbites diagonales dans $\Omega$ qui permettra de démontrer les résultats d'approximation diophantienne présentés ci-dessus.
Dans la construction de la fonction $c$, l'idée sous-jacente est qu'un élément $g\Gamma$ dans $\Omega$ doit s'étudier simultanément dans chacune des représentations fondamentales de $G$.
Cela nous est indiqué par les méthodes de la théorie de la réduction, et en particulier \cite[\S\S 14 et 16]{borel_iga}.

\bigskip

Notons $T$ un $\Q$-tore déployé maximal dans $G$, $A=T^0(\R)$ la composante connexe des points réels de $T$, et $\ka$ l'algèbre de Lie de $A$.
Le système de racines $\Sigma$ de $G$ par rapport à $T$ s'identifie à un système de racines dans l'espace dual $\ka^*$.
Nous fixons une base de racines simples $\Pi=\{\alpha_1,\dots,\alpha_r\}$ dans $\Sigma$ et notons $\varpi_1,\dots,\varpi_r$ les poids fondamentaux associés.
Pour chaque $i$, on fixe une $\Q$-représentation $V_i$ de $G$ engendrée par une unique droite rationnelle de plus haut poids $\omega_i=b_i\varpi_i$, avec $b_i\in\N^*$ minimal, et $V_i(\Z)$ un réseau rationnel dans $V_i$ stable par l'action de $\Gamma$ et contenant un vecteur de plus haut poids $e_i$.
Suivant Borel et Tits \cite[\S12.13]{boreltits}, nous dirons que les représentations $V_i$, $i=1,\dots,r$, sont les représentations \emph{fondamentales} de $G$.
Dans l'espace vectoriel $V_i$, nous noterons $\tx_i=G\cdot e_i$ l'orbite du vecteur de plus haut poids $e_i$ sous l'action de $G$.

Les \emph{covolumes successifs} d'un élément $\bar{g}=g\Gamma$ de l'espace de réseaux $\Omega$ sont les quantités
\[
\mu_i(g) = \min\{\norm{gv}\ |\ v\in V_i(\Z)\cap\tx_i\}, \quad i=1,\dots,r.
\]
Comme les poids fondamentaux forment une base de $\ka^*$, il existe un unique élément $c_0(g)\in\ka$ tel que
\[
\forall i\in[1,r],\ \omega_i(c_0(g)) = \log\mu_i(g).
\]
Rappelons que la chambre de Weyl négative $\ka^-$ est la partie convexe fermée de $\ka$ définie par
\[
\ka^- = \{Y\in\ka\ |\ \forall \alpha\in\Pi,\, \alpha(Y)\leq 0\}.
\]
L'espace $\ka$ est muni d'une norme euclidienne invariante par l'action du groupe de Weyl, et cela permet de définir la projection $p_{\ka^-}:\ka\to\ka^-$, qui associe à $Y_0$ l'unique élément $Y\in\ka^-$ tel que $d(Y_0,\ka^-)=d(Y_0,Y)$.
\note{S'il n'y a qu'une seule pointe, et si $\Omega$ est muni d'une distance riemannienne induite par une métrique riemannienne invariante à droite sur $G$, un corollaire de la théorie de la réduction est que la distance entre deux éléments $\bar{g_1}$ et $\bar{g_2}$ dans $\Omega$ est essentiellement majorée par la norme de la différence $c(g_1)-c(g_2)$:
\[
d(\bar{g_1},\bar{g_2}) \ll \norm{c(g_1)-c(g_2)} + 1.
\]}
La fonction $c:\Omega\to\ka^-$ est alors définie par la formule
\[
c(g) = p_{\ka^-}(c_0(g)).
\]
C'est une fonction propre sur $\Omega$ que l'on peut comprendre de la façon suivante: si $g_1$ et $g_2$ sont dans la même pointe de $\Omega$, la distance de $g_1$ à $g_2$ est comparable à $\norm{c(g_1)-c(g_2)}$, à une constante additive près.

\subsection*{Le théorème du sous-espace}

Avec les notations utilisées ci-dessus, nous considérons maintenant un sous-groupe à un paramètre $(a_t)$ dans $A$.
Les démonstrations des théorèmes~\ref{extalgi} et \ref{expalgi} sont basées sur le théorème fort du sous-espace de Schmidt \cite[Theorem~3A, page 163]{schmidt_da}, grâce auquel nous déterminons un équivalent asymptotique, lorsque $t$ tend vers l'infini, de la fonction $c(a_ts)$, si $s$ appartient à $G(\QQ)$.
Pour énoncer ce résultat, il est commode d'introduire la relation d'ordre partiel sur $\ka$ donnée par
\[
Y_1\prec Y_2 \quad\Longleftrightarrow\quad \forall i\in[1,r],\ \omega_i(Y_1)\leq\omega_i(Y_2).
\]
Par exemple, si $(Y_j)_{j\in J}$ est une famille d'éléments de $\ka$, nous noterons $Y=\inf_{j\in J} Y_i$ le plus grand minorant de la famille $(Y_j)_{j\in J}$, défini par
\[
\forall i\in[1,r],\ \omega_i(Y) = \inf_{j\in J} \omega_i(Y_j).
\]
%Nous reprenons les notations introduites avant le théorème~\ref{extalgi}.
Rappelons que pour $Y\in\ka$ et $w$ dans le groupe de Weyl, on note $Y^w=(\Ad w)^{-1}Y$.

\begin{thm}[Orbites diagonales des points algébriques dans $\Omega$]
Soient $Y\in\ka^-$ un élément arbitraire, $(a_t)=(e^{tY})$ le sous-groupe à un paramètre associé, et $P$ le sous-groupe parabolique associé à $(a_t)$:
\[
P=\{g\in G\ |\ \lim_{t\to+\infty}a_tga_{-t}\ \mbox{existe}\}.
\]
Pour $w\in W_P$, nous notons $X_w=\overline{PwB}$ la variété de Bruhat standard correspondante.
Étant donné $s\in G(\QQ)$ posons
\[
c_\infty = \inf\{ p_{\ka^-}(Y^w)\ ;\ w\in W_P:\,\exists\gamma\in G(\Q):\, s \in X_w\gamma\} \in \ka^-.
\]
Alors,
\[
\lim_{t\to+\infty} \frac{1}{t}c(a_ts) = c_\infty
\]
et il existe une variété de Bruhat\footnote{Une variété de la forme $X_w\gamma$ sera dite variété de Bruhat si elle est vue comme une partie de $G$, et variété de Schubert si elle est vue comme une partie de $X=P\bcs G$.}
rationnelle $X_w\gamma$, $w\in W_P$, $\gamma\in G(\Q)$ contenant $x$ et telle que $p_{\ka^-}(Y^w)=c_\infty$.
\end{thm}

\comm{
La variété de Bruhat $X_w\gamma=\overline{PwB\gamma}$ n'est pas unique, car elle dépend du choix de l'élément $\gamma_\infty$ introduit au théorème~\ref{diagalg}.
Mais si l'on introduit le sous-groupe parabolique $Q_\infty$ associé à $c_\infty$, alors la variété $\overline{PwQ_\infty\gamma}$ contenant $x$ et vérifiant $p_{\ka^-}(Y^w)=c_\infty$ est uniquement déterminée.
}

En fait, la démonstration de ce théorème permet aussi de contrôler partiellement l'élément $\gamma_t\in G(\Q)$ qui apparaît pour $t>0$ grand dans une décomposition de Siegel $a_ts=k_tb_tn_t\gamma_t$.
Nous renvoyons au chapitre~\ref{chap:algebrique}, théorème~\ref{diagalg} pour un énoncé précis.

\subsection*{Non divergence quantitative}

Le dernier résultat dont nous aurons besoin sur l'espace de réseaux $\Omega$ est un énoncé de non divergence quantitative: étant donnée une mesure borélienne $\mu$ sur $G$ on cherche à contrôler la valeur de $c(g)$ lorsque l'élément $g$ est choisi aléatoirement suivant la mesure $\mu$.
C'est pour étudier le comportement des flots unipotents dans l'espace $\SL_d(\R)/\SL_d(\Z)$ que ce type de résultat a été introduit par Margulis \cite{margulis}.
Ensuite, les travaux de Dani \cite{dani_nd2}, puis de Kleinbock et Margulis \cite{kleinbockmargulis} ont permis d'aboutir à un énoncé quantitatif adapté à l'approximation diophantienne.
Mais les théorèmes montrés par Kleinbock et Margulis, ou plus récemment par Lindenstrauss, Margulis, Mohammadi et Shah \cite{lmms} ne concernent que le groupe $\SL_d$.
Pour les applications que nous avons en vue, il est important d'avoir un résultat qui s'applique dans un $\Q$-groupe semi-simple arbitraire, et qui en respecte la géométrie.
Ici encore, c'est la fonction $c:G\to\ka^-$ qui permet de formuler l'énoncé adéquat.

\bigskip

Nous aurons même besoin d'étendre la fonction $c$ aux parties compactes de $G$.
Si $S\subset G$ est une partie compacte, on pose donc, pour chaque $i\in[1,r]$,
\[
\mu_i(S) = \min_{v\in V_i(\Z)\cap\tx_i} \max_{s\in S} \norm{sv}.
\]
Cela permet de définir un élément $c_0(S)\in\ka$ par
\[
\forall i\in[1,r],\ \omega_i(c_0(S)) = \log\mu_i(S),
\]
puis, comme précédemment,
\[
c(S) = p_{\ka^-}(c_0(S)).
\]
Dans toute la suite, le groupe $G$ est muni d'une métrique riemannienne arbitraire.
Si $\mu$ est une mesure borélienne sur $G$, $U$ un ouvert de $G$, et $C,\alpha>0$ deux constantes, nous dirons qu'une fonction mesurable $f:U\to\R$ est \emph{$(C,\alpha)$}-régulière sur $U$ si pour toute boule $B=B(x,r)\cap U$ de $U$ centrée en $x\in\Supp\mu$, et pour tout $\eps>0$,
\[
\mu(\{g\in B\ |\ \abs{f(g)} \leq \eps \norm{f}_{B,\mu}\}) \leq C \eps^\alpha \mu(B),
\]
où l'on note
\[
\norm{f}_{B,\mu}=\sup\{\abs{f(y)}\ ;\ y\in B\cap\Supp\mu\}.
\]
Nous montrerons le théorème suivant.

\begin{thm}[Non divergence quantitative dans $G$]
\label{ndi}
%Soient $G$ un $\Q$-groupe semi-simple, $\Gamma$ un sous-groupe arithmétique, 
Étant données deux \- constantes $C_0,\alpha_0>0$, il existe $C,\alpha>0$ tels qu'on ait la propriété suivante.

%Notons $r\in\N^*$ le rang rationnel de $G$, et $V_i$, $i=1,\dots,r$ les représentations rationnelles fondamentales de $G$.
Soit $\mu$ une mesure de Radon sur $G$ et $B(x,\rho)\subset G$ une boule satisfaisant
\[
\forall i\in[1,r],\,\forall v\in\tx_i\cap V_i(\Z),\ 
g\mapsto\norm{g v}
\ \mbox{est $(C_0,\alpha_0)$-régulière sur $B(x,5\rho)$ pour $\mu$.}
\]
Alors, pour tout $\eps>0$, notant $S=\Supp\mu\cap B(x,\rho)$,
\[
\mu(\{g\in B(x,\rho)\ |\ \norm{c(g)-c(S)} \geq -\log\eps\}) \leq C\eps^\alpha\mu(B(x,\rho)).
\]
\end{thm}

Pour l'approximation diophantienne dans les sous-variétés, ce théorème jouera un rôle analogue à celui du théorème du sous-espace de Schmidt pour l'approximation diophantienne des points algébriques, et nous permettra notamment de démontrer les deux résultats ci-dessous.
Nous dirons qu'une mesure borélienne sur $G$ est \emph{localement régulière} en un point $s_0$ dans $G$ s'il existe une boule ouverte $B=B(s_0,r)$ et des constantes $C,\alpha>0$ telles que
\[
\begin{array}{l}
\forall i\in[1,r],\ \forall v\in\tx_i\cap V_i(\Z),\ \forall g\in G,\\
\qquad s\mapsto\norm{g s v}\ \mbox{est $(C,\alpha)$-régulière sur}\ B\ \mbox{pour}\ \mu.
\end{array}
\]

\begin{thm}[Orbites diagonales partant d'une mesure régulière]
\label{diagani}
Soit $(a_t)_{t>0}$ un sous-groupe à un paramètre dans $A$, et $\mu$ une mesure sur $G$ localement régulière en $s_0$.
Il existe une boule ouverte $B$ centrée en $s_0$ telle que pour $\mu$-presque tout $s\in B$, notant $S=B\cap\Supp\mu$,
\[
\lim_{t\to+\infty} \frac{1}{t}(c(a_ts) - c(a_tS)) = 0.
\]
\end{thm}

En d'autres termes, au voisinage de $s_0$, du point de vue de la fonction $c$, presque toutes les orbites suivent le même comportement asymptotique, qui ne dépend que du support de $\mu$ au voisinage de $s_0$.
Si $S$ est un ensemble algébrique irréductible de dimension $m$ dans $G$, on le munit de sa mesure de Lebesgue $\mu$, qui n'est autre que la restriction à $S$ de la mesure de Hausdorff de dimension $m$.
En tout point non singulier de $S$, la mesure $\mu$ est localement régulière, et cela permet d'appliquer le théorème ci-dessus.
Dans le cas où l'ensemble algébrique $S$ est défini sur $\QQ$, nous pourrons même montrer, à l'aide de nos résultats sur les nombres algébriques, que l'expression $\frac{1}{t}c(a_tS)$ converge lorsque $t$ tend vers l'infini, et en déduire le résultat suivant.

\begin{thm}[Orbites diagonales et ensembles algébriques]
\label{diaganalgi}
Soient $Y\in\ka^-$ un élément arbitraire, $(a_t)=(e^{tY})$ le sous-groupe à un paramètre associé, et $P$ le sous-groupe parabolique associé à $(a_t)$:
\[
P=\{g\in G\ |\ \lim_{t\to+\infty}a_tga_{-t}\ \mbox{existe}\}.
\]
Pour $w\in W_P$, nous notons $X_w=\overline{PwB}$ la variété de Bruhat standard correspondante.
Étant donné un ensemble algébrique irréductible $S$ dans $G$ défini sur $\QQ$, posons
\[
c_\infty = \inf\{ p_{\ka^-}(Y^w)\ ;\ w\in W_P:\,\exists\gamma\in G(\Q):\, S \subset X_w\gamma\} \in \ka^-.
\]
Alors, pour presque tout $s$ dans $S$,
\[
\lim_{t\to+\infty} \frac{1}{t}c(a_ts) = c_\infty.
\]
et il existe une variété de Bruhat rationnelle $X_w\gamma$, $w\in W_P$, $\gamma\in G(\Q)$ contenant $S$ et telle que $p_{\ka^-}(Y^w)=c_\infty$.
\end{thm}

Ici encore, on renvoie au théorème~\ref{diaganalg} pour un énoncé plus précis.
C'est à l'aide de ce dernier résultat que nous démontrerons le théorème~\ref{expanalgi}.

\chapter{Une correspondance}
\label{chap:correspondance}

\note{On pourrait prendre $G$ réductif, mais cela ne permet pas d'étudier plus de variétés de drapeaux $X$. D'ailleurs, si on impose même que $G$ soit sans facteur compact et simplement connexe, on obtient encore le même ensemble de variétés $X=P\bcs G$.}

Dans tout ce chapitre, nous considérerons une variété de drapeaux $X$, obtenue comme un espace quotient $X=P\bcs G$ d'un $\Q$-groupe semi-simple $G$ par un $\Q$-sous-groupe parabolique $P$.
Après avoir défini des hauteurs et des distances sur $X$, nous définirons l'exposant diophantien d'un point de $X$, que nous mettrons en relation avec le comportement asymptotique de certaines orbites dans l'espace des réseaux d'une représentation rationnelle bien choisie de $G$.

\section{Hauteurs}
\label{sec:hauteur}

Quitte à remplacer $G$ par son revêtement universel, nous pouvons supposer sans perte de généralité que le groupe $G$ est simplement connexe.
%comme groupe algébrique, i.e. $G(\C)$ simplement connexe.
Fixons dans $G$ un $\Q$-sous-groupe parabolique minimal $B\subset P$, et $T\subset B$ un sous-tore $\Q$-déployé maximal.
Le choix de $B$ induit un ordre sur le groupe $X^*(T)$ des caractères de $T$, et donc une base $\Pi$ du système de racines associé à $G$ et $T$.
Pour la théorie générale des représentations rationnelles de $G$, on renvoie à Borel et Tits \cite[\S12]{boreltits}.

Pour définir une hauteur sur $X=P\bcs G$, nous partirons d'une $\Q$-représentation linéaire à gauche
\[
\rho:G\to\GL(V_\chi)
\]
engendrée par un unique vecteur $e_\chi\in V_\chi(\Q)$ de plus haut poids $\chi$.
On suppose en outre que si $[e_\chi]\in\PP(V)$ est la direction engendrée par $e_\chi$, alors
\[ \Stab_G[e_\chi] = P.\]
%ou \[ \Stab_G(v_\rho) \supset P.\]
La variété $X=P\bcs G$ s'identifie naturellement à l'orbite de la droite $[e_\chi]$ dans l'espace projectif $\PP(V)$ via l'application
\note{Comme toutes les hauteurs sur $\PP(V_\chi)(\Q)$ sont équivalentes, la hauteur $H_\chi$ ne dépend pas du choix de la hauteur sur $\PP(V_\chi)(\Q)$, à équivalence près.
%cf Lang, Diophantine geometry, chap. IV.
}
\[\begin{array}{cccc}
\iota: & X & \to & \PP(V)\\
& Pg & \mapsto & g^{-1}[e_\chi],
\end{array}\]
et comme cette identification préserve les points rationnels, cela permet de définir la hauteur $H_\chi$ sur $X$, par restriction.
\note{Plus généralement, on pourrait considérer toutes les hauteurs sur $X$ obtenues à partir d'un morphisme rationnel $X\to\PP^N$.
Mais la description de ces morphismes à l'aide des classes de diviseurs montre qu'à équivalence près, toute hauteur sur $X(\Q)$ est de la forme $H_\chi$, pour une certaine représentation irréductible $\rho$ totalement rationnelle de poids dominant $\chi$.
(Trouver une référence explicite sur le sujet, ou écrire la démonstration.)
}
Plus précisément, ayant fixé une base rationnelle $(u_i)_{1\leq i\leq d}$ de $V$ et une norme euclidienne sur $V$ pour laquelle cette base est orthonormée, on définit une hauteur sur $\PP(V)$ par la formule
\[
H(u) = \norm{\bu},
\]
où $\bu$ est un représentant primitif de $u$ dans le réseau $\oplus_{1\leq i\leq d}\Z u_i$, et la hauteur sur $X$ est alors donnée par
\[
H_\chi(v) = H(\iota_\chi(v)).
\]

Pour les variétés de drapeaux munies d'une hauteur comme ci-dessus, on dispose d'un équivalent asymptotique pour le nombre de points de hauteur bornée.
Cela a été observé en premier lieu par Schanuel \cite{schanuel} pour l'espace projectif, puis démontré par Franke \cite{fmt} pour une variété de drapeaux $X$ munie de la hauteur anti-canonique.
La version générale que nous citons ici est due à Mohammadi et Salehi Golsefidy \cite[Theorem~4]{msg}.

\begin{theorem}[Nombre de points rationnels de hauteur bornée]
\label{hborne}
Soit $X$ une variété de drapeaux définie sur $\Q$, munie d'une hauteur $H_\chi$ associée au poids dominant $\chi$.
Il existe des constantes $c, u_\chi>0$ et $v_\chi\in\N^*$ telles que la quantité
\[
N_\chi(T) = \card\{ v\in X(\Q)\ |\ H_\chi(v)\leq T\}
\]
vérifie
\[
N_\chi(T) \sim  c\cdot T^{u_\chi}(\log T)^{v_\chi-1}
\quad\mbox{lorsque $T$ tend vers}\ 
+\infty.
\]
\end{theorem}

Les constantes $u_\chi$ et $v_\chi$ sont facilement calculables à l'aide d'un système de racines associé à $G$, comme expliqué dans \cite{msg}.

\section{Distance de Carnot-Carathéodory}
\label{sec:cc}

Nous fixons maintenant un point réel $x$ de $X=P\bcs G$, et expliquons comment évaluer la distance à $x$ d'une approximation $v$.
La distance pour laquelle nous pourrons obtenir des résultats satisfaisants sur les exposants diophantiens est une distance de Carnot-Carathéodory, qui n'est pas riemannienne en général, sauf si le radical unipotent de $P$ est abélien.
%ou le sous-groupe unipotent opposé à $P$, ce qui est équivalent.

\bigskip

Soit $T$ un sous-tore $\Q$-déployé maximal de $G$ inclus dans $P$.
L'algèbre de Lie $\g$ se décompose sous l'action de $T$ en sous-espaces de racines:
\[
\g = \kz \oplus\left(\bigoplus_{\alpha\in\Sigma} \g_\alpha\right),
\]
où
\[
\kz=\{u\in\g\ |\ \forall t\in T,\, (\Ad t)u=u\}
\]
et
\[
\g_\alpha=\{u\in\g\ |\ \forall t\in T,\, (\Ad t)u=\alpha(t)u\}.
\]
D'après \cite[\S11.7]{borel_iga}, si $\Pi$ est une base du système de racines $\Sigma$, il existe une partie $\theta\subset\Pi$ telle que l'algèbre de Lie $\p$ du sous-groupe parabolique $P$ s'écrive
\[
\p = \kz\oplus\left(\bigoplus_{\alpha\in\Sigma^+\cup\bracket{\theta}^-} \g_\alpha\right),
\]
où $\Sigma^+$ désigne l'ensemble des racines positives, et $\bracket{\theta}^-\subset\Sigma$ l'ensemble des racines négatives qui s'écrivent
\[
\alpha = -\sum_{\beta\in\theta} n_{\beta}\beta, \quad n_{\beta}\in\N.
\]
L'algèbre de Lie $\uu^-$ du sous-groupe unipotent $U^-$ opposé à $P$ se décompose en somme directe
\[
\uu^- = \bigoplus_{i\geq 1} m_i,
\]
où $m_i$ est la somme de tous les espaces de racines $\g_\alpha$, où $-\alpha$ est une racine positive contenant exactement $i$ éléments hors de $\theta$ dans sa décomposition en racines simples, avec multiplicité.
La proposition suivante peut se vérifier au cas par cas pour chaque système de racines, en utilisant les tables de Bourbaki \cite[planches~I à IX]{bourbaki_gal4-6};
%\note{et due à Sébastien Miquel?}
la démonstration plus conceptuelle incluse ci-dessous provient de la thèse de doctorat de Sébastien Miquel, et nous a été communiquée par Yves Benoist.

\begin{proposition}[Stratification de $\uu^-$]
\label{strat}
La décomposition $\uu^-=\bigoplus_{i\geq 1} m_i$ est une stratification de $\uu^-$, i.e. pour chaque $i$,
\[
[m_1,m_i] = m_{i+1}.
\]
\end{proposition}
\begin{proof}
Si $X\in\g_\alpha$ et $Y\in\g_\beta$, alors $[X,Y]\in\g_{\alpha+\beta}$, donc pour tout $i$, $[m_1,m_i]\subset m_{i+1}$.
Pour l'inclusion réciproque, supposons tout d'abord que le groupe est déployé sur $\Q$, de sorte que si $\alpha$, $\beta$ sont deux racines telles que $\alpha+\beta\in\Sigma$ est non nulle, alors $[\g_\alpha,\g_\beta]=\g_{\alpha+\beta}$.
\note{Si $G$ est déployé, chaque $\g_\alpha$ est de dimension 1, et donc $[\g_\alpha,\g_\beta]=\g_{\alpha+\beta}$.
Ce résultat est d'ailleurs vrai en général, référence ?}
%Il s'agit alors de vérifier que si $\beta$ est une racine positive dans laquelle apparaissent au moins 2 éléments hors de $\theta$, on peut écrire $\beta=\alpha+\beta'$, où $\alpha$ est une racine positive dans laquelle apparaît exactement un élément hors de $\theta$.
Soit $i\geq 2$ et $\beta$ une racine contenant $i$ racines simples hors de $\theta$; on veut voir que $\g_\beta\subset[m_1,m_{i-1}]$.
\note{\cite[Chapter 5, \S9, Proposition~5, page 32]{serre_alssc}}
D'après \cite[chapitre 5, \S9, proposition~5, page 32]{serre_alssc}
%ou Bourbaki \cite[Chapitre~VI, no~1.6, proposition~19, page~159]{bourbaki_gal4-6}
on peut décomposer $\beta$ en somme de racines simples
\[ \beta = \alpha_1+\dots+\alpha_k\]
de sorte que pour chaque $i$, $\alpha_1+\dots+\alpha_i$ soit une racine.
Montrons par récurrence sur $k\geq i$ que
\[
\g_\beta \subset [m_1,m_{i-1}].
\]
Le résultat est clair pour $k=i$, car alors pour tout $j\in[1,k]$, $\alpha_j\not\in\theta$, et par suite $\g_\beta=[\g_{\alpha_1},\g_{\alpha_2+\dots+\alpha_k}]\subset[m_1,m_{i-1}]$.
Supposons donc $k>i$.
Si $\alpha_k\not\in\theta$, on a ce qu'on veut: avec $\beta'=\alpha_1+\dots+\alpha_{k-1}$, on a $\beta=\beta'+\alpha_k$ et donc $\g_\beta=[\g_{\alpha_k},\g_{\beta'}]\subset[m_1,m_{i-1}]$.
Si $\alpha_k\in\theta$, l'hypothèse de récurrence sur $k$ permet d'écrire
\[
\g_{\beta'} \subset [m_1,m_{i-1}].
\]
Par conséquent
\[
\g_\beta = [\g_{\beta'},\g_{\alpha_k}] \subset [[m_1,m_{i-1}],\g_{\alpha_k}]
\subset [[m_1,\g_{\alpha_k}],m_{i-1}] + [m_1,[m_{i-1},\g_{\alpha_k}]].
\]
Comme $\alpha_k\in\theta$, $[m_1,\g_{\alpha_k}]\subset m_1$ et $[m_{i-1},\g_{\alpha_k}]\subset m_{i-1}$ et donc
\[
\g_\beta \subset [m_1,m_{i-1}] + [m_1,m_{i-1}].
\]
Si $G$ n'est pas déployé sur $\Q$, on remplace $\Q$ par le corps algébriquement clos $\C$, sur lequel $G$ est déployé.
Soit alors $T'$ un tore maximal contenant $T$, et $\Sigma'$ le système de racines associé à $G$ et $T'$.
%-- non nécessairement déployé sur $\Q$ -- 
Le système de racines $\Sigma$ de $G$ par rapport à $T$ s'obtient à partir de $\Sigma'$ par restriction à $T$.
Fixons un ordre sur $\Sigma'$ compatible avec l'ordre choisi sur $\Sigma$, i.e. tel que la projection d'une racine positive est positive, et notons $\Pi'$ la base associée à cet ordre.
D'après \cite[proposition~6.8]{boreltits}
\note{Cela devrait aussi se vérifier directement, voir commentaire dans le fichier .tex.}
les éléments de $\Pi'$ sont envoyés par restriction à $T$ sur $\Pi\cup\{0\}$.
%Commençons par voir que tout élément irréductible $\alpha\in\Sigma^+$ admet un relevé $\alpha'$ irréductible dans $\Sigma^{\prime +}$. Pour cela, on prend un relevé $\alpha'$ de longueur minimale. Si $\alpha'=\alpha'_1+\alpha'_2$, avec $\alpha'_1,\alpha'_2\in\Sigma^{\prime +}$, alors $\alpha=\alpha_1+\alpha_2$, et par irréductibilité de $\alpha$, on doit avoir $\alpha_1=0$ ou $\alpha_2=0$. Si par exemple $\alpha_2=0$, on trouve que $\alpha'_1$ est un relevé de $\alpha$, et par minimalité de $\alpha'$, on doit avoir $\alpha'_2=0$. Donc $\alpha'$ est irréductible, puis $\alpha'\in\Pi'$.
%Soit maintenant $\alpha'\in\Pi'$, d'image $\alpha$ dans $\Sigma^+$. On veut voir que $\alpha$ est irréductible. Soit $\alpha_1\in\Pi$ tel que $\alpha-\alpha_1\in\Sigma^+$. (C'est le cas par exemple, si $(\alpha,\alpha_1)>0$.) D'après ce qui précède il existe un relevé $\alpha'_1$ de $\alpha_1$ dans $\Pi'$. Alors $\alpha'-\alpha_1'$ est un élément positif. Malheureusement, si on a mal choisi $\alpha_1'$, $\alpha'-\alpha_1'$ n'est pas nécessairement une racine (cela se voit déjà en projetant A_2 sur un petit A_1)... à étudier.
Soit $\theta'\subset\Pi'$ l'image réciproque de $\theta\cup\{0\}$.
Sur $\C$, et pour le tore $T'$, le sous-groupe parabolique $P$ est associé à la partie $\theta'\subset\Pi'$ et l'on vérifie facilement que la filtration de $\uu^-$ est aussi construite comme ci-dessus, à partir de $\theta'$.
%En effet, on peut écrire $\g_\alpha=\oplus_{\alpha'}\g_\alpha'$, où $\alpha'$ décrit les relevés de $\alpha$ dans $\Sigma'$. Si la décomposition $\alpha=\alpha_1+\dots+\alpha_k$ contient $i$ éléments hors de $\theta$, alors toute décomposition $\alpha'=\alpha_1'+\dots+\alpha_k'+\alpha_{k+1}'+\dots+\alpha_\ell'$, avec $\alpha_j'|_T=0$ et donc $\alpha_j'\in\theta'$ si $j>k$, contient aussi exactement $i$ éléments hors de $\theta'$.
Le résultat découle donc du cas où $G$ est déployé.
\end{proof}

L'espace vectoriel $m_1\leq\uu^-$ s'identifie naturellement à un sous-espace de l'espace tangent à $X=P\bcs G$ au point base $P$; ce sous-espace est invariant par l'action du stabilisateur $P$ du point base.
%En effet, $m_1\leq\p\bcs\g$ s'identifie à la somme des espaces de racines pour les racines dont la somme des coefficients de $\Pi\setminus\theta$ est supérieure ou égale à $-1$, et $\p$ à la somme des espaces de racines pour les racines dont la somme des coefficients de $\Pi\setminus\theta$ est supérieure ou égale à $0$. Donc $[\p,m_1]\subset m_1$. Puis $\Ad P \cdot m_1\subset m_1$.
% Cet argument n'est pas tout à fait rigoureux car on n'est pas sûr que $P=\exp\p$. Mais $P$ est engendré par $\exp\p$, donc tout va bien.
Nous dirons qu'un chemin $\gamma:[0,1]\to X$ est \emph{horizontal} si pour tout $t$ dans $[0,1]$,
\[
\gamma'(t)\cdot s(\gamma(t))^{-1} \in m_1,
\]
% $= T_{\gamma(t)}R_{s(\gamma(t))^{-1}}$ où $R_g:X\to X$ est la multiplication (à droite) par $g$.
où $s(x)$ désigne un élément de $G$ tel que $x=Ps(x)$.
Comme $m_1$ est invariant par $P$, cette notion ne dépend pas du choix de l'élément $s(x)$.
Nous noterons $\cH$ l'ensemble des chemins horizontaux.
% On aurait aussi pu considérer le sous-fibré $F\leq TX$ dont la fibre au point $x=Pg$ est égale à $m_1\cdot g$.
% Ce fibré est invariant par l'action de $G$, et les chemins horizontaux sont ceux dont la dérivée est toujours dans $F$.
Ayant fixé un sous-groupe compact maximal $K\subset G$ et une norme euclidienne sur $m_1$ invariante sous l'action de $K\cap P$, on définit la \emph{longueur} d'un chemin horizontal $\gamma$ par la formule
\[
\ell(\gamma) = \int_0^1 \norm{\gamma'(t)\cdot s(\gamma(t))^{-1}}\,\dd t,
\]
où cette fois l'élément $s(\gamma(t))$ tel que $\gamma(t)=Ps(\gamma(t))$ est choisi dans $K$.

\begin{definition}[Distance de Carnot-Carathéodory sur $P\bcs G$]
On définit une distance sous-riemannienne sur $X$ par la formule
\[
d(x,y) = \inf\{\ell(\gamma)\ ;\ \gamma\in\cH\ \mbox{tel que}\ \gamma(0)=x\ \mbox{et}\ \gamma(1)=y \}.
\]
\end{definition}

\begin{remark}
La distance de Carnot-Carathéodory que nous avons construite dépend du choix du sous-groupe compact maximal $K$.
Mais toutes les distances construites de cette manière sont équivalentes, et pour les problèmes d'approximation diophantienne que nous étudierons, ce choix n'aura donc pas d'importance.
\end{remark}

D'après le théorème de Chow
%-Rashevskii
\cite[Theorem 0.4]{gromov_cc}
% ou \cite[Theorem~3.31]{agrachevbarilariboscain}
cette formule définit bien une distance sur $X$, et la topologie associée est équivalente à la topologie usuelle sur $X$.
De plus, vue notre construction, cette distance est $K$-invariante.
Remarquons aussi que tout élément $g$ dans $G$ agit sur $X$ en préservant les chemins horizontaux ; par conséquent, $g$ induit une transformation bi-lipschitzienne de $X$, muni de la métrique de Carnot-Carathéodory.
Pour comprendre la géométrie associée à la distance $d$, on peut s'aider de la carte locale
\[
\begin{array}{lll}
U^- & \to & X\\
s & \mapsto & P\cdot s
\end{array}
\] 
Si $U^-$ est muni de la distance de Carnot-Carathéodory \cite[\S 3.3, page 79]{ledonne_lecturesinsubriemanniangeometry}
 associée à la stratification $\uu^-=m_1\oplus\dots\oplus m_r$ 
cette carte est localement bi-lipschitzienne.
%Car les chemins horizontaux dans $U^-$ sont envoyés sur des chemins horizontaux dans $X$, et réciproquement, tandis que les normes sur $m_1$ ou ses translatés sont équivalentes.
On définit ensuite une quasi-norme sur $\uu^-$ par la formule
\[
\abs{x} = \max_{1\leq i\leq r} \norm{x_i}^{\frac{1}{i}},
\]
où $x=\sum_i x_i$ est la décomposition de $x$ suivant la stratification de $\uu^-$, et $\norm{\cdot}$ désigne une norme arbitraire fixée sur $\uu^-$.
Pour $r>0$, nous noterons
\[
B_{\uu^-}(0,r) = \{ u\in\uu^-\ |\ \abs{u}\leq r\}.
\]
La figure ci-dessous représente la boule $B_{\uu^-}(0,r)$ pour $r>0$ petit, lorsque $U^-$ est le groupe de Heisenberg de dimension 3, identifié à l'espace $\R^3$ muni de l'opération $(x,y,z)*(x',y',z')=(x+x',y+y',z+z'+xy')$.

\begin{figure}[H]
\begin{center}
\begin{tikzpicture}

\draw[blue,thick] (-2-1.5,-1.05) -- (2-1.5,-1.05) -- (2+1.5,0.65);
\draw[blue] (2+1.5,0.65) -- (-2+1.5,0.65) -- (-2-1.5,-1.05);

\draw[->,gray] (0.7*2.3,0.7*1.3) -- (-0.7*2.3,-0.7*1.3) node[anchor=south] {\tiny{x}};
\draw[->,gray] (-0.7*3,0) -- (0.7*3,0) node[anchor=south] {\hspace{7pt} \tiny{y}};
\draw[->,gray] (0,-1) -- (0,1.5) node[anchor=south] {\tiny{z}};

\draw[blue,thick] (-2-1.5,-0.65) -- (2-1.5,-0.65) -- (2+1.5,1.05) -- (-2+1.5,1.05) -- (-2-1.5,-0.65);
\draw[blue,thick] (-2-1.5,-1.05) -- (-2-1.5,-0.65);
\draw[blue,thick] (2-1.5,-1.05) -- (2-1.5,-0.65);
\draw[blue,thick] (2+1.5,0.65) -- (2+1.5,1.05);
\draw[blue] (-2+1.5,0.65) -- (-2+1.5,1.05);
\draw[blue,thick] (-2-1.5,-1.05) -- (-2-1.5,-0.65);

\draw[<->,thick] (-2-1.65,-0.65) -- (-2-1.65,-0.9) node[anchor=east] {\tiny{$r^2$}} -- (-2-1.65,-1.05);
\draw[<->,thick] (-2-1.5,-1.2) -- (-1.5,-1.2) node[anchor=north] {\tiny{$r$}} -- (2-1.5,-1.2);
\end{tikzpicture}
\end{center}
\caption{Boule de rayon $r>0$ pour $U^-=\mathrm{Heisenberg}(3)$}
\end{figure}

La proposition ci-dessus permet de comparer les boules de rayon $r>0$ pour la distance de Carnot-Carathéodory sur $X$ à des boules pour la quasi-norme $\abs{\cdot}$ associée à la stratification de $\uu^-$.
En anglais, cet énoncé est souvent appelé \enquote{Ball-box Theorem} ce qui résume bien son contenu.

\begin{proposition}[Ball-box Theorem]
\label{ballbox}
Soit $x\in X$ et $s_x\in G$ tel que $x=Ps_x$.
Il existe une constante $C>0$ telle que pour tout $r>0$ suffisamment petit,
\[
P\exp(B_{\uu^-}(0,\frac{r}{C}))\cdot s_x
\subset B(x,r)
\subset P\exp(B_{\uu^-}(0,Cr))\cdot s_x.
\]
\end{proposition}
\begin{proof}
Comme l'application $y\mapsto ys_x$ est bi-lipschitzienne sur $X$, il suffit de vérifier la proposition pour le point $x=P$.
Vue l'équivalence locale des distances de Carnot-Carathéodory sur $X$ et sur $U^-$, il suffit de démontrer le résultat analogue sur $U^-$.
On définit une famille de dilatations $\delta_t:\uu^-\to\uu^-$ par $\delta_t(\sum x_i) = \sum t^ix_i$.
C'est un sous-groupe à un paramètre d'automorphismes de $\uu^-$, qui induit donc un sous-groupe à un paramètre d'automorphismes de $U^-$, toujours noté $\delta_t$.
Comme $\delta_t$ préserve les chemins horizontaux et dilate la norme sur $m_1$ par un facteur $t$, on a pour tous $x$ et $y$, $d(\delta_t(x),\delta_t(y))=td(x,y)$.
Par conséquent,
\[
B(1,r) = \delta_r(B(1,1)) = \exp(\delta_r(B_{\uu^-}(0,1))) = \exp( B_{\uu^-}(0,r)).
\]
\end{proof}

Pour deviner par un argument heuristique la valeur presque sûre de l'exposant diophantien d'un point de $X$ (cf. paragraphe suivant), il est utile de connaître le nombre de recouvrement de $X$ par des boules de petit rayon $r>0$ pour la métrique de Carnot-Carathéodory.

\begin{proposition}[Dimension de Carnot-Carathéodory]
Le nombre de recouvrement de $X$ à l'échelle $\delta>0$ est
\[
N(X,\delta)\asymp\delta^{-\dim_{cc} X}.
\]
où $\dim_{cc}X$ est la dimension de $X$ pour la distance $d$, égale à
\[ \dim_{cc} X = \sum_{i\geq 1} i\dim m_i.\]
%\[ = \sum_{\alpha\in\Sigma^+\setminus\bracket{\theta}} n_{\theta^c}(\alpha).\]
%Dans cette formule, $n_{\theta^c}(\alpha)$ désigne le nombre de racines de $\theta^c$ qui apparaissent dans $\alpha$, avec multiplicité.
\end{proposition}
\begin{proof}
Cette formule découle de la proposition précédente, qui décrit la forme des boules de rayon $r$ pour la métrique $d$ sur $X$.
\end{proof}

%%%%%%%%%%%%%%%%%
%Une remarque de Jean-François sur les cocycles:
%Si $G$ agit à droite sur $X$, un cocycle $\sigma:X\times G\to H$ est une application vérifiant $\sigma(x,g_1g_2)=\sigma(xg_1,g_2)\sigma(x,g_1)$.
%Si $X=P\bcs G$, la donnée d'un cocycle $\sigma$ est équivalente à celle d'un morphisme de groupes $\theta:P\to H^{op}$ (groupe opposé) et d'une section $s:X\to G$ telle que $s(P)=1$: il existe un unique cocycle tel que
%\[ \left\{\begin{array}{ll}
%\sigma(P,p)=\theta(p) & \forall p\in P\\
%\sigma(P,s(x))=1 & \forall x\in X.
%\end{array}\right.\]
%À titre d'exemple, on obtient le cocycle d'Iwasawa $\sigma:X\times G\to\R^*$ via la formule $\sigma(P,kan)=\det(\Ad a)^{-1}|_{\p\bcs\g}$.
%Plus généralement, l'action de $G$ sur le fibré tangent de $X$ est décrite par un cocycle $\sigma:X\times G\to\GL(\p\bcs\g)$.

\begin{remark}
Les méthodes développées dans ce mémoire permettent de traiter des quasi-distances un peu plus générales: on peut partir d'une quasi-norme sur $\p\bcs\g$ homogène pour un sous-groupe diagonal à un paramètre $a_t=e^{tY}$ dans $P$.
Mais l'avantage de la métrique de Carnot-Carathéodory est que le drapeau associé dans $\p\bcs\g$ est invariant par $P$, ce qui fait que l'on obtient bien une distance sur $X=P\bcs G$. 
Dans le cadre de l'espace projectif, les quasi-distances sont utilisées dans \cite{abrs2} pour étudier l'approximation diophantienne dans les groupes nilpotents; on obtient ainsi des résultats d'approximation diophantienne pour des quasi-normes dans $\R^d$.
\end{remark}

\section{L'espace des réseaux}

Soit $V$ un espace vectoriel euclidien, isomorphe à $\R^d$, $d\in\N^*$.
Un \emph{réseau} dans $V$ est un sous-groupe discret de rang maximal, égal à $\dim V$.
Pour décrire la forme d'un réseau $\Delta$, on définit la suite de ses \emph{minima successifs}
\[
\lambda_1(\Delta) \leq \lambda_2(\Delta) \leq \dots \leq \lambda_d(\Delta)
\]
par
\[
\forall i\in[1,d],\quad
\lambda_i(\Delta) = \inf\{\lambda>0\ |\ \rang(B(0,\lambda)\cap\Delta)\geq i\},
\]
où, pour $E\subset V$, on note $\rang(E)$ le rang linéaire de $E$, i.e. le cardinal maximal d'une famille libre d'éléments de $E$.
On définit aussi  le \emph{covolume} de $\Delta$ dans $\R^d$, noté $\covol(\Delta)$, comme le volume d'un domaine fondamental de $\R^d$ sous l'action de $\Delta$.
Le second théorème de Minkowski \cite{minkowski} relie les minima successifs d'un réseau et son covolume.
Il sera d'une importance capitale dans la suite ce mémoire.

\begin{theorem}[Second théorème de Minkowski]
\label{minkowski}
Soit $d\in\N^*$ et $\Delta$ un réseau dans $\R^d$.
Alors,
\[
\frac{2^{d}}{d!}\covol(\Delta) \leq \lambda_1(\Delta)\dots\lambda_d(\Delta) \leq 2^d\covol(\Delta).
\]
\end{theorem}

Soit maintenant $X$ une variété de drapeaux, obtenue comme un quotient $X=P\bcs G$ d'un $\Q$-groupe semi-simple $G$ par un sous-groupe parabolique $P$.
Nous fixons une $\Q$-représentation $V_\chi$ de $G$ engendrée par un vecteur rationnel de plus haut poids $e_\chi$ tel que
\[
\forall p\in P,\quad p\cdot e_\chi = \chi(p) e_\chi.
\]
Cela permet en particulier de définir une hauteur $H_\chi$ comme expliqué au paragraphe~\ref{sec:hauteur}.
On munit aussi $V_\chi$ d'une norme euclidienne.
Pour l'étude de l'approximation diophantienne dans $X$, en plus des minima successifs, nous devons définir une certaine fonction $r_\chi$ sur l'espace des réseaux de $V_\chi$.
Ayant fixé un sous-tore $\Q$-déployé maximal $T$ dans $P$, l'espace $V_\chi$ se décompose en somme directe d'espaces de poids.
Soit
\[
\pi^+:V_\chi\to\R e_\chi
\]
la projection parallèlement à la somme des espaces de poids distincts de $\R e_\chi$.
Notons aussi
\[
\tx = G\cdot e_\chi
\]
le cône dans $V_\chi$ engendré par $\iota(X)$, privé du point $0$.
Pour $r>0$, on considère la partie suivante de $V_\chi$:
\[
C_r = \{ \bv\in \tx\ |\ \norm{\bv}\leq r\ \mbox{et}\ \norm{\pi^+(\bv)}>\frac{\norm{\bv}}{2}\}.
\]

\begin{remark}
Dans cette définition le choix de la constante $\frac{1}{2}$ dans l'inégalité $\norm{\pi^+(\bv)}>\frac{\norm{\bv}}{2}$ est arbitraire.
La correspondance que nous mettrons en évidence ci-dessous à la proposition~\ref{dani} est encore valable si l'on remplace cette condition par $\norm{\pi^+(\bv)}>c\norm{\bv}$, pour une constante $c>0$ arbitraire.
\end{remark}

\begin{figure}[H]
\begin{center}
\begin{tikzpicture}
\draw[->,gray] [xshift=-5cm] (0,0) -- (-0.7*2.3,-0.7*1.3) node[anchor=south] {\tiny{z}};
\draw[->,gray] [xshift=-5cm] (0,0) -- (0.7*3,0) node[anchor=south] {\hspace{7pt} \tiny{y}};
\draw[->,gray] [xshift=-5cm] (0,0) -- (0,1.5) node[anchor=south] {\tiny{x}};
\draw [xshift=-7cm] (-1.3,1) node[anchor=west] {$\tx=\R^3\setminus\{0\}$};
\draw [xshift=-7cm] (-1.3,.6) node[anchor=west] {$X=\PP^2$};
\draw [xshift=-7cm] (-1.3,.2) node[anchor=west] {$G=\SL_3$};
\draw (-0.9*1.6,0.9*1.5) -- (0.9*1.6,-0.9*1.5);
\draw (0.9*1.6,0.9*1.5) -- (-0.9*1.6,-0.9*1.5);
\draw (-1.36,-1.28) arc (160:380:1.45 and .45);
\filldraw [xshift=-5cm,draw=blue,fill=blue!40!white] (-.3*0.9*1.6,.3*0.9*1.5) -- (.3*0.9*1.6,-.3*0.9*1.5) -- (-.3*.9*1.6,-.3*.9*1.5) -- (.3*.9*1.6,.3*.9*1.6) -- cycle;
\filldraw [xshift=-5cm,draw=blue,fill=blue!40!white] (0,.5) ellipse (.5 and 0.3*.5);
\filldraw [xshift=-5cm,draw=blue,fill=blue!40!white] (0,-.5) ellipse (.5 and 0.3*.5);
\filldraw [fill=blue!30!white, draw=blue] (0,0) -- (-.44,.42) arc (200:325:.49 and .10)  -- cycle;
\filldraw [fill=blue!40!white, draw=blue] (0,0) -- (.44,-.42) arc (20:145:.49 and .10)  -- cycle;
\draw [xshift=-5cm,blue,->,thick] (0,0) -- (0,1) node[anchor=west] {\tiny{$e_\chi=(1,0,0)$}};
\draw [xshift=-5cm,blue] (-.2,0) node[anchor=east] {$C_r$};

\draw (0,1.5) ellipse (1.5 and .5);
\draw (1.3,1) node[anchor=west] {$\tx=\{x^2+y^2-z^2=0\}\setminus\{0\}$};
\draw (1.3,.6) node[anchor=west] {$X=\BS^1$};
\draw (1.3,.2) node[anchor=west] {$G=\SO_{2,1}$};
\draw (-0.9*1.6,0.9*1.5) -- (0.9*1.6,-0.9*1.5);
\draw (0.9*1.6,0.9*1.5) -- (-0.9*1.6,-0.9*1.5);
\draw (-1.36,-1.28) arc (160:380:1.45 and .45);
\draw (0,.5) ellipse (.5 and 0.3*.5);
\draw (0,-.5) ellipse (.5 and 0.3*.5);
\draw (-.3*0.9*1.6,.3*0.9*1.5) -- (.3*0.9*1.6,-.3*0.9*1.5);
\draw (.3*0.9*1.6,.3*0.9*1.5) -- (-.3*0.9*1.6,-.3*0.9*1.5);
\filldraw [fill=blue!40!white, draw=blue] (0,0) -- (-.44,.42) arc (200:325:.49 and .10)  -- cycle;
\filldraw [fill=blue!15!white, draw=blue] (0,0) -- (.44,-.42) arc (20:145:.49 and .10)  -- cycle;
\draw [blue,->,thick] (0,0) -- (-.5,1.) node[anchor=south] {\tiny{\hspace{7pt}$e_\chi=(1,0,1)$}};
\draw[->,gray] (0,0) -- (-1.5,-1) node[anchor=north] {\tiny{x}};
\draw[->,gray] (0,0) -- (1.,0) node[anchor=south] {\tiny{y}};
\draw[->,gray] (0,0) -- (0,1.5) node[anchor=south] {\tiny{z}};
\draw [blue] (-.2,0) node[anchor=east] {$C_r$};
\end{tikzpicture}
\end{center}
\caption{L'ensemble $C_r$ pour $X=\PP^2$ et $X=\BS^1$}
\end{figure}

\begin{definition}[La fonction $r_\chi$]
Si $\Delta$ est un réseau de $V_\chi$, on pose 
\[
r_\chi(\Delta) = \inf\{ r>0\ |\ C_r\cap \Delta\neq\{0\}\}\in \R\cup\{+\infty\}.
\]
\end{definition}

\begin{remark}
Attention! 
Les parties $C_r$ ne sont pas convexes, et leur volume est nul en général, car $C_r\subset\tx$.
Il n'est donc pas question d'appliquer ici le premier théorème de Minkowski, et il est d'ailleurs facile de construire un réseau $\Delta$ dans $V_\chi$ tel que $r_\chi(\Delta)=\infty$.
%Dans la suite, nous nous intéresserons seulement à des réseaux de $V_\chi$ dont une base est constituée de vecteurs de $\tilde{X}$; pour de tels réseaux, on a $r_\chi(\Delta)<\infty$, quitte à multiplier $\Delta$ par un élément $a_t$ de taille bornée.
\end{remark}

\section{L'exposant diophantien}
\label{sec:ed}

Pour chaque choix de hauteur $H_\chi$ sur la variété de drapeaux $X$, nous définissons un exposant diophantien $\beta_{\chi}$ sur $X$,
que nous interprétons ensuite en termes d'orbites diagonales dans l'espace des réseaux de la représentation $V_\chi$ de plus haut poids $\chi$.

\begin{definition}
L'\emph{exposant diophantien} d'un point $x\in X$ pour la distance $d$ de Carnot-Carathéodory est
\[
\beta_{\chi}(x) = \inf\{\beta>0\ |\ \exists c>0:\,\forall v\in X(\Q),\, d(x,v)\geq c.H_\chi(v)^{-\beta}\}.
\]
\end{definition}

\begin{exercise}[Minoration uniforme de l'exposant diophantien]
\label{stereo}
Le but de cet exercice est de minorer l'exposant diophantien d'un point $x$ arbitraire dans $X(\R)$, grâce à la projection stéréographique
\(
\begin{array}{lll}
\uu^- & \to & X\\
u & \mapsto & Pe^u.
\end{array}
\)
On fixe un réseau rationnel $\uu^-(\Z)$ dans l'algèbre de Lie $\uu^-$.
\begin{enumerate}
\item Pour $u\in\uu^-$ irrationnel, montrer qu'il existe une infinité de rationnels $\frac{\bp}{q}\in\uu^-(\Q)$, où $\bp\in\uu^-(\Z)$ et $q\in\N^*$ tels que $\Abs{u-\frac{\bp}{q}}\ll q^{-\frac{1+\dim X}{\dim_{cc} X}}$.
\item Si $x=e^u$, vérifier que $d(x,Pe^{\frac{\bp}{q}})\asymp d(P,e^{\frac{\bp}{q}}e^{-u}) \ll \Abs{u-\frac{\bp}{q}}$.
\item Justifier que l'application $\uu^-\to V_\chi;\ u\mapsto e^{-u}e_\chi$ est polynomiale; on note $d_\chi$ son degré.
\item Montrer que pour tout $x\in X(\R)$, $\beta_\chi(x)\geq \frac{1+\dim X}{d_\chi\dim_{cc}X}$.
\item Si $X=\Grass(\ell,d)$ est une variété grassmannienne, munie de la hauteur associée au plongement de Plücker $X\hookleftarrow\PP(\wedge^\ell\R^d)$, montrer que pour tout $x\in X(\R)$, $\beta_\chi(x)\geq\frac{1}{\ell}+\frac{1}{\ell(d-\ell)\min(\ell,d-\ell)}$.
\end{enumerate}
\end{exercise}

Suivant la méthode introduite par Dani \cite{dani_correspondence}, et exploitée ensuite avec succès par divers auteurs, notamment Kleinbock et Margulis \cite{kleinbockmargulis, kleinbock_dichotomy}, nous voulons traduire les propriétés diophantiennes d'un point $x\in X=P\bcs G$ en termes d'orbites diagonales dans l'espace des réseaux de $V_\chi$.
\note{Un réseau $\Delta$ dans $V_\chi$ est dit \emph{rationnel} si $\Delta\subset V_\chi(\Q)$.}
Pour cela, nous fixons dans $V_\chi$ un réseau rationnel $V_\chi(\Z)$, et si $x=Ps_x$ est un élément de $X$, nous lui associons le réseau
\[
\Delta_x = s_xV_\chi(\Z).
\]
Par ailleurs, ayant fixé un sous-tore $\Q$-déployé maximal $T\subset P$, nous noterons $A=T^0(\R)$ la composante neutre de ses points réels, et $\ka$ l'algèbre de Lie de $A$.
Rappelons que l'on note $\Pi$ une base du système de racines de $G$ pour $T$, et $\theta\subset\Pi$ la partie associée au sous-groupe parabolique $P$, telle que toutes les racines négatives de $P$ se décomposent en éléments de $\theta$.
On définit alors un sous-groupe diagonal $(a_t)$ à un paramètre dans $G$ en posant
\begin{equation}\label{at1}
a_t=e^{tY}
\quad\mbox{où}\quad Y\in\ka\ \mbox{est défini par}\
\alpha(Y) = 
\left\{
\begin{array}{ll}
0 & \mbox{si}\ \alpha\in\theta\\
-1 & \mbox{si}\ \alpha\not\in\theta.
\end{array}
\right.
\end{equation}
Ce flot est choisi de sorte que la quasi-norme sur $\uu^-$ définie au paragraphe ci-dessus satisfasse la propriété d'homogénéité suivante.

\begin{lemma}
Pour tout $t\in\R$ et tout $u\in\uu^-$,
\[
\abs{(\Ad a_t)u} = e^t\abs{u}.
\]
\end{lemma}
\begin{proof}
Écrivons $u=\sum_{i\geq 1}u_i$ suivant la décomposition en somme directe $\uu^-=m_1\oplus\dots\oplus m_s$.
Comme $\Ad a_t=e^{t\ad Y}$, cette décomposition est préservée par $\Ad a_t$ et de plus, vue la définition de $Y$, pour chaque $i$, $(\Ad a_t)u_i = e^{it}u_i$.
Avec la formule $\abs{u}=\max_{1\leq i\leq s}\norm{u_i}^{\frac{1}{i}}$, on trouve ce qu'on veut.
\end{proof}

\begin{definition}[Taux de fuite dans l'espace des réseaux de $V_\chi$]
Le \emph{taux de fuite} du réseau $\Delta_x=s_xV_\chi(\Z)$ sous l'action de $(a_t)$ est
\[
\gamma_{\chi}(x) = \liminf_{t\to\infty}\frac{-1}{t}\log r_\chi(a_t\Delta_x)% \in \R.
\]
\end{definition}

%Quelques remarques s'imposent.

\begin{remark}
Ce taux de fuite ne dépend pas du choix de l'élément $s_x$ tel que $x=Ps_x$, car pour tout $p\in P$, l'élément $a_tpa_t^{-1}$ converge vers $p_\infty\in P$ lorsque $t$ tend vers l'infini.
En effet, $p_\infty e_\chi=\chi(p_\infty)e_\chi$, donc si $\norm{a_t\bv}\leq e^{-\gamma t}$ et $\norm{\pi^+(a_t\bv)}\geq\frac{\norm{a_t\bv}}{2}$ alors $\norm{a_tp\bv}\leq Ce^{-\gamma t}$ et $\norm{\pi^+(a_tp\bv)}\geq\frac{\norm{a_tp\bv}}{2C}$, où $C$ est une constante qui dépend de $p_\infty$.
En suite, prenant $t'=t-C$, avec $C>0$ suffisamment grand, comme la direction $e_\chi$ est plus contractée que les autres par un facteur exponentiel, on trouve $\norm{\pi^+(a_{t'}p\bv)}\geq\frac{\norm{a_{t'}p\bv}}{2}$ et $\norm{a_{t'}p\bv}\leq C'e^{-\gamma t}$.
\end{remark}

\note{Peut-on montrer que la fonction $\gamma_{\chi}$ est à valeurs dans l'intervalle $[\chi(Y),-\chi(Y)]$, où $a_t=e^{tY}$ ?
Cela devrait découler du fait que les valeurs propres de $a_t$ dans $V_\chi$ sont toutes dans l'intervalle  $[e^{-t\chi(Y)},e^{t\chi(Y)}]$.}

\begin{remark}
Il existe une constante $c=c_x>0$ telle que pour tout $\bv$ non nul dans $\Delta_x$, $\norm{\bv}\geq c$.
Par suite, pour tout $t>0$, tout vecteur $\bv$ non nul de $a_t\Delta_x$ satisfait $\norm{\bv}\geq ce^{t\chi(Y)}$, ce qui montre que pour tout $x$, $\gamma_\chi(x)\leq -\chi(Y)$.
Cette majoration est optimale, comme le montre le cas $s_x=1$.
Il semble beaucoup plus difficile de déterminer la borne inférieure optimale de $\gamma_\chi$ sur $X$.
Toutefois, l'exercice~\ref{stereo} ci-dessus montre que pour tout $x$, $\beta_\chi(x)\geq\frac{1+\dim X}{d_\chi\dim_{cc} X}>0$.
Avec la proposition~\ref{dani} ci-dessous, cela implique que $\gamma_\chi(x)$ est uniformément minorée sur $X$, ce qui n'était pas évident a priori.
\end{remark}
%
%\begin{remark}
%Le réseau rationnel $L$ contient nécessairement un multiple du vecteur de plus haut poids $e_\chi$.
%En considérant l'orbite $G(\Q)\cdot e_\chi$, on obtient même dans $L$ une base de vecteurs constituée d'éléments de $\tilde{X}$.
%Par suite, pour tout $g$ dans $G$, le réseau $gL$ admet une base constituée d'éléments de $\tx$, car $g$ stabilise $\tx$.
%C'est le cas en particulier du réseau $L_x$.
%Mais il ne semble pas clair a priori que le taux de fuite $\gamma_\chi(x)$ est fini.
%\end{remark}

Dans son article \cite{dani_correspondence}, Dani formule en termes dynamiques la correspondance entre les propriétés diophantiennes d'un point $x$ de l'espace projectif $\PP^{d-1}(\R)$ et le comportement d'une orbite diagonale $(a_ts_x\Z^d)_{t>0}$ dans l'espace des réseaux de $\R^d$: les bonnes approximations rationnelles de $x$ correspondent aux petits vecteurs dans les réseaux de l'orbite.
La proposition suivante généralise cette correspondance pour les variétés drapeaux.
Une différence majeure apparaît cependant: s'il est toujours vrai qu'une bonne approximation rationnelle du point $x$ permet de construire un petit vecteur dans un réseau $a_ts_xV_\chi(\Z)$, pour $t>0$ bien choisi, la réciproque n'est pas toujours vraie.
Pour pouvoir construire une bonne approximation à partir d'un petit vecteur $\bv$ dans $a_ts_xV_\chi(\Z)$, il est nécessaire d'imposer que ce vecteur appartienne à l'orbite $\tx$ d'un vecteur de plus haut poids dans $V_\chi$, et surtout, que sa projection dans la direction de plus haut poids soit comparable à sa norme.

\comm{Ajouter un exemple (dans la grassmannienne) d'un petit vecteur qui ne correspond pas à une bonne approximation rationnelle.}

\begin{proposition}[Correspondance drapeau-réseau]
\label{dani}
Soit $G$ un $\Q$-groupe semi-simple, $P$ un $\Q$-sous-groupe parabolique, et $X=P\bcs G$ la variété quotient.
Fixons aussi une hauteur $H_\chi$ sur $X(\Q)$, donnée par un poids dominant $\chi$ associé à $P$.

Avec les notations ci-dessus, pour tout $x\in X(\R)$,
\[ \beta_{\chi}(x) = \frac{1}{-\chi(Y)-\gamma_{\chi}(x)}.\]
\end{proposition}
\begin{proof}
%Puisque la formule ne dépend pas de la hauteur $H_\chi$ à équivalence près, nous pouvons supposer que la hauteur $H_\chi$ sur $X(\Q)$ est donnée par la formule
La hauteur $H_\chi$ sur $X(\Q)$ est donnée par la formule
\[ H_\chi(v) = \norm{\bv},\]
où $\bv\in V_\chi(\Z)$ est un représentant primitif de $v\in X(\Q)$.

Soit $\beta<\beta_{\chi}(x)$, de sorte qu'il existe un point rationnel $v\in X(\Q)$ arbitrairement proche de $x$ tel que
$d(x,v)\leq H_\chi(v)^{-\beta}$.
Écrivons
\[
x=Ps_x \quad\mbox{et}\quad v=Pe^us_x,\ \mbox{avec}\ u\in\uu^-,
\]
ce qui peut aussi s'écrire, en identifiant $X\simeq G\cdot[e_\chi]$,
\[
x=s_x^{-1}\cdot[e_\chi] \quad\mbox{et}\quad v=s_x^{-1}e^{-u}\cdot[e_\chi].
\]
Si $\bv$ est un représentant primitif de $v$ dans $V_\chi(\Z)$, le vecteur $s_x\bv$ est porté par la direction $\R e^{-u}e_\chi$, et par conséquent, pour tout $t>0$,
\begin{align*}
a_ts_x\bv & \asymp H_\chi(v) a_te^{-u}e_\chi\\
& = H_\chi(v) \chi(a_t) e^{-(\Ad a_t)u}e_\chi\\
& = H_\chi(v) \chi(a_t) [e_\chi -((\Ad a_t)u)e_\chi+\frac{1}{2}((\Ad a_t)u)^2e_\chi+\dots]
\end{align*}
\comm{Il serait plus rigoureux d'écrire
\begin{align*}
\rho(a_ts_x)\bv & = H_\chi(v) \rho(a_te^{-u})e_\chi\\
& = H_\chi(v) \chi(a_t) e^{-\rho'((\Ad a_t)u)}e_\chi\\
& = \dots
\end{align*}
où $\rho':\g\to\gl(V_\chi)$ est la différentielle de la représentation $\rho:G\to\GL(V_\chi)$.
En fait, on omet tout simplement la représentation $\rho$, ce qui revient à voir $G$ comme un sous-groupe de $\GL(V_\chi)$ et $\g$ comme une sous-algèbre de Lie de $\gl(V_\chi)$.}
Notons que dans la somme ci-dessus, seul le premier terme $e_\chi$ n'est pas dans le noyau de $\pi^+$.
%En effet, $u=\sum_{\alpha<0}u_\alpha$, et pour chaque $\alpha$, $u_\alpha\cdot e_\chi$ est un vecteur de poids $\chi-\alpha$.

D'après la proposition~\ref{ballbox}, la distance $d(x,v)$ est comparable à la quasi-norme du vecteur $u$ dans $\uu^-$:
\[
d(x,v) \asymp |u| \leq H_\chi(v)^{-\beta}.
\]
Soit alors $c>0$ tel que pour tout $w\in\uu^-$ tel que $\abs{w}\leq c$, on ait $\norm{w}\leq\frac{1}{4}$,
\comm{Ou plus exactement, $\norm{\rho'(w)}\leq 1/4$.}
et $t>0$ tel que $e^t=cH_\chi(v)^\beta$.
Avec ce choix de $t$, par homogénéité de la quasi-norme $\abs{\cdot}$ pour le flot $\Ad a_t$,
\[
\abs{(\Ad a_t)u}=e^t\abs{u}\leq ce^td(x,v) \leq c
\]
et donc 
\[
\norm{(\Ad a_t)u}\leq \frac{1}{4},
\]
\details{
Or, pour $\norm{w}\leq 1/4$, l'inégalité $\norm{e^w-I}\leq 1/2$ implique d'une part
\[ \norm{-((\Ad a_t)u)e_\chi+\frac{1}{2}((\Ad a_t)u)^2e_\chi+\dots]} = \norm{(e^{-(\Ad a_t)u}-I)e_\chi} \leq 1/2\]
d'où
\[ \norm{e_\chi}=1 \geq \frac{1}{2}\norm{e_\chi -((\Ad a_t)u)e_\chi+\frac{1}{2}((\Ad a_t)u)^2e_\chi+\dots}\]
}
puis
\[
\norm{\pi^+(a_ts_x\bv)} \geq \frac{1}{2}\norm{a_ts_x\bv}.
\]
D'autre part
\[
\norm{a_ts_x\bv} \leq 2H_\chi(v)\chi(a_t) = 2c^{-\beta}e^{t[\chi(Y)+\frac{1}{\beta}]}.
\]
Cela montre que $\gamma_{\chi}(x) \geq -\chi(Y)-\frac{1}{\beta}$, et donc
\[
\beta_{\chi}(x) \leq \frac{1}{-\chi(Y)-\gamma_{\chi}(x)}.
\]

Pour montrer l'inégalité réciproque, fixons $\gamma<\gamma_{\chi}(x)$.
On peut donc trouver $t>0$ arbitrairement grand, et $\bv\in\tx\cap V_\chi(\Z)$ tel que
\begin{equation}\label{norme1}
 \norm{a_ts_x\bv} \leq e^{-\gamma t}
\end{equation}
et
\begin{equation}\label{projection1}
 \norm{\pi^+(a_ts_x\bv)} \geq \frac{1}{2}\norm{a_ts_x\bv}.
\end{equation}
Notons $v$ l'image de $\bv$ dans $X$, et comme ci-dessus, soit $u\in\uu^-$ tel que $v=Pe^us_x$.
Grâce à l'expression de $a_ts_x\bv$ déjà utilisée précédemment, nous avons
%\begin{align*}
% a_ts_x\bv & = H_\chi(v) \chi(a_t) e^{-(\Ad a_t)u}e_\chi\\
%& = H_\chi(v) \chi(a_t) (e_\chi - ((\Ad a_t)u)e_\chi+\frac{1}{2}((\Ad a_t)u)^2e_\chi+\dots).
%\end{align*}
%et par conséquent
\note{L'égalité $\pi^+(a_ts_x\bv)\asymp H_\chi(v)\chi(a_t)e_\chi$ signifie que les deux vecteurs sont colinéaires, avec une constante de proportionalité bornée.}
$\pi^+(a_ts_x\bv)\asymp H_\chi(v)\chi(a_t)e_\chi$, et donc, d'après \eqref{projection1},
\[
\frac{1}{H_\chi(v)\chi(a_t)}\norm{a_ts_x\bv} \ll 1.
\]
%\begin{equation}\label{vplus}
% H_\chi(v) \chi(a_t) = \|\pi^+(a_ts_x\bv)\| \geq \frac{\|a_ts_x\bv\|}{2}\geq e^{-\gamma t}.
%\end{equation}
On utilise ensuite la stratification $\uu^-=m_1\oplus\dots\oplus m_r$, et l'on décompose $u=\sum_i u_i$ suivant cette somme directe.
Dans l'égalité
\begin{equation}\label{egdec}
e_\chi - \frac{1}{H_\chi(v)\chi(a_t)} a_ts_x\bv = \big(\sum_i((\Ad a_t)u_i)e_\chi\big) - \frac{1}{2}((\Ad a_t)u)^2e_\chi+\dots,
\end{equation}
le terme $((\Ad a_t)u_1)e_\chi$ est en somme directe avec tous les autres,
%ce terme correspond à la somme des espaces de poids de la forme $\chi-\alpha$, où $\alpha$ ne contient qu'une seule racine simple hors de $\theta$;
%tous les autres sont dans des espaces de poids associés à $\chi-\alpha$, où $\alpha$ contient au moins deux racines simples hors de $\theta$.
 et l'on peut donc majorer
\[ \norm{(\Ad a_t)u_1e_\chi} \ll 1.\]
Comme $P=\Stab[e_\chi]$, l'application $u\mapsto u\cdot e_\chi$ est un difféomorphisme local au voisinage de $0$, et par conséquent, 
\[ \norm{(\Ad a_t)u_1} \ll 1.\]
Dans \eqref{egdec}, on peut alors faire passer tous les termes du membre de droite de la forme $((\Ad a_t)u_1)^ie_\chi$ dans le membre de gauche, et cela n'augmente pas significativement la norme de ce dernier.
Le terme $((\Ad a_t)u_2)e_\chi$ est alors en somme directe avec tous les autres termes du membre de droite, ce qui permet de voir que $\|((\Ad a_t)u_2)e_\chi\| \ll 1$, puis
\[ \norm{(\Ad a_t)u_2} \ll 1.\]
Ainsi de proche en proche, on montre que pour chaque $i$,
$\norm{(\Ad a_t)u_i} \ll 1$,
de sorte qu'à la fin
$\norm{((\Ad a_t)u)e_\chi} \ll 1$, et donc aussi
\[ \abs{(\Ad a_t)u} \ll 1.\]
Par suite,
\[ d(x,v) \asymp \abs{u} = e^{-t} \abs{(\Ad a_t)u} \ll e^{-t}.\]
Or, l'inégalité \eqref{norme1} implique
$H_\chi(v)\chi(a_t) = H_\chi(v) e^{t\chi(Y)} \leq e^{-\gamma t}$,
d'où l'on tire $e^{t(\chi(Y)+\gamma)}\leq H_\chi(v)^{-1}$, puis\footnote{Noter que $\chi(Y)+\gamma<0$.}
\[
d(x,v) \ll H_\chi(v)^{\frac{1}{\chi(Y)+\gamma}}.
\]
Comme $\gamma$ peut être choisi arbitrairement proche de $\gamma_\chi(x)$, on obtient bien
\[
\beta_{\chi}(x) \geq \frac{1}{-\chi(Y)-\gamma_\chi(x)}.
\]
\end{proof}

Le premier résultat important impliqué par la proposition~\ref{dani} est le théorème suivant, selon lequel la fonction $\beta_\chi$ est constante presque partout sur $X$.

\begin{theorem}[Valeur presque sûre de l'exposant diophantien]
\label{exposantps}
Soit $G$ un $\Q$-groupe semi-simple, $P$ un $\Q$-sous-groupe parabolique, $X=P\bcs G$ la variété quotient, et $H_\chi$ la hauteur sur $X(\Q)$ associée au poids dominant $\chi$.
Il existe une constante $\beta_\chi(X)>0$ telle que pour presque tout $x$ dans $X(\R)$,
\[
\beta_\chi(x) = \beta_\chi(X).
\]
De plus, si $\Pi$ est une base d'un système de racines de $G$ et $\theta\subset\Pi$ est telle que $P=P_\theta$, alors
\[
\beta_\chi(X) = \frac{-1}{\chi(Y)},
\]
où $Y\in\ka$ est l'élément défini par
\[
\alpha(Y) = \left\{
\begin{array}{ll}
0 & \mathrm{si}\ \alpha\in\theta\\
-1 & \mathrm{si}\ \alpha\in\Pi\setminus\theta.
\end{array}
\right.
\]
\end{theorem}

\begin{remark}
La formule $\beta_\chi(X)=\frac{-1}{\chi(Y)}$ implique en particulier que $\beta_\chi(X)$ est rationnel.
En effet, il existe un entier $\ell\geq 1$ tel que pour tout élément $\chi$ dans le réseau des poids, on ait $\chi(Y)\in\frac{1}{\ell}\Z$.
Si le poids $\chi$ est choisi dans le réseau engendré par les racines, on a même $-\chi(Y)\in\N^*$.
\end{remark}

\begin{proof}
D'après la proposition~\ref{dani}, il suffit de faire voir que pour presque tout $x$ dans $X(\R)$, $\gamma_\chi(x)=0$.
Notons $\Gamma$ le sous-groupe arithmétique de $G$ qui stabilise le réseau $V_\chi(\Z)$.
% et choisissons pour $V_\chi(\Z)$ un réseau stable par $\Gamma$; cela est toujours possible, d'après \cite[proposition~7.12]{borel_iga}.
Tout d'abord, par ergodicité du flot $(a_t)$ sur l'espace $\Omega=G/\Gamma$, pour presque tout $s\in G$,
\[
\lim_{t\to\infty} \frac{1}{t}\log\lambda_1(a_tsV_\chi(\Z))=0.
%\dots=\lim_{t\to\infty}\log\lambda_{\dim V_\chi}(a_tsV_\chi(\Z))
\]
\details{
En effet, si $\lambda_1(a_tsV_\chi(\Z))\leq e^{-\tau t}$, alors pour tout $t'\in [(1-\tau/2)t,t]$, $\lambda_1(a_{t'}sV_\chi(\Z))\leq e^{-\tau t/2}\leq \eps$, et par conséquent $\frac{1}{t}\abs{\{t'\in[0,t]\ |\ \lambda_1(a_{t'}sV_\chi(\Z)\leq\eps\})}\geq \tau/2$. Mais par ergodicité, le membre de gauche converge vers $m(\{\Delta\ |\ \lambda_1(\Delta)\leq\eps\})$, qui tend vers zéro lorsque $\eps$ tend vers zéro.
Noter que $g\mapsto\lambda_1(gV_\chi(\Z))$ définit bien une fonction sur $G/\Gamma$, car $V_\chi(\Z)$ est stable par $\Gamma$.
}
Or, cette limite ne dépend que de la classe $Ps$ de $s$ modulo $P$, et donc, pour presque tout $x=Ps_x$ dans $X(\R)$,
\begin{equation}\label{lambda1}
\lim_{t\to\infty} \frac{1}{t}\log\lambda_1(a_ts_xV_\chi(\Z))=0.
\end{equation}
Comme $\lambda_1(a_ts_xV_\chi(\Z))\leq r_\chi(a_ts_xV_\chi(\Z))$, cela montre déjà, pour presque tout $x$ dans $X(\R)$, $\gamma_\chi(x)\leq 0$.

Pour l'inégalité réciproque, nous utiliserons le théorème~\ref{reduction}, tiré de la théorie de la réduction des groupes arithmétiques, dont les résultats principaux sont rappelés au chapitre~\ref{chap:reduction}.
Pour $t>0$ grand, ce théorème nous donne une décomposition de Siegel de $a_ts_x$:
\[
a_ts_x = kan\gamma \quad\mbox{avec}\ k\in K,\ a\in A_\tau,\ n\in\omega, \gamma\in C\Gamma,
\]
où $C$ est une partie finie de $G(\Q)$, $\Gamma$ le stabilisateur de $V_\chi(\Z)$ dans $G$, et $A_\tau$ et $\omega$ sont définis au paragraphe~\ref{sec:redbase}.
Comme $\lambda_1(a_ts_xV_\chi(\Z))= e^{o(t)}$, on a aussi
\[
\norm{a} = e^{o(t)}.
\]
Soit $(u_1,\dots,u_k)$ une base de $V_\chi(\Q)$ constituée d'éléments de $\tx\cap V_\chi(\Z)$.
Le sous-réseau qu'elle engendre est d'indice fini dans $V_\chi(\Z)$.
Si $D\in\N^*$ est un dénominateur commun aux coefficients des éléments de $C$ dans la représentation sur $V_\chi$, la base $(Da_ts_x\gamma^{-1}u_1,\dots,Da_ts_x\gamma^{-1}u_k)$ est constituée d'éléments de norme $e^{o(t)}$ appartenant à $\tx\cap a_ts_xV_\chi(\Z)$, et engendre un réseau de covolume borné indépendamment de $t$.
Par conséquent, pour tout $\eps>0$, on peut trouver $i$ tel que l'élément $\bv_i=a_ts_x\gamma^{-1}u_i$ vérifie $\frac{\norm{\pi^{+}(\bv_i)}}{\norm{\bv_i}}\geq e^{-\eps t}$, et quitte à diminuer un peu $t$, on aura même $\frac{\norm{\pi^{+}(\bv_i)}}{\norm{\bv_i}}\geq \frac{1}{2}$.
Comme on a aussi $\norm{\bv_i}=e^{o(t)}$, avec $\eps>0$ arbitrairement petit, cela donne bien $\gamma_\chi(x)\geq 0$, ce qu'il fallait démontrer.
\end{proof}

Notons que la démonstration ci-dessus se fonde implicitement sur le lemme suivant, qui découle de la proposition~\ref{dani}, et dont nous ferons encore usage dans les chapitres~\ref{chap:algebrique} et \ref{chap:sousvariete}.

\begin{lemma}
\label{daniextremal}
Soit $G$ un $\Q$-groupe semi-simple, $P$ un $\Q$-sous-groupe parabolique, et $X=P\bcs G$ la variété quotient.
Fixons aussi une hauteur $H_\chi$ sur $X(\Q)$, donnée par un poids dominant $\chi$ associé à $P$.

Soit $x\in X(\R)$ et $s_x\in G$ tel que $x=Ps_x$.
Si $\lim_{t\to\infty}\frac{1}{t}\log\lambda_1(a_ts_xV_\chi(\Z))=0$, alors
\[
\beta_\chi(x) = \beta_\chi(X).
\]
\end{lemma}

\section{Points mal approchables}

Par analogie avec la terminologie utilisée pour $\PP^1(\R)$, nous définissons l'ensemble $\BA_X$ des points \emph{mal approchables} dans $X$ par
\[
\BA_X = \{x\in X(\R)\ |\ \exists c>0:\forall v\in X(\Q),\ d(x,v)\geq cH_\chi(v)^{-\beta_\chi}\}.
\]
Nous montrons ci-dessous que cet ensemble est non vide, et même que sa dimension de Hausdorff est maximale, ce qui généralise un résultat bien connu de Schmidt pour l'espace projectif \cite{schmidt_games}.
Notre démonstration utilise une variante des jeux de Schmidt, similaire à celle utilisée dans un article en commun avec Dmitry Kleinbock \cite{simplex} pour montrer la propriété dans le cas particulier où $X$ est une quadrique projective.
Insistons cependant sur le fait que la distance utilisée ci-dessous sur $X(\R)$ est celle associée à la métrique de Carnot-Carathéodory introduite à la partie~\ref{sec:cc}.

Commençons par expliquer la règle du jeu de Schmidt, qui dépend d'un paramètre fixe $\alpha\in]0,\frac{1}{3}[$.
Le jeu se joue entre deux joueurs Alice et Bob.
Pour commencer la partie, Alice choisit une boule $B_0=B(x_0,\rho_0)$ dans $X(\R)$.
Bob peut alors ôter à $B_0$ un voisinage de taille au plus $\rho_1=\alpha\rho_0$ d'une sous-variété linéaire stricte $L$, i.e. de l'intersection de $X(\R)$, vue comme une partie de $V_\chi(\R)$, avec un hyperplan de $V_\chi$.
Puis Alice doit choisir une boule $B_1$ de rayon $\rho_1$ dans la partie restante $B_0\setminus L^{(\alpha\rho_0)}$.
Ensuite, Bob peut supprimer de $B_1$ un voisinage de taille au plus $\rho_2=\alpha\rho_1$ d'une sous-variété linéaire, et ainsi de suite.
Comme la suite des rayons $(\rho_n)_{n\geq 0}$ converge vers $0$, l'intersection décroissante $\cap_{n\geq 0} B_i$ est réduite à un singleton:
\[
\bigcap_{n\geq 0} B_n = \{x_\infty\}.
\]
Une partie $S\subset X(\R)$ est dite \emph{$\alpha$-gagnante} si Bob peut toujours jouer de sorte que $x_\infty\in S$.
Une partie $S$ est dite \emph{gagnante} si elle est $\alpha$-gagnante pour $\alpha>0$ arbitrairement proche de zéro.
Deux propriétés importantes des ensembles gagnants sont données dans la proposition ci-dessous, qui se démontre avec des arguments très semblables à ceux utilisés par Schmidt dans son article fondateur \cite{schmidt_games}, auquel on renvoie le lecteur pour plus de détails sur le sujet.

\begin{proposition}
Soit $X$ une variété de drapeaux munie de la distance de Carnot-Carathéodory naturelle et d'une hauteur associée à un poids dominant $\chi$.
\begin{enumerate}
\item Si $(S_i)$ est une famille dénombrable de parties gagnantes dans $X(\R)$, alors $S=\bigcap_iS_i$ est aussi une partie gagnante.
\item Si $S$ est une partie gagnante de $X(\R)$, alors $S$ est dense dans $X(\R)$ et $\dim_HS=\dim X$.
\end{enumerate}
\end{proposition}
\begin{proof}
Pour le premier point, fixons $\alpha\in]0,1/3[$.
En suivant aux temps $1,3,5,\dots$ une stratégie $\alpha^2$-gagnante pour $S_1$, Bob peut s'assurer que $x_\infty\subset B_0\cap B_2\cap B_4 \dots$ appartient à $S_1$.
Aux temps $2,6,10,\dots$ une stratégie $\alpha^4$-gagnante pour $S_2$ permet aussi de s'assurer que $x_\infty\subset B_1\cap B_5\cap\dots$ soit dans $S_2$.
Et ainsi de suite, pour chaque $k\geq 1$, Bob peut choisir une stratégie $\alpha^{2^k}$ gagnante pour $S_k$, qu'il joue aux temps $2^{k-1}(2n+1)$, $n\in\N$ pour avoir toujours $x_\infty\in S_k$.
Cela donne bien $x_\infty\in S=\bigcap_iS_i$, et $S$ est donc bien un ensemble gagnant.

Pour le second point, la propriété de $S$ pour un petit paramètre $\alpha>0$ permet de construire une suite décroissante de fermés $(K_n)_{n\geq 1}$ telle qu'à chaque étape, pour une constante $c>0$ indépendante de $n$ et $\alpha$, la partie $K_n$ soit réunion d'au moins $(c\cdot\alpha^{\dim_{cc} X})^n$ boules de rayon $\alpha^n$, et telle que l'intersection $K=\bigcap_{n\geq 1}K_n$ soit incluse dans $S$.
Pour la métrique de Carnot-Carathéodory, on vérifie sans peine que $\dim_{H,cc} K \geq \frac{(\dim_{cc}X)\abs{\log\alpha}-\abs{\log c}}{\abs{\log\alpha}}$, d'où l'on tire, en faisant tendre $\alpha$ vers $0$, $\dim_{H,cc} S = \dim_{cc}X$.
Cela implique la même égalité pour la métrique riemannienne: $\dim_H S=\dim X$.
\end{proof}

\begin{remark}
Dans le premier point, la dimension de Hausdorff est celle associée à une métrique riemannienne sur $X(\R)$, et $\dim X$ désigne la dimension de $X$ au sens usuel du terme.
Mais la démonstration de la proposition montre plutôt l'égalité analogue pour la métrique de Carnot-Carathéodory $\dim_{H,cc}S=\dim_{cc}X$, qui lui est en fait équivalente.
\end{remark}

\begin{theorem}[Propriété de Schmidt des points mal approchables]
\label{bax}
Soit $X$ une variété de drapeaux munie de la distance de Carnot-Carathéodory naturelle et d'une hauteur associée à un poids dominant $\chi$.
L'ensemble $\BA_X$ est un ensemble gagnant au sens défini ci-dessus.
En particulier,
\[
\dim_H\BA_X = \dim X.
\]
\end{theorem}

Pour démontrer ce théorème, nous utiliserons le lemme suivant.

\begin{lemma}[Lemme du simplexe]
\label{simplex}
Soit $X$ une variété de drapeaux munie de la distance de Carnot-Carathéodory naturelle et d'une hauteur associée à un poids dominant $\chi$.
Il existe une constante $c>0$ telle que pour toute boule $B(x,\rho)\subset X(\R)$ de rayon $\rho\in]0,1[$, l'ensemble
\[
E_{x,\rho}= B(x,\rho) \cap \{v\in X(\Q)\ |\ H(v)\leq c\rho^{-\frac{1}{\beta_\chi(X)}}\}
\]
soit inclus dans une sous-variété linéaire stricte de $X$.
\end{lemma}
\begin{proof}
On reprend les notations utilisées dans la démonstration de la proposition~\ref{dani}.
En particulier si $B_\rho=B(x,\rho)$, on note $s_x\in K$ un élément tel que $Ps_x=x$.
Soit $v\in X(\Q)\cap B_\rho$ tel que $H(v)\leq c\rho^{-\frac{1}{\beta_\chi(X)}}$, et $\bv$ un représentant dans $V_\chi(\Z)$.

Soit $t>0$ tel que $e^{-t\chi(Y)}=\rho^{-\frac{1}{\beta_\chi(X)}}$.
Si $c>0$ est choisi suffisamment petit, le calcul utilisé dans la première partie de la démonstration de la proposition~\ref{dani} montre que $\norm{a_ts_x\bv} < 1$.
Comme le réseau $a_ts_xV_\chi(\Z)$ est de covolume égal à 1, cela implique que tous les représentants $\bv$ des éléments $v\in E_{x,\rho}$ sont inclus dans un hyperplan de $V_\chi$, i.e. que $E_{x,\rho}$ est inclus dans une sous-variété linéaire de $X$.
\end{proof}

\begin{proof}[Démonstration du théorème~\ref{bax}]
Notons pour simplifier $\beta=\beta_\chi(X)$, et fixons un petit paramètre $\alpha>0$.
Pour commencer, Alice choisit une boule $B_0=B(x_0,\rho_0)$.
D'après le lemme~\ref{simplex}, il existe une constante $c>0$ telle que tous les points rationnels $v\in X(\Q)$ dans $2B_0$ tels que $H(v)\leq c\rho_0^{-\frac{1}{\beta}}$ soient inclus dans une sous-variété linéaire $L_0$.
Bob peut alors supprimer le voisinage $L_0^{(\alpha\rho_0)}$ de cette sous-variété.
Et de même, dans toute la suite de la partie, chaque fois qu'Alice a choisi une boule $B_i=B(x_i,\rho_i)$, où $\rho_i=\rho_0\alpha^i$, les points rationnels $v\in 2B_i$ tels que $H(v)\leq c\rho_i^{-\frac{1}{\beta}}$ sont tous dans une sous-variété linéaire $L_i$, et Bob peut supprimer le petit voisinage $L_i^{(\alpha\rho_i)}$.

Cette stratégie implique que $x_\infty \in \BA_X$.
En effet, si $v\in X(\Q)$, on peut choisir $i$ tel que
\begin{equation}\label{rhoi}
c\rho_{i-1}^{-\frac{1}{\beta}}
\leq H(v) \leq c\rho_{i}^{-\frac{1}{\beta}}.
\end{equation}
Si $v\not\in 2B_i$, alors, comme $x_\infty\in B_i$, on trouve
\[
d(x_\infty,v)
\geq \rho_i
= \alpha\rho_{i-1}
\geq \alpha c^\beta H(v)^{-\beta}
\]
tandis que si $v\in 2B_i$, l'inégalité \eqref{rhoi} implique $v\in L_i$, et comme $x_\infty\in B_{i+1}$,
\[
d(x_\infty,v)
\geq \alpha\rho_i
= \alpha^2\rho_{i-1}
\geq \alpha^2 c^\beta H(v)^{-\beta}.
\]
Posant $c_0=c^\beta\alpha^2$, on trouve
\[ \forall\, v\in X(\Q),\ d(x_\infty,v)\geq c_0H(v)^{-1},\]
et donc $x_\infty\in\BA_X$.
\end{proof}

\section{Un cas particulier: la hauteur anti-canonique}

D'après un résultat de Franke \cite{fmt}, lorsque la variété de drapeaux est munie de la hauteur anti-canonique,
le nombre $N_X(T)$ de points rationnels de hauteur anti-canonique au plus $T$ dans $X$ vérifie, pour certaines constantes $c>0$ et $b\in\N^*$,
%\[
%\log N_X(T) \sim \log T, \hspace{1cm} (T\to\infty).
%\]
%$b=\rang G-\rang P$ désigne le rang de la variété de drapeaux, le nombre de points rationnels de hauteur anti-canonique au plus $T$ dans $X$ vérifie
\[
N_X(T) \sim c\cdot T (\log T)^{b-1}, \hspace{1cm} (T\to\infty).
\]
Si ces points sont bien répartis, on peut donc s'attendre à ce que la valeur presque sûre $\beta_\chi(X)$ de l'exposant diophantien soit donnée par $\frac{1}{\dim_{cc}X}$, où $\dim_{cc}X$ est la dimension de $X$ pour la distance de Carnot-Carathéodory, donnée par l'équivalent asymptotique $N(X,\delta)\asymp \delta^{-\dim_{cc}X}$, où $N(X,\delta)$ désigne le nombre de boules de rayon $\delta>0$ nécessaire pour recouvrir l'espace compact $X(\R)$.
Le théorème suivant confirme cet argument heuristique.

\begin{theorem}[Exposant diophantien pour la hauteur anti-canonique]
\label{dac}
Soit $G$ un $\Q$-groupe semi-simple, $P$ un $\Q$-sous-groupe parabolique, et $X=P\bcs G$ la variété quotient.
On munit $X$ de la hauteur anti-canonique $H_\chi$ associée à la somme des racines apparaissant dans le radical unipotent de $P$.
Alors
\[
\beta_\chi(X) = \frac{1}{\dim_{cc} X}
\]
où $\dim_{cc} X$ est la dimension de Carnot-Carathéodory de $X$.
\end{theorem}
\begin{proof}
Il suffit de vérifier que la formule générale donnée par le théorème~\ref{exposantps} donne la valeur souhaitée dans ce cas particulier.

Les racines qui apparaissent dans le radical unipotent de $U$ sont exactement celles dont la décomposition en racines simples contient un élément hors de la partie $\theta$ associée à $P$.
Soit $\chi=\sum_{\alpha\in \Sigma_U} \alpha$ la somme de ces racines.
On décompose cette somme suivant le nombre d'éléments hors de $\theta$ contenus dans la racine $\alpha$.
Vue la définition de l'élément $Y$, et en utilisant aussi le fait que $U$ est naturellement isomorphe au sous-groupe unipotent $U^-$ opposé à $P$,
\[
\chi(Y) = -\sum_{i\geq 1} i\dim m_i = -\dim_{cc} X.
\]
\end{proof}

Cette valeur explicite de $\beta_\chi(X)$, combinée avec l'équivalent asymptotique de Franke $\log N_X(T)\sim\log T$ et un principe de transfert dû à Beresnevich et Velani~\cite{bv_transfert} permet même de calculer la dimension de Hausdorff des points de $X(\R)$ tels que $\beta_\chi(x)=\beta$.

\begin{theorem}[Théorème de Jarník pour une variété de drapeaux]
\label{jarnik}
Soit $X$ une variété de drapeaux rationnelle munie de sa métrique de Carnot-Carathéodory et de la hauteur anti-canonique $H_\chi$.
Pour tout $\beta\geq\frac{1}{\dim_{cc}X}$, l'ensemble
\[
E_\beta = \{x\in X(\R)\ |\ \beta(x)\geq\beta\}
\]
vérifie
\[
\dim_{H,cc} E_\beta = \frac{1}{\beta}.
\]
\end{theorem}
\begin{proof}
Soit $\beta\geq\frac{1}{\dim_{cc}X}$.
Commençons par vérifier que si $s>\frac{1}{\beta}$, alors la mesure de Hausdorff de dimension $s$ de $E_\beta$ est nulle: $\cH^s_{cc}(E_\beta)=0$.
Pour cela, on utilise le recouvrement de $E_\beta$ par les boules $B(v,H(v)^{-\beta})$, $v\in X(\Q)$:
\[
E_\beta \subset \bigcup_{H(v)\geq 2^{n_0}} B(v,H(v)^{-\beta})
\]
et on majore
\begin{align*}
\sum_{H(v)\geq 2^{n_0}} H^(v)^{-\beta s} & \leq \sum_{n\geq n_0} 2^{-n\beta s} N_X(2^{n+1})\\
& \leq C_\eps \sum_{n\geq n_0} 2^{-n(\beta s-1-\eps)},
\end{align*}
où la deuxième inégalité provient de la majoration $N_X(T)\leq C_\eps T^{1+\eps}$, qui découle de l'équivalent asymptotique de Franke rappelé ci-dessus.
Si $\eps>0$ est choisi assez proche de zéro pour que $\beta s-1-\eps>0$, la somme converge, et en faisant tendre $n_0$ vers l'infini, on trouve donc $\cH^s_{cc}(E_\beta)=0$.
Cela montre déjà $\dim_{H,cc} E_\beta \leq \frac{1}{\beta}$.

Pour montrer l'inégalité réciproque, on utilise le principe de transfert de Beresnevich et Velani \cite[Theorem~3]{bv_transfert}.
Remarquons que la mesure de Lebesgue sur $X(\R)$ est équivalente à la mesure de Hausdorff en dimension $\dim_{cc} X$ définie pour la métrique de Carnot-Carathéodory.
Pour abréger les notations, posons $g:x\mapsto x^{\dim_{cc}X}$ et notons $\cH^g_{cc}$ la mesure de Hausdorff associée à cette fonction, pour la métrique de Carnot-Carathéodory.
Pour certaines constantes $c_1,C_2>0$, la mesure d'une boule de rayon $r\leq 1$ dans $X(\R)$ est contrôlée par
\[
c_1g(r) \leq \cH^g_{cc}(B(x,r))  \leq C_2 g(r).
\]
Suivant les notations de \cite{bv_transfert}, si $f:\R^+\to\R^+$ est une fonction croissante bijective, et $B=B(x,r)$ une boule dans $X$, on note
\[
B^f=B(x,g^{-1}f(r)) = B(x, (f(r))^{\frac{1}{\dim_{cc}X}}).
\]
En particulier, $B^g=B$.
Fixons $\beta>\frac{1}{\dim_{cc}X}$, et notons $(B_i)_{i\in\N}$ la famille de boules
\[
B(v,H(v)^{-\beta}),\quad v\in X(\Q)
\]
de sorte que
\[
E_\beta \supset \limsup_i B_i^g.
\]
Soit $\eps>0$ et $f:x\mapsto x^{\frac{1}{\beta+\eps}}$.
D'après le théorème~\ref{dac} l'ensemble $\limsup_i B_i^f$ est de mesure pleine:
\[
\forall B\subset X(\R),\quad \cH^g(B\cap \limsup_iB_i^f) = \cH^g(B).
\]
D'après \cite[Theorem~3]{bv_transfert}, cela implique
\[
\forall B\subset X(\R),\quad \cH^f(B\cap \limsup_i B_i^g) = \cH^f(B)
\]
et donc 
\[
\dim_{H,cc} E_\beta \geq \frac{1}{\beta+\eps}.
\]
\end{proof}

\begin{remark}
Lorsque $X$ n'est pas munie de la hauteur anti-canonique, on peut toujours utiliser l'équivalent de Mohammadi et Salehi Golsefidy \cite{msg}
\[
N_X(T)\sim c T^{u_\chi}(\log T)^{v_\chi}
\]
et le lemme de Borel-Cantelli pour montrer que $\beta_\chi(X)\leq\frac{u_\chi}{\dim_{cc}X}$.
Mais cette inégalité peut être stricte.
\comm{En effet, la formule $\beta_\chi(X)=\frac{1}{\chi(Y)}$ peut s'écrire
\(
\beta_\chi(X) = \frac{1}{\bracket{\chi,\sum_{i\in\theta}\varpi_i}}
\)
et $\dim_{cc}X = \bracket{\rho,\sum_{i\in\theta}\varpi_i}$, où $\rho$ est la somme des racines de l'unipotent de $P$, donc
\[
(\dim_{cc}X)\beta_\chi(X) = \frac{\bracket{\rho,\sum_{i\in\theta}\varpi_i}}{\bracket{\chi,\sum_{i\in\theta}\varpi_i}}.
\]
D'autre part, la formule donnée pour $u_\chi$ dans \cite{msg} est
\[
u_\chi = \max_{i\in\theta}\frac{\bracket{\rho,\alpha_i}}{\bracket{\chi,\alpha_i}}.
\]
Et il est facile de construire un exemple où ces deux quotients prennent des valeurs différentes.
}
\comm{Remarquons qu'une inégalité stricte $\beta_\chi(X)<\frac{u_\chi}{\dim_{cc}X}$ signifie que les points rationnels de hauteur au plus $T$ ne sont pas répartis uniformément dans $X$. On peut supposer qu'ils s'accumulent sur certaines familles de sous-variétés, sans doute des variétés de Schubert.
À titre d'analogie, on peut remarquer que si on mesure la hauteur d'un point rationnel $v=(p_1/q_1,p_2/q_2)$ dans $[0,1]^2$ par $H(v)=q_1+q_2^\tau$, il y a environ $T^{1+\frac{1}{\tau}}$ points de hauteur inférieure à $T$, et pourtant, l'exposant diophantien presque sûr n'est jamais supérieur à $2$, qui est l'exposant diophantien obtenu pour la première coordonnée. Cela s'explique par le fait que les points rationnels sont concentrés sur les droites rationnelles verticales. Faire un dessin.}

De même, la première partie de la démonstration du théorème~\ref{jarnik} ci-dessus permet de voir que
\[
\forall\beta\geq\frac{u_\chi}{\dim_{cc}X},\quad
\dim_{H,cc} E_\beta \leq \frac{u_\chi}{\beta}
\]
tandis que le principe de transfert de Beresnevich et Velani implique que
\[
\forall \beta\geq\beta_\chi(X),\quad
\dim_{H,cc} E_\beta \geq \frac{\beta_\chi(X)\dim_{cc}X}{\beta}.
\]
Mais ces deux bornes ne coïncident pas en général.

\end{remark}

\chapter{Le théorème de Khintchine}
\label{chap:khintchine}

Comme précédemment, $X$ désigne une variété de drapeaux, donnée sous la forme $X=P\bcs G$, où $G$ est un $\Q$-groupe semi-simple, et $P$ un $\Q$-sous-groupe parabolique.
La hauteur $H_\chi$ sur $X(\Q)$ est donnée par le choix d'un poids dominant $\chi$ de $G$ associé à $P$, et la distance est celle associée à la métrique de Carnot-Carathéodory usuelle sur $X$.
Cela permet de définir l'exposant diophantien $\beta_\chi(x)$ d'un point $x$ dans $X(\R)$.
Nous avons vu à la partie précédente que cet exposant diophantien est constant presque partout sur $X(\R)$: il existe une constante $\beta_\chi(X)$ telle que pour presque tout $x$ dans $X(\R)$, l'inégalité
\[
d(x,v) \leq H(v)^{-\beta},
\]
admet une infinité de solutions $v\in X(\Q)$ si $\beta<\beta_\chi(X)$, et un nombre fini de solutions si $\beta>\beta_\chi(X)$.
Pour étudier plus précisément les propriétés diophantiennes d'un point aléatoire de $X(\R)$, nous nous intéressons donc à l'inégalité
\begin{equation}\label{khineq}
d(x,v) \leq H(v)^{-\beta_\chi(X)}\psi(H(v)),
\end{equation}
où $\psi:\R^+\to\R^+$ est une fonction décroissante.
Le théorème~\ref{khintchine} ci-dessous, analogue du célèbre théorème de Khintchine \cite{khintchine}, donne une condition nécessaire et suffisante sur $\psi$ pour que \eqref{khineq} ait une infinité de solutions pour presque tout $x$ dans $X(\R)$.

\bigskip

Pour sa démonstration, nous suivons l'approche développée par Kleinbock et Margulis \cite{km_loglaws,km_loglaws_erratum} pour démontrer le théorème de Khintchine dans l'espace projectif: nous ramenons ce théorème à un énoncé sur le comportement asymptotique de certaines orbites diagonales dans l'espace de réseaux $\Omega=G/\Gamma$, et utilisons la propriété de mélange exponentiel du flot $(a_t)$ sur $\Omega$.
Dans le cas où $X$ est une sphère projective, le théorème~\ref{khintchine} est dû à Kleinbock et Merrill \cite{kleinbockmerrill}, qui ont ensuite généralisé leur résultat à une quadrique projective arbitraire, dans un travail en commun avec Fishman et Simmons \cite{fkms}.
Quelques complications interviennent dans notre situation, notamment parce que les ensembles à atteindre ne sont plus à proprement parler des voisinages de l'infini dans $\Omega$.
Nous commençons la démonstration par un paragraphe un peu calculatoire, sur le volume de certaines parties de $\Omega$, parce que cela fera apparaître les paramètres utiles à l'énoncé précis du théorème de Khintchine.

\section{Voisinages de l'infini}

Ce paragraphe a pour but la proposition~\ref{asymp} ci-dessous, qui décrit le comportement asymptotique du volume de certaines portions de  voisinages de l'infini dans $G/\Gamma$.
Dans tout le paragraphe, $G$ désigne un $\Q$-groupe semi-simple.
Étant donnée une $\Q$-représentation linéaire $V_\chi$ de $G$ engendrée par une unique droite rationnelle $\R e_\chi$ de plus haut poids $\chi$, on note
\[
P = \Stab_G \R e_\chi
\]
le stabilisateur dans $G$ de la direction engendrée par $e_\chi$, et
\[
\tx=G\cdot e_\chi
\]
l'orbite du vecteur de plus haut poids sous l'action de $G$.
On considère un réseau rationnel $V_\chi(\Z)$ dans $V_\chi$, et le stabilisateur de $V_\chi(\Z)$ dans $G$ est noté $\Gamma$; c'est un sous-groupe arithmétique de $G$.

Notons enfin $\Omega=G/\Gamma$, que nous identifierons naturellement à une partie de l'espace des réseaux de $V_\chi$, via l'application $g\Gamma\mapsto gV_\chi(\Z)$.
Rappelons que sur cet espace de réseaux, la fonction $r_\chi$ est définie par
\[
r_\chi(\Delta) = \inf\{ r>0\ |\ \exists \bv\in\Delta\cap\tx:\ \norm{\bv}\leq r\ \mathrm{et}\ \norm{\pi^+(\bv)}\geq\frac{\norm{\bv}}{2}\},
\]
où $\pi^+:V_\chi\to V_\chi$ désigne la projection sur $e_\chi$ parallèlement aux autres espaces de poids.

\begin{proposition}[Volume des voisinages de l'infini]
\label{asymp}
Pour $r>0$, on considère le voisinage de l'infini dans $\Omega$ défini par
\[
\Omega_r = \{ \Delta\in\Omega \ |\ \min_{\bv\in\Delta\cap\tx}\norm{\bv}\leq r\},
\]
et la sous-partie
\[
\Omega_r' = \{ \Delta\in\Omega \ |\ r_\chi(\Delta) \leq r\}.
\]
Il existe des constantes $C,a_\chi,b_\chi>0$ telles que pour tout $r$ suffisamment petit,
\[
\frac{1}{C}r^{a_\chi}\abs{\log r}^{b_\chi-1}
\leq m_\Omega(\Omega_r')
\leq m_\Omega(\Omega_r) \leq Cr^{a_\chi}\abs{\log r}^{b_\chi-1}.
\]
Si $(Y_i)_{1\leq i\leq r}$ désigne la base duale de la base des racines simples, et $\rho$ la somme des racines positives comptées avec multiplicité, alors
% \comm{i.e. les vecteurs sur les arêtes de $\ka^+$},
\[
a_\chi=\min_{1\leq i\leq r} \frac{\rho(Y_i)}{\chi(Y_i)}
\quad\mathrm{et}\quad
b_\chi=\card\{ i\in[1,r]\ |\ \frac{\rho(Y_i)}{\chi(Y_i)}=a_\chi\}.
\]
\end{proposition}

La démonstration de cette proposition est basée sur la théorie de la réduction, telle qu'elle est exposée dans Borel \cite[\S\S 15 et 16]{borel_iga}.
Nous aurons en particulier besoin de la notion d'ensemble de Siegel.
Notons $B$ un $\Q$-sous-groupe parabolique minimal de $G$ inclus dans $P$, $T$ un tore $\Q$-déployé maximal dans $B$, $A=T^0(\R)$ la composante connexe des points réels de $T$, $N$ le radical unipotent de $B$, et $M$ le $\Q$-sous-groupe anisotrope maximal du centralisateur $Z(T)^0$ de $T$ dans $G^0$.
Le lecteur est renvoyé à Borel \cite[\S11]{borel_iga} et Borel et Tits \cite{boreltits} pour les résultats fondamentaux concernant la structure des $\Q$-sous-groupes paraboliques de $G$.

Soit $\ka$ l'algèbre de Lie de $A$.
Le système de racines $\Sigma$ de $G$ par rapport à $T$ s'identifie à un système de racines dans l'espace dual $\ka^*$.
On fixe un système de racines simples $\Pi=\{\alpha_1,\dots,\alpha_r\}$ pour un ordre associé à $B$.
Pour $\tau\geq 0$, on définit un voisinage $\ka_\tau^-$ de $\ka^-$ par
\[
\ka^-_\tau = \{Y\in\ka\ |\ \forall \alpha\in\Pi,\, \alpha(Y)\leq \tau\},
\]
et
\[
A_\tau = \exp\ka^-_\tau \subset A.
\]
Un \emph{ensemble de Siegel} $\FS$ de $G$ sur $\Q$ est un ensemble de la forme
\[
\FS=KA_\tau \omega,
\]
où $K$ désigne un sous-groupe compact maximal de $G$, et $\omega$ un voisinage compact de l'identité dans les points réels de $MN$.

\begin{proof}[Démonstration de la proposition~\ref{asymp}]
Nous supposerons dans cette démonstration que la norme sur $V_\chi$ est invariante par le sous-groupe compact maximal $K$; on peut toujours se ramener à ce cas, par équivalence des normes en dimension finie.
Soit $C\subset G(\Q)$ un ensemble de représentants des classes de $P(\Q)\bcs G(\Q)/\Gamma$.
D'après Borel \cite[proposition~15.6]{borel_iga}, l'ensemble $C$ est fini.
Soit
\[
\tOmega_r=\{g\Gamma\in\Omega \ |\ \exists \gamma\in\Gamma,\, c\in C:\ \norm{g\gamma^{-1}c^{-1}e_\chi}\leq r\}.
\]
Comme $C$ est fini, il existe une constante $C_0>0$ telle que pour tout $r>0$,
\[
\Omega_{r/C_0}\subset \tOmega_r\subset \Omega_{C_0r}.
\]
De même, avec
\[
\tOmega_r'=\{g\Gamma\in\Omega \ |\ \exists \gamma\in\Gamma,\, c\in C:\ \norm{g\gamma^{-1}c^{-1}e_\chi}\leq r\ \mathrm{et}\ \frac{\norm{\pi^+(g\gamma^{-1}c^{-1}e_\chi)}}{\norm{g\gamma^{-1}c^{-1}e_\chi}}\geq\frac{1}{2}\},
\]
on a
\[
\Omega'_{r/C_0}\subset \tOmega'_r\subset \Omega'_{C_0r}.
\]
Il suffit donc de démontrer la proposition pour les ensembles $\tOmega_r$ et $\tOmega_r'$.
\comm{Les relations de comparaisons ci-dessus sont encore valables si l'on définit $\tOmega_r$ et $\tOmega_r'$ à partir d'un ensemble fini $C$ contenant strictement une famille de représentants de $P\bcs G(\Q)/\Gamma$. C'est ce que nous ferons ci-dessous en appliquant Borel \cite[théorème~16.9]{borel_iga}.}

Fixons maintenant un ensemble de Siegel $\FS$ qui satisfasse la conclusion de \cite[Théorème~16.9]{borel_iga} pour la fonction de type $(P,\chi)$ définie par $\Phi(g)=\norm{g e_\chi}$.
\comm{On vérifie sans peine que cette fonction est bien invariante par $\Gamma\cap P$ et par $L_{\theta'}$, où $L_{\theta'}$ est un sous-groupe semi-simple de $P_{\theta'}=\Stab e_\chi$.}
Soit $\tphi$ la fonction sur $G$ définie par
\[
\tphi(g) = \sum_{c\in C} \mathbbm{1}_{\{gc^{-1}\in\FS\ \mathrm{et}\ \norm{gc^{-1}e_\chi}\leq r\}}.
\]
\comm{Dans cette formule et ci-dessous, $C$ désigne l'ensemble $C'$ donné par \cite[théorème~16.9]{borel_iga}.}
La projection de $\tphi$ sur $G/\Gamma$ est donnée par
\begin{align*}
\phi(g\Gamma) & = \sum_{\gamma\in\Gamma}\sum_{c\in C} \mathbbm{1}_{\{g\gamma c^{-1}\in\FS\ \mathrm{et}\ \norm{g\gamma c^{-1}e_\chi}\leq r\}}\\
& \geq \mathbbm{1}_{\{\exists\gamma, c\ :\ g\gamma c^{-1}\in\FS\ \mathrm{et}\ \norm{g\gamma c^{-1}e_\chi}\leq r\}}\\
& = \mathbbm{1}_{\{\exists\gamma, c\ :\ \norm{g\gamma c^{-1}e_\chi}\leq r\}}\\
& = \mathbbm{1}_{\tOmega_r}(g\Gamma),
\end{align*}
où l'avant-dernière égalité découle du choix de $\FS$, suivant \cite[Théorème~16.9]{borel_iga}.
Par suite,
\begin{align*}
m_{G/\Gamma}(\tOmega_r) & \leq m_{G/\Gamma}(\phi)\\
& = m_G(\tphi)\\
& = \sum_{c\in C} m_G(\{g\ |\ gc^{-1}\in\FS\ \mathrm{et}\ \norm{gc^{-1}e_\chi}\leq r\})\\
& = \abs{C}\cdot m_G(\{g\in\FS\ |\ \norm{ge_\chi}\leq r\}.
\end{align*}
Pour évaluer le dernier terme, on utilise la définition de $\FS$ et on décompose la mesure de Haar sur $G$ suivant la décomposition $g=kman$.
D'après \cite[Proposition~8.32]{knapp_lgbi}, pour les mesures de Haar à gauche sur $G$, $K$ et $MAN$, notant $\Delta_{MAN}:MAN\to\R_+^*$ la fonction modulaire de $MAN$,
\[
\dd g = \Delta_{MAN}(man)\,\dd k\,\dd(man) = \Delta_{MAN}(m)\,\rho(a)\,\dd k\,\dd(man),
\]
où $\rho$ est le caractère de $A$ associé à la somme des racines restreintes positives, comptées avec multiplicité.
Toujours grâce à \cite[Proposition~8.32]{knapp_lgbi}, on peut encore décomposer
\[
%\dd(man) = \frac{\Delta_{AN}(an)}{\Delta_{MAN}(an)}\dd m \dd(an) = \frac{\Delta_{AN}(a)}{\Delta_{MAN}(a)}\dd m\dd(an) = \dd m \dd(an)
%car comme l'élément $a$ commute à $M$, $\Delta_{AN}(a)= \abs{\det\Ad_{\a+\n}a} = \abs{\det\Ad_{\m+\a+\n}a} =\Delta_{MAN}(a)$.
\dd(man)=\dd m\,\dd(an) = \dd m\,\dd a\,\dd n,
\]
et donc
\begin{align*}
m_G(\{g\in\FS\ |\ \norm{ge_\chi}\leq r\})
& \asymp \int_{A^-} \mathbbm{1}_{\{\chi(a)\leq r\}}\rho(a)\,\dd a\\
& = \int_{\ka^-} \mathbbm{1}_{\{\chi(h)\leq\log r\}} e^{\rho(h)}\,\dd h.
\end{align*}
Cette dernière intégrale est l'intégrale d'une fonction exponentielle sur un polytope convexe de $\ka$, et est donc comparable au maximum de la fonction sur cet ensemble, multiplié par un facteur correspondant à la dimension de la face sur laquelle ce maximum est réalisé.
Naturellement, ce maximum est réalisé en l'un des sommets du convexe, qui sont les points $A_i=\frac{\log r}{\chi(Y_i)}Y_i$, $i=1,\dots,r$.
On trouve donc
\[
\int_{\ka^-} \mathbbm{1}_{\{\chi(h)\leq\log r\}} e^{\rho(h)}\dd h
\asymp r^{a_\chi} (\log1/r)^{b_\chi-1}
\]
où 
\[ a_\chi= \min_{1\leq i\leq r} \frac{\rho(Y_i)}{\chi(Y_i)}
\quad\mathrm{et}\quad
b_\chi= \card\{i\in[1,r] \ |\ \frac{\rho(Y_i)}{\chi(Y_i)}=a\}.
\]
Cela montre déjà une partie des inégalités souhaitées:
\[
m_\Omega(\Omega_r') \leq m_\Omega(\Omega_r) \ll r^{a_\chi}\abs{\log r}^{b_\chi-1}.
\]
Le calcul pour minorer $m_\Omega(\Omega_r')\asymp m_\Omega(\tOmega_r')$ est similaire: ayant fixé un ensemble de Siegel $\FS$, on utilise la fonction
\[
\tpsi(g) = \mathbbm{1}_{\{g\in\FS\ \mathrm{et}\ \norm{ge_\chi}\leq r\ \mathrm{et}\ \norm{\pi^+(ge_\chi)}\geq\frac{1}{2}\norm{ge_\chi}
\}}.
\]
Par la propriété de Siegel pour le domaine $\FS$, i.e. la dernière assertion de \cite[théorème~15.5]{borel_iga}, la projection de $\tpsi$ sur $G/\Gamma$ vérifie, pour une certaine constante $C_1$, 
\begin{align*}
\psi(g\Gamma) & = \sum_{\gamma\in\Gamma} \mathbbm{1}_{\{g\gamma \in\FS\ \mathrm{et}\ \norm{g\gamma e_\chi}\leq r\ \mathrm{et}\ \norm{\pi^+(g\gamma e_\chi)}\geq\frac{1}{2}\norm{g\gamma e_\chi}\}}\\
%& \leq C_1\mathbbm{1}_{\{\exists\gamma\ :\ g\gamma \in\FS\ \mathrm{et}\ \norm{g\gamma e_\chi}\leq r\ \mathrm{et}\ \norm{\pi^+(g\gamma e_\chi)}\geq\frac{1}{2}\norm{g\gamma e_\chi}\}}\\
& \leq C_1\mathbbm{1}_{\{\exists\gamma\in\Gamma\ :\ \norm{g\gamma e_\chi}\leq r\ \mathrm{et}\ \norm{\pi^+(g\gamma e_\chi)}\geq\frac{1}{2}\norm{g\gamma e_\chi}\}}\\
& = C_1\mathbbm{1}_{\tOmega_r'}(g\Gamma).
\end{align*}
Par conséquent,
\begin{align*}
m_{G/\Gamma}(\tOmega_r') & \gg m_{G/\Gamma}(\psi)\\
& = m_G(\tpsi)\\
& = m_G(\{g\in\FS\ |\ \norm{ge_\chi}\leq r\ \mathrm{et}\ \norm{\pi^+(g e_\chi)}\geq\frac{1}{2}\norm{g e_\chi}\}.
\end{align*}
Comme précédemment, pour évaluer le dernier terme, on utilise la décomposition de la mesure de Haar sur $G$ suivant la décomposition $g=kman$.
Soit  $\omega\subset MN$ un voisinage compact de l'identité tel que
\[
\FS\supset K(\exp\ka^-)\omega
\]
et $V_K\subset K$ un voisinage de l'identité tel que
\[
\forall k\in V_K,\quad \norm{\pi^+(ke_\chi)} \geq \frac{1}{2}\norm{ke_\chi}=\frac{1}{2}\norm{e_\chi}.
\]
\note{car $ue_\chi=e_\chi$}
Alors, pour tout $h$ dans la chambre de Weyl $\ka^-$ et tout $u\in MN$,
\[
\norm{\pi^+(ke^hue_\chi)} = e^{\chi(h)}\norm{\pi^+(ke_\chi)}
 \geq \frac{1}{2}e^{\chi(h)}\norm{e_\chi} = \frac{1}{2}\norm{ke^hue_\chi},
\]
et l'on peut donc minorer
%Notant $\rho$ le caractère de $A$ correspondant à la somme des racines positives (ou à la double somme des poids fondamentaux)
\begin{IEEEeqnarray*}{l}
m_G(\{g\in\FS\ |\ \norm{ge_\chi}\leq r\ \mathrm{et}\ \norm{\pi^+(g e_\chi)}\geq\frac{1}{2}\norm{g e_\chi}\})\\
\qquad \geq  \iint_{V_K\times\omega}\int_{\ka^-} \mathbbm{1}_{\{\norm{ke^hu e_\chi}\leq r\ \mathrm{et}\ \norm{\pi^+(ke^hu e_\chi)}\geq\frac{1}{2}\norm{ke^hu e_\chi}\}}e^{\rho(h)}\,\dd h\, \dd k\, \dd u \\
\qquad \geq  \iint_{V_K\times \omega}\int_{\ka^-} \mathbbm{1}_{\{\chi(h)\leq\log r\}}e^{\rho(h)}\,\dd h\, \dd k\, \dd u \\
\qquad \gg \int_{\ka^-} \mathbbm{1}_{\{\chi(h)\leq\log r\}}e^{\rho(h)}\,\dd h \\
\qquad \gg  r^{a_\chi}\abs{\log r}^{b_\chi-1}
\end{IEEEeqnarray*}
par le calcul déjà expliqué dans la première partie de la démonstration.
\end{proof}

\begin{remark}
Les coefficients $u_\chi$ et $v_\chi$ qui apparaissent dans le théorème de Mohammadi et Salehi Golsefidy \cite[Theorem~4]{msg}, énoncé comme théorème~\ref{hborne} ci-dessus, ne sont pas égales à $a_\chi$ et $b_\chi$ en général.
L'équivalent asymptotique du nombre de points rationnels de $X$ de hauteur au plus $T$ est donné par
\[
N_\chi(T) \sim c\cdot T^{u_\chi}(\log T)^{v_\chi-1}.
\]
Dans \cite{msg}, ces coefficients sont définis de la façon suivante.
On note $\kb^+=\{\phi\in\ka^*\ |\ \forall i,\ \phi(\alpha_i)\geq 0\}$ et $\rho_X$ la somme des racines associées à l'unipotent de $P$, et on pose
\[
u_\chi=\inf\{u\ |\ u\chi-\rho_X\in\kb^+\} = \max_i \frac{\bracket{\rho_X,\alpha_i}}{\bracket{\chi,\alpha_i}}
,
\]
tandis que $v_\chi$ est la codimension de la face de $\kb^+$ contenant $u_\chi\chi-\rho$.
Dans le cas particulier où $X$ est munie de la hauteur anti-canonique, $\rho_X=\chi$ et donc $u_\chi=a_\chi=1$ et $v_\chi=b_\chi=\rang G-\rang P$.
En effet, dans ce cas, $a_\chi=\min_i\frac{\rho(Y_i)}{\rho_X(Y_i)}$ est atteint pour tout $i$ correspondant à une racine simple hors de la partie $\theta$ associée à $P$.
\end{remark}

Notons $\varpi_1,\dots,\varpi_r$ les poids fondamentaux associés à la base de racines simples $\Pi=\{\alpha_1,\dots,\alpha_r\}$.
Pour chaques$i$, on fixe une $\Q$-représentation $V_i$ de $G$ engendrée par une unique droite rationnelle de plus haut poids $\omega_i=b_i\varpi_i$, avec $b_i\in\N^*$ minimal.
%, et $V_i(\Z)$ un réseau rationnel dans $V_i$ stable par l'action de $\Gamma$ et contenant un vecteur de plus haut poids $e_i$.
Suivant Borel et Tits \cite[\S12.13]{boreltits}, nous dirons que les représentations $V_i$, $i=1,\dots,r$, sont les représentations \emph{fondamentales} de $G$.
Lorsque $V_\chi=V_j$ est une représentation fondamentale de $G$, on peut déterminer quel indice réalise le minimum définissant $a_\chi$; c'est le contenu du lemme ci-dessous.

\begin{lemma}
\label{maximal}
Soit $V_\chi=V_j$ une représentation fondamentale de $G$ et $P$ le sous-groupe parabolique maximal associé, stabilisateur de la droite de plus haut poids.
Il existe $r_j\in\Q$ tel que la somme des racines du radical unipotent $U$ de $P$, comptées avec multiplicité, vérifie
\note{Si $G$ est simplement connexe, on a $\omega_j=\varpi_j$ et $r_j\in\N^*$. En règle générale, on peut borner le dénominateur de $r_j$.}
\[
\sum_{\alpha\in\Sigma_j^+}m_\alpha\alpha = r_j\omega_j,
\]
et alors, les quantités $a_\chi$ et $b_\chi$ de la proposition précédente sont données par
\[
a_\chi = r_j
\quad\mathrm{et}\quad
b_\chi = 1.
\]
\end{lemma}
\begin{proof}
Soit $P$ le parabolique maximal correspondant à $\chi=\omega_j$.
Le radical unipotent $U$ de $P$ correspond à l'ensemble $\Sigma_j^+$ des racines positives contenant $\alpha_j$ dans leur décomposition en racines simples.
D'après \cite[Lemma~5]{msg},
%(autre référence, ou démonstration?)
la somme $\rho_j$ de ces racines, comptées avec multiplicité, est proportionnelle à $\omega_j$:
\[
r_j\omega_j = \sum_{\alpha\in\Sigma_j^+}m_\alpha\alpha =\rho_j.
\]
Par conséquent, il suffit de comprendre pour quel $i$ la quantité $\frac{\rho(Y_i)}{\rho_j(Y_i)}$ est minimale.
Naturellement, comme $Y_j$ est orthogonal à toutes les racines qui ne contiennent pas $\alpha_j$, on trouve
\[
\frac{\rho(Y_j)}{\rho_j(Y_j)} = 1,
\]
tandis que si $i\neq j$,
\[
\frac{\rho(Y_i)}{\rho_j(Y_i)} = 1 + \frac{\rho'_j(Y_i)}{\rho_j(Y_i)} > 1,
\]
où $\rho'_j$ désigne la somme des racines positives qui ne contiennent pas $\alpha_j$.
Ainsi, $b_\chi=1$ et
\[
a_\chi = \frac{\rho(Y_j)}{\omega_j(Y_j)} = \frac{\rho_j(Y_j)}{\omega_j(Y_j)} = r_j.
\]
\end{proof}

Ce lemme et sa démonstration se généralisent de la façon suivante.

\begin{lemma}
\label{somrac}
Soit $\theta$ une partie de la base du système de racines de $G$, et $P_\theta$ le sous-groupe parabolique associé.
Si $\chi=\sum_{\alpha\in\Sigma_\theta^+} m_\alpha\alpha$, est la somme des racines du radical unipotent de $P_\theta$, comptées avec multiplicité, alors
\[
a_\chi = 1
\quad\mathrm{et}\quad
b_\chi = \rang_{\Q}G-\abs{\theta}.
\]
\end{lemma}
\begin{proof}
Pour chaque $i$, on observe que
\[
\rho(Y_i) = \sum_{\alpha\in\Sigma_i^+}m_\alpha\alpha(Y_i),
\]
où $\Sigma_i^+$ désigne l'ensemble des racines positives contenant $\alpha_i$.
Comme par ailleurs $\Sigma_\theta^+$ est l'ensemble des racines positives qui contiennent un élément hors de $\theta$, on calcule
\[
\chi(Y_i) = \sum_{\alpha\in\Sigma_i^+\cap\Sigma_\theta^+}m_\alpha\alpha(Y_i)
\left\{\begin{array}{ll}
= \sum_{\alpha\in\Sigma_i^+}m_\alpha\alpha(Y_i) & \mathrm{si}\ i\not\in\theta\\
< \sum_{\alpha\in\Sigma_i^+}m_\alpha\alpha(Y_i) & \mathrm{si}\ i\in\theta,
\end{array}\right.
\]
cela montre ce qu'on veut.
\end{proof}

\section{Théorème de Khintchine et flots diagonaux}

Comme précédemment, $G$ désigne un $\Q$-groupe semi-simple, $P$ un $\Q$-sous-groupe parabolique, et $X=P\bcs G$ la variété quotient.
On fixe un poids dominant $\chi$ associé au sous-groupe parabolique $P$.
D'après la proposition~\ref{asymp}, il existe des constantes $a_\chi>0$ et $b_\chi\in\N^*$ telles que pour $r>0$ petit, le voisinage $\Omega_r$ de l'infini dans $\Omega$ défini par
\[
\Omega_r = \{ g\Gamma\ |\ \min_{v\in \tx\cap V_\chi(\Z)} \norm{gv} \leq r\}
\]
vérifie
\[
m_\Omega(\Omega_r) \asymp r^{a_\chi}\abs{\log r}^{b_\chi-1}.
\]
Avec ces notations, le théorème de Khintchine sur la variété projective $X$, munie de la hauteur $H_\chi$ associée à $\chi$ et de la distance de Carnot-Carathéodory usuelle, s'énonce comme suit.

\note{D'un point de vue physique, quelles sont les dimensions des quantités $d(x,v)$, $H(v)$ et $\psi(H(v))$?}

\begin{theorem}[Théorème de Khintchine sur $X$]
\label{khintchine}
Notons $\beta_\chi=\beta_\chi(X)$ l'exposant donné par le théorème~\ref{exposantps}.
Soit $\psi:\R^+\to\R^+$ une fonction décroissante.
Pour $x$ dans $X$, on s'intéresse à l'inégalité
\begin{equation}\label{psi}
d(x,v) \leq H_\chi(v)^{-\beta_\chi}\psi(H_\chi(v)).
\end{equation}
\begin{itemize}
\item Si $\int_e^{+\infty} \psi(u)^{\frac{a_\chi}{\beta_\chi}}(\log\log u)^{b_\chi-1}\frac{\dd u}{u}<+\infty$, alors, pour presque tout $x$ dans $X$, l'inégalité \eqref{psi} n'a qu'un nombre fini de solutions.
\item Si $\int_e^{+\infty} \psi(u)^{\frac{a_\chi}{\beta_\chi}}(\log\log u)^{b_\chi-1}\frac{\dd u}{u}=+\infty$, alors, pour presque tout $x$ dans $X$, l'inégalité \eqref{psi} admet une infinité de solutions.
\end{itemize}
\end{theorem}

Rappelons qu'un point $x\in X(\R)$ est dit \emph{mal approchable} dans $X$ s'il existe une constante $c>0$ telle que pour tout $v$ dans $X(\Q)$, $d(x,v)\geq cH_\chi(v)^{-\beta_\chi}$.
Nous avons vu au chapitre~\ref{chap:correspondance} que l'ensemble $\mathrm{BA}_X$ des points mal approchables dans $X$ est de dimension de Hausdorff maximale.
Le théorème de Khintchine permet de montrer qu'il est toutefois négligeable au sens de la mesure de Lebesgue.

\begin{corollary}[Points mal approchables sur $X$]
L'ensemble $\mathrm{BA}_X$ des points mal approchables dans $X$ est de mesure de Lebesgue nulle.
\end{corollary}
\begin{proof}
Pour la fonction constante définie par $\psi(u)=c>0$ pour tout $u\in\R^+$, l'intégrale est divergente.
Le résultat découle donc de la seconde partie du théorème.
\end{proof}

On peut aussi étudier les inégalités obtenues à partir de fonctions élémentaires, et montrer le résultat suivant.

\begin{corollary}
\begin{enumerate}
\item L'inégalité $d(x,v) \leq H_\chi(v)^{-\beta_\chi} (\log H_\chi(v))^{-\gamma}$ admet une infinité de solutions $v\in X(\Q)$ pour presque tout $x\in X(\R)$, si et seulement si $\gamma\leq\frac{\beta_\chi}{a_\chi}$.
\item L'inégalité $d(x,v) \leq H_\chi(v)^{-\beta_\chi} (\log H_\chi(v))^{-\frac{\beta_\chi}{a_\chi}} (\log\log H_\chi(v))^{-\delta}$ admet une infinité de solutions $v\in X(\Q)$ pour presque tout $x\in X(\R)$, si et seulement si $\delta\leq\frac{b_\chi\beta_\chi}{a_\chi}$.
\end{enumerate}
\end{corollary}
\begin{proof}
D'après le théorème~\ref{khintchine}, il suffit pour la première partie d'étudier l'intégrabilité au voisinage de l'infini de la fonction $\phi$ définie par
\[
\phi(u) = \frac{(\log\log u)^{b_\chi-1}}{u (\log u)^{\frac{a_\chi\gamma}{\beta_\chi}}}.
\]
L'intégrale diverge si et seulement si $\gamma\leq\frac{\beta_\chi}{a_\chi}$.
La démonstration de la seconde assertion est analogue, une fois observé que l'intégrale
\[
\int_e^{+\infty} \frac{\dd u}{u(\log u) (\log\log u)^{\frac{a_\chi\delta}{\beta_\chi}-b_\chi+1}}
\]
diverge si et seulement si $\delta\leq\frac{b_\chi\beta_\chi}{a_\chi}$.
% Noter que $\int\frac{\dd u}{u(log u)(\log\log u)} = \log\log\log u$ diverge, tandis que pour $\eps>0$, $\int\frac{\dd u}{u(\log u)(\log\log u)^{1+\eps}}=\frac{-1}{\eps(\log\log u)^{\eps}}$ converge.
\end{proof}

Pour traduire le théorème~\ref{khintchine} en termes d'orbites diagonales dans l'espace des réseaux, nous aurons besoin d'une version un peu plus précise de la correspondance démontrée au chapitre~\ref{chap:correspondance}, que nous énonçons maintenant.
Rappelons que l'on définit un sous-groupe à un paramètre $(a_t)$ dans le groupe $A$ des points réels du tore $\Q$-déployé maximal $T$ en posant
\begin{equation}\label{at}
a_t=e^{tY}
\quad\mbox{où}\quad Y\in\ka\ \mbox{est défini par}\
\alpha(Y) = 
\left\{
\begin{array}{ll}
0 & \mbox{si}\ \alpha\in\theta\\
-1 & \mbox{si}\ \alpha\not\in\theta,
\end{array}
\right.
\end{equation}
où $\theta\subset\Pi$ est l'ensemble de racines simples associé au sous-groupe parabolique $P$, tel que toutes les racines négatives de $P$ se décomposent en éléments de $\theta$.
Enfin, si $\Delta$ est un réseau dans $V_\chi$, nous avons défini ci-dessus
\[
r_\chi(\Delta) = \inf\{r>0\ |\ \exists v\in\tx\cap B(0,r):\ \norm{\pi^+(v)} \geq \frac{1}{2}\norm{v}\},
\]
où $\pi^+:V_\chi\to V_\chi$ désigne la projection sur $\R e_\chi$ parallèlement aux autres espaces de poids de $a_t$.

\begin{proposition}[Correspondance drapeau-réseau]
\label{daniprecis}
%Soit $G$ un $\Q$-groupe semi-simple, $P$ un $\Q$-sous-groupe parabolique, et $X=P\bcs G$ la variété quotient.
%La variété $X$ est munie de la hauteur $H_\chi$ associée à la représentation de plus haut poids $\chi$, et de la distance de Carnot-Carathéodory usuelle.
Fixons $x\in X(\R)$ et choisissons $s_x\in G$ tel que $x=Ps_x$.
On note $\beta_\chi=\beta_\chi(X)$ l'exposant diophantien donné par la proposition~\ref{exposantps}.
Il existe une constante $C> 0$ telle que les énoncés suivants soient vérifiés.

Soit $\psi:\R_+\to\R_+$ une fonction décroissante, et $\Psi:\R^+\to\R^+$ la fonction définie par
\[
\Psi(u)=Cu^{-\beta_\chi}\psi(u).
\]
\begin{itemize}
\item Si l'inégalité $d(x,v)\leq H_\chi(v)^{-\beta_\chi}\psi(H_\chi(v))$ admet une infinité de solutions $v\in X(\Q)$, alors il existe $t>0$ arbitrairement grand tel que
\[
r_\chi(a_ts_xV_\chi(\Z)) \leq 2 e^{-\frac{t}{\beta_\chi}}\Psi^{-1}(e^{-t}).
\]
\item Si on a pour $t>0$ arbitrairement grand $r_\chi(a_ts_xV_\chi(\Z))\leq e^{-\frac{t}{\beta_\chi}}\Psi^{-1}(e^{-t})$, alors l'inégalité
\[
d(x,v)\leq C^2H_\chi(v)^{-\beta_\chi}\psi(H_\chi(v))
\]
admet une infinité de solutions $v\in X(\Q)$. 
%, où $C$ est une constante dépendant de $x$.
\end{itemize}
\end{proposition}
\begin{proof}
La démonstration est identique à celle de la proposition~\ref{dani}.
Soit $v\in X(\Q)$ arbitrairement proche de $x$ tel que
$d(x,v)\leq H_\chi(v)^{-\beta_\chi}\psi(H_\chi(v))$.
Écrivons
\[
x=Ps_x \qquad v=Pe^us_x,\ u\in\uu^-,
\]
de sorte que
\[
d(x,v) \asymp \abs{u} \leq H_\chi(v)^{-\beta_\chi}\psi(H_\chi(v)).
\]
Soit $C>0$ tel que pour tout $w\in\uu^-$ tel que $C\abs{w}\leq 1$, on ait $\norm{w}\leq\frac{1}{4}$.
On choisit $t>0$ tel que 
\[
e^{-t}=\Psi(H_\chi(v)) = CH_\chi(v)^{-\beta_\chi}\psi(H_\chi(v)).
\]
Cela implique $C\abs{(\Ad a_t)u}=Ce^t\abs{u}\leq 1$ et donc
\[
\norm{(\Ad a_t)u} \leq \frac{1}{4}.
\]

Soit maintenant $\bv$ un représentant primitif de $v$ dans $V_\chi(\Z)$.
Le vecteur $s_x\bv$ est porté par la direction $\R e^{-u}e_\chi$, et par conséquent, pour tout $t>0$,
\begin{align*}
a_ts_x\bv & = H_\chi(v) a_te^{-u}e_\chi\\
& = H_\chi(v) e^{t\chi(Y)} [\exp(-(\Ad a_t)u)]e_\chi\\
& = H_\chi(v) e^{t\chi(Y)} [e_\chi - (\Ad a_t)u)e_\chi + ((\Ad a_t)u)^2e_\chi -\dots]
\end{align*}
L'inégalité $\norm{e^w-I}\leq 2\norm{w}$ appliquée à $w=(\Ad a_t)u$ donne d'une part
\[
\norm{\pi^+(a_ts_x\bv)} \geq \frac{1}{2}\norm{a_ts_x\bv},
\]
et d'autre part
\[
\|a_ts_x\bv\| \leq 2H_\chi(v) e^{t\chi(Y)}.
\]
Comme $H_\chi(v)=\Psi^{-1}(e^{-t})$ et $\beta_\chi=\frac{-1}{\chi(Y)}$, cela montre que 
\[
r_\chi(a_ts_x\bv) \leq 2H_\chi(v)e^{t\chi(Y)} \leq 2e^{-\frac{t}{\beta_\chi}}\Psi^{-1}(e^{-t}).
\]

Montrons l'inégalité réciproque.
Pour $t>0$ arbitrairement grand, soit $\bv\in\tx\cap V_\chi(\Z)$ tel que
\begin{equation}\label{norme}
 \norm{a_ts_x\bv} \leq e^{-\frac{t}{\beta_\chi}}\Psi^{-1}(e^{-t})
\end{equation}
et
\begin{equation}\label{projection}
 \norm{\pi^+(a_ts_x\bv)} \geq \frac{1}{2}\norm{a_ts_x\bv}.
\end{equation}
Notons $v$ l'image de $\bv$ dans $X$.
Grâce à l'expression de $a_ts_x\bv$ utilisée ci-dessus, nous avons
%\begin{align*}
% a_ts_x\bv & = H_\chi(v) \chi(a_t) e^{(\Ad a_t)u}e_\chi\\
%& = H_\chi(v) \chi(a_t) (e_\chi - ((\Ad a_t)u)e_\chi+\frac{1}{2}((\Ad a_t)u)^2e_\chi+\dots).
%\end{align*}
\[
\pi^+(a_ts_x\bv)=e^{t\chi(Y)}H_\chi(v)e_\chi
\]
et donc, d'après \eqref{projection},
\[
H_\chi(v)^{-1}e^{-t\chi(Y)}\norm{a_ts_x\bv} \leq 2.
\]
%\begin{equation}\label{vplus}
% H_\chi(v) \chi(a_t) = \|\pi^+(a_ts_x\bv)\| \geq \frac{\|a_ts_x\bv\|}{2}\geq e^{-(\gamma+\eps)t}.
%\end{equation}
On utilise ensuite la stratification $\uu^-=m_1\oplus\dots\oplus m_k$, et l'on décompose $u=\sum_i u_i$ suivant cette somme directe.
Dans l'égalité
\begin{equation}\label{egdec}
e_\chi - e^{-t\chi(Y)}H_\chi(v)^{-1} a_ts_x\bv = \sum_i((\Ad a_t)u_i)e_\chi-\frac{1}{2}((\Ad a_t)u)^2e_\chi+\dots,
\end{equation}
le terme $((\Ad a_t)u_1)e_\chi$ est en somme directe avec tous les autres, et l'on peut donc majorer
\[
\|(\Ad a_t)u_1e_\chi\| \leq 3.
\]
Comme $P=\Stab[e_\chi]$, l'application $u\mapsto u\cdot e_\chi$ est un difféomorphisme local au voisinage de $0$, et par conséquent, 
\[ \|(\Ad a_t)u_1\| \ll 1.\]
Dans \eqref{egdec}, on peut alors faire passer tous les termes du membre de droite de la forme $((\Ad a_t)u_1)^ie_\chi$ dans le membre de gauche, et cela n'augmente pas significativement la norme de ce dernier.
Le terme $((\Ad a_t)u_2)e_\chi$ est alors en somme directe avec tous les autres termes du membre de droite, ce qui permet de voir que $\|((\Ad a_t)u_2)e_\chi\| \ll 1$, puis
\[ \|(\Ad a_t)u_2\| \ll 1.\]
Ainsi de proche en proche, on montre que pour chaque $i$,
$\|(\Ad a_t)u_i\| \ll 1$,
de sorte qu'à la fin
$\|((\Ad a_t)u)e_\chi\| \ll 1$, et donc aussi
\[
\abs{(\Ad a_t)u}\ll 1.
\]
Par suite,
\[
d(x,v) = \abs{u} = e^{-t} \abs{(\Ad a_t)u} \ll e^{-t}.
\]
Or, l'inégalité \eqref{norme} et la relation $\beta_\chi=-\frac{1}{\chi(Y)}$ impliquent
\[
e^{t\chi(Y)}H_\chi(v)= \norm{\pi^+(a_ts_xv)} \leq \norm{a_ts_xv} \leq e^{t\chi(Y)}\Psi^{-1}(e^{-t}),
\]
d'où l'on tire $e^{-t}\leq \Psi(H_\chi(v))$,
ce qui permet de conclure
\[
d(x,v) \ll \Psi(H_\chi(v)) \ll H_\chi(v)^{-\beta_\chi}\psi(H_\chi(v)).
\]
\end{proof}

\section{Somme convergente}

Nous donnons ici la démonstration de la première partie du théorème de Khintchine pour $X$.
L'énoncé correspondant concernant le comportement asymptotique des orbites diagonales dans l'espace $\Omega=G/\Gamma$ est une simple application du lemme de Borel-Cantelli.
Rappelons que l'espace $\Omega$ s'identifie à un ensemble de réseaux de l'espace euclidien $V_\chi$ associé à la représentation de plus haut poids $\chi$, via l'application $g\Gamma\mapsto gV_\chi(\Z)$.
Avec cette identification, notant aussi $\tx=G\cdot e_\chi$ l'orbite du vecteur de plus haut poids dans $V$, on pose, pour $r>0$,
\[
\Omega_r = \{ \Delta\in\Omega\ |\ \exists v\in\Delta\cap\tx:\, \norm{v}\leq r\}.
\]
Enfin, la mesure de Haar sur $\Omega=G/\Gamma$ est notée $m_\Omega$.

\begin{proposition}
\label{kcr}
Soit $(a_t)_{t\in\N}$ une suite d'éléments de $G$ et $(r_t)_{t\in\N}$ une suite de réels positifs telle que $\sum_{t\geq 1} m_\Omega(\Omega_{r_t})<\infty$.
Pour presque tout $\Delta\in\Omega$, pour tout $t\in\N$ suffisamment grand,
\[
a_t\Delta \not\in \Omega_{r_t}.
\]
Par conséquent, pour presque tout $x=Ps_x$ dans $X$, pour tout $t\in\N$ suffisamment grand,
\[
r_\chi(a_ts_xV_\chi(\Z)) \geq r_t.
\]
\end{proposition}
\begin{proof}
Comme chaque élément $a_t$ préserve la mesure de Haar sur $\Omega$,
\[
m_\Omega(\{\Delta\ |\ a_t\Delta\in\Omega_{r_t}\}) = m_\Omega(\Omega_{r_t}),
\]
est le terme général d'une série convergente, et le lemme de Borel-Cantelli montre donc que pour presque tout $\Delta$, pour tout $t$ suffisamment grand, $a_t\Delta\not\in\Omega_{r_t}$.
La seconde partie découle de la première, car
\[
r_\chi(a_ts_x\Gamma)\geq \min\{\norm{v}\ ;\ v\in a_ts_xV_\chi(\Z)\cap\tx\},
\]
et ce deuxième terme ne dépend pas du choix de $s_x$, à une constante multiplicative près.
En appliquant la première partie à une suite $(r'_t)$ telle que $\sum_{t\in\N}m_\Omega(\Omega_{r_t'})<+\infty$ et $r_t=o(r_t')$ on fait disparaître cette constante multiplicative.
\end{proof}

Nous aurons besoin du lemme suivant, qui permet de supposer dans la démonstration du théorème~\ref{khintchine} que la fonction $\psi$ est majorée par une fonction de la forme $t\mapsto (\log t)^{-c}$.
Le dessin ci-dessous explique le calcul effectué pour la démonstration.

\begin{figure}[H]
\begin{center}
\begin{tikzpicture}[domain=.65:6]
%\draw[very thin,color=gray] (-0.1,-1.1) grid (3.9,1.9);
\draw[->] (-0.2,0) -- (5.5,0) node[below] {$x$};
\draw[->] (0,-.2) -- (0,4.2) node[above] {$y$};
\draw[color=red] plot (\x,3/\x);
\draw[color=red] (6,.7) node[right] {$y =(\log x)^{-c}$};
\draw[color=blue] plot (\x,1/\x);
\draw[color=blue] (6,-.1) node[right] {$y = (\log x)^{-2c}$};
\draw[color=green] (.65, 1.33) -- (.75,1.33) .. controls (4.5,1) .. (5,.6) .. controls (5.5,.3) .. (6,.3); 
\draw[color=green] (6,.3) node[right] {$y=\psi(x)$};
\draw[color=gray] (5,.6) -- (5,0) node[below] {$T$};
\draw[color=gray] (5,.6) -- (1.67,.6) -- (1.67,0) node[below] {$e^{\sqrt{\log T}}$};
\draw[color=gray] (1.67,.6) -- (0,.6);
\draw[color=gray] (.75,1.33) -- (.75,0) node[below] {$T_1$};

%\draw[color=orange] plot (\x,{0.05*exp(\x)}) node[right] {$f(x) = \frac{1}{20} \mathrm e^x$};
\end{tikzpicture}
\end{center}
\caption{Graphes des fonctions $\psi$, $x\mapsto(\log x)^{-c}$ et $x\mapsto (\log x)^{-2c}$.}
\end{figure}
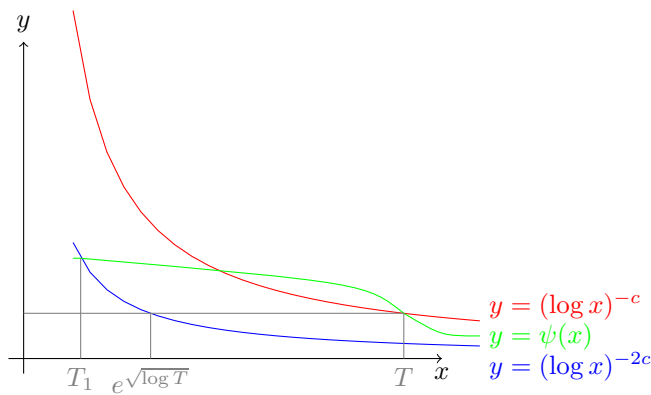

\begin{lemma}
\label{calc}
Soient $a,\beta >0$ et $b\in\N^*$.
Il existe $c>0$ tel que pour toute fonction $\psi:\R^+\to\R^+$ décroissante telle que
\[
I(\psi)=\int_e^\infty \psi(u)^{\frac{a}{\beta}}(\log\log u)^{b-1}\frac{\dd u}{u}<+\infty,
\]
pour tout $u>0$ suffisamment grand, $\psi(u)\leq (\log u)^{-c}$.
\end{lemma}
\begin{proof}
Posons $c=\frac{\beta}{2a}$ et supposons que pour $T>0$ arbitrairement grand, $\psi(T)\geq (\log T)^{-c}$.
Soit
\[
T_1= \sup\{ u\in[0,T]\ |\ \psi(u)\leq (\log u)^{-2c}\}.
\]
Comme $\psi$ est décroissante,
\[
(\log T_1)^{-2c} \geq \psi(T_1) \geq \psi(T) \geq (\log T)^{-c},
\]
et donc $(\log T_1)\leq (\log T)^{1/2}$.
Par suite, %pour $T$ assez grand,
\begin{align*}
\int_{T_1}^T \psi(u)^{\frac{a}{\beta}}(\log\log u)^{b-1}\frac{\dd u}{u}
%& \geq \int_{T_1}^T \frac{(\log\log u)^{b-1}}{\log u}}\frac{\dd u}{u}\\
& \geq \int_{T_1}^T \frac{\dd u}{u\log u}\\
& = \log\log T- \log\log T_1\\
& \geq \frac{1}{2}\log\log T.
\end{align*}
Cela montre que l'intégrale $I(\psi)$ est divergente, et le lemme s'en déduit, par contraposée.
\end{proof}

Nous pouvons maintenant démontrer la première partie du théorème~\ref{khintchine}.

\begin{proof}[Démonstration du théorème~\ref{khintchine}, cas convergent]
Soit $\psi:\R^+\to\R^+$ décroissante telle que
\[
\int_e^\infty \psi(u)^{\frac{a_\chi}{\beta_\chi}}(\log\log u)^{b_\chi-1}\frac{\dd u}{u} <\infty.
\]
Quitte à remplacer $\psi(u)$ par $\psi(u)+(\log u)^{-C}$, avec $C>\frac{\beta_\chi}{a_\chi}$, on peut toujours supposer que $\psi(u)\geq(\log u)^{-C}$.
D'autre part, le lemme~\ref{calc} montre que pour tout $u>0$ suffisamment grand, $\psi(u)\leq(\log u)^{-c}$.
Ainsi,
\begin{equation}\label{encadre}
\forall u\geq 0,\quad (\log u)^{-C} \leq \psi(u) \leq (\log u)^{-c}.
\end{equation}
Posant 
\[
\Psi(u)=u^{-\beta_\chi}\psi(u)
\qquad\mbox{et}\qquad
r_t=2e^{-\frac{t}{\beta_\chi}}\Psi^{-1}(e^{-t}),
\]
les propositions~\ref{daniprecis} et \ref{kcr} montrent qu'il suffit de démontrer que la somme $\sum_{t\geq 1} m_\Omega(\Omega_{r_t})$ converge, ce qui équivaut, d'après la proposition~\ref{asymp}, à
\[
\sum_{t\geq 1} r_t^{a_\chi} \abs{\log r_t}^{b_\chi-1} <+\infty.
\]
L'encadrement \eqref{encadre} ci-dessus se réécrit
\begin{equation}\label{encadre2}
u^{-\beta_\chi}(\log u)^{-C} \leq \Psi(u) \leq u^{-\beta_\chi}(\log u)^{-c},
\end{equation}
ce qui implique
\[
s^{-\frac{1}{\beta_\chi}}\abs{\log s}^{-\frac{C}{\beta_\chi}} \ll \Psi^{-1}(s)\ll s^{-\frac{1}{\beta_\chi}}\abs{\log s}^{-\frac{c}{\beta_\chi}}.
\]
\details{
En effet, si $\Phi_-(u)=u^{-\beta_\chi}(\log u)^{-C}$ et $\Phi_+(u)=u^{-\beta_\chi}(\log u)^{-c}$, on a $\Phi_-^{-1}(s)\leq \Psi^{-1}(s)\leq\Phi_+^{-1}(s)$, et il suffit donc de vérifier que \[
s^{-1/\beta}\abs{\log s}^{-\frac{C}{\beta}} \ll \Phi_-^{-1}(s) \leq \Phi_+^{-1}(s)\ll s^{-1/\beta}\abs{\log s}^{-\frac{c}{\beta}}.
\]
Comme $\Phi_-$ et $\Phi_+$ sont décroissantes, cela est équivalent au fait que pour certains $C_0,c_0>0$, 
\begin{align*}
\Phi_-(c_0s^{-1/\beta}\abs{\log s}^{-\frac{C}{\beta}}) & \geq  s  \geq \Phi_+(C_0s^{-1/\beta}\abs{\log s}^{-\frac{c}{\beta}})\\
c_0^{-\beta} s \abs{\log s}^C[ \frac{-1}{\beta}\log s -\frac{C}{\beta}\log\abs{\log s}]^{-C}
& \geq  s  \geq
C_0^{-\beta} s \abs{\log s}^c[ \frac{-1}{\beta}\log s -\frac{c}{\beta}\log\abs{\log s}]^{-c}\\
c_0^{-\beta} s [\frac{1}{\beta}-o(1)]^{-C} & \geq s \geq  C_0^{-\beta} s [\frac{1}{\beta}-o(1)]^{-c},
\end{align*}
et le résultat s'ensuit, si $C_0, c_0$ sont choisis de sorte que $c_0^{-\beta}\beta^C>1>C_0^{-\beta}\beta^c$.
}
Avec la borne supérieure de cet encadrement, on trouve
\[
r_t = 2 e^{-\frac{t}{\beta_\chi}}\Psi^{-1}(e^{-t}) \ll t^{-\frac{c}{\beta_\chi}},
\]
d'où 
\[
\abs{\log r_t} \ll \log t.
\]
Par suite, il suffit de vérifier que
\[
\sum_{t\geq 1} r_t^{a_\chi} (\log t)^{b_\chi-1} <+\infty.
\]
Pour cela, majorons
\note{On utilise la borne inférieure de \eqref{encadre2}: $\Psi^{-1}(s)\geq c_0s^{-\frac{1}{\beta_\chi}}\abs{\log s}^{-\frac{C}{\beta_\chi}}$, et la décroissance de $\psi$.}
\[
\Psi^{-1}(s) = s^{-\frac{1}{\beta_\chi}}\psi(\Psi^{-1}(s)) \leq s^{-\frac{1}{\beta_\chi}}\psi(c_0s^{-\frac{1}{\beta_\chi}}\abs{\log s}^{-\frac{C}{\beta_\chi}}),
\]
et donc
\[
r_t = 2e^{-\frac{t}{\beta_\chi}}\Psi^{-1}(e^{-t}) \leq 2\psi(c_0e^{\frac{t}{\beta_\chi}} t^{-\frac{C}{\beta_\chi}}).
\]
Cela donne
\begin{align*}
\sum_{t\geq 1} r_t^{a_\chi} (\log t)^{b_\chi-1} 
& \leq 2\sum_{t\geq 1} \psi(c_0e^{\frac{t}{\beta_\chi}}t^{-\frac{C}{\beta_\chi}})^{a_\chi} (\log t)^{b_\chi-1}\\
& \leq 2 \int_e^\infty \psi(c_0e^{\frac{t}{\beta_\chi}}t^{-\frac{C}{\beta_\chi}})^{a_\chi} (\log t)^{b_\chi-1}\dd t
\end{align*}
et avec le changement de variable $u=c_0e^{\frac{t}{\beta_\chi}}t^{-\frac{C}{\beta_\chi}}$, $\frac{\dd u}{u} = (\frac{1}{\beta_\chi} + o(1))\dd t$,
\[
\sum_{t\geq 1} r_t^{a_\chi} (\log t)^{b_\chi-1}
\ll \int_1^\infty \psi(u)^{a_\chi} (\log\log u)^{b_\chi-1} \frac{\dd u}{u} < +\infty.
\]
\end{proof}

\section{Somme divergente}

Ici encore, on se ramène à un énoncé sur les orbites diagonales dans l'espace $\Omega=G/\Gamma$.
La démonstration de la proposition suivante occupe la majeure partie de la fin de la démonstration du théorème~\ref{khintchine} ; la traduction en termes d'approximation diophantienne sur $X$ se fera facilement, par un calcul analogue à celui que nous avons donné ci-dessus dans le cas où la somme est convergente.

\begin{proposition}
\label{kdr}
Soit $(r_t)$ une suite de réels positifs telle que $\sum_{t\geq 1}m_\Omega(\Omega_{r_t})=+\infty$.
Alors, pour presque tout $\Delta\in \Omega$, il existe $t\in\N$ arbitrairement grand tel que
\[
r_\chi(a_t\Delta) \leq r_t.
\]
En particulier, pour presque tout $x=Ps_x\in X$, pour $t\in\N$ arbitrairement grand,
\[
r_\chi(a_ts_xV(\Z)) \leq r_t.
\]
\end{proposition}

Naturellement, la démonstration de cette proposition est basée sur une version du lemme de Borel-Cantelli dans le cas où la somme des probabilités des événements considérés est divergente, avec une hypothèse supplémentaire d'indépendance.
Plus précisément, nous utiliserons le lemme élémentaire ci-dessous.

\begin{lemma}[Une version du lemme de Borel-Cantelli]
\label{bcd}
Soit $(X,m)$ un espace de probabilité, et $(h_t)_{t\in\N}$ une suite de fonctions intégrables à valeurs positives sur $X$ satisfaisant:
\begin{enumerate}
%\item $\forall t\in\N,\quad m(h_t)\leq 1$;
\item $\sum_{t\geq 1} m(h_t) = +\infty$;
\item $\exists C\geq 0:\ \forall N\geq 1,\quad \int_X (\sum_{t=1}^Nh_t(x)-m(h_t))^2\,m(\dd x) \leq C\cdot \sum_{t=1}^Nm(h_t)$.
\end{enumerate}
Alors pour presque tout $x$, l'ensemble
\(
I(x) = \{t\in\N\ |\ h_t(x)>0\}
\)
est infini.
\end{lemma}
\begin{proof}
Posons
\[
\phi_N(x) = \frac{1}{\sum_{t=1}^Nm(h_t)}\cdot\sum_{t=1}^Nh_t(x).
\]
Par hypothèse,
\[
\norm{\sum_{t=1}^N h_t-m(h_t)}^2_2
 \leq C\cdot \sum_{t=1}^N m(h_t)
\]
et donc
\[
\norm{\phi_N-1}_2^2 \leq \frac{1}{\sum_{t=1}^N m(h_t)} \to_{N\to\infty} 0.
\]
Par suite, on peut extraire de $(\phi_N)$ une sous-suite $(\phi_{N_k})$ convergeant pour presque tout $x$ vers $1$.
Alors, presque sûrement,
\[
\sum_{t=1}^{N_k} h_t(x) \sim_{k\to\infty} \sum_{t=1}^{N_k}m(h_t)\to +\infty,
\]
ce qui montre en particulier que l'ensemble $I(x)$ est presque sûrement infini.
\comm{Si l'on veut éviter d'extraire une sous-suite qui converge presque partout, on peut raisonner comme suit: on extrait une sous-suite $(N_k)$ telle que $(\sum_{t=1}^{N_k}m(h_t))^{-1/2}\leq 2^{-k}$, de sorte que $m(\{x\ |\ \phi_{N_k}(x)\leq 0\}) \leq m(\{x\ |\ \abs{\phi_{N_k}(x)-1}\geq 1\}) \leq  \norm{\phi_{N_k}-1} \leq 2^{-k}$, et comme cette somme converge, le lemme de Borel-Cantelli usuel montre ce qu'on veut.}
\end{proof}

\comm{On peut en fait démontrer le résultat plus précis suivant, qui est démontré par exemple dans Philipp \cite[Theorem~3]{philipp} ou dans Sprindzuk \cite[Chapter~1, Lemma~10]{sprindzuk}.
Ce lemme est une loi des grands nombres avec une majoration du terme d'erreur, pour des variables aléatoires non identiquement distribuées, mais bornées et faiblement indépendantes.
}
\comm{
\begin{lemma}
Soit $(X,m)$ un espace de probabilité, et $(h_t)_{t\in\N}$ une suite de fonctions à valeurs réelles sur $X$ satisfaisant
\begin{enumerate}
\item $\forall t\in\N,\quad m(h_t)\leq 1$;
\item $\exists C\geq 0:\ \forall N>M\geq 1,\quad \int_X (\sum_{t=M}^Nh_t(x)-\sum_{t=M}^Nm(h_t))^2\,m(\dd x) \leq C\cdot \sum_{t=M}^Nm(h_t)$.
\end{enumerate}
Posons
\[
S_N(x) = \sum_{t\leq N}h_t(x)
\qquad\mathrm{et}\qquad
E_N = \int_X S_N(x)m(\dd x).
\]
Alors, pour tout $\eps>0$, pour $m$-presque tout $x$ dans $X$, lorsque $N$ tend vers $+\infty$,
\[
S_N(x) = E_N + O(E_N^{1/2}\log^{3/2+\eps}E_N).
\]
\end{lemma}
\begin{proof}
À faire en exercice.
\end{proof}
Avec ce lemme on peut obtenir un équivalent asymptotique du nombre $\card\{t\leq T\ |\ a_t\Delta\in\Omega_{r_t}\}$, du moins si l'on est capable de montrer que $m(\Omega_r)\sim r^a\abs{\log r}^b$ lorsque $r$ tend vers zéro.
Cela peut se faire à l'aide d'une formule de Siegel pour $G/\Gamma$.
Malheureusement, cet équivalent plus précis ne semble pas donner un analogue du théorème de Schmidt \cite[Theorem~3B]{schmidt_da}; cette question reste donc à étudier.
}

Pour démontrer la deuxième partie du théorème~\ref{khintchine}, on souhaiterait appliquer le lemme~\ref{bcd} à la famille de fonctions $h_t=\mathbbm{1}_{a_{t}^{-1}\Omega_{r_t}'}$, où
\[
\Omega_{r_t}' = \{\Delta\in\Omega\ |\ r_\chi(\Delta) \leq r_t\}.
\]
Cependant, la seconde condition du lemme n'est pas évidente à vérifier.
Pour cela, nous nous ramenons à une suite de fonctions lisses $h_t$ qui approchent $\mathbbm{1}_{a_{t}^{-1}\Omega_{r_t}'}$ et dont on contrôle les normes de Sobolev.

\begin{lemma}
Soit $\Upsilon$ l'opérateur $\Upsilon=1-\sum_i Y_i^2$, où $(Y_i)$ est une base orthonormée de l'algèbre de Lie $\mathfrak{k}$ d'un sous-groupe compact maximal de $G(\R)$.
Il existe des constantes $C,c>0$ et une famille de fonctions $\eta_r:\Omega\to[0,1]$, $r\in]0,1]$ telles que
\begin{enumerate}
\item $\eta_r(\Delta)>0 \quad\Rightarrow\quad r_\chi(\Delta) \leq r$;
\item $\int_{\Omega}\eta_r(\Delta)\, m_\Omega(\dd\Delta) \geq cr^{a_\chi}\abs{\log r}^{b_\chi-1}$;
\item $\norm{\Upsilon^\ell(\eta_r)}_2 \leq C \norm{\eta_r}_2$.
\end{enumerate}
\end{lemma}
\begin{proof}
Pour $\eps>0$ fixé suffisamment petit, on fixe une fonction positive $P\in C_c^\infty(G)$ telle que $\int_G P = 1$, $P(g)=P(g^{-1})$, et $\Supp P\subset B_G(1,\eps)$.
Ensuite, pour tout $r>0$, on pose
\[
\eta_r = P * \mathbbm{1}_{\Omega_r''},
\]
où
\[
\Omega_r'' = \{ \Delta\in\Omega\ |\ \exists v\in\Delta\cap\tx:\, \norm{v}\leq\frac{r}{2}\ \mathrm{et}\ \frac{\norm{\pi^+(v)}}{\norm{v}}\geq\frac{3}{4}\}.
\]
Vérifions que ces fonctions satisfont les conditions requises.

\smallskip

\noindent 1. Si $\eta_r(\Delta)=\int_G P(g)\mathbbm{1}_{\Omega_r''}(g\Delta)\,\dd u>0$, il existe $g\in\Supp P\subset B_G(1,\eps)$ tel que $g\Delta\in \Omega_r''$.
En d'autres termes, pour un certain vecteur $v\in\Delta\cap\tx$ et $g\in B_G(1,\eps)$,
\[
\norm{\pi^+(gv)}\geq \frac{3}{4}\norm{gv}
\quad\mbox{et}\quad
\norm{gv}\leq\frac{r}{2}.
\] 
Si $\eps>0$ est suffisamment petit, alors $\norm{g-1}\leq 1/10$, $\norm{g^{-1}-1}\leq 1/5$, et par conséquent, d'une part
\[
\norm{v} \leq \norm{g^{-1}}\norm{gv} \leq r,
\]
et d'autre part
\begin{align*}
\norm{\pi^+(v)}
& \geq \norm{\pi^+(gv)} - \norm{\pi^+(gv-v)}\\
& \geq \frac{3}{4}\norm{gv}-\frac{1}{10}\norm{v}\\
& \geq \frac{3\norm{v}}{4\norm{g^{-1}}} - \frac{\norm{v}}{10} \geq \frac{\norm{v}}{2}.
\end{align*}
Cela montre bien que $\Delta\in\Omega_r'$.

\smallskip

\noindent 2. Il suffit d'appliquer la proposition~\ref{asymp} pour obtenir
\[
\int_\Omega \eta_r(\Delta)\, m_\Omega(\dd\Delta) = m_\Omega(\Omega_r'')
\geq c r^{a_\chi} \abs{\log r}^{b_\chi-1},
\]
pour un certain $c>0$.

\smallskip

\noindent 3. On contrôle $\norm{\Upsilon^\ell\eta_r}_2$ grâce à l'inégalité de Young:
\[
\norm{\Upsilon^\ell\eta_r}_2
% = \norm{\Upsilon^\ell(P*\mathbbm{1}_{\Omega_r''})}_2
= \norm{\Upsilon^\ell P * \mathbbm{1}_{\Omega_r''}}_2
\leq \norm{\Upsilon P}_2\norm{\mathbbm{1}_{\Omega_r''}}_1\\
\leq C m_\Omega(\Omega_r'') = C \norm{\eta_r}_1.
\]
\end{proof}

Enfin, nous aurons besoin du théorème suivant, qui découle des résultats de \cite[\S3.4]{km_loglaws} et \cite[\S4.1]{emmv}.
Malheureusement, la démonstration de ce théorème est trop longue pour être incluse dans ce mémoire ; plusieurs résultats intermédiaires importants ont permis d'y aboutir, et l'on renvoie à \cite[\S\S~4.2 et 4.3]{emmv} pour plus de détails sur ce sujet.

\begin{theorem}[Décroissance exponentielle des coefficients dans $L^2(\Omega)$]
\label{decay}
Soit $G$ un $\Q$-groupe semi-simple simplement connexe, $\Gamma$ un sous-groupe arithmétique, et $\Omega=G/\Gamma$.
Soit $T$ un $\Q$-tore déployé maximal de $G$, $A=T^0(\R)$ et $(a_t)$ un sous-groupe à un paramètre de $A$.
On suppose que
\note{Avec $G=\SL_2\times\SL_2$, $\Gamma=\SL_2(\Z)\times\SL_2(\Z)$, $a_t=e^{tY}$ et $f$ constante sur l'un des facteurs de $\Omega=G/\Gamma$, on voit qu'il est nécessaire que $p_i(Y)\neq 0$ pour chaque $\Q$-facteur simple.}
\note{Cette condition implique que $G$ est sans facteur $\Q$-simple compact. En effet, si $G_i$ est un $\Q$-facteur compact, comme $T$ est un $\Q$-tore déployé, on aura nécessairement $p_i(Y)=0$.}
$\forall t>0,\ a_t=e^{tY}$, avec $Y\in\ka$ tel que pour toute projection $p_i:\g\to \g_i$ sur un facteur $\Q$-simple, $p_i(Y)\neq 0$.
Il existe des constantes $\ell\in\N$ et $C,\tau>0$ telles que pour toutes fonctions $f_1,f_2\in C^\infty(\Omega)$ et tout $t>0$,
\note{On peut à droite prendre la norme $L^2$ ou $L^\infty$; cela n'a pas d'importance, car on peut contrôler la norme $L^\infty$ avec la norme $L^2$ pour un $\ell$ un peu plus grand.}
\[
\left\lvert\int_\Omega f_1(a_t\Delta)f_2(\Delta)\,m_\Omega(\dd\Delta) - m_\Omega(f_1)m_\Omega(f_2)\right\rvert
\leq C  e^{-\tau t} \norm{\Upsilon^\ell f_1}_2 \norm{\Upsilon^\ell f_2}_2,
\]
où $\Upsilon$ désigne l'opérateur différentiel $\Upsilon=1-\sum_i Y_i^2$, où $(Y_i)$ est une base orthonormée de l'algèbre de Lie $\mathfrak{k}$ d'un sous-groupe compact maximal de $G(\R)$.
\end{theorem}

\comm{Un mot sur la démonstration. D'après Kleinbock et Margulis \cite[Theorem~3.4]{km_loglaws} il suffit de montrer que la restriction de $L^2(\Omega)$ à chaque facteur simple est isolée de la représentation triviale. À vérifier: 1) Pourquoi a-t-on besoin de la restriction à chaque facteur simple \cite{cowling, howe}? Dans \cite[\S 4.3]{emmv}, où sont cités Cowling, Haagerup et Howe \cite{chh}, cela n'est pas évoqué. 2) Lire l'article de Katok et Spazier \cite{ks} où l'on passe des vecteur $K$-finis aux vecteurs $C^\infty$. S'il n'y a pas besoin de considérer les restrictions aux facteurs simples, alors on conclut simplement avec \cite[\S 4.1]{emmv}.}

\begin{remark}
L'inégalité démontrée dans \cite{emmv} est plus générale: pour tout $g$ dans $G(\R)$, $\left\lvert \bracket{g f_1,f_2} - m_\Omega(f_1)m_\Omega(f_2)\right\rvert \leq C \norm{g}^{-\tau}\norm{\Upsilon^\ell f_1}_2 \norm{\Upsilon^\ell f_2}_2$, et la constante $\tau$ ne dépend pas de $G$.
Le résultat est même encore valable sur les adèles.
\end{remark}

Nous pouvons enfin conclure la démonstration de la proposition~\ref{kdr}.

\begin{proof}[Démonstration de la proposition~\ref{kdr}]
On considère la suite de fonctions définies par $h_t(\Delta)=\eta_{r_t}(a_t\Delta)$.
On peut supposer que le $\Q$-groupe semi-simple $G$ est simplement connexe, et que le sous-groupe parabolique $P$ ne contient aucun $\Q$-facteur de $G$.
Dans ce cas, vu la définition de l'élément $Y$, pour toute projection $p_i:\g\to\g_i$ sur un facteur simple, $p_i(Y)\neq 0$, et l'on peut appliquer le théorème~\ref{decay}:
\begin{align*}
\int_X (\sum_{t=1}^N h_t(\Delta)-m_\Omega(h_t))^2\,m_\Omega(\dd \Delta)
& = \sum_{1\leq t_1,t_2\leq N} \int_X h_{t_1}(\Delta)h_{t_2}(\Delta)\,m_\Omega(\dd \Delta) - m_\Omega(h_{t_1})m_\Omega(h_{t_2})\\
& \ll \sum_{1\leq t_1,t_2\leq N} e^{-\tau\abs{t_1-t_2}} \norm{\Upsilon^\ell\eta_{r_{t_1}}}_2\norm{\Upsilon^\ell\eta_{r_{t_2}}}_2\\
%& \leq C\cdot \sum_{1\leq t_1,t_2\leq N} e^{-\tau\abs{t_1-t_2}} \norm{\eta_{r_{t_1}}}_1 \norm{\eta_{r_{t_2}}}_1\\
& \ll \sum_{1\leq t_1,t_2\leq N} e^{-\tau\abs{t_1-t_2}} \norm{h_{t_1}}_1\norm{h_{t_2}}_1\\
%& = \sum_{t_1=1}^N m_\Omega(h_{t_1}) \sum_{t_2} e^{-\tau\abs{t_1-t_2}}\norm{h_{t_2}}\\
& \ll \sum_{t=1}^N m_\Omega(h_t).
\end{align*}
\note{Noter que $\norm{h_t}\leq 1$.}
Par ailleurs, 
\[
m_\Omega(h_t)=m_\Omega(\eta_{r_t})\geq c r_t^{a_\chi}\abs{\log r_t}^{b_\chi-1} \gg m_\Omega(\Omega_{r_t})
\]
et donc
\[
\sum_{t\in\N} m_\Omega(h_t)\gg \sum_{t\in\N} m_\Omega(\Omega_{r_t}) = +\infty.
\]
Ainsi, le lemme~\ref{bcd} s'applique à la suite de fonctions $(h_t)$, et montre que pour presque tout $\Delta$, l'ensemble $I(\Delta)=\{t\in\N\ |\ h_t(\Delta)>0\}$ est infini.
Cela implique que pour presque tout $\Delta\in\Omega$, il existe $t$ arbitrairement grand tel que $a_t\Delta\in\Omega_{r_t}'$.
C'est ce qu'on voulait.
\end{proof}

Nous pouvons enfin conclure cette partie avec la fin de la démonstration du théorème de Khintchine pour la variété de drapeaux $X=P\bcs G$.
%Rappelons qu'étant donnée une fonction $\psi:\R^+\to\R^+$ décroissante, nous posons, pour $t>0$,
%\[
%\Psi(u)=u^{-\beta_\chi}\psi(u)
%\quad\mbox{et}\quad
%r_t = e^{-\frac{t}{\beta_\chi}} \Psi^{-1}(e^{-t}).
%\]
%Nous voulons montrer que pour presque tout $\Delta$ il existe $t>0$ arbitrairement grand tel que
%\[
%r_\chi(a_t\Delta) \leq r_t.
%\]

\begin{proof}[Démonstration du théorème~\ref{khintchine}, cas divergent]
Soit $C>\frac{\beta_\chi}{a_\chi}$.
Quitte à remplacer $\psi$ par la fonction $\tpsi$ définie par $\tpsi(u) = \max(\psi(u), (\log u)^{-C})$, on peut supposer que
\begin{equation}\label{minor}
\forall t>0,\quad \psi(u)\geq(\log u)^{-C}.
\end{equation}
En effet,
comme $\int_e^{+\infty} (\log u)^{-\frac{Ca_\chi}{\beta_\chi}}(\log\log u)^{b_\chi-1}\frac{\dd u}{u}<+\infty$ la première partie du théorème montre que l'inégalité $d(x,v)\leq H_\chi(v)^{-\beta_\chi}(\log H_\chi(v))^{-C}$ n'a qu'un nombre fini de solutions.
Par suite, si $d(x,v)\leq H_\chi(v)^{-\beta}\tpsi(H_\chi(v))$ a une infinité de solutions, alors c'est aussi le cas pour l'inégalité $d(x,v)\leq H_\chi(v)^{-\beta}\psi(H_\chi(v))$.

Rappelons qu'étant donnée une fonction $\psi:\R^+\to\R^+$ décroissante, nous posons, pour $t>0$,
\[
\Psi(u)=u^{-\beta_\chi}\psi(u)
\quad\mbox{et}\quad
r_t = e^{-\frac{t}{\beta_\chi}} \Psi^{-1}(e^{-t}).
\]
La proposition~\ref{daniprecis} montre qu'il suffit de montrer pour presque tout $x=Ps_x$ dans $X(\R)$, pour $t>0$ arbitrairement grand,
\[
r_\chi(a_ts_xV_\chi(\Z)) \leq r_t.
\]
Cela découlera de la proposition~\ref{kdr}, si nous pouvons montrer que $\sum_{t\geq 1}m_\Omega(\Omega_{r_t})=+\infty$.

Sous l'hypothèse~\eqref{minor}, $\Psi(u)\geq u^{-\beta_\chi}(\log u)^{-C}$, donc $\Psi^{-1}(s)\gg s^{-\frac{1}{\beta_\chi}}\abs{\log s}^{-\frac{C}{\beta_\chi}}$, et
\[
r_t = e^{-\frac{t}{\beta_\chi}}\Psi^{-1}(e^{-t}) \geq t^{-\frac{C}{\beta_\chi}},
\]
et encore
\details{
\[
s = \Psi^{-1}(s)^{-\beta_\chi}\psi(\Psi^{-1}(s))
\]
\[
\Psi^{-1}(s) = s^{-\frac{1}{\beta_\chi}} \psi(\Psi^{-1}(s))^{\frac{1}{\beta_\chi}}
\]
}
\[
\Psi^{-1}(s) = s^{-\frac{1}{\beta_\chi}}\psi(\Psi^{-1}(s))^{\frac{1}{\beta_\chi}}
\gg s^{-\frac{1}{\beta_\chi}}\psi(c_0 s^{-\frac{1}{\beta_\chi}}\abs{\log s}^{-\frac{C}{\beta_\chi}})^{\frac{1}{\beta_\chi}}.
\]
Par suite,
\begin{align*}
\sum_{t\geq 1} r_t^{a_\chi} \abs{\log r_t}^{b_\chi-1}
& \gg \sum_{t\geq 1} r_t^{a_\chi} (\log t)^{b_\chi-1}\\
& = \sum_{t\geq 1} e^{-\frac{ta_\chi}{\beta_\chi}} \Psi^{-1}(e^{-t})^{a_\chi} (\log t)^{b_\chi-1}\\
& \geq \sum_{t\geq 1} \psi(c_0e^{\frac{t}{\beta_\chi}}t^{-C})^{\frac{a_\chi}{\beta_\chi}} (\log t)^{b_\chi-1}\\
& \gg \int_1^\infty \psi(c_0e^{\frac{t}{\beta_\chi}}t^{-C})^{\frac{a_\chi}{\beta_\chi}} (\log t)^{b_\chi-1} \dd t\\
& \gg \int_e^\infty \psi(u)^{\frac{a_\chi}{\beta_\chi}}(\log\log u)^{b_\chi-1} \frac{\dd u}{u},
\end{align*}
où la dernière inégalité découle du changement de variable
\[
u=c_0e^{\frac{t}{\beta_\chi}}t^{-C}, \qquad \frac{\dd u}{u} = (\frac{1}{\beta_\chi} +o(1))\dd t.
\]
Ainsi, $\sum_{t\geq 1} m_\Omega(\Omega_{r_t})=+\infty$, et le théorème est démontré.
\end{proof}

\chapter{Géométrie des espaces de réseaux}
\label{chap:reduction}

Nous rappelons dans cette partie les résultats principaux de la théorie de la réduction, due à Borel et Harish-Chandra \cite{bhc}, et exposée très clairement dans Borel \cite{borel_iga}.
Cette théorie -- que nous avons déjà utilisée dans la partie précédente pour évaluer le volume de certaines parties de $\Omega$ -- décrit quels paramètres sont nécessaires pour situer un élément $\Delta$ dans l'espace de réseaux $\Omega=G/\Gamma$.
Plus tard, nous étudierons le comportement asymptotique de ces paramètres le long de certaines orbites diagonales dans $\Omega$, pour l'appliquer à nos problèmes d'approximation diophantienne.

\section{Théorie de la réduction}
\label{sec:redbase}

Soit $G$ un $\Q$-groupe semi-simple et $\Gamma=G(\Z)$ un sous-groupe arithmétique de $G$.
Nous notons $B$ un $\Q$-sous-groupe parabolique minimal de $G$, $U$ le radical unipotent de $B$, $T$ un $\Q$-tore déployé maximal dans $B$, $M$ le $\Q$-sous-groupe anisotrope maximal du centralisateur $Z(T)^0$ de $T$ dans $G^0$, et $A=T^0(\R)$ la composante connexe des points réels de $T$.

Soit $\ka$ l'algèbre de Lie de $A$.
Le système de racines $\Sigma$ de $G$ par rapport à $T$ s'identifie à un système de racines dans l'espace dual $\ka^*$.
Fixons un système de racines simples $\Pi=\{\alpha_1,\dots,\alpha_r\}$ pour un ordre associé à $B$ et notons $\ka^-$ la chambre de Weyl opposée dans $\ka$, définie par
\[
\ka^- = \{Y\in\ka\ |\ \forall \alpha\in\Pi,\, \alpha(Y)\leq 0\}.
\]
Plus généralement, pour $\tau\geq 0$, on définit un voisinage $\ka_\tau^-$ de $\ka^-$ par
\[
\ka^-_\tau = \{Y\in\ka\ |\ \forall \alpha\in\Pi,\, \alpha(Y)\leq\tau \},
\]
et
\[
A_\tau = \exp\ka^-_\tau \subset A.
\]

\begin{figure}[H]
\begin{center}
\begin{tikzpicture}

\filldraw [blue!15!white] (-.5-3*.75, -3*.435) -- (.25,.435) -- (.25,-2) -- cycle;
\draw[blue] (-1,-1.5) node[right] {$\ka^-_\tau$};

\draw (-.5-3*.75,-3*.435) node[left] {$\{\alpha_2=\tau\}$}-- (-.5+2*.75,+2*.435);
\draw (.25,1.2) -- (.25,-2) node[right] {$\{\alpha_1=\tau\}$};
\draw[<->, color=gray] (-1,0) -- (1,0) node[right] {$\alpha_1$};
\draw[<->, color=gray] (-.5,-.87) -- (.5,0.87);
\draw[<->, color=gray] (.5,-.87) -- (-.5,.87) node[above] {$\alpha_2$};
\end{tikzpicture}
\end{center}
\caption{Dessin de $\ka^-_\tau$ pour $\SL_3$}
\end{figure}

\begin{definition}[Ensemble de Siegel]
Un \emph{ensemble de Siegel} $\FS$ de $G$ sur $\Q$ est un ensemble de la forme
\[
\FS=KA_\tau \omega,
\]
où $K$ désigne un sous-groupe compact maximal de $G$, et $\omega$ un voisinage compact de l'identité dans les points réels de $MU$.
\note{Dans la décomposition de Bruhat $PwB$ cela permet de supposer que $w\in K$, ce qui est important en particulier dans la démonstration du théorème~\ref{expalg}.}
\note{Pour voir qu'on peut toujours prendre des représentants des doubles classes $PwB$ dans $K$, il suffit de voir que $P\bcs G$ est un quotient de $K$, ce qui découle simplement du fait que $G=KMAN$. Cela ne montre pas qu'on peut voir le groupe de Weyl comme un quotient d'un sous-groupe de $K$, mais suffit peut-être pour la démonstration du théorème~\ref{expalg}, à vérifier...}
En outre, on supposera toujours que $K$ contient un ensemble de représentants du groupe de Weyl relatif $W=N(A)/A$ de $G$ sur $\Q$, où $N(A)$ est égal au normalisateur de $A$ dans $G$.
Cela est possible d'après Borel et Tits \cite[\S5]{boreltits}.
\note{Lorsque $A$ est le tore déployé sur $\R$ maximal, cela est en fait bien connu, et se montre à l'aide de l'involution de Cartan, cf. Knapp \cite[Theorem~6.57]{knapp_lgbi}. Ensuite, \cite[corollaire~5.5]{boreltits} permet de voir le groupe de Weyl sur $\Q$ comme un quotient du groupe de Weyl sur $\R$, et on a donc ce qu'on veut.}
\end{definition}

La théorie de la réduction pour les groupes arithmétiques \cite[Théorème~15.5]{borel_iga} nous assure qu'il existe un ensemble de Siegel dans $G$ qui est à peu près un domaine fondamental pour l'action de $\Gamma$ sur $G$.

\begin{theorem}[Domaine fondamental du second type]
\label{reduction}
Il existe un ensemble de Siegel $\FS$ sur $\Q$ de $G$ et une partie finie $C$ de $G(\Q)$ tels que $G=\FS C\Gamma$.
%De plus, pour toute partie finie $F$ dans $G(\Q)$, l'ensemble $\Gamma\cap(\FS F)(\FS F)^{-1}$ est fini.
\end{theorem}

Dans la suite, nous fixons une partie finie $C$ de $G(\Q)$ et un ensemble de Siegel $\FS=KA_\tau\omega$ qui satisfont la conclusion du théorème ci-dessus.
Pour $g$ dans $G$, nous dirons que l'expression
\[
g = k a n c \gamma,
\quad\mbox{avec}\ k\in K,\ a\in A_\tau,\ n\in\omega,\ c\in C,\ \mbox{et}\ \gamma\in\Gamma
\]
est une \emph{décomposition de Siegel} de $g$.
Une telle décomposition n'est pas unique, mais nous verrons plus tard que certains de ses éléments sont essentiellement indépendants du choix de la décomposition.
En général, le théorème ci-dessous \cite[théorème~15.4]{borel_iga} nous assure qu'il n'existe qu'un nombre fini de décompositions de Siegel d'un élément $g$ de $G$.

\begin{theorem}[Propriété de Siegel]
\label{siegel}
Soit $\FS$ un ensemble de Siegel dans $G$.
Pour toute partie finie $F\subset G(\Q)$, l'ensemble $\Gamma\cap(\FS F)(\FS F)^{-1}$ est fini.
\end{theorem}

\section{Représentations et ensembles fondamentaux}
\label{sec:fonctionc}

%Soit $G$ un $\Q$-groupe semi-simple \comm{cette hypothèse est-elle vraiment nécessaire?}, et $\Gamma=G(\Z)$ un sous-groupe arithmétique de $G$.
Ayant fixé un $\Q$-tore déployé maximal $T$ dans $G$, nous choisisons une base $\Pi=\{\alpha_1,\dots,\alpha_r\}$ du système de racines $\Sigma$ associé à $T$, et notons $\varpi_1,\dots,\varpi_r$ les poids fondamentaux associés.
Pour chaque $i$, on fixe une $\Q$-représentation $V_i$ de $G$ engendrée par une unique droite rationnelle de plus haut poids $\omega_i=b_i\varpi_i$, avec $b_i\in\N^*$ minimal, et $V_i(\Z)$ un réseau rationnel dans $V_i$ stable par l'action de $\Gamma$ et contenant un vecteur de plus haut poids $e_i$.
Suivant Borel et Tits \cite[\S12.13]{boreltits}, nous dirons que les représentations $V_i$, $i=1,\dots,r$, sont les représentations \emph{fondamentales} de $G$.
Dans chaque $V_i$, on fixe aussi un réseau rationnel $V_i(\Z)$ stable par l'action de $\Gamma$ et contenant un vecteur de plus haut poids $e_i$.
Dans l'espace vectoriel $V_i$, nous noterons $\tx_i=G\cdot e_i$ l'orbite du vecteur de plus haut poids $e_i$ sous l'action de $G$.

\begin{definition}[Covolumes successifs]
Pour $g$ dans $G$, on définit les \emph{covolumes successifs} $\mu_1(g),\dots,\mu_r(g)$ de $g\Gamma$ par
\[
\mu_i(g) = \min\{ \norm{gv}\ ;\ v\in V_i(\Z)\cap\tx_i\}.
\]
\end{definition}

\begin{remark}
Les quantités $\mu_i(g)$ ne dépendent en fait que de la projection de $g$ dans $G/\Gamma$, car l'ensemble $V_i(\Z)\cap\tx_i$ est stable par l'action de $\Gamma$.
Si $\bar{g}=g\Gamma$ est un élément de $G/\Gamma$, nous écrirons indifféremment $\mu_i(\bar{g})$ ou $\mu_i(g)$ pour ses covolumes successifs.
\end{remark}

\begin{remark}
Le choix d'une autre norme sur $V_i$ change $\mu_i$ en une fonction comparable, à une constante multiplicative près.
En fait, c'est souvent à cette constante multiplicative près qu'il faut comprendre $\mu_i$.
Par exemple, comme l'ensemble $P\bcs G(\Q)/\Gamma$ est fini, il existe une partie finie $C$ dans $G(\Q)$ telle que pour tout $g$,
\[
\mu_i(g) \asymp \min\{ \norm{g \gamma c e_i}\ ;\ c\in C,\, \gamma\in\Gamma\}.
\]
\comm{En effet, si on identifie les points rationnels de $P\bcs G$ aux vecteurs primitifs de $V_i(\Z)\cap\tx_i$, on observe qu'il n'y a qu'un nombre fini d'orbites sous l'action de $\Gamma$.
Vu que cette identification se fait via l'application $g\mapsto g^{-1}[e_i]$, il serait plus naturel dans la formule ci-dessus d'écrire $\norm{g\gamma^{-1}c^{-1} e_i}$.}
\end{remark}

%Notons $X(T)$ le groupe des caractères de $T$ et $Y(T)$ le groupe des sous-groupes à un paramètre de $T$.
%Sur $\ka=Y(T)\otimes\R$, tout élément de $X(T)$ définit par dualité une forme linéaire.

Les plus hauts poids $\omega_1,\dots,\omega_r$ des représentations fondamentales s'identifient naturellement à une base de $\ka^*$,
et il existe donc un unique élément $c_0(g)\in\ka$ tel que
\[
\forall i\in[1,r],\quad \omega_i(c_0(g)) = \log\mu_i(g).
\]
La théorie de la réduction permet de montrer que l'application $c_0$ que nous venons de définir est essentiellement à valeurs dans $\ka^-$.
Ci-dessous, et dans toute la suite, on munit $\ka$ d'une structure euclidienne invariante par l'action du groupe de Weyl.

\begin{proposition}[Covolumes successifs et chambre de Weyl]
\label{camoins}
Il existe $C_0\geq 0$ tel que pour tout $g$ dans $G$, $d(c_0(g),\ka^-)\leq C_0$.
\end{proposition}
\begin{proof}
Soient $\FS=KA_\tau\omega$ et $C\subset G(\Q)$ l'ensemble de Siegel et la partie finie donnés par le théorème~\ref{reduction}.
Pour tout $g$ dans $G$, il existe une décomposition de Siegel 
\[
g=kan c\gamma,\quad \mbox{avec}\ k\in K,\ n\in\omega,\ c\in C,\ \gamma\in\Gamma\ \mbox{et}\ a=e^Y,\ Y\in\ka^-_t.
\]
Alors, pour chaque $i$, $\mu_i(g)=\mu_i(anc)\asymp \mu_i(an)$, car l'ensemble $c V_i(\Z)\cap\tx_i$ est commensurable à $V_i(\Z)\cap\tx_i$: si $D$ est un dénominateur commun des coefficients des éléments de $C$ et de leurs inverses dans la représentation $V_i$, 
\[
D V_i(\Z)\cap\tx_i \subset cV_i(\Z)\cap\tx_i \subset \frac{1}{D} V_i(\Z)\cap\tx_i.
\]
%Autrement dit, tout point de la forme $v'=cv$, avec $v\in V_i(\Z)\cap\tx_i$, vérifie $Dv'\in V_i(\Z)\cap\tx_i$, et réciproquement.
Ensuite, on remarque que $ana^{-1}$ est un élément borné dans $G$ lorsque $a$ varie dans $A_\tau$ et $n$ dans $\omega$, de sorte que pour tout $v\in V_i(\Z)\setminus\{0\}$,
\[
\norm{an v} \asymp \norm{a v} \gg e^{\omega_i(Y)},
\]
car $e^{\omega_i(Y)}$ est essentiellement la plus petite valeur propre de $a$ dans $V_i$.
Cela montre que $\mu_i(g)\gg e^{\omega_i(Y)}$, i.e. $\log\mu_i(g) \geq \omega_i(Y) - O(1)$.
Comme $\norm{g\gamma^{-1}c^{-1}e_i}\asymp e^{\omega_i(Y)}$, on a en fait $\mu_i(g) \asymp e^{\omega_i(Y)}$ et donc $\log\mu_i(g) = \omega_i(Y)+ O(1)$.
En d'autres termes,
\[
d(c_0(g),Y) = O(1),
\]
et cela montre ce qu'on veut, car par définition de $A_\tau=\exp\ka^-_\tau$, l'élément $Y$ est à distance bornée de $\ka^-$.
\end{proof}

\begin{exercise}
Vérifier que pour $G=\GL_d$, la proposition ci-dessus est exactement équivalente au second théorème de Minkowski, sans l'optimalité des constantes.
\end{exercise}

Comme il est plus commode de travailler avec des éléments qui sont vraiment dans $\ka^-$, plutôt que dans un petit voisinage, nous remplaçons la fonction $c_0$ sur $G$ par une fonction $c$, qui lui est proche, mais à valeurs dans $\ka^-$.
On munit $\ka$ d'un produit scalaire invariant par l'action du groupe de Weyl.
Comme $\ka^-$ est une partie convexe de $\ka$ on dispose d'une projection $p_{\ka^-}:\ka\to\ka^-$, qui à $Y_0$ associe l'unique élément $Y$ tel que $d(Y_0,Y)=d(Y_0,\ka^-)$.

\begin{definition}
On définit la fonction $c:G\to\ka^-$ par la formule
\[
c(g) = p_{\ka^-}(c_0(g)).
\]
\end{definition}

\begin{remark}
D'après la proposition~\ref{camoins}, on a $c(g)=c_0(g)+O(1)$.
En effet, $c(g)$ est le point de $\ka^-$ le plus proche de $c_0(g)$.
\end{remark}

La fonction $c(g)$ sur $G/\Gamma$ s'interprète naturellement à partir de la théorie de la réduction, c'est le contenu de la proposition suivante, qui apparaît implicitement dans la démonstration de la proposition~\ref{camoins}.

\begin{proposition}[Composante diagonale d'une décomposition de Siegel]
Si $g=kanc\gamma$ est une décomposition de Siegel de $g$, alors la composante $a$ est bien déterminée à une constante multiplicative près.
En fait,
\[
a=e^{c(g)+O(1)}.
\]
\end{proposition}
\begin{proof}
Dans la représentation fondamentale $V_i$, les réseaux $V_i(\Z)$ et $\sum_{c\in C} cV_i(\Z)$ sont commensurables, donc leurs minima successifs sont comparables.
L'écriture $g=kanc\gamma$ montre donc que le premier minimum $\lambda_1(gV_i(\Z))$ est comparable à $e^{\omega_i(Y)}$, où $a=e^Y$, $Y\in\ka^-_t$, et ce premier minimum est atteint sur un point de $V_i(\Z)\cap\tx_i$.
Par conséquent, pour chaque $i$, $\omega_i(Y)=\log\lambda_1(gV_i(\Z))+O(1) = \log \mu_i(g)+O(1) = \omega_i(c(g))+O(1)$, ce qu'il fallait démontrer.
\end{proof}

\section{Drapeau partiel associé à un réseau}

Nous avons vu que dans une décomposition de Siegel $g=kan c\gamma$, l'élément $a$ ne dépend pas du choix de la décomposition, à constante multiplicative près.
Selon la position de $a$ dans $A_\tau$, la classe de l'élément $c\gamma$ modulo un certain sous-groupe parabolique $P_g$ dépendant de $g$ est aussi indépendante de la décomposition.
Nous détaillons maintenant cette construction, dont les idées nous seront utiles plus tard.

\bigskip

Rappelons que $\Pi=\{\alpha_1,\dots,\alpha_r\}$ est une base du système de racines de $G$ pour $T$, pour un ordre associé au parabolique minimal $B$.
À chaque partie $\theta\subset\Pi$ on associe un sous-groupe parabolique de la façon suivante:
on définit un sous-tore de $T$ par
\[
S_\theta = (\bigcap_{\alpha\in\theta}\ker\alpha)^0
\]
puis
\[
P_\theta = Z(S_\theta) U,
\]
où $Z(S_\theta)$ désigne le centralisateur de $S_\theta$ dans $G$, et $U$ le radical unipotent de $B$.
Les poids de $T$ dans la représentation adjointe sur l'algèbre de Lie $\mathfrak{p}_\theta$ correspondent aux racines positives et aux racines négatives dont la décomposition en racines simples est faite d'éléments de $\theta$.

\begin{proposition}[Drapeau partiel associé à $g\Gamma$]
\label{drapartiel}
Soit $\FS$ un ensemble de Siegel pour $G$ et $C$ une partie finie de $G(\Q)$ telle que $G=\FS C \Gamma$.
Il existe une constante $C_0\geq 0$ telle que l'énoncé suivant soit vérifié.
Pour $g\in G$,
soit
\[
\theta_g = \{i\in[1,r]\ |\ \alpha_i(c(g))\geq -C_0\}
\]
et $Q_g=P_{\theta_g}$ le sous-groupe parabolique associé.
Si $Q$ est un sous-groupe parabolique contenant $Q_g$, alors dans une décomposition de Siegel
\[
g = kan c\gamma,\quad\mbox{avec}\ kan\in\FS,\ c\in C,\ \mbox{et}\ \gamma\in\Gamma,
\]
la classe $D_g = Q c\gamma\in Q\bcs G$ est indépendante du choix de la décomposition de Siegel.
De plus, l'application $h\mapsto D_h$ est localement constante.
\end{proposition}

\begin{remark}
Si $c(g)=0$, on a nécessairement $Q=G$, et la proposition ne donne aucune information supplémentaire sur $g\Gamma$.
\end{remark}

\begin{example}
Dans le cas où $G=\SL_d$, l'espace $\Omega$ s'identifie à l'espace des réseaux unimodulaires de $\R^d$.
Notons 
\[
\lambda_1(g)\leq\dots\leq \lambda_d(g)
\]
les minima successifs de $g\Z^d$, et choisissons une famille linéairement indépendante $v_1,\dots,v_d$ dans $\Z^d$ telle que
\[
\forall i\in[1,d],\quad \norm{gv_i} = \lambda_i(g).
\]
La condition $\alpha_i(c(g))\geq -C_0$ se traduit en termes des minima successifs par l'inégalité
\[
\lambda_{i+1}(g) \leq e^{C_0}\lambda_i(g).
\]
Si $i\not\in\theta_g$, on a donc $\lambda_{i+1}(g)>\lambda_i(g)$, et le sous-réseau $\Delta_i=\Z v_1\oplus\dots\oplus\Z v_i$ est uniquement défini.
Le drapeau $D_g$ s'identifie au drapeau partiel
\[
\{0\}<\Delta_{i_1} < \dots < \Delta_{i_k}<\Z^d, \quad \{i_1,\dots,i_k\}=[1,r]\setminus\theta_g.
\]
\end{example}

À cause du cas particulier donné en exemple ci-dessus, nous dirons que $D_g$ est le \emph{drapeau partiel} associé à l'élément $g\in \Omega$.
La démonstration de la proposition~\ref{drapartiel} repose sur l'observation importante suivante.

\begin{proposition}[Racine et covolume]
\label{racine}
Soit $G$ un $\Q$-groupe semi-simple et $V_k$ la représentation fondamentale de $G$ associée au poids $\omega_k$.
Soit $g$ dans $G$ et $v_0$ dans $V_k(\Z)\cap\tx_k$ tel que $\mu_k(g)=\norm{gv_0}$.
Pour tout vecteur $v\in V_k(\Z)$ linéairement indépendant de $v_0$.
\[
\norm{gv} \gg e^{-\alpha_k(c(g))}\mu_k(g).
\]
La constante implicite dans la notation de Vinogradov ne dépend pas de $g\in G$.
En outre, il existe une constante $C_0\geq 0$ telle que si $\alpha_k(c(g))\leq -C_0$, et $g=kanc\gamma$ est une décomposition de Siegel de $g$, alors $v_0$ est colinéaire à $\gamma^{-1}c^{-1}e_k$.
\end{proposition}
\begin{proof}
Dans cette démonstration, on munit l'espace des racines d'un produit scalaire invariant par l'action du groupe de Weyl.
Quitte à changer le réseau rationnel $V_k(\Z)$, on peut supposer qu'on dispose d'une base $(u_i)$ de $V_k(\Z)$ constituée de vecteurs de poids.
Le plus haut poids $\omega_k$ est associé au vecteur $e_k=u_1$, et tout autre poids de la représentation peut s'écrire
\[
\omega = \omega_k -\sum_i n_i\alpha_i,
\]
où les $n_i$ sont des entiers naturels.
De plus, comme $\omega_k$ est le plus haut poids de la représentation et $\bracket{\omega_k,\alpha_i}=0$ si $i\neq k$,
\[
\norm{\omega_k}^2 \geq  \norm{\omega}^2
= \norm{\omega_k}^2+\norm{\sum_i n_i\alpha_i}^2 -2n_k\bracket{\omega_k,\alpha_k},
\]
\note{D'après un résultat de Satake, démontré aussi dans Borel-Tits \cite[Proposition~12.16]{boreltits}, l'ensemble $I=\{i\ |\ n_i\geq 1\}$ est connexe et contient $k$.
Cela implique en particulier que $n_k\geq 1$.}
et donc $n_k\geq 1$ si $\omega\neq\omega_k$.

Par abus de notation, on note encore $\omega_k$ le caractère de $B$ associé au poids $\omega_k$.
Posons $\Phi:g\mapsto\norm{ge_k}$.
Bien sûr,
\[
\forall b\in B,\quad \Phi(gb)=\abs{\omega_k(b)}\cdot \Phi(g),
\]
i.e. $\Phi$ est une fonction de type $(B,\omega_k)$, au sens de Borel~\cite[\S14.1]{borel_iga}.
%\comm{Tel qu'il est énoncé, le théorème~16.9 de Borel s'applique à une fonction de type $(B,\omega_k)$, où $B$ est un $\Q$-sous-groupe parabolique minimal. C'est bien le cas de la fonction que nous considérons ici, donc ce n'est pas un problème. Pour ne pas alourdir la rédaction, nous ne mentionnons pas ce détail.}
En outre,
\[
\forall \gamma\in \Gamma\cap B,\quad \abs{\omega_k(\gamma)}=1,
\]
et donc $\Phi$ est invariante à droite par $\Gamma\cap B$.
Soit enfin
\[
\theta=\{i\ |\ \bracket{\alpha_i,\omega_k}>0\}=\{k\}
\quad\mbox{et}\quad
\theta'=[1,r]\setminus\{k\}.
\]
Alors, $P_{\theta'}=\Stab\R e_k$, donc, si $L_{\theta'}$ désigne un sous-groupe semi-simple (donc sans caractère) de $P_{\theta'}$, on doit avoir
\[
\forall \ell\in L_{\theta'},\quad \Phi(g\ell)
=\norm{g\ell e_k} = \abs{\omega_k(\ell)}\cdot \norm{g e_k}= \norm{g e_k}
= \Phi(g),
\]
ce qui montre que $\Phi$ est invariante à droite par $L_{\theta'}$.
D'après Borel~\cite[théorème~16.9]{borel_iga}, il existe une partie finie $C'\subset G(\Q)$ et un ensemble de Siegel $\FS=KA_\tau\omega$ tels que pour tout $g\in G$, la fonction
\[
\begin{array}{llcr}
\Phi_g: & \Gamma C' & \to & \R^+\\ 
& u &\mapsto & \Phi(gu),
\end{array}
\]
atteigne son minimum en un point de $\Gamma C'\cap g^{-1}\FS$.
Remarquons qu'à une constante multiplicative près,
\[
\mu_k(g) \asymp \min\{ \Phi_g(u)\ ;\ u\in\Gamma C'\}.
\]
Soit donc $u=\gamma' c'\in\Gamma C'$ tel que 
\[
\norm{g\gamma'c' e_k} = \Phi_g(\gamma'c') = \min\{ \norm{g\gamma'c' e_k}\ ;\ c'\in C',\,\gamma'\in\Gamma\},
\]
et, suivant l'ensemble de Siegel choisi ci-dessus,
\[
g\gamma'c'= kan.
\]
Comme les réseaux $gV_k(\Z)$ et $g\gamma'c'V_k(\Z)$ sont commensurables, leurs minima successifs sont comparables.
Dans la base de $V_k(\Z)$ constituée de vecteurs de poids, l'élément $a=e^{c(g)+O(1)}$ agit suivant la matrice $\diag(e^{\omega(c(g))})$, où $\omega$ décrit les poids de la représentation $V_k$; à une constante multiplicative près, cela donne tous les minima successifs de $gV_k(\Z)$.
La plus petite valeur propre de $a$ est $\mu_k(g)\asymp e^{\omega_k(c(g))}$.
Si $\omega=\omega_k-\sum_in_i\alpha_i$ est un poids différent de $\omega_k$, alors d'après le calcul ci-dessus, $n_k\geq 1$, et donc
\[
e^{\omega(c(g))} \geq e^{\omega_k(c(g))-\alpha_k(c(g))} = e^{-\alpha_k(c(g))}\mu_k(g).
\]
Soit maintenant $v\in V_k(\Z)$.
Si $(e_\omega)_\omega$ désigne une base de $V_k(\Z)$ constituée de vecteurs de poids, le réseau engendré par les vecteurs $\gamma'c'e_{\omega}$
%, où $\omega$ décrit les poids de $V_k(\Z)$,
est commensurable à $V_k(\Z)$ à un facteur borné près.
Pour un certain entier $D$ indépendant de $v$, on peut donc écrire
\[
Dgv = \sum_{\omega} n_\omega g\gamma'c'e_\omega = \sum_{\omega} n_\omega kan e_\omega,
\]
avec $n_\omega\in\N$, et les termes de cette somme sont essentiellement orthogonaux.
Si $v$ n'est pas colinéaire à $v_0$, il existe $\omega\neq\omega_k$ tel que $\abs{n_\omega}\geq 1$, et cela implique
\[
\norm{gv} \gg \frac{1}{D} \norm{kan e_\omega} \gg e^{\omega(c(g))} \geq e^{-\alpha_k(c(g))}\mu_k(g).
\]
La dernière assertion découle de cette inégalité et de ce que si $g=kanc\gamma$ est une décomposition de Siegel de $g$, alors
\[
\norm{g\gamma^{-1}c^{-1}e_k} \asymp \norm{ae_k} \asymp e^{\omega_k(c(g))} \asymp \mu_k(g).
\]
\end{proof}

\begin{proof}[Démonstration de la proposition~\ref{drapartiel}]
Pour $k\in[1,r]$, soit $P_k = P_{[1,r]\setminus\{k\}}$ le $k$-ième sous-groupe parabolique maximal.
Comme
\[
Q_g = P_{\theta_g} = \bigcap_{k\not\in\theta_g} P_k,
\]
il suffit de montrer que pour chaque $k\not\in\theta_g$, l'élément $P_kc\gamma$ ne dépend pas du choix de la décomposition de Siegel.

Soit donc $k\not\in\theta_g$ fixé, et $V_k$ la représentation fondamentale associée.
Si $C_0$ est choisi assez grand, l'inégalité $\alpha_k(c(g))<-C_0$, avec la proposition~\ref{racine}, 
%montre qu'au signe près, il existe un unique vecteur primitif $v\in V_k(\Z)$ tel que 
%\[
%\norm{gv} = \min_{u\in V_k(\Z)} \norm{gu}.
%\]
%Plus précisément, la démonstration de la proposition~\ref{racine}
montre que si $g=kan c\gamma$ est une décomposition de Siegel et $u\in V_k(\Z)$, l'égalité
\[
\norm{gu} = \mu_k(g),
\]
ne peut être atteinte que si $u$ est colinéaire à $\gamma^{-1}c^{-1}e_k$.
Identifiant $P_k\bcs G$ à l'orbite de la droite $[e_k]$ de plus haut poids dans la représentation fondamentale $V_k$, grâce à l'application $g\mapsto [g^{-1}e_k]$, cela montre que $P_kc\gamma$ ne dépend pas du choix de la représentation de Siegel: c'est la direction qui réalise le premier minimum $\lambda_1(gV_k(\Z))$ du réseau $gV_k(\Z)$.

Pour chaque $k\not\in\theta_g$, la proposition~\ref{racine} montre que $\lambda_1(gV_k(\Z))<\lambda_2(gV_k(\Z))$, donc le vecteur de $V_k(\Z)$ qui réalise le premier minimum est localement constant au voisinage de $g$.
Cela montre la dernière assertion de la proposition.
\end{proof}

%
%On peut généraliser la proposition ci-dessus de la façon suivante.
%
%\begin{proposition}[Racine et covolume]
%\label{racine}
%Soit $G$ un $\Q$-groupe semi-simple et $V_k$ la représentation fondamentale de $G$ associée au poids $\omega=\sum_{i=1}^rn_i\omega_i$.
%Soit $g$ dans $G$ et $v_0$ dans $V_\omega(\Z)\cap\tx_\omega$ tel que 
%\[
%\mu_\omega(g)=\norm{gv_0} \asymp \lambda_1(gV_\omega(\Z)).
%\]
%Notons aussi
%\[
%\alpha_\omega(g) = \max_{i: n_i\neq 0} \alpha_i(c(g)).
%\]
%Pour tout vecteur $v\in V_\omega(\Z)$ linéairement indépendant de $v_0$.
%\[
%\norm{gv} \gg e^{-\alpha_\omega(c(g))}\mu_\omega(g).
%\]
%\comm{La constante implicite dans la notation de Vinogradov dépend de $G$ et $\omega$.}
%\end{proposition}
%\begin{proof}
%Comme la démonstration est très similaire à celle du cas particulier expliqué ci-dessus, et que nous n'aurons pas besoin de cette généralisation, nous omettons les détails.
%\end{proof}
%

\section{Une relation d'ordre sur $\ka$}
\label{sec:ordre}

Nous utiliserons la relation d'ordre partiel sur $\ka$ donnée par
\[
Y_1\prec Y_2
\quad\Longleftrightarrow\quad
\forall i,\ \omega_i(Y_1)\leq\omega_i(Y_2).
\]
En particulier, nous nous intéresserons à l'ensemble des minorants de $c_0(g)$ dans $\ka^-$.
C'est un ensemble convexe, et la proposition suivante montre qu'il est stable par une opération de maximum.

\begin{proposition}
Si $(Y_s)_{s\in S}$ est une famille d'éléments de $\ka^-$ alors l'élément $Y\in\ka$ défini par
\[
\forall i,\ \omega_i(Y) = \sup_{s\in S} \omega_i(Y_s)
\]
est dans $\ka^-$.
En particulier, pour tout $Y_0\in\ka$ l'ensemble
\[
m_{Y_0}=\{Y\in\ka^-\ |\ Y\prec Y_0\}.
\]
admet un unique plus grand élément.
\end{proposition}

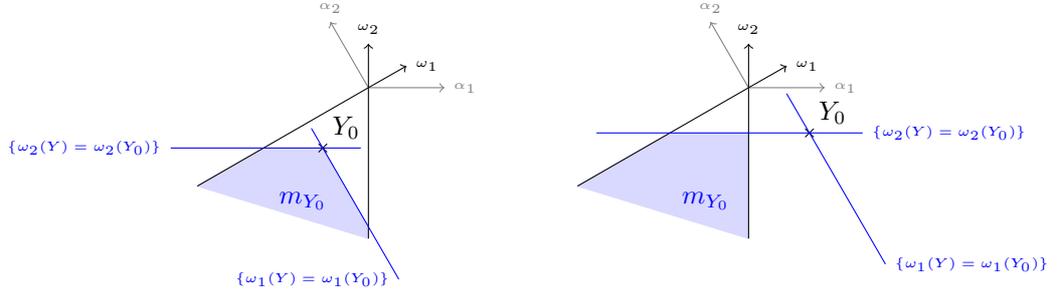
\begin{figure}[H]
\begin{center}
\begin{tikzpicture}

\filldraw [blue!15!white,xshift=-5cm] (-3*.75, -3*.435) -- (-1.37,-.8) -- (-.6,-.8) -- (0,-1.85) -- (0,-2) -- cycle;
\draw[blue,xshift=-5cm] (-1.3,-1.5) node[right] {$m_{Y_0}$};
\draw [xshift=-5cm] (-3*.75,-3*.435) -- (0,0);
\draw [xshift=-5cm] (0,0) -- (0,-2);

\draw [xshift=-5cm] (-.6,-.8) node[cross=2pt]{} node[anchor=south west]{$Y_0$};
\draw [color=blue,xshift=-5cm] (-2.6,-.8) node[left]{\tiny{$\{\omega_2(Y)=\omega_2(Y_0)\}$}}-- (-.1,-.8);
\draw [color=blue,xshift=-5cm] (-.6-.15,-.8+.15*2*.87) -- (-.6+1,-.8-2*.87) node[left]{\tiny{$\{\omega_1(Y)=\omega_1(Y_0)\}$}};

\draw[->, color=gray,xshift=-5cm] (0,0) -- (1,0) node[right] {\tiny{$\alpha_1$}};
%\draw[<->, color=gray,xshift=-5cm] (-.5,-.87) -- (.5,0.87);
\draw[->, color=gray,xshift=-5cm] (0,0) -- (-.5,.87) node[above] {\tiny{$\alpha_2$}};
\draw[->, xshift=-5cm] (0,0) -- (.5,.33*.87) node [right] {\tiny{$\omega_1$}};
\draw[->, xshift=-5cm] (0,0) -- (0,.67*.87) node [above] {\tiny{$\omega_2$}};

\filldraw [blue!15!white] (-3*.75, -3*.435) -- (-1.03,-.6) -- (0,-.6) -- (0,-2) -- cycle;
\draw[blue] (-1,-1.5) node[right] {$m_{Y_0}$};
\draw (-3*.75,-3*.435) -- (0,0);
\draw (0,0) -- (0,-2);

\draw (.8,-.6) node[cross=2pt]{} node[anchor=south west]{$Y_0$};
\draw [color=blue] (-2,-.6) -- (1.5,-.6) node[right]{\tiny{$\{\omega_2(Y)=\omega_2(Y_0)\}$}};
\draw [color=blue] (.8-.3,-.6+.3*2*.87) -- (.8+1,-.6-2*.87) node[right]{\tiny{$\{\omega_1(Y)=\omega_1(Y_0)\}$}};

\draw[->, color=gray] (0,0) -- (1,0) node[right] {\tiny{$\alpha_1$}};
%\draw[<->, color=gray,xshift=-5cm] (-.5,-.87) -- (.5,0.87);
\draw[->, color=gray] (0,0) -- (-.5,.87) node[above] {\tiny{$\alpha_2$}};
\draw[->] (0,0) -- (.5,.33*.87) node [right] {\tiny{$\omega_1$}};
\draw[->] (0,0) -- (0,.67*.87) node [above] {\tiny{$\omega_2$}};

\end{tikzpicture}
\end{center}
\caption{L'ensemble $m_{Y_0}$ pour $\SL_3$}
\end{figure}

\begin{proof}
Dans cette démonstration on identifie $\ka$ à $\ka^*$ grâce au produit scalaire usuel.
\note{Référence pour la construction est les propriétés de ce produit scalaire.}
La famille des poids fondamentaux $(\varpi_j)$ s'identifie alors à la base duale de la base $(\alpha_i)$ des racines simples:
\[
\forall i,j,\quad \bracket{\alpha_i,\varpi_j}=\delta_{ij}.
\]
Rappelons que le plus haut poids $\omega_i$ est relié à $\varpi_i$ par $\omega_i=b_i\varpi_i$, pour un certain $b_i\in\N^*$.
Si l'on décompose
\[
Y_s=\sum_i t_i^{(s)}\alpha_i,\ \mbox{avec}\ t_i^{(s)}\in\R,
\]
alors $Y=\sum_i (\sup_s t_i^{(s)})\alpha_i$ et
\[
\bracket{\alpha_k,Y} = \bracket{\alpha_k,\sum_i (\sup_s t_i^{(s)})\alpha_i}
= \sum_i (\sup_s t_i^{(s)})\bracket{\alpha_k,\alpha_i}.
\]
Il s'agit de voir que cette quantité est négative.
Cela découle de ce que pour $i\neq k$, $\bracket{\alpha_k,\alpha_i}\leq 0$, tandis que $\bracket{\alpha_k,\alpha_k}\geq 0$.
En effet, pour chaque $s$, l'inégalité $\bracket{\alpha_k,Y_s}\leq 0$ donne
\[
-\sum_{i\neq k} t_i^{(s)}\bracket{\alpha_k,\alpha_i} \geq \bracket{\alpha_k,\alpha_k} t_k^{(s)}
\]
et comme pour tout $i\neq k$, $\bracket{\alpha_k,\alpha_i}\leq 0$, cela implique
\[
-\sum_{i\neq k} (\sup_u t_i^{(u)})\bracket{\alpha_k,\alpha_i} \geq \bracket{\alpha_k,\alpha_k} t_k^{(s)}.
\]
Comme ceci vaut pour tout $s$, on trouve bien
\[
-\sum_{i\neq k} (\sup_u t_i^{(u)})\bracket{\alpha_k,\alpha_i} \geq \bracket{\alpha_k,\alpha_k} \sup_s t_k^{(s)},
\]
\comm{C'est dans cette dernière inégalité qu'intervient le fait que seul $\bracket{\alpha_k,\alpha_k}\geq 0$.
S'il y avait un autre terme $\bracket{\alpha_k,\alpha_{k'}}$, on pourrait seulement majorer $\sup_s (\bracket{\alpha_k,\alpha_k}t_k^{(s)}+\bracket{\alpha_k,\alpha_{k'}}t_{k'}^{(s)})$, qui est a priori strictement plus petit que $\bracket{\alpha_k,\alpha_k}(\sup_s t_k^{(s)}) + \bracket{\alpha_k,\alpha_{k'}}\sup_s t_{k'}^{(s)}$.}\\
c'est-à-dire $\alpha_k(Y)\leq 0$.
Soit maintenant $Y_0\in\ka$ un élément quelconque.
Comme $m_{Y_0}$ est non vide, la propriété de stabilité que nous venons de démontrer permet de définir son plus grand élément $Y$ par
\[
\forall i,\ \omega_i(Y) = \sup_{Y'\in m_{Y_0}} \omega_i(Y').
\]
\end{proof}

\begin{exercise}
Soit $G=\SL_d$.
On identifie $\ka$ aux fonctions sur $\{0,\dots,d\}$ qui s'annulent en $0$ et $d$ par l'application $Y_0\mapsto (\omega_i(Y_0))_{1\leq i\leq d-1}$.
\begin{enumerate}
\item Montrer que $\ka^-$ s'identifie à l'ensemble des fonctions convexes négatives.
\item Vérifier que la proposition ci-dessus traduit simplement le fait que la borne supérieure d'une famille de fonctions convexes négatives est encore une fonction convexe négative.
\end{enumerate}
\end{exercise}

Le dessin ci-dessous donne un exemple pour $\SL_4$. On identifie $\ka$ aux fonctions sur $\{0,\dots,4\}$ qui s'annulent en $0$ et $4$, et $\ka^-$ au sous-ensemble des fonctions convexes.
\begin{figure}[H]
\begin{center}
\begin{tikzpicture}
\draw[->] (-1,0) -- (5,0);
\draw (0,-2) -- (0,0.2);
\draw[->] (0,0.37) -- (0,1);
\foreach \x in {0,...,4}
{
\draw (\x,-0.1) -- (\x,0.1) node[anchor=south] {\tiny{\x}};
}
\draw[red,thin] (0,0) -- (1,-0.3) -- (2,0) -- (3,-1.95) -- (4,0.05);
\filldraw[red] (0,0) circle (1pt);
\filldraw[red] (1,-0.3) circle (1pt);
\filldraw[red] (2,0) circle (1pt);
\filldraw[red] (3,-2) circle (1pt);
\filldraw[red] (4,0) circle (1pt);
%\draw[red] (0,0) node[cross=2pt,thick]{};
%\draw[red] (1,-0.3) node[cross=2pt,thick]{};
%\draw[red] (2,0) node[cross=2pt,thick]{};
%\draw[red] (3,-2) node[cross=2pt,thick]{};
%\draw[red] (4,0) node[cross=2pt,thick]{};
\draw[red] (2.7,-0.5) node {$Y_0$};
\draw[blue,thin] (0,0) -- (3,-2) -- (4,0);
%\filldraw[blue] (0,0) circle (1pt);
%\filldraw[blue] (3,-2) circle (1pt);
%\filldraw[blue] (4,0) circle (1pt);
\draw[blue] (0,0) node[cross=2pt]{};
\draw[blue] (3,-2) node[cross=2pt]{};
\draw[blue] (4,0) node[cross=2pt]{};
\draw[blue] (3.8,-1) node {$Y$};
\end{tikzpicture}
\caption{$Y$ est le plus grand minorant de $Y_0$ dans $\ka^-$.}
\label{hn}
\end{center}
\end{figure}
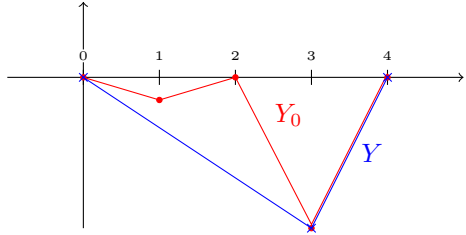

En fait, le plus grand élément de $m_{Y_0}$ s'interprète géométriquement comme une projection de $Y_0$ sur $\ka^-$.
On munit $\ka$ du produit scalaire usuel, invariant par l'action du groupe de Weyl.
Comme $\ka^-$ est une partie convexe de $\ka$ on dispose d'une projection $p_{\ka^-}:\ka\to\ka^-$, qui à $Y_0$ associe l'unique élément $Y$ tel que $d(Y_0,Y)=d(Y_0,\ka^-)$.

\begin{proposition}[Plus grand minorant et projection orthogonale]
\label{inf}
Pour tout $Y_0$ dans $\ka$, le point $p_{\ka^-}(Y_0)$ est le plus grand élément de $m_{Y_0}$.
\end{proposition}
\begin{proof}
Ici encore, on identifie $\ka$ à $\ka^*$ à l'aide du produit scalaire usuel.
Soit $Y_0\in\ka$, et
\[
I_0=\{i\ |\ \bracket{\alpha_i,Y_0}\geq 0\}.
\]
Soit $Y_1$ la projection orthogonale de $Y_0$ sur la face $\bigcap_{i\in I_0}\alpha_i^\perp$, donnée par
\[
Y_1 = Y_0 - \sum_{i\in I_0} t_i^{(0)}\alpha_i,
\]
où les $t_i^{(0)}$ sont choisis de sorte que pour tout $j\in I_0$, $\bracket{\alpha_j,Y_0}=\sum_{i\in I_0} t_i^{(0)}\bracket{\alpha_j,\alpha_i}$.
Si $A=\begin{pmatrix}\bracket{\alpha_i,Y_0}\end{pmatrix}_{i\in I_0}$, ces équations s'écrivent matriciellement $A=G\cdot T^{(0)}$, où $G=(\bracket{\alpha_j,\alpha_i})_{i,j\in I_0}$.
Comme $G$ est une matrice de Gram telle que $\bracket{\alpha_j,\alpha_i}\leq 0$ si $i\neq j$, c'est une matrice inversible, et $G^{-1}$ est à coefficients positifs, d'après \cite[Chapitre~V, \S3, \no 6, Lemme~6, page~79]{bourbaki_gal4-6}.
Or $A$ est à coefficients positifs aussi, et donc le vecteur $T^{(0)}=G^{-1}A$ est à coefficients positifs.
Cela revient à dire que pour tout $i$, $\omega_i(Y_1)\leq\omega_i(Y_0)$, i.e. $Y_1\prec Y_0$.

De même, à partir de $Y_1$, on définit $I_1=\{i\ |\ \bracket{\alpha_i,Y_1}\geq 0\}$, et $Y_2$ la projection orthogonale de $Y_1$ sur $\bigcap_{i\in I_1}\alpha_i^\perp$, ...etc.
On obtient ainsi une suite de points
\[
Y_0\succ Y_1 \succ Y_2 \succ \dots
\]
Notons que pour tout $n$, $I_n\subset I_{n+1}$.
En effet, si $\bracket{\alpha_i,Y_n}\geq 0$, alors par définition, $\bracket{\alpha_i,Y_{n+1}}=0$.
En particulier, la suite $(I_n)$ est stationnaire, donc la suite $(Y_n)$ aussi.
Soit $Y$ sa limite et $I=\{i\ |\ \bracket{\alpha_i,Y}=0\}$.

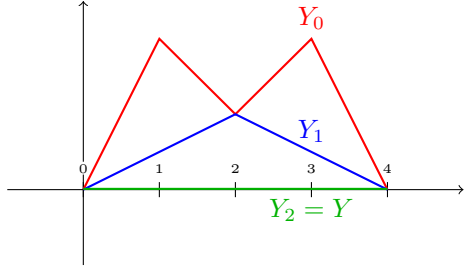
\begin{figure}[H]
\begin{center}
\begin{tikzpicture}
\draw[->] (-1,0) -- (5,0);
\draw (0,-1) -- (0,0.2);
\draw[->] (0,0.37) -- (0,2.5);
\foreach \x in {0,...,4}
{
\draw (\x,-0.1) -- (\x,0.1) node[anchor=south] {\tiny{\x}};
}
\draw[red,thick] (0,0) -- (1,2) -- (2,1) -- (3,2) node[above]{$Y_0$} -- (4,0);
\draw[blue,thick] (0,0) -- (2,1) -- (3,.5) node[above]{$Y_1$} -- (4,0);
\draw[green!70!black,thick] (0,0.01) -- (3,.01) node[below]{$Y_2=Y$} -- (4,0.01);
\end{tikzpicture}
\caption{Une suite $Y_0\succ Y_1\succ Y_2$ pour $G=\SL_4$}
\end{center}
\end{figure}

Naturellement $Y\in\ka^-$, et par construction,
\[
Y = Y_0 - \sum_{i\in I} t_i\alpha_i,
\qquad \forall i,\ t_i\geq 0.
\]
Donc $Y$ est la projection orthogonale de $Y_0$ sur $\bigcap_{i\in I}\alpha_i^\perp$.

Pour voir que $Y$ est bien égal à la projection de $Y_0$ sur $\ka^-$, on vérifie que pour tout $Y'$ dans $\ka^-$, $\bracket{Y'-Y,Y_0-Y}\leq 0$.
Écrivons $Y'-Y=\sum_{i}s_i\varpi_i$.
Pour $i\in I$, $\bracket{\alpha_i,Y}=0$ et $\bracket{\alpha_i,Y'}\leq 0$, donc $s_i\leq 0$.
Par conséquent, 
\[
\bracket{Y'-Y,Y_0-Y} = \sum_{i\in[1,r],j\in I} s_it_j\bracket{\alpha_j,\varpi_i}
= \sum_{i\in[1,r],j\in I} s_it_j\delta_{ij}
= \sum_{i\in I} s_it_i \leq 0.
\]
Donc $Y$ est la projection de $Y_0$ sur $\ka^-$.
Reste à voir que c'est bien le plus grand minorant de $Y_0$ dans $\ka^-$.

Tout d'abord, par construction, $Y$ est bien un minorant de $Y_0$.
Soit maintenant $Y'\in\ka^-$ un autre minorant de $Y_0$.
Pour tout $i\not\in I$, $\omega_i(Y)=\omega_i(Y_0)\geq \omega_i(Y')$, et le lemme ci-dessous permet d'en conclure l'inégalité souhaitée: $Y'\prec Y$.
\end{proof}

Nous concluons cette partie par le lemme technique utilisé dans la démonstration ci-dessus, et dont nous ferons encore usage au chapitre suivant.

\begin{lemma}
\label{pointsangulaires}
Soit $Y\in\ka^-$, et $I=\{i\ |\ \alpha_i(Y)=0\}$.
Si $Y'\in\ka^-$ vérifie pour chaque $i\not\in I$, $\omega_i(Y')\leq\omega_i(Y)$, alors $Y'\prec Y$.
\end{lemma}
\begin{proof}
Écrivons
\[
Y-Y'=\sum_i s_i\alpha_i,
\quad\mbox{avec}\ s_i\in\R.
\]
Par hypothèse $\omega_j(Y)\geq \omega_j(Y')$ si $j\not\in I$, et donc
\[
\forall j\not\in I,\quad s_j\geq 0.
\]
D'autre part, pour $i\in I$, $\alpha_i(Y-Y')=-\alpha_i(Y')\geq 0$, et donc
\[
\forall i\in I,\quad \sum_j s_j\bracket{\alpha_i,\alpha_j}\geq 0.
\]
Par conséquent,
\[
\forall i\in I,\quad \sum_{j\in I} s_j\bracket{\alpha_i,\alpha_j}\geq -\sum_{j\not\in I}s_j\bracket{\alpha_i,\alpha_j} \geq 0.
\]
Cette inégalité peut encore s'écrire $G\cdot S\geq 0$, où $G=(\bracket{\alpha_i,\alpha_j})_{i,j\in I}$ et $S=(s_j)_{j\in I}$.
Comme $G^{-1}$ est à coefficients positifs, cela implique $S\geq 0$.
Ainsi, pour tout $i$, $s_i\geq 0$, i.e. $Y'\prec Y$.
\end{proof}

\begin{remark}
Dans le cas $G=\SL_d$, le contenu de ce dernier lemme est assez clair.
Si $Y$ est une fonction convexe sur $[0,d]$ (segment d'entiers) avec $Y(0)=Y(d)=0$, et si $Y'$ est une autre fonction convexe sur $[0,d]$, inférieure à $Y$ en tout point angulaire (i.e. extrémal du graphe), alors la fonction $Y'$ est toujours inférieure à la fonction $Y$.
\end{remark}

\chapter{Approximation des points algébriques}
\label{chap:algebrique}

\emph{Dans tout ce mémoire, contrairement à l'usage le plus répandu, on note $\QQ$ le sous-corps de $\R$ constitué des nombres algébriques sur $\Q$.}

\bigskip

Comme précédemment, $X=P\bcs G$ est une variété de drapeaux, munie de sa distance de Carnot-Carathéodory, et d'une hauteur $H_\chi$ induite par une représentation linéaire irréductible de $G$ de plus haut poids $\chi$.
Rappelons que nous avons défini l'exposant diophantien d'un point $x$ dans $X(\R)$ par
\[
\beta_\chi(x) = \inf\{ \beta\in\R\ |\ \exists c>0:\,\forall v\in X(\Q),\, d(x,v)\geq cH_\chi(v)^{-\beta}\},
\]
et qu'il existe une constante $\beta_\chi(X)$ telle que pour presque tout $x$ dans $X(\R)$, $\beta_\chi(x)=\beta_\chi(X)$.
%Nous dirons qu'un point $x$ dans $X(\R)$ est \emph{très bien approchable} par les rationnels dans $X$ si $\beta_\chi(x)>\beta_\chi(X)$; ainsi, l'ensemble des points très bien approchables de $X(\R)$ est de mesure de Lebesgue nulle.

\bigskip

Dans ce chapitre, nous étudions l'exposant diophantien $\beta_\chi(x)$ d'un point $x\in X(\QQ)$, où $\QQ$ désigne le sous-corps de $\R$ constitué des éléments algébriques sur $\Q$.
À l'aide des outils de la théorie de la réduction rappelés à la partie précédente, le théorème du sous-espace de Schmidt nous permettra de décrire le comportement asymptotique des orbites diagonales dans l'espace de réseaux $\Omega=G/\Gamma$.
Ensuite, grâce à la correspondance développée au chapitre~\ref{chap:correspondance} nous pourrons traduire ces résultats en termes d'approximation diophantienne des points de $X$.

\section{Orbites diagonales algébriques dans $G/\Gamma$}

%La méthode de ce paragraphe est indépendante de celle du paragraphe précédent, et n'utilise pas les résultats de théorie géométrique des invariants \cite{mumford, kempf} sur lesquels se fonde Yang \cite{yang}; cela donne donc une autre démonstration du théorème~\ref{extqq}.

Dans ce paragraphe, $G$ est un $\Q$-groupe semi-simple, $T\subset G$ un tore $\Q$-déployé maximal, $A=T^0(\R)$ la composante neutre des points réels de $T$, et $\ka$ son algèbre de Lie.
Soit $(a_t)_{t>0}$ un sous-groupe à un paramètre de $A$.
Écrivant $a_t=e^{tY}$ pour un certain $Y\in\ka$, on peut choisir un ordre sur $\ka$ de sorte que $Y$ appartienne à la chambre de Weyl négative $\ka^-$ dans $\ka$:
\[
\forall t\in\R,\ a_t=e^{tY},\qquad Y\in\ka^-.
\]

\begin{remark}
Dans ce paragraphe, l'élément $Y$ est quelconque dans $\ka$, et nous n'étudions que l'espace de réseaux $\Omega=G/\Gamma$.
Plus tard, nous choisirons l'élément $Y$ comme en \eqref{at1}, et la correspondance établie au chapitre~\ref{chap:correspondance} nous permettra d'obtenir les résultats l'approximation diophantienne que nous avons en vue.
\end{remark}

Soit alors $B$ le $\Q$-sous-groupe parabolique minimal correspondant à l'ensemble des racines positives pour l'ordre choisi sur $\ka$, et 
\[
P = \{g\in G\ |\ \lim_{t\to\infty} a_tga_{-t}\ \mathrm{existe}\}
\]
le sous-groupe parabolique associé à $(a_t)$.
Soit $W$ le groupe de Weyl associé à $G$ et $T$, quotient du normalisateur de $T$ par son centralisateur, et
\[
W_P = (W\cap P)\bcs W.
\]
La décomposition de Bruhat \cite[Théorème~5.15]{boreltits} de $G$ permet d'écrire
\[
G = \bigsqcup_{w\in W_P} PwB.
\]
On fixe aussi un sous-groupe compact maximal $K$ dans $G(\R)$.
Enfin, étant donné un sous-groupe arithmétique $\Gamma$ dans $G$, on fixe une partie finie $C\subset G(\Q)$ et un ensemble de Siegel $\FS=KA_u\omega$ par rapport à $K$, $B$ et $T$ tel que
\[
G = \FS C \Gamma.
\]
Rappelons qu'une décomposition de Siegel d'un élément $g\in G$ est une écriture de $g$ sous la forme $g=kan\gamma$, avec $k\in K$, $a\in A_u=\exp\ka_u^-$, $n\in\omega$ et $\gamma\in C\Gamma$.
Ce paragraphe a pour but le théorème suivant.

\begin{theorem}[Orbites diagonales des points algébriques dans $\Omega$]
\label{diagalg}
Soit $Y\in\ka^-$, et $(a_t)=(e^{tY})_{t>0}$ le sous-groupe à un paramètre associé.
Soit $s\in G(\QQ)$.
Pour tout $t>0$, soit
\[
a_t s = k_t b_t n_t \gamma_t,
\]
une décomposition de Siegel de $a_ts$.
\begin{enumerate}
\item Il existe un élément $c_\infty\in\ka^-$ tel que
\(
\lim_{t\to\infty}\frac{1}{t}\log b_t = c_\infty.
\)
\item Si $\theta_\infty=\{\alpha\in\Pi\ |\ \alpha(c_\infty)=0\}$ et $Q_\infty=P_{\theta_\infty}$ est le sous-groupe parabolique associé à $\theta_\infty$, alors $Q_\infty\gamma_t=Q_\infty\gamma_\infty$ est indépendant de $t$ au voisinage de l'infini.
\note{$w\in W_P$ est uniquement défini par la condition $s\gamma_\infty^{-1}\in PwB$ mais dépend du choix de $\gamma_\infty$ si $Q_\infty$ contient strictement $P$.}
\item Soit $w\in W_P$ tel que $s\gamma_\infty^{-1}\in PwB$.
Alors $c_\infty=p_{\ka^-}(Y^w)$, où $Y^w=(\Ad w)^{-1}Y$, et $p_{\ka^-}:\ka\to\ka^-$ désigne la projection sur le convexe $\ka^-$.
\end{enumerate}
\end{theorem}

\begin{remark}
Soit $c:G\to\ka^-$ la fonction définie au chapitre~\ref{chap:reduction}, telle que si $g=kan\gamma$ est une décomposition de Siegel, alors $a=e^{c(g)+O(1)}$.
Le premier point du théorème se réécrit $\lim_{t\to\infty}\frac{1}{t}c(a_ts)=c_\infty$.
Avec les notations de la proposition~\ref{drapartiel}, cela implique en particulier que pour tout $t>0$ assez grand, $Q_{a_ts}\subset Q_\infty$.
Le second point énonce alors que le drapeau partiel $D_{a_ts}\in Q_\infty\bcs G$ associé à $a_ts$ est constant au voisinage de l'infini.
\end{remark}

\begin{remark}
Vu l'énoncé du théorème, la formule pour $c_\infty$ ne doit pas dépendre du choix de $\gamma_\infty$.
En d'autres termes $c_\infty$ est déterminé par la variété $PwQ_\infty\gamma_\infty$ contenant $s$.
C'est d'ailleurs ce que donne la démonstration du théorème.
%L'élément $w\in P\bcs W/Q_\infty$ tel que $p_{\ka^-}(Y^w)=c_\infty$ est unique.
\end{remark}

\note{
Soit $\delta_\infty=q\gamma_\infty$ avec $q\in Q_\infty$.
Alors, $s\delta_\infty^{-1}\in Pw\sigma B$, avec $\sigma\in ??$.
Il s'agit de voir que $p_{\ka^-}(Y^w)=p_{\ka^-}(Y^{w\sigma})$.
Cela ne semble pas tout à fait évident.
}

La démonstration du théorème~\ref{diagalg} repose sur le théorème suivant, dû à Schmidt \cite[Theorem~3A]{schmidt_da}.

\begin{theorem}[Théorème du sous-espace fort]
\label{sst}
Soit $(a_t)$ un sous-groupe diagonal à un paramètre dans $\GL_d(\R)$ et $L\in\GL_d(\QQ)$.
Notons $\lambda_1(t)\leq\dots\leq\lambda_d(t)$ les minima successifs du réseau $a_tL\Z^d$, pour $t>0$, et supposons qu'il existe $\delta>0$, $i\in[1,d[$, et un ensemble non borné $\mathfrak{N}$ tels que,
\[
\forall t\in\mathfrak{N},\quad \lambda_i(t) \leq e^{-\delta t} \lambda_{i+1}(t).
\]
Alors, il existe un sous-espace rationnel $T\leq\Q^d$ de dimension $i$ et un sous-ensemble non borné $\mathfrak{N}'\subset\mathfrak{N}$ tels que pour tout $t\in\mathfrak{N}'$, les $i$ premiers minima successifs de $a_tL\Z^d$ sont réalisés en des vecteurs $v_1,\dots,v_i$ de $a_tLT$.
\end{theorem}

Ce théorème implique d'ailleurs la proposition suivante, qui décrit de façon plus précise le comportement de l'orbite dans l'espace des réseaux.
%Cette proposition précise le théorème du sous-espace paramétrique \cite[]{faltingswustholz}.

\begin{proposition}
\label{sstdrap}
Soit $(a_t)$ un sous-groupe diagonal à un paramètre dans $\GL_d(\R)$ et $L\in\GL_d(\QQ)$.
Pour $t>0$, notons $\lambda_1(t)\leq\dots\leq\lambda_d(t)$ les minima successifs du réseau $a_tL\Z^d$.
Alors, pour chaque $i\in[1,d]$ la limite
\[
\Lambda_i = \lim_{t\to\infty} \frac{1}{t}\log\lambda_i(t)
\]
existe.
De plus, si $\Lambda_i<\Lambda_{i+1}$, alors il existe un unique sous-espace rationnel $S_i$ de dimension $i$ tel que pour tout $t>0$ suffisamment grand, les $i$ premiers minima successifs de $a_tL\Z^d$ sont atteints dans $S_i$.
\end{proposition}

\begin{remark}
Les sous-espaces $S_i$ forment un drapeau partiel dans $\Q^d$.
\end{remark}

\begin{proof}
Pour $i=1,\dots,d$, posons
\[
\Lambda_i=\liminf_{t\to\infty} \frac{1}{t}\log\lambda_i(t).
\]
On définit les indices $i_0,i_1,\dots,i_s,i_{s+1}$ par $i_0=0$, $i_{s+1}=d$, et 
\[
\Lambda_1=\dots=\Lambda_{i_1}<\Lambda_{i_1+1}=\dots=\Lambda_{i_2}<\dots<\Lambda_{i_s+1}=\dots=\Lambda_d.
\]
Nous allons montrer par récurrence sur $k\in[0,s]$ la propriété suivante:
Il existe un drapeau partiel $0=S_0<S_1<\dots<S_k$ de sous-espaces rationnels tel que pour chaque $\ell\in[1,k]$,
\begin{itemize}
\item $\dim S_\ell = i_\ell$;
\item $\forall t>0$ assez grand, $a_tLS_\ell$ contient les $i_\ell$ premiers minima de $a_tL\Z^d$;
\item si $i_{\ell-1}<i\leq i_\ell$, alors $\lim_{t\to\infty}\frac{1}{t}\log\lambda_i(t)=\Lambda_{i_\ell} = \frac{\tau(S_\ell)-\tau(S_{\ell-1})}{i_\ell-i_{\ell-1}}$.
\end{itemize}
Dans la formule ci-dessus, $\tau(W)$ désigne le taux de dilatation du sous-espace $W$ par le flot $(a_tL)$.
Si $\bw$ représente $W$ dans une puissance extérieure de $\R^d$, dans la décomposition de $L\bw$ suivant les espaces propres de $a_t$, la plus grande valeur propre apparaissant est égale à $e^{t\tau(W)}$.
De manière équivalente,
\[
\tau(W) = \lim_{t\to\infty} \frac{1}{t}\log\norm{a_tL\bw}.
\]
Pour $k=0$, la propriété est triviale.
Supposons donc la propriété vraie pour $k-1$, avec $k\in[1,s+1]$.
Soit $(t_n)$ une suite qui tend vers l'infini telle que $\Lambda_{i_k}=\lim\frac{1}{t_n}\log\lambda_{i_k}(t_n)$.
Posant $\delta=\frac{1}{3}(\Lambda_{i_k+1}-\Lambda_{i_k})$, on a, pour $n$ suffisamment grand,
\[
\lambda_{i_k}(a_{t_n}\Delta)\leq e^{t_n(\Lambda_{i_k}+\delta)}
\leq e^{t_n(\Lambda_{i_k+1}-2\delta)}
\leq e^{-t_n\delta}\lambda_{i_k+1}(a_{t_n}\Delta).
\]
D'après le théorème~\ref{sst} ci-dessus, il existe un sous-espace rationnel $S_k$ de dimension $i_k$ tel que le long d'une sous-suite de $(t_n)$, les $i_k$ premiers minima de $a_{t_n}L\Z^d$ sont toujours atteints dans $a_{t_n}LS_k$.

Par hypothèse de récurrence, pour $\ell<k$, si $i_{\ell-1}<i\leq i_\ell$,
\[ \lim\frac{1}{t_n}\log\lambda_i(t_n)=\Lambda_{i_\ell},\]
et par définition de $(t_n)$ et des indices $i_{k-1}$ et $i_k$, 
\[ \lim\frac{1}{t_n}\log\lambda_{i_k}(t_n) = \Lambda_{i_k} =\dots= \Lambda_{i_{k-1}+1}=\lim\frac{1}{t_n}\log\lambda_{i_{k-1}+1}(t_n).\]
%la dernière inégalité $\leq\liminf\frac{1}{t_n}\dots$ est claire, par définition de $\Lambda_{i_{k-1}+1$, et on doit avoir égalité, sans quoi on aurait $\Lambda_{i_{k-1}+1}<\Lambda_{i_k}$.
Comme $a_{t_n}LS_k$ contient les $i_k$ premiers minima, cela implique 
\[ \covol_{a_{t_n}LS_k}(a_{t_n}L(S_k\cap\Z^d))\leq e^{t_n\sum_{\ell=1}^k(i_\ell-i_{\ell-1})\Lambda_{i_\ell}+o(t_n)}\]
d'où l'on tire
\[ \tau(S_k)\leq \sum_{\ell=1}^k (i_\ell-i_{\ell-1})\Lambda_{i_\ell}.\]
Mais d'après l'hypothèse de récurrence, $\tau(S_{k-1}) = \sum_{\ell=1}^{k-1} (i_\ell-i_{\ell-1})\Lambda_{i_\ell}$, et donc
\begin{equation}
\label{ikmoins}
 \frac{\tau(S_k)-\tau(S_{k-1})}{i_k-i_{k-1}} \leq \Lambda_{i_k}=\dots=\Lambda_{i_{k-1}+1}.
\end{equation}
Comme l'hypothèse de récurrence implique $S_{k-1}\leq S_k$, le second théorème de Minkowski appliqué dans $a_tLS_k$, donne, pour tout $t>0$ assez grand,
%\[ \lambda_{i_{k-1}+1}(a_t\Delta) \leq \dots \leq \lambda_{i_k}(a_t\Delta) \]
%et 
\begin{align*}
\lambda_{i_{k-1}+1}(t) \dots \lambda_{i_k}(t)
& \ll \left(\frac{\covol(a_{t}(S_k\cap\Z^d))}{\covol(a_{t}(S_{k-1}\cap\Z^d))}\right)\\
& \leq \exp\big[t(\tau(S_k)-\tau(S_{k-1}))+o(t)\big].
\end{align*}
Avec \eqref{ikmoins}, cela montre qu'on doit avoir, pour chaque $i\in[i_{k-1}+1,i_k]$,
%Cela montre que $\limsup\frac{1}{t}\log\lambda_{i_{k-1}+1}(a_t\Delta)\leq\frac{\tau(S_k)-\tau(S_{k-1}}{i_k-i_{k-1}}$, et donc
%\[ \Lambda_{i_k}=\Lambda_{i_{k-1}+1}=\lim\frac{1}{t}\log\lambda_{i_{k-1}+1}(t)=\frac{\tau(S_k)-\tau(S_{k-1}}{i_k-i_{k-1}}.\]
\[ \Lambda_i=\lim\frac{1}{t}\log\lambda_i(t)=\frac{\tau(S_k)-\tau(S_{k-1})}{i_k-i_{k-1}}.\]
Pour $\eps=\frac{1}{3}(\Lambda_{i_k+1}-\Lambda_{i_k})$, pour $t>0$ assez grand, le sous-réseau $a_tL(S_k\cap\Z^d)$ contient $i_k$ vecteurs linéairement indépendants de norme inférieure à $e^{t(\Lambda_{i_k}+\eps)}$, et comme $\Lambda_{i_k+1}>\Lambda_{i_k}+2\eps$, tout vecteur de $a_tL\Z^d$ hors de $a_tLS_k$ est de norme supérieure à $e^{t(\Lambda_{i_k}+2\eps)}$.
Cela montre que pour $t>0$ assez grand, les $i_k$ premiers minima de $a_tL\Z^d$ sont atteints dans $a_tLS_k$.
\end{proof}

Nous pouvons enfin démontrer le théorème~\ref{diagalg}.

\begin{proof}[Démonstration du théorème~\ref{diagalg}]
Soit $\rho:G\to\SL_d$ une représentation rationnelle fidèle de $G$.
Pour $t>0$, et $i\in[1,d]$, notons $\lambda_i(t)=\lambda_i(\rho(a_ts)\Z^d)$.
D'après le corollaire~\ref{sstdrap}, les limites
\[
\Lambda_i = \lim_{t\to\infty}\frac{1}{t}\log\lambda_i(t)
\]
sont bien définies.
Si $a_ts = k_tb_tn_t\gamma_t$ est une décomposition de Siegel de $a_ts$, alors, à une constante multiplicative près, les minima successifs de $\rho(a_ts)\Z^d$ s'obtiennent en ordonnant les valeurs propres de la matrice diagonale $\rho(b_t)$.
Cela montre que la limite $\lim_{t\to+\infty}\frac{1}{t}\log\rho(b_t)$ est bien définie, et comme $\rho:A \to (\R_+^*)^d$ est un morphisme de groupes injectif, il en découle que la limite
\[
c_\infty = \lim_{t\to\infty} \frac{1}{t}\log b_t
\]
existe.
Comme $d(\log b_t,\ka^-)=O(1)$, on a bien sûr $c_\infty\in\ka^-$.
Cela montre le premier point du théorème.
Ensuite, la proposition~\ref{drapartiel} montre que pour tout $t>0$ suffisamment grand, l'application $t\mapsto Q_\infty \gamma_t$ est localement constante.
Par connexité, elle est donc constante au voisinage de l'infini.

\smallskip

Fixons maintenant $\gamma_\infty\in G(\Q)$ tel que pour tout $t$ suffisamment grand, $Q_\infty\gamma_t = Q_\infty\gamma_\infty$.
D'après la proposition~\ref{racine}, pour chaque $k\not\in\theta_\infty$, le plus petit vecteur de $a_tsV_k(\Z)$ est atteint dans la direction $a_ts\gamma_\infty^{-1}e_k$, et
\[
\mu_k(a_ts) = \min\{\norm{g\bv}\ ;\ \bv\in V_k(\Z)\cap\tx_k\} \asymp \norm{a_ts\gamma_\infty^{-1}e_k}.
\]
Écrivons $s\gamma_\infty^{-1}=pwb$, avec $p\in P$, $w\in W_P$ et $b\in B$.
Alors, pour $t>0$ grand,
\[
\norm{a_ts\gamma_\infty^{-1}e_k} \asymp \norm{a_twbe_k} \asymp \norm{a_twe_k} \asymp e^{t\omega_k(Y^w)}.
\]
Cela montre déjà que pour $k\not\in\theta_\infty$,
\[
\omega_k(c_\infty) = \omega_k(Y^w).
\]
Le même calcul pour $k\in\theta_\infty$ montre au moins que $\omega_k(c_\infty)\leq\omega_k(Y^w)$, et donc $c_\infty\prec Y^w$ puis
\[
c_\infty \prec p_{\ka^-}(Y^w).
\]
Réciproquement, notons que pour $k\not\in\theta_\infty$, $\omega_k(c_\infty)=\omega_k(Y^w)\geq \omega_k(p_{\ka^-}(Y^w))$, car $p_{\ka^-}(Y^w)$ minore $Y^w$.
Donc le lemme~\ref{pointsangulaires}, appliqué à $Y=c_\infty$ et $Y'=p_{\ka^-}(Y^w)$ implique que $c_\infty\succ p_{\ka^-}(Y^w)$.
Cela montre bien l'égalité souhaitée:
\[
c_\infty = p_{\ka^-}(Y^w).
\]
\end{proof}

\section{Flots semi-stables et extrémalité}
\label{sec:extalg}

Nous nous intéressons ici à un cas particulier du théorème~\ref{diagalg}, pour lequel on a toujours $c_\infty=0$.
Commençons par rappeler une terminologie introduite par Pengyu Yang.

\begin{definition}[Variété de Bruhat instable]
La variété de Bruhat standard $X_w$ est définie comme l'adhérence de Zariski d'une cellule de Bruhat:
\[
X_w = \overline{PwB},\quad w\in W_P.
\]
Une \emph{variété de Bruhat} est une variété de la forme $X_wg$, pour $g\in G$ et $w\in W_P$.
Une variété de Bruhat est dite \emph{rationnelle} si elle peut s'écrire $X_w\gamma$, avec $\gamma\in G(\Q)$.
La variété $X_wg$ est \emph{instable} pour le flot $a_t=e^{tY}$ s'il existe un poids dominant $\omega$ tel que $\omega(Y^w)<0$.
\end{definition}

Rappelons que l'on note en exposant l'action adjointe à droite de $W$ sur $\ka$: $Y^w=(\Ad w)^{-1}Y$.
La notion d'instabilité ne dépend pas du choix du représentant de $w\in W_P=(W\cap P)\bcs W$, car si $p\in W\cap P$, alors $Y^p=Y$.

\begin{remark}
Si $PwB$ est instable, alors il existe un poids fondamental $\omega_i$, $i\in[1,r]$ tel que $\omega_i(Y^w)<0$.
En effet, tout poids dominant est combinaison linéaire à coefficients entiers positifs de poids fondamentaux.
\end{remark}

Le théorème~\ref{diagalg} admet le corollaire important suivant.

\begin{corollary}[Points semi-stables dans $G/\Gamma$]
\label{sstable}
Avec les notations du théorème~\ref{diagalg}, si $s\in G(\QQ)$ n'est inclus dans aucune variété de Bruhat rationnelle instable, alors $c_\infty=0$.
Par conséquent, dans toute représentation rationnelle $V$ de $G$,
\[
\lim_{t\to\infty} \frac{1}{t}\log\lambda_1(a_tsV(\Z)) = 0.
\]
\end{corollary}
\begin{proof}
D'après le troisième point du théorème~\ref{diagalg}, on peut écrire $c_\infty=p_{\ka^-}(Y^w)$, où $w$ est tel que $s\in PwB\gamma$.
On raisonne alors par contraposée.
Si $c_\infty\neq 0$, il existe donc un poids fondamental $\omega_i$ tel que $\omega_i(Y^w)<0$, et la cellule de Bruhat $X_w\gamma$ est instable.
La seconde assertion du corollaire découle de la première, car dans la représentation irréductible rationnelle $V$ de $G$ proximale de plus haut poids $\chi$,
\[
\lambda_1(a_tsV(\Z)) \asymp \chi(b_t).
\]
\end{proof}

\begin{remark}
En fait, $c_\infty=0$ si et seulement si $s$ n'est inclus dans aucune variété de Schubert rationnelle instable.
Dans ce cas, le second point du théorème est trivial, car $Q_\infty=G$.
\end{remark}

Le corollaire ci-dessus permet déjà de démontrer un analogue du célèbre théorème de Roth \cite{roth} pour une variété de drapeaux générale.
Soit $X=P\bcs G$ une variété de drapeaux, obtenue comme quotient d'un $\Q$-groupe semi-simple $G$ par un sous-groupe parabolique $P$.
On munit $X$ de la distance de Carnot-Carathéodory usuelle, et d'une hauteur $H_\chi$ induite par une représentation irréductible de $G$ de poids dominant $\chi$, et on considère l'exposant diophantien $\beta_\chi$ sur $X$, défini au chapitre~\ref{chap:correspondance}.

\begin{definition}
Un point $x$ dans $X(\R)$ est dit \emph{extrémal} si $\beta_\chi(x)=\beta_\chi(X)$.
\end{definition}

Le théorème~\ref{exposantps} nous assure que, pour la mesure de Lebsegue, presque tous les points de $X(\R)$ sont extrémaux.
Nous donnons maintenant un critère suffisant pour qu'un point $x\in X(\QQ)$ soit extrémal.
Dorénavant, le sous-groupe à un paramètre $(a_t)$ est celui construit au chapitre~\ref{chap:correspondance}, donné par la formule
\begin{equation}\label{at3}
a_t=e^{tY}
\quad\mbox{où}\quad Y\in\ka\ \mbox{est défini par}\
\alpha(Y) = 
\left\{
\begin{array}{ll}
0 & \mbox{si}\ \alpha\in\theta\\
-1 & \mbox{si}\ \alpha\not\in\theta.
\end{array}
\right.
\end{equation}
Vues comme des parties de la variété de drapeaux $X$, les variétés de Bruhat seront plutôt appelées \emph{variétés de Schubert}.
Une variété de Schubert sera dite instable si la variété de Bruhat correspondante l'est pour le sous-groupe $(a_t)$ ci-dessus.

\begin{theorem}[Critère d'extrémalité pour les points algébriques]
\label{extqq}
Soit $X$ une variété de drapeaux rationnelle, munie de la distance de Carnot-Carathéodory usuelle, et d'une hauteur $H_\chi$ induite par une représentation irréductible de $G$.
Si $x\in X(\QQ)$ n'est inclus dans aucune variété de Schubert rationnelle instable, alors $x$ est extrémal.
\end{theorem}
\begin{proof}
Soit $s_x\in G$ tel que $x=Ps_x$.
Comme $x$ n'appartient à aucune variété de Schubert instable, $s_x$ n'appartient à aucune variété de Bruhat stable.
Le corollaire~\ref{sstable} ci-dessus montre donc que dans la représentation $V_\chi$ de poids dominant $\chi$,
\[
\lim_{t\to\infty} \frac{1}{t} \log\lambda_1(a_t s_x V_\chi(\Z)) = 0.
\]
D'après le lemme~\ref{daniextremal}, cela implique que $\beta_\chi(x)=\beta_\chi(X)$.
\end{proof}

\section{L'exposant diophantien d'un point algébrique}
\label{sec:expalg}

\emph{Comme ci-dessus, $X=P\bcs G$ est une variété de drapeaux, obtenue comme quotient d'un $\Q$-groupe semi-simple $G$ par un sous-groupe parabolique $P$.
On munit $X$ de la distance de Carnot-Carathéodory usuelle, et d'une hauteur $H_\chi$ induite par une représentation irréductible de $G$ de poids dominant $\chi$, et on considère l'exposant diophantien $\beta_\chi$ sur $X$, défini au chapitre~\ref{chap:correspondance}.
Ayant fixé dans $V_\chi$ un réseau rationnel $V_\chi(\Z)$, nous noterons $\Gamma=\Stab_GV_\chi(\Z)$ le sous-groupe arithmétique de $G$ associé.
Le sous-groupe à un paramètre $(a_t)$ est celui défini par la formule \eqref{at3}.}

\bigskip

En règle générale, pour un point $x\in X(\QQ)$ non nécessairement extrémal, le théorème~\ref{diagalg} donne une méthode pour le calcul de l'exposant diophantien $\beta_\chi(x)$.
C'est le contenu du théorème ci-dessous.
Rappelons que l'action adjointe à droite du groupe de Weyl est notée en exposant: $Y^w=(\Ad w)^{-1}Y$.
Et de même pour l'action co-adjointe: $\chi^w=\chi\circ\Ad w$.

%\begin{theorem}[Exposant d'un point algébrique]
%\label{expalg}
%Soit $x\in X$, soit $s_x\in G$ tel que $x=Ps_x$, et soit $(a_t)$ le sous-groupe diagonal défini en \eqref{at3}.
%Soit $c_\infty$, $Q_\infty$, $\gamma_\infty$ les éléments donnés par le théorème~\ref{diagalg} pour décrire l'orbite $(a_ts_x\Gamma)$ dans $G/\Gamma$, et 
%\[
%s_x\gamma_\infty^{-1} = pwb, \quad p\in P,\ w\in W_P,\ b\in B,
%\]
%une décomposition de Bruhat de $s_x\gamma_\infty^{-1}$.
%
%Soit encore $\theta_\infty$ l'ensemble de racines simples associé à $Q_\infty$, et $\pi_\infty:\ka^*\to\ka^*$ la projection orthogonale sur $\bigcap_{\alpha\in\theta_\infty}\alpha^\perp$.
%Alors
%\[
%\gamma_\chi(x) = -\pi_\infty(w^{-1}\chi)(w^{-1}Y),
%\]
%et par conséquent,
%\[
%\beta_\chi(x) = \frac{1}{-\chi(Y) + \pi_\infty(w^{-1}\chi)(w^{-1}Y)}.
%\]
%En particulier, l'exposant $\beta_\chi(x)$ est entièrement déterminé par la variété de Schubert $PwB\gamma_\infty\ni x$.
%\end{theorem}
%
%\begin{remark}
%Si la variété de Bruhat rationnelle $PwB\gamma_\infty$ est semi-stable, $\theta_\infty=\Pi$, et $\pi_\infty$ est l'application nulle.
%Donc on a bien $\gamma_\chi(x)=0$.
%\end{remark}
%
%\begin{remark}
%Si $w=I$, alors $\theta_\infty=\Pi$.
%Dans ce cas, $gamma_\chi(x)=-\chi(Y)$ est maximal, et l'exposant diophantien $\beta_\chi(x)$ vaut $+\infty$.
%On peut dire que le point $x$ est \emph{de Liouville}.
%\end{remark}

\begin{theorem}[Exposant d'un point algébrique]
\label{expalg}
Soit $x\in X$, soit $s_x\in G$ tel que $x=Ps_x$, et soit $(a_t)$ le sous-groupe diagonal défini en \eqref{at3}.
Soit $c_\infty$, $Q_\infty$, $\gamma_\infty$ les éléments donnés par le théorème~\ref{diagalg} pour décrire l'orbite $(a_ts_x\Gamma)$ dans $G/\Gamma$, et 
\[
s_x\gamma_\infty^{-1} = pwb, \quad p\in P,\ w\in W_P,\ b\in B,
\]
une décomposition de Bruhat de $s_x\gamma_\infty^{-1}$.

%Soit encore $\theta_\infty$ l'ensemble de racines simples associé à $Q_\infty$, et $\pi_\infty:\ka^*\to\ka^*$ la projection orthogonale sur $\bigcap_{\alpha\in\theta_\infty}\alpha^\perp$.
Alors
\[
\gamma_\chi(x) = -\bracket{\chi^w,p_{\ka^-}(Y^w)},
\]
et par conséquent,
\[
\beta_\chi(x) = \frac{1}{-\bracket{\chi,Y} + \bracket{\chi^w,p_{\ka^-}(Y^w)}}.
\]
En particulier, l'exposant $\beta_\chi(x)$ est entièrement déterminé par la variété de Schubert $PwB\gamma_\infty\ni x$.
\end{theorem}

\begin{remark}
Si la variété de Bruhat rationnelle $PwB\gamma_\infty$ est semi-stable, alors pour tout $i\in[1,r]$, $\omega_i(Y^w)\geq 0$, et donc $p_{\ka^-}(Y^w)=0$.
On retrouve bien $\gamma_\chi(x)=0$ et $\beta_\chi(x)=\beta_\chi(X)$.
\end{remark}

\begin{remark}
Si $w=I$, alors $Y^w=Y\in\ka^-$.
Dans ce cas, $\gamma_\chi(x)=-\bracket{\chi,Y}$ est maximal, et l'exposant diophantien $\beta_\chi(x)$ vaut $+\infty$.
Cela n'est pas surprenant, car alors $x=P\gamma_\infty$ est un point de $X(\Q)$.
%On peut dire que le point $x$ est \emph{de Liouville}.
\end{remark}

\begin{proof}
Écrivons $s_x \gamma_\infty^{-1} = p w b$.
Pour tout $p\in P$, le comportement asymptotique de $s_x\Gamma$ est équivalent à celui de $p^{-1}s_x\Gamma$, et nous pouvons donc supposer sans perte de généralité que l'élément $s_x$ est de la forme
\[
s_x = w b \gamma_\infty.
\]
D'après le théorème~\ref{diagalg}, pour $t$ suffisamment grand, dans la décomposition de Siegel
\[
a_t s_x \gamma_\infty^{-1} = k a n \gamma,
\]
l'élément $\gamma$ appartient à $Q_\infty$.
Par conséquent, dans la décomposition de Siegel
\[
a_t^wb= (w^{-1}k) a n \gamma,
\]
\note{Attention ! $B$ n'est pas un sous-groupe de Borel, mais un $\Q$-parabolique minimal. Donc $B=MAN$ contient bien $an$.}
\note{Ici, on utilise le fait que $K$ contient un système de représentants du groupe de Weyl.}
on a $\gamma\in Q_\infty$, $an\in B\subset Q_\infty$, et donc $w^{-1}k\in K\cap Q_\infty$.

%Pour $t>0$, on cherche à comprendre le réseau
%\[
%a_ts_xV_\chi(\Z) = k a n V_\chi(\Z) = w a_t^w b V_\chi(\Z).
%\]
Soit 
\[
Q_\infty=L_\infty R_\infty,
\]
\note{Ajouter un argument sur le choix de $K$ pour assurer que $K\cap Q_\infty\subset L_\infty$.}
la décomposition de Levi standard de $Q_\infty$, où $R_\infty$ désigne le radical unipotent de $Q_\infty$ et $L_\infty$ le $\Q$-sous-groupe réductif engendré par l'ensemble de racines $\theta_\infty$ associé à $Q_\infty$.
Suivant cette décomposition, écrivons
\note{Noter que $\gamma_U\in R_\infty(\Z)$ et $\gamma_L\in L_\infty(\Z)$.}
\[
b=\ell u,\quad  n=n_Ln_U\quad\mbox{et}\quad \gamma=\gamma_L\gamma_U.
\]
Comme $\gamma_L^{-1}n_U\gamma_L\gamma_U \in R_\infty$ et
\[
a_t^w \ell u = w^{-1}k a n_L\gamma_L \gamma_L^{-1}n_U\gamma_L\gamma_U, 
\]
on doit avoir
\[
a_t^w\ell = (w^{-1}k) a n_L \gamma_L,
\]
ce qui donne une décomposition de Siegel de $a_t^w\ell$.
Cela montre que les minima successifs de $a_ts_x V_\chi(\Z)$ sont essentiellement égaux à ceux de $a_t^w\ell V_\chi(\Z)$.
De plus, les drapeaux partiels de $V_\chi(\Z)$ (à $e^{\eps t}$ près) pour $a_t^w\ell$ et $a_t^w b$ sont égaux.
%De plus, posant $\tu=\gamma_L u \gamma_L^{-1}$, on peut écrire
%\[
%a_t^w \ell u = w^{-1} k a n_L \gamma_L u
%= w^{-1} k a n_L \tu \gamma_L,
%\]
%et choisissant un élément $\gamma_{\tu}\in R_u(P_\infty)(\Z)$ tel que $\tu\gamma_{\tu}$ soit dans un domaine fondamental compact de $R_u(P_{\infty})$, on obtient une décomposition de Siegel:
%\[
%a_t^w \ell u = (w^{-1} k) a (n_L \tu \gamma_{\tu}) (\gamma_{\tu}\gamma_L).
%\]
En effet, le drapeau partiel de $V_\chi(\Z)$ qui réalise les minima (à $e^{\eps t}$ près) de $a_ts_x\gamma_\infty^{-1}=a_twb$ est stable par $Q_\infty$.
Ce drapeau se décrit de la façon suivante:
\[
D_0=\{ V_1 < V_1\oplus V_2 < \dots\}
\]
où
\[
V_1 = \Vect\{ e_\beta;\ \beta(c_\infty)=\beta_1\ \mbox{minimal}\},
\]
\[
V_2 = \Vect\{ e_\beta;\ \beta(c_\infty)=\beta_2>\beta_1\ \mbox{minimal}\}
\]
et ainsi de suite.
\note{Noter que $\beta(c_\infty)$ est minimal si et seulement si $\beta=\chi-\sum_{\alpha\in\theta}n_\alpha\alpha$.}
Chacun de ces espaces est stable par $L_\infty$, qui est engendré par les racines $\alpha\in\theta_\infty$, pour lesquelles $\alpha(c_\infty)=0$.
Et chaque $V_1\oplus\dots\oplus V_k$ est stable par $Q_\infty=L_\infty R_\infty$.
Et comme $\gamma_L\in Q_\infty$ et $\gamma\in Q_\infty$, on trouve bien que les filtrations pour $a_t^w\ell$ et $a_t^wb$, données par $\gamma_L^{-1}D_0=\gamma^{-1}D_0=D_0$ sont égales.

Nous sommes donc ramenés à comprendre la géométrie du réseau $a_t^w\ell V_\chi(\Z)$, et plus particulièrement la quantité $r_\chi(wa_t^w\ell)$.

\note{Justifier que $V_\chi'$ est irréductible pour $L_\infty$? Même si cela n'a pas grande importance pour la suite.}
Comme $a_t^w\ell$ est dans le $\Q$-groupe réductif $L_\infty$, les minima successifs de $a_t^w\ell V_\chi(\Z)$ sont compatibles avec la décomposition de $V_\chi$ en sous-$L_\infty$-modules irréductibles.
\note{Semi-stabilité d'un $L_\infty$-module $V$ irréductible: soit $(\phi_i)$ la famille des poids de $V$. Les minima successifs de $a_t^w\ell V(\Z)$ sont donc les $e^{\phi_i(c_\infty)}$.
Si $\phi_0$ est le plus haut poids, chaque $\phi_i$ s'écrit $\phi_i=\phi_0-\sum_{\alpha\in\theta_\infty}n_\alpha\alpha$, et comme $\alpha(c_\infty)=0$ si $\alpha\in\theta_\infty$, on trouve bien que tous les minima successifs sont égaux à $e^{\phi_0(c_\infty)}$.}
Chacun de ces sous-modules est semi-stable pour le flot $(a_t^w\ell)$, et il nous suffit de comprendre le module $V_\chi'$ engendré par $w^{-1}e_\chi$.
Si l'on note
\[
\tau(V_\chi',a_t^w) = \log\det(a_1^w|_{V_\chi'})
\]
son taux de dilatation par le flot $(a_t^w)$, on obtient donc
\[
\gamma_\chi(x) = -\frac{\tau(V_\chi',a_t^w)}{\dim V_\chi'}.
\]
Le taux de dilatation est donné par la somme (chaque poids est compté autant de fois qu'il apparaît dans $V_\chi'$)
\[
\tau(V_\chi',a_t^w) = \sum_{\omega\prec V_\chi'} \omega(Y^w).
\]
Nous identifions maintenant $\ka$ et $\ka^*$ grâce au produit scalaire usuel.
Alors, la somme $\sum_{\omega\prec V_\chi'}\omega$ est un élément de $\ka_\infty=\bigcap_{\alpha\in\theta_\infty}\alpha^\perp$, et chaque élément $\omega\prec V_\chi'$ a la même projection sur $\ka_\infty$, qui est donc égale à $p_\infty(\chi^w)$, où $p_\infty:\ka\to\ka_\infty$ est la projection orthogonale.
Cela donne déjà une expression pour $\gamma_\chi(x)$:
\[
\gamma_\chi(x) = -\bracket{p_\infty(\chi^w),Y^w} = -\bracket{\chi^w,p_\infty(Y^w)}.
\]
Mais d'après le théorème~\ref{diagalg}, $p_\infty(Y^w)=c_\infty=p_{\ka^-}(Y^w)$, et donc
\[
\gamma_\chi(x) = -\bracket{\chi^w,p_{\ka^-}(Y^w)}.
\]
La formule pour $\beta_\chi(x)$ est alors une conséquence de la proposition~\ref{dani}.
\end{proof}

Dans certains cas, le théorème~\ref{expalg} permet de montrer que l'exposant d'un point algébrique quelconque de $X$ est minoré par l'exposant générique.
Les deux corollaires suivants sont ainsi des résultats analogues au théorème de Dirichlet \cite{dirichlet}, pour le moment valables uniquement pour les points algébriques, hélas.

\begin{corollary}[Variétés de drapeaux de rang 1]
\label{rang1}
Supposons que la variété de drapeaux $X$ s'écrive $X=P\bcs G$ avec $P$ un $\Q$-sous-groupe parabolique strict \emph{maximal} de $G$.
Alors, pour tout $x\in X(\QQ)$,
\[
\beta_\chi(x) \geq \beta_\chi(X).
\]
De plus, l'égalité a lieu si et seulement si $x$ n'appartient à aucune sous-variété de Schubert rationnelle instable.
\end{corollary}
\begin{proof}
Il existe $i\in[1,r]$ tel que l'ensemble de racines simples associé au parabolique maximal $P$ soit
\[
\theta = [1,r]\setminus\{i\}.
\]
Soit $V_\chi$ une représentation irréductible de $G$ engendrée par un vecteur de plus haut poids $e_\chi$ tel que $\Stab_G[e_\chi]=P$.
On doit avoir
\[
\chi=n\varpi_i,\quad n\geq 1.
\]
Dans la suite, nous identifions encore $\ka$ à $\ka^*$ grâce au produit scalaire usuel.
Avec cette identification,
\[
Y = -\varpi_i.
\]
Par conséquent,
\[
\gamma_\chi(x) = -\bracket{\chi^w,p_{\ka^-}(Y^w)} 
= n \bracket{Y^w,p_{\ka^-}(Y^w)} = n \norm{p_{\ka^-}(Y^w)}^2 \geq 0,
\]
et avec la proposition~\ref{dani} et la définition de $\beta_\chi(X)$ donnée au théorème~\ref{exposantps},
\[
\beta_\chi(x) \geq \beta_\chi(X).
\]
De plus, l'égalité a lieu si et seulement si $p_{\ka^-}(Y^w)=0$, i.e. $c_\infty=0$.
Cela équivaut à ce que $x$ n'appartienne à aucune sous-variété de Schubert rationnelle instable.
\end{proof}

Le même argument donne un résultat analogue si $X$ est la variété drapeau complète d'un $\Q$-groupe déployé, munie de la hauteur anti-canonique.

\begin{corollary}[Variétés de drapeaux déployée]
\label{deploy}
Soit $X=P\bcs G$, avec $G$ un $\Q$-groupe semi-simple $\Q$-déployé et $P$ un sous-groupe parabolique \emph{minimal} de $G$.
On munit $X$ de la hauteur anti-canonique $H_\chi$ associée à $\chi=\sum_{\alpha\in\Sigma^+}\alpha$, somme des racines positives.
Pour tout $x\in X(\QQ)$,
\[
\beta_\chi(x) \geq \beta_\chi(X),
\]
avec égalité si et seulement si $x$ n'appartient à aucune sous-variété de Schubert instable.
\end{corollary}
\begin{proof}
D'après \cite[Proposition~2.69]{knapp_lgbi}, la somme des racines positives est égale à la double somme des poids fondamentaux
\[
\chi = \sum_{\Sigma^+}\alpha = 2 \sum_i \varpi_i = - 2Y.
\]
On peut donc reprendre l'argument de la démonstration précédente:
\[
\gamma_\chi(x) = 2 \norm{p_{\ka^-}(Y^w)}^2 \geq 0.
\]
Cela implique $\beta_\chi(x)\geq\beta_\chi(X)$, avec égalité si et seulement si $c_\infty=p_{\ka^-}(Y^w)=0$, i.e. $x$ n'appartient à aucune sous-variété de Schubert rationnelle instable.
\end{proof}

\begin{remark}
L'inégalité
\[
\inf_{x\in X(\QQ)}\beta_\chi(x) \geq \beta_\chi(X),
\]
n'est pas valable en général.
\note{Pour ce contre-exemple, peut-on prendre pour $H_\chi$ la hauteur anti-canonique?}
Nous verrons au paragraphe~\ref{sec:drap} un exemple de variété de drapeaux $X$ munie de la distance de Carnot-Carathéodory et d'une hauteur $H_\chi$ telle qu'il existe un point algébrique $x\in X(\QQ)$ tel que $\beta_\chi(x)<\beta_\chi(X)$.
\end{remark}

Concluons cette partie par une question qui nous semble digne d'intérêt, quoique nous ne puissions pas y apporter de réponse.
Dans le cas où $X$ est une variété grassmannienne, c'est une formulation un peu plus précise d'une question de Schmidt \cite[\S16]{schmidt_grass}.

\bigskip

\note{Vue la formule pour l'exposant d'un point algébrique, la borne inférieure dans le membre de gauche est toujours atteinte; on aurait pu écrire $\min_{x\in X(\QQ)}\beta_\chi(x)$.}
\noindent\textbf{Question:}
L'égalité suivante est-elle toujours valable:
\[
\inf_{x\in X(\QQ)}\beta_\chi(x) = \inf_{x\in X(\R)}\beta_\chi(x)\ ?
\]

\chapter{Non divergence quantitative}
\label{chap:nondivergence}

\emph{Dans toute ce chapitre, $G$ est un $\Q$-groupe semi-simple de rang $r$, et $\Gamma$ un sous-groupe arithmétique.
On fixe un tore $\Q$-déployé maximal $T$ dans $G$, et un $\Q$-sous-groupe parabolique minimal $B$ contenant $T$.
On note $\{\alpha_1,\dots,\alpha_r\}$ la base du système de racines de $(G,T)$ pour l'ordre associé à $B$, et $\{\varpi_1,\dots,\varpi_r\}$ les poids fondamentaux correspondants.
Pour $i=1,\dots,r$ on note $V_i$ une représentation irréductible de $G$ engendrée par une unique droite de plus haut poids $\omega_i=b_i\varpi_i$, avec $b_i\in\N^*$ minimal.
L'algèbre de Lie réelle de $T$ est notée $\ka$, et la chambre de Weyl négative associée à $B$ est notée $\ka^-$.}

\bigskip

Le but de ce chapitre est de développer un analogue, pour un $\Q$-groupe semi-simple $G$ arbitraire, des résultats obtenus petit à petit par Margulis, Dani et Kleinbock \cite{margulis, dani_nd2, kleinbockmargulis, kleinbock_anextension} pour le groupe $\SL_d$.

\note{On devrait pouvoir travailler dans le cadre plus général d'un groupe réductif. Il faut alors définir les représentations fondamentales de $G$, en ajoutant les caractères du centre de $G$. Pour $\GL_d$, par rapport à $\SL_d$, on ajoute la représentation déterminant.}

La version précise qui nous intéresse est celle de Kleinbock \cite[Theorem~0.2]{kleinbock_anextension}.
Cela dit, même pour $\SL_d$, les énoncés que nous démontrons généralisent strictement ceux qui étaient déjà connus, en ce que la forme des voisinages de la pointe que l'on cherche à éviter est plus flexible.
Alors que ce travail était en cours de rédaction, Lindenstrauss, Margulis, Mohammadi et Shah \cite{lmms} ont aussi amélioré la non divergence pour $\SL_d$, et leurs observations permettent de retrouver une partie des résultats exposés ici, notamment le théorème~\ref{nd}, dans ce cas particulier.
Cela dit, notre approche est différente, et a le mérite de s'appliquer sans distinction à tout $\Q$-groupe semi-simple. 

\bigskip

Les résultats de la théorie de la réduction pour les groupes arithmétiques, rappelés au chapitre~\ref{chap:reduction}, nous ont permis d'associer à chaque point $x$ dans $\Omega=G/\Gamma$ un vecteur $c(x)$ dans $\ka^-$ qui décrit la position de $x$ dans $\Omega$.
Nous considérons maintenant une mesure de probabilité $\mu$ sur $G$ satisfaisant certaines propriétés de régularité, et nous expliquerons comment contrôler les valeurs de $c(g)$, lorsque $g$ est choisi aléatoirement suivant $\mu$.

Ce contrôle de la fonction $c$ hors d'un ensemble de petite mesure est ce qu'on appelle la non divergence quantitative.

\section{Sous-ensembles de $G/\Gamma$}

Pour énoncer la non divergence quantitative, nous devons en premier lieu généraliser la définition de la fonction $c:G\to\ka^-$.
Rappelons que dans la représentation fondamentale $V_i$, on note $V_i(\Z)$ un réseau rationnel stable par $\Gamma$, et
\[
\tx_i = G\cdot e_i
\]
l'orbite du vecteur de plus haut poids sous l'action de $G$.
Étant donnée une partie compacte $S\subset G$ nous lui associons un élément $c(S)\in\ka^-$, de manière analogue à ce qui a été fait en \ref{sec:fonctionc}.
Pour chaque $i$, on pose
\[
\mu_i(S) = \min_{v\in V_i(\Z)\cap\tx_i} \max_{g\in S} \norm{gv}.
\]
Puis on considère l'élément $c_0(S)$ de $\ka$ défini par
\[
\forall i\in[1,r],\quad \omega_i(c_0(S)) = \log\mu_i(S),
\]
et enfin on pose
\[
c(S) = p_{\ka^-}(c_0(S)).
\]

\begin{remark}
Si $S=\{g\}$ est réduit à un singleton, on retrouve la définition précédente: $c(S)=c(g)$.
\end{remark}

\begin{exercise}
Soit $G=\SL_3$. Pour $\eps>0$, on considère
\[
S_\eps=\{\begin{pmatrix} t & 1 & 0\\ t^2-\eps & t & 0\\ 0 & 0 & \eps^{-1}\end{pmatrix}\ ;\ t\in[0,1]\}.
\]
Montrer que lorsque $\eps>0$ tend vers zéro, la distance de $c_0(S_\eps)$ à $\ka^-$ tend vers l'infini.
\end{exercise}

\noindent Nous utiliserons aussi la relation d'ordre sur $\ka$ définie au paragraphe~\ref{sec:ordre} par
\[
Y_1\prec Y_2
\quad\Longleftrightarrow\quad
\forall i,\ \omega_i(Y_1)\leq\omega_i(Y_2).
\]

\begin{proposition}
Soit $S\subset G$ une partie compacte. 
\[
\forall g\in S,\qquad c(g) \prec c(S).
\]
\end{proposition}
\begin{proof}
Soit $g\in S$ quelconque.
Naturellement, pour chaque $i$, $\mu_i(g)\leq \mu_i(S)$.
Par conséquent, $c_0(g)\prec c_0(S)$.
Comme $c(g)$ et $c(S)$ sont les plus grands minorants respectifs dans $\ka^-$ de $c_0(g)$ et $c_0(S)$, la proposition est claire.
\end{proof}

\begin{remark}
On pourrait vouloir aussi définir un drapeau partiel associé à $S$.
Cependant, un tel drapeau (s'il existe?) n'est pas uniquement défini.
Par exemple, pour $G=\SL_4$, si $S=\{g_1,g_2\}$, avec $g_1=\diag(\eps^3,\eps^{-1},\eps^{-1},\eps^{-1})$ et $g_2=\diag(\eps,\eps,\eps,\eps^{-3})$, on trouve
$c(S)=(\log\eps,2\log\eps,\log\eps)$.
Le point $i=2$ est anguleux, au sens où $\alpha_2(c(S))<0$; et il existe pourtant deux 2-plans distincts, à savoir $v_1=e_1\wedge e_2$ et $v_2=e_1\wedge e_3$ pour lesquels $\log \sup_{g\in S}\norm{gv_i} = \omega_2(c(S))$.
Toutefois, si $S$ est le support d'une mesure régulière, le théorème de non divergence nous permettra de définir un drapeau partiel associé à $S$.
\end{remark}

Nous voulons maintenant montrer que si $S=\Supp\mu$ est le support d'une mesure suffisamment régulière sur $G$, alors l'ensemble des points $g$ dans $S$ tels que $c(g)$ est distant de $c(S)$ est de petite mesure.
Mais pour cela, nous devons d'abord expliquer quelles conditions de régularité nous allons imposer sur $\mu$.

\section{Fonctions régulières et non divergence}

Dans cette partie, $\mu$ désigne une mesure de Radon sur $G$, et $S=\Supp\mu$.
%Nous supposerons en outre que $X$ vérifie la propriété de Besicovitch pour une certaine constante $C$: si $(B(a,r_a))_{a\in A}$ est une famille de boules ouvertes, on peut en extraire une sous-famille dénombrable $(B_i)_{i\geq 1}$ de multiplicité d'intersection inférieure à $C$ telle que $A\subset\bigcup_{i\geq 1}B_i$.
Soit $U$ un ouvert de $G$.
Une fonction $f:U\to\R$ borélienne est dite \emph{$(C,\alpha)$-régulière} pour $\mu$ si elle vérifie, pour toute boule $B=B(x,r)\cap U$ dans $U$ et tout $\eps>0$,
\[
\mu(\{g\in B\ |\ \abs{f(g)}\leq \eps\norm{f}_{B,\mu}\}) \leq C\eps^\alpha\mu(B),
\]
où $\norm{f}_{B,\mu} = \sup_{g\in B\cap \Supp\mu} \abs{f(g)}$.
Ci-dessous, et dans toute la suite, lorsque $B=B(x,r)$ est une boule dans un espace métrique, et $\lambda\in\R_+^*$, on note $\lambda B=B(x,\lambda r)$.
Par exemple, $5B=B(x,5r)$.
La mesure $\mu$ sera dite $C$-doublante si pour toute boule $B(x,r)$ dans $G$,
\[
\mu(2B) \leq C\mu(B).
\]
Nous pouvons maintenant donner une version de la non divergence quantitative dans $G/\Gamma$.

\begin{theorem}[Non divergence quantitative]
\label{nd}
Soient $G$ un $\Q$-groupe semi-simple, $\Gamma$ un sous-groupe arithmétique, et $C_0, C_1,\alpha_0>0$.
Il existe deux constantes positives $C,\alpha>0$ telles qu'on ait la propriété suivante.

Soit $\mu$ une mesure borélienne finie $C_1$-doublante sur $G$ et $B\subset G$ une boule satisfaisant:
\[
\forall i\in[1,r],\,\forall v\in\tx_i\cap V_i(\Z),\quad
g\mapsto\norm{g v}
\ \mbox{est $(C_0,\alpha_0)$-régulière sur $5B$ pour $\mu$.}
\]
Alors, pour tout $\eps>0$,
\[
\mu(\{g\in B\ |\ \norm{c(g)-c(S)} \geq -\log\eps\}) \leq C\eps^\alpha\mu(B).
\]
\end{theorem}

\begin{remark}
Écrivons $g=k_ga_gn_g\gamma_g$, suivant une décomposition de Siegel, de sorte que $a_g=e^{c(g)+O(1)}$.
Le théorème nous assure que pour un ensemble de grande mesure, l'élément $a_g$ est déterminé par $S$, à une constante multiplicative près.
Si $\theta_S=\{\alpha\in\Pi\ |\ \alpha(c(S))<\log\eps\}$, on peut ajouter que l'élément $P_{\theta_S}\gamma_g$ est constant sur un ensemble de mesure relative au moins $1-C\eps^\alpha$.
Nous expliquerons cela ci-dessous, au paragraphe~\ref{ajout}.
\end{remark}

Le restant de cette partie est consacré à la démonstration de ce théorème.
On pourrait peut-être donner une démonstration directe, dans l'esprit de celle de Margulis et ses co-auteurs pour $\SL_d$, mais ce n'est pas tout à fait évident, car il faut comprendre les relations entre les petits vecteurs de $g\Gamma$ dans les différentes représentations fondamentales $V_i$.
%Pour $\SL_d$, il s'agit simplement des relations d'inclusions, si l'on identifie les vecteurs purs de $\wedge^k\Z^d$ aux sous-réseaux de rang $k$ de $\Z^d$.
%Pour $\SO_{p,q}$ aussi, mais il faut se restreindre aux sous-réseaux totalement isotropes, ce qui complique un peu l'argument. Il serait intéressant de vérifier que l'argument de Margulis fonctionne dans ce cas.
%Mais le cas d'un $\Q$-groupe semi-simple la situation est moins claire.
Nous proposons une autre approche, qui consiste à étudier d'abord séparément chaque condition de la forme $\mu_k(g)\leq\eps\mu_k(S)$, pour $k$ fixé,
puis à réunir toutes ces conditions pour obtenir l'énoncé global donné ci-dessus.
Ces deux étapes correspondent aux deux paragraphes suivants.

\section{Non divergence pour une racine fixée}

%Pour alléger les notations, nous noterons $c(x)=c(h(x))$, $\mu_i(x)=\mu_i(h(x))$, etc. pour $x\in X$ et $i\in\{1,\dots,r\}$.
Ayant fixé $k$, nous voulons comprendre l'ensemble des points $g\in S$ qui satisfont $\mu_k(g)\leq\eps\mu_k(S)$.
Cet ensemble peut être de grande mesure et c'est ce qui fait que dans l'énoncé~\ref{nd}, on doit considérer simultanément toutes les représentations fondamentales, via la fonction $c$.

\begin{exercise}
Si $h:X\to G$ est une fonction, on note $c(h)=c(\{h(x); x\in X\})$.
Soit $\eps>0$. Vérifier qu'avec l'application $h:[0,1]\to\SL_3(\R)$ définie par
\[
h(x) =\begin{pmatrix} x & 1 & 0\\ x^2-\eps^2 & x & 0\\ 0 & 0 & \eps^{-2}\end{pmatrix}
\]
on a pour tout $x$, $\mu_1(x)\leq\eps\mu_1(h)$.
\end{exercise}

L'observation cruciale, exprimée dans le théorème~\ref{ndk} ci-dessous, est en quelque sorte que si un point $g$ satisfait $\mu_k(g)\leq\eps\mu_k(S)$, i.e. 
\[
\omega_k(c(g))\leq \omega_k(c(S))+\log\eps,
\]
alors la probabilité d'avoir
\[
\alpha_k(c(g))\leq \log\eps.
\]
est très petite.
Pour la démonstration, il nous est nécessaire de contrôler tous les $\mu_i(g)$ à une petite constante multiplicative $\eps^{1/3}$ près, mais cela doit plutôt être vu comme un détail technique.

\begin{theorem}[Non divergence dans une représentation fondamentale]
\label{ndk}
Soient $C_0,C_1,\alpha_0$ des paramètres strictement positifs, et $k\in[1,r]$ \emph{fixé}.

%On suppose que $X$ est un espace métrique $K_1$-Besicovitch, $\nu$ une mesure $K_2$-doublante sur $X$, et $h:X\to G$ une application continue telle que
Soit $\mu$ une mesure de Radon $C_1$-doublante sur $G$, et $B$ une boule dans $G$.
On suppose que pour tout $v\in V_k(\Z)\cap\tx_k$,
\[
g\mapsto \norm{gv}\ \mbox{est}\ (C_0,\alpha_0)-\mbox{régulière sur}\ 5B\ \mbox{pour}\ \mu.
\]
Soient $\rho\in]0,1[$ et $\eps\in]0,1]$.
On note $S=B\cap\Supp\mu$ et on suppose que $\rho\leq \eps\mu_k(S)$.
Alors l'ensemble
\[
\cA_\eps^{(k)}(\rho) = \left\{x\in B\ \left|\ 
\begin{array}{l}
\eps^{\frac{1}{3}}\rho \leq \mu_k(x) \leq \rho,\\
\alpha_k(c(x))\leq\log\eps
\end{array}\right.
\right\}
\]
%
%\[
%\cA_\eps(\rho_1,\dots,\rho_r) = \left\{x\in B\ \left|\ 
%\begin{array}{l}
%\eps^{\frac{1}{3}}\rho_i \leq \mu_i(x) \leq \rho_i,\quad\forall i\in[1,r]\\
%\alpha_k(c(x))\leq\log\eps
%\end{array}\right.
%\right\}
%\]
vérifie
\[ \nu(\cA_\eps^{(k)}(\rho))
\leq C_2\eps^{\alpha_2}\nu(B),\]
où $\alpha_2=\frac{\alpha_0}{2}$ et $C_2=K_0C_1^3C_0$, avec $K_0$ une constante de Besicovitch pour $G\cap 5B$.
\end{theorem}

\begin{remark}
Étant donné $M\geq 0$ l'ensemble
\[ \cA_{M,\eps}^{(k)}(\rho) = \left\{x\in B\ \left|\ 
\begin{array}{l}
\eps^{M}\rho \leq \mu_k(x) \leq \rho\\
\alpha_k(c(x))\leq\log\eps
\end{array}\right.
\right\}
\]
peut être recouvert par $3M$ ensembles du type $\cA_\eps^{(k)}(\rho')$, avec $\rho'\in[\eps^M\rho,\rho]$, dont le théorème ci-dessus permet de majorer la mesure.
On a donc aussi une borne
\[
\mu(\cA_{M,\eps}^{(k)}(\rho) \leq CM\eps^\alpha\mu(B).
\]
\end{remark}

\begin{proof}[Démonstration du théorème~\ref{ndk}]
Soit $x\in\cA_\eps^{(k)}(\rho)$.
Par définition, il existe un vecteur primitif $v_x\in\tx_k\cap V_k(\Z)$ tel que
\[
\norm{xv_x}\leq\rho.
\]
Soit $B_x$ la boule de rayon maximal telle que
\[
\forall y\in B_x\cap S,\quad \norm{yv_x} \leq \eps^{-1/2}\rho.
\]
Notons que $B_x\subset 3B$, $2B_x\subset 5B$ et
\[
\sup_{y\in 2B_x\cap S} \norm{yv_x}\geq \eps^{-1/2}\rho.
\]
Soit $y\in B_x\cap S\cap\cA_\eps^{(k)}(\rho)$, et $v\in V_k(\Z)$.
Si $v$ est linéairement indépendant de $v_x$, alors, d'après la proposition~\ref{racine},
\[
\norm{yv} \gg e^{-\alpha_k(c(y))}\mu_k(y)
\geq \eps^{-1}\eps^{1/3}\rho
= \eps^{-2/3}\rho
> \norm{yv_x}.
\]
Par conséquent, $v_x$ est l'unique vecteur de $V_k(\Z)$ tel que $\mu_k(y)=\norm{yv_x}$.
Comme $y\in\cA_\eps^{(k)}(\rho)$, cela implique en particulier
\[
\norm{yv_x}\leq\rho.
\]
Or, la fonction $f:y\mapsto \norm{yv_x}$ est $(C_0,\alpha_0)$-régulière sur $2B_x$ et vérifie $\norm{f}_{2B_x,\mu}\geq\eps^{-1/2}\rho$, on trouve donc
\[
\mu(B_x\cap\cA_\eps^{(k)}(\rho)) = \mu(B_x\cap S\cap\cA_\eps^{(k)}(\rho))
\leq C_0 \eps^{\alpha_0/2}\mu(2B_x)
\leq C_0 C_1\eps^{\alpha_0/2}\mu(B_x).
\]
Par la propriété de Besicovitch dans l'espace métrique $G\cap 5B$, il existe une sous-famille dénombrable $(B_{x_i})_{i\in\N}$ de multiplicité au plus $K_0$ telle que $\cA_\eps^{(k)}(\rho)\subset\bigcup_{i\in\N}B_{x_i}\subset 3B$, ce qui permet d'écrire
\[
\mu(\cA_\eps^{(k)}(\rho)) \leq \sum_{i\in\N} \mu(B_{x_i}\cap\cA_\eps^{(k)}(\rho))
 \leq C_0C_1\eps^{\frac{\alpha_0}{2}} \sum_{i\in\N} \mu(B_{x_i})
%\leq C_0C_1\eps^{\frac{\alpha_0}{2}}\mu(3B)
\leq C_0C_1^3K_0\eps^{\frac{\alpha_0}{2}}\mu(B).
\]
\end{proof}

\section{Non divergence globale}

La fin de la démonstration du théorème~\ref{nd} consiste à mettre ensemble les bornes obtenues pour chaque racine $\alpha_k$, $k=1,\dots,r$, après avoir observé que si $\norm{c(x)-c(S)}\geq -\log\eps$, alors il doit exister $k$ tel que $\mu_k(x)<\eps^{\tau}\mu_k(S)$ et $\alpha_k(c(x))<\tau\log\eps$.
Nous mettons cette observation sous la forme d'un lemme élémentaire sur les systèmes de racines.

\begin{lemma}
\label{angle}
Soit $\Sigma$ un système de racines dans un espace euclidien $\ka$ de dimension $r$.
Il existe $\tau>0$ tel que la propriété suivante soit satisfaite.\\
Pour $\eps\in(0,1)$ et $Y_1, Y_2\in\ka^-$, si $Y_1\prec Y_2$ et $\norm{Y_2-Y_1}\geq-\log\eps$, alors il existe $k\in[1,r]$ tel que
\begin{enumerate}
\item $\omega_k(Y_1) \leq \omega_k(Y_2) + \tau\log \eps$;
\item $\alpha_k(Y_1) \leq \tau\log\eps$.
\end{enumerate}
\end{lemma}
\begin{proof}
Posons
\[
Y = Y_2-Y_1 = \sum_i t_i\alpha_i.
\]
Comme $Y_1\prec Y_2$, on doit avoir $\forall i,\ t_i\geq 0$.
%En d'autres termes, si $Y=Y_2-Y_1$, on a pour chaque $i$,
%\[
%\omega_i(Y)\geq 0.
%\]
Comme
\[
\sum t_i\norm{\alpha_i} \geq \norm{Y_2-Y_1} \geq -\log\eps
\]
il doit exister $j$ tel que $t_j=\varpi_j(Y)\geq \frac{-\log\eps}{C_1}$, où $C_1=\sum_i\norm{\alpha_i}$.
Or, écrivant
\[
Y=\sum_k \alpha_k(Y)\varpi_k,
\]
on observe que
\[
\varpi_j(Y)=\sum_k\bracket{\varpi_k,\varpi_j}\alpha_k(Y) \leq \max_k\alpha_k(Y)\sum_k\bracket{\varpi_k,\varpi_j},
\]
car pour tous $j,k$, $\bracket{\varpi_k,\varpi_j}\geq 0$ et $\max_k\alpha_k(Y)\geq 0$.
Avec $C_2=\sum_k\bracket{\varpi_k,\varpi_j}$, cela montre qu'il existe $k$ tel que $\alpha_k(Y)\geq\frac{-\log\eps}{C_1C_2}$.
Posant $c=(C_1C_2)^{-1}$, on trouve
\[
\alpha_k(Y_1) = \alpha_k(Y_2)-\alpha_k(Y) \leq  \alpha_k(Y_2) + c\log\eps
\leq c\log\eps.
\]
De plus, comme $Y=\sum_i t_i\alpha_i$, cela implique
\[ \sum_i t_i\bracket{\alpha_k,\alpha_i} \geq -c\log\eps \]
et comme $\bracket{\alpha_k,\alpha_i}\leq 0$ pour $i\neq k$,
\[
\norm{\alpha_k}^2\varpi_k(Y) = \norm{\alpha_k}^2 t_k \geq -c \log\eps.
\]
Comme $\omega_k=b_k\varpi_k$ pour un certain $b_k\in\N^*$, cela montre ce qu'on voulait, avec $\tau=\frac{c}{\max_k\norm{\alpha_k}^2}$:
\[
\omega_k(Y_1) \leq \omega_k(Y_2) + \tau \log\eps.
\]
\end{proof}

Avec ce lemme, nous pouvons maintenant déduire le théorème~\ref{nd} du théorème~\ref{ndk}.

\begin{proof}[Démonstration du théorème~\ref{nd}]
Il suffit de montrer que sous les hypothèses du théorème,
\begin{equation}
\label{eps2eps}
 \mu(\{x\in B\ |\ -2\log\eps \geq \norm{c(x)-c(S)} \geq -\log \eps\})
\leq C\eps^\alpha\mu(B).
\end{equation}
En effet, on aura alors, quitte à supposer $\eps<\frac{1}{2}$,
\begin{align*}
\mu(\{x\in B\ |\ \norm{c(x)-c(S)} \geq -\log\eps\})
& \leq \sum_{m\geq 0} \mu(\{x\in B\ |\ -2\log\eps^{2^m} \geq \norm{c(x)-c(S)} \geq -\log\eps^{2^m}\})\\
& \leq C \sum_{m\geq 0} \eps^{2^m\alpha}\mu(B)\\
& \leq \frac{C}{1-2^{-\alpha}} \eps^\alpha \mu(B).
\end{align*}
Pour voir \eqref{eps2eps}, notons $\cA=\{x\in B\ |\ -2\log\eps \geq \norm{c(x)-c(S)} \geq -\log \eps\}$.
D'après le lemme~\ref{angle}, il existe $\tau>0$ tel que $\cA$ est recouvert par les ensembles
\[
\cA^{(k)} = \left\{x\in B\ \left| 
\begin{array}{c}
-2\log\eps \geq \norm{c(x)-c(S)} \geq-\log \eps\\
\omega_k(c(x)) \leq \omega_k(c(S))+\tau\log\eps\\
\alpha_k(c(x)) \leq \tau\log\eps
\end{array}
\right.
\right\}
\]
Enfin, grâce au contrôle $-2\log\eps\geq\norm{c(x)-c(S)}$ sur la norme de $c(x)-c(S)$, chaque $\cA^{(k)}$ peut être recouvert par $M\asymp (\frac{6}{\tau})^r$ ensembles de la forme
\[
\cA^{(k)}_\eps(\rho_1,\dots,\rho_r) = \left\{x\in B\ \left|\ 
\begin{array}{l}
\eps^{\frac{\tau}{3}}\rho_i \leq \mu_i(x) \leq \rho_i,\quad\forall i\in[1,r]\\
\alpha_k(c(x))\leq\tau\log\eps
\end{array}\right.
\right\}
\]
Pour chacun de ces ensembles, le théorème~\ref{ndk} donne
\[
\mu(\cA^{(k)}_\eps(\rho_1,\dots,\rho_r)) \leq C_2\eps^{\alpha_2\frac{\tau}{3}}\mu(B),
\]
et par suite,
\[
\mu(\cA) \leq \sum_k\sum_{\rho_1,\dots,\rho_r} \mu(\cA^{(k)}_\eps(\rho_1,\dots,\rho_r)) \leq C\eps^{\alpha}\mu(B),
\]
avec $\alpha=\frac{\tau\alpha_2}{3}$ et $C=rMC_2\asymp r(\frac{6}{\tau})^r C_2$.
\end{proof}

\section{Drapeau partiel pour une mesure régulière}
\label{ajout}

Comme conséquence du théorème~\ref{nd} de non divergence, nous montrons maintenant que si $S$ est le support d'une mesure régulière sur $G$, on peut définir un drapeau partiel associé à $S$.

\begin{proposition}
\label{drapartiel2}
Soient $G$ un $\Q$-groupe semi-simple, $\Gamma$ un sous-groupe arithmétique, et $C_0,\alpha_0>0$.
Il existe deux constantes $C',\alpha'>0$ telles qu'on ait la propriété suivante.

Soit $\mu$ une mesure borélienne finie sur $G$ et $B\subset G$ une boule satisfaisant:
\[
\forall i\in[1,r],\,\forall v\in\tx_i\cap V_i(\Z),\quad
g\mapsto\norm{g v}
\ \mbox{est $(C_0,\alpha_0)$-régulière sur $5B$ pour $\mu$.}
\]
Soit $S=B\cap\Supp\mu$, et
\[
I_S(\eps) = \{i\in[1,r]\ |\ \alpha_i(c(S))\geq \log\eps\}.
\]
Pour $\eps>0$ suffisamment petit, pour chaque $i\not \in I_S(\eps)$, il existe une unique direction dans $V_i$ contenant un vecteur $v_i\in V_i(\Z)\cap\tx_i$ tel que
\[
\sup_{g\in S}\norm{gv_i} = \min_{v\in V_i(\Z)\cap\tx_i} \sup_{g\in S} \norm{gv}.
\]
De plus, il existe un élément $\gamma_S\in G(\Q)$ tel que
\[
\forall i\not \in I_S(\eps),\quad v_i=\gamma_S e_i,
\]
et si on note $g=k_ga_gn_g\gamma_g$ une décomposition de Siegel de $g\in S$,
\[
\mu(\{g\in B\ |\ P_{I_S(\eps)}\gamma_g = P_{I_S(\eps)}\gamma_S\})
\geq 1 - C'\eps^{\alpha'}\mu(B).
\]
\end{proposition}

\begin{remark}
L'élément $P_{I_S(\eps)}\gamma_S\in P_{I_S(\eps)}\bcs G$ est le drapeau partiel associé à l'ensemble $S$, pour le paramètre $\eps$.
\end{remark}
\begin{proof}
D'après le théorème~\ref{nd},
\[
\mu(\{g\in B\ |\ \norm{c(g)-c(S)}\geq -\log\eps\}) \leq C\eps^\alpha\mu(B).
\]
Soit $C_1>0$ une constante dépendant de $G$, et dont nous choisirons la valeur ci-dessous.
Pour $\eps>0$ suffisamment petit, l'inégalité ci-dessus montre en particulier qu'il existe $g_1\in S$ tel que
\[
\norm{c(g_1)-c(S)} \leq -\frac{\log\eps}{C_1}.
\]
Cela implique en particulier que pour toute racine simple $\alpha$,
\[
\abs{\alpha(c(g_1))-\alpha(c(S))} \leq -\frac{\log\eps}{2}.
\]
Par conséquent, si $i\not\in I_S(\eps)$,
\[
\alpha_i(c(g_1)) < \frac{\log\eps}{2},
\]
et avec la proposition~\ref{racine}, si $v\in V_i(\Z)$ n'est pas colinéaire à la direction $v_i$ qui réalise le premier minimum de $g_1V_i(\Z)$, on doit avoir
\[
\norm{g_1v} \gg e^{-\alpha_i(c(g_1))}\mu_i(g_1)
\geq \eps^{-\frac{1}{2}} \mu_i(g_1)
\geq \eps^{-\frac{1}{2}+\frac{1}{C_1}} \mu_i(S).
\] 
Cela implique que pour chaque $i\not\in I_S(\eps)$, le vecteur $v_i\in V_i(\Z)\cap\tx_i$ engendre l'unique direction telle que
\[
\sup_{g\in S}\norm{gv_i} = \mu_i(S). %\min_{v\in V_i(\Z)\cap\tx_i} \sup_{g\in S} \norm{gv}.
\]
Comme cette direction est aussi uniquement déterminée par l'égalité $\norm{g_1v_i}=\mu_i(g_1)$ et qu'on a $\alpha_i(c(g_1))<\frac{\log\eps}{2}$, une décomposition de Siegel $g_1=kan\gamma$ permet d'écrire, pour chaque $i\not\in I_S(\eps)$,
avec $\gamma_S=\gamma$,
\[
v_i = \gamma_S^{-1} e_i.
\]
Enfin, si $g\in S$ vérifie $\norm{c(g)-c(S)}\leq-\frac{\log\eps}{C_1}$ le raisonnement fait ci-dessus pour $g_1$ s'applique aussi à $g$, et montre que si $i\not\in I_S(\eps)$, alors $v_i$ est aussi l'unique vecteur de $V_i(\Z)$ tel que $\norm{gv_i}=\mu_i(g)$.
En d'autres termes, pour chaque $i\not\in I_S(\eps)$, $\gamma_g^{-1}e_i=\gamma_S^{-1}e_i$, i.e.
\[
P_{I_S(\eps)}\gamma_g = P_{I_S(\eps)}\gamma_S.
\]
Cela démontre le dernier point de la proposition, car
\[
\mu(\{g\in S\ |\ \norm{c(g)-c(S)}\geq -\frac{\log\eps}{C_1}\}) \leq C'\eps^{\alpha'}\mu(S).
\]
\end{proof}

\chapter{Flots diagonaux dans $G/\Gamma$}
\label{chap:orbites}

\emph{Dans ce chapitre, $G$ est un $\Q$-groupe semi-simple de $\Q$-rang $r$, et $\Gamma$ un sous-groupe arithmétique.
On fixe un tore $\Q$-déployé maximal $T$ dans $G$, et un $\Q$-sous-groupe parabolique minimal $B$ contenant $T$.
On note $\{\alpha_1,\dots,\alpha_r\}$ la base du système de racines de $(G,T)$ pour l'ordre associé à $B$, et $\{\varpi_1,\dots,\varpi_r\}$ les poids fondamentaux correspondants.
Pour $i=1,\dots,r$, on note $V_i$ la représentation irréductible de $G$ engendrée par une unique droite rationnelle de plus haut poids $\omega_i=b_i\varpi_i$, avec $b_i\in\N^*$ minimal.
On note $A=T^0(\R)$ la composante neutre des points réels de $T$.
L'algèbre de Lie de $A$ est notée $\ka$, et la chambre de Weyl négative associée à $B$ est notée $\ka^-$.
On considère un sous-groupe à un paramètre $(a_t)$ dans $A$, donné par $a_t=e^{tY}$, avec $Y\in\ka^-$.}

\bigskip

Ce chapitre a pour but de décrire le comportement asymptotique d'une orbite diagonale dans l'espace de réseaux $\Omega=G/\Gamma$, lorsque le point de départ est choisi aléatoirement, suivant une mesure suffisamment régulière.
On observe que ces résultats sont très similaires à ceux obtenus au chapitre~\ref{chap:algebrique} lorsque le point de départ était un réseau algébrique.
On donne d'ailleurs une généralisation du théorème~\ref{diagalg}, pour une orbite partant d'un point choisi aléatoirement sur une variété algébrique définie sur $\QQ$.

Ces résultats généraux sur l'espace de réseaux $\Omega$ serviront de base à notre étude de l'approximation diophantienne dans $X=P\bcs G$, détaillée au chapitre suivant.

\section{Mesures régulières}

Commençons par une définition qui résume les propriétés de régularité que doit satisfaire une mesure pour que s'applique le théorème de non divergence.

\begin{definition}
Nous dirons qu'une mesure borélienne sur $G$ est \emph{localement régulière} en un point $s_0$ dans $G$ s'il existe une boule ouverte $B=B(s_0,r)$ et des constantes $C,\alpha>0$ telles que
\[
\forall i\in[1,r],\ \forall v\in\tx_i\cap V_i(\Z),\ \forall g\in G,\quad s\mapsto\norm{g s v}\ \mbox{est $(C,\alpha)$-régulière sur}\ B\ \mbox{pour}\ \mu.
\]
\end{definition}

Cette définition est stable par translation par un élément de $G$: si $\mu$ est localement régulière en $s_0$, alors $g_*\mu$ est localement régulière en $gs_0$.
Grâce au théorème~\ref{nd} de non divergence quantitative, nous allons voir que si $(a_t)$ est un flot diagonal fixé dans $G$, et $\mu$ une mesure localement régulière en $s_0$, alors il existe une boule ouverte $B=B(s_0,r)$ telle que pour $s\in B$, les orbites $(a_ts\Gamma)_{t>0}$ ont presque toutes le même comportement asymptotique.
Rappelons que pour une partie compacte $S\subset G$, on définit $c_0(S)$ par
\[
\forall i\in [1,r],\quad \omega_i(c_0(S)) = \log \mu_i(S)
\]
où $\mu_i(S) =\min_{v\in\tx_i\cap V_i(\Z)} \max_{s\in S} \norm{sv}$,
et
\[
c(S) = p_{\ka^-}(c_0(S)).
\]
Avec le théorème~\ref{nd} de non divergence, ce vecteur permet de contrôler la position dans $\Omega$ de $s\Gamma$, lorsque $s$ est choisi aléatoirement dans $S$.

\begin{theorem}[Orbites diagonales partant d'une mesure régulière]
\label{diagan}
Soit $(a_t)_{t>0}$ un sous-groupe à un paramètre dans $A$, et $\mu$ une mesure sur $G$ localement régulière en un point $s_0\in G$.
Il existe une boule ouverte $B$ centrée en $s_0$ telle que pour $\mu$-presque tout $s\in B$, notant $S=B\cap\Supp\mu$,
\[
\lim_{t\to\infty} \frac{1}{t}(c(a_ts) - c(a_tS)) = 0.
\]
\end{theorem}
\begin{proof}
Soient $\delta>0$ et $t>0$.
D'après le théorème~\ref{nd} appliqué à la mesure $(a_t)_*\mu$, avec $\eps=e^{-\delta t}$.
\[
\mu(\{s\in B\ |\ \norm{c(a_ts)-c(a_tS)}\geq \delta t\}) \leq C e^{-\alpha\delta t} \mu(B),
\]
qui est le terme général d'une série convergente.
Le lemme de Borel-Cantelli montre donc que pour tout $\delta>0$, pour presque tout $s$ dans $B$, pour tout $t\in\N$ suffisamment grand,
\[
\norm{c(a_ts)-c(a_tS)} < \delta t.
\]
Or il existe une constante $C$ telle que pour tous $t,t'\in\R$, $\norm{c(a_ts)-c(a_{t'}s)}+\norm{c(a_tS)-c(a_{t'}S)}\leq C\abs{t-t'}$, et donc, pour tout $t>0$ suffisamment grand,
\[
\norm{c(a_ts)-c(a_tS)} < \delta t + C.
\]
Comme $\delta>0$ est arbitrairement proche de $0$, cela montre bien la limite souhaitée.
\end{proof}

\begin{remark}
On peut en outre contrôler le drapeau partiel associé à $a_ts$.
Pour tout $\delta>0$, soit
\[
I_{a_tS}(e^{-\delta t})=\{i\in[1,r]\ |\ \alpha(c(a_tS))\geq -\delta t\}.
\]
Suivant la proposition~\ref{drapartiel2} on définit le drapeau partiel $P_{I_{a_tS}(e^{-\delta t})} \gamma_{a_tS}$ associé à $a_tS$.
Notons aussi $a_ts=k_{t,s}a_{t,s}n_{t,s}\gamma_{t,s}$ une décomposition de Siegel de $a_ts$.
Alors, pour presque tout $s\in S$, pour tout $t>0$ suffisamment grand,
\[
P_{I_{a_tS}(e^{-\delta t})}\gamma_{t,s} = P_{I_{a_t S}(e^{-\delta t})}\gamma_{a_tS}.
\]
%Il faut faire attention que le drapeau soit bien défini par rapport au même sous-groupe parabolique, ce qui n'est pas clair si on a des angles de taille justement $\eps t$; en fait, ce n'est pas très grave, il suffit de projeter sur le plus grand des deux paraboliques. C'est pourquoi on prend le même ensemble $I_{a_tS}(...)$ de part et d'autre.
\end{remark}

\section{Adhérence de Zariski et hérédité}
\label{sec:chs}

Soit $S$ une partie de $G$.
Dans chaque représentation fondamentale $\rho_i:G\to\GL(V_i)$ on considère dans $\End V_i$ le sous-espace vectoriel
\[
\cH_i(S) = \Vect \{\rho_i(s)\,;\,s\in S\}
\]
et on définit
\[
\cH(S) = \{g\in G\ |\ \forall i\in[1,r],\ \rho_i(g)\in\cH_i(S)\}.
\]

\comm{
\begin{remark}
D'un point de vue plus géométrique, l'ensemble $\cH(S)$ est construit de sorte que dans chaque plongement $P\bcs G\hookrightarrow \PP(V_i)$, les ensembles $S$ et $\cH(S)$ engendrent le même sous-espace projectif.
(Il y a en fait ici un petit problème à cause de l'inverse $s^{-1}$...)
\end{remark}
}

\note{Noter que $\cH(aS)=a\cH(S)$, car pour chaque $i$, $\cH_i(aS)=\rho_i(a)\cH_i(S)$, puisque la multiplication par $\rho_i(a)$ est linéaire sur $\End V_i$.}

\begin{proposition}
\label{heredite}
Il existe une boule $B_0$ dans $G$ telle que pour toute partie compacte $S\subset G$, il existe une constante $C$ telle que pour tout élément $a\in G$,
\[
\norm{c(aS) - c(a\cH(S)\cap B_0)} \leq C.
\]
\end{proposition}
\begin{proof}
Pour $i\in[1,r]$, considérons la partie $\rho_i(S)$ dans l'espace vectoriel $\End V_i$.
À une constante multiplicative près ne dépendant que de $S$, si $\varphi:\End V_i\to V_i$ est une application linéaire, alors
\[
\sup_{u\in\rho_i(S)} \norm{\varphi(u)} \asymp \sup_{u\in\cH_i(S)\cap B(0,1)}\norm{\varphi(u)}.
\]
Cela s'applique en particulier aux applications de la forme $u\mapsto auv$, pour $a\in G$ et $v\in V_i$, et l'on obtient ainsi, pour une boule $B_0\subset G$,
\note{Détailler le choix de $B_0$: il faut que pour chaque $i$, pour tout $S\subset G$, on ait l'égalité $\cH_i(S)=\cH_i(\cH(S)\cap B_0)$.}
\[
\sup_{g\in S} \norm{a g v} \asymp \sup_{g\in\cH(S)\cap B_0} \norm{agv}.
\]
\end{proof}

Étant donnée une mesure $\mu$ localement régulière en $s_0$, on pose
\[
\cH_\mu(s_0) =\bigcap_{\eps>0} \cH(B(s_0,\eps)\cap\Supp\mu)
\]
et
\[
\cH'_\mu(s_0) = B_0 \cap \bigcap_{\eps>0} \cH(B(s_0,\eps)\cap\Supp\mu),
\]
où $B_0$ est la boule dans $G$ donnée par la proposition ci-dessus.
L'ensemble $\cH_\mu(s_0)$ est un \emph{ensemble algébrique} dans $G$: il s'obtient comme l'ensemble des zéros d'une famille de polynômes.
En outre, lorsque le point $s$ est choisi aléatoirement suivant $\mu$ au voisinage de $s_0$, le comportement asymptotique de l'orbite $(a_ts\Gamma)$ est déterminé presque sûrement par $\cH_\mu(s_0)$.

\begin{corollary}
Soit $\mu$ une mesure sur $G$ localement régulière en $s_0$.
Pour $\mu$-presque tout $s$ au voisinage de $s_0$,
\[
\lim_{t\to\infty} \frac{1}{t}[c(a_ts)-c(a_t\cH'_\mu(s_0))] = 0.
\]
\end{corollary}
\begin{proof}
Cela découle immédiatement du théorème~\ref{diagan}, et de la proposition~\ref{heredite}.
\end{proof}

\section{Encadrement du taux de fuite}

Dans la suite, nous considérons un sous-groupe à un paramètre $(a_t)$ dans $A$, donné par
\[
a_t = e^{tY},\quad\mbox{avec}\ Y\in\ka^-,
\]
et nous notons $P$ le sous-groupe parabolique associé à $(a_t)$:
\[
P = \{g\in G\ |\ \lim_{t\to\infty} a_tga_{-t}\ \mbox{existe}\}.
\]
Nous voulons étudier le comportement asymptotique de la fonction $c$ le long des orbites de $(a_t)$ dans $\Omega$.
Notre premier résultat concerne un point choisi aléatoirement suivant une mesure localement régulière $\mu$ qui satisfait une condition géométrique naturelle.
Rappelons qu'une sous-variété de Bruhat dans $G$ est une sous-variété de la forme 
\[
X_wg = \overline{PwB}g, \quad\mbox{avec}\quad w\in W_P\ \mbox{et}\ g\in G.
\]
Étant donnée une famille $(A_i)_{i\in I}$ d'éléments de $\ka$, son plus grand minorant $A=\inf_{i\in I} A_i$ pour la relation d'ordre $\prec$ est défini par
\[
\forall k\in[1,r],\quad \omega_k(A) = \inf_{i\in I} \omega_k(A_i).
\]
Pour le sous-groupe diagonal $(a_t)=(e^{tY})$, avec $Y\in\ka$, on définit aussi le taux de contraction $\tau_{\R}(M,a_t)\in\ka^-$ de $M$ par $(a_t)$ par la formule
\[
\tau_{\R}(M,a_t) = \inf\{ p_{\ka^-}(Y^w)\ ;\ w\in W_P\ \mbox{tel que}\ \exists g\in G:\, M\subset X_wg \}.
\]
Nous aurons aussi besoin de la quantité analogue
\[
\tau_{\Q}(M,a_t) = \inf\{ p_{\ka^-}(Y^w)\ ;\ w\in W_P\ \mbox{tel que}\ \exists g\in G(\Q):\, M\subset X_wg \}.
\]
Remarquons que les ensembles ci-dessus ne sont jamais vides, car on a toujours $M\subset X_{w_0}g=G$ si $w_0$ est l'élément de longueur maximale dans le groupe de Weyl.

\begin{example}
Suivant Pengyu Yang, nous dirons qu'une variété de Bruhat $X_wg$ est \emph{instable} pour le flot $a_t=e^{tY}$ s'il existe un poids fondamental $\omega$ tel que $\omega(Y^w)<0$.
Cela revient à dire que $\tau_{\R}(X_w,a_t)\neq 0$.
En effet, s'il existe $i$ tel que $\omega_i(Y^w)<0$, alors $\omega_i(p_{\ka^-}(Y^w))\leq \omega_i(Y^w) < 0$, donc $p_{\ka^-}(Y^w)\neq 0$;
et réciproquement, si $\tau_{\R}(X_w,a_t)\neq 0$, alors il existe $i$ tel que $\omega_i(Y^w)<0$.
\end{example}

Rappelons qu'un ensemble algébrique dans $G$ est dit \emph{irréductible} s'il ne peut pas s'écrire comme réunion non triviale de deux sous-ensembles algébriques.

\begin{theorem}[Encadrement du taux de fuite]
\label{diagstab}
%Soit $G$, $(a_t)$, ...etc.
Soit $\mu$ une probabilité sur $G$ régulière en $s_0$, telle que $\cH_\mu(s_0)$ est irréductible.
Alors, pour presque tout $s$ au voisinage de $s_0$,
\[
\tau_{\R}(\cH'_\mu(s_0),a_t) \prec \liminf_{t\to\infty} \frac{1}{t} c(a_t s)
\prec \limsup_{t\to\infty} \frac{1}{t}c(a_ts) \prec \tau_{\Q}(\cH_\mu(s_0),a_t).
\]
\end{theorem}

\begin{remark}
En général, si $\cH_\mu(s_0)$ n'est pas irréductible, on peut l'écrire comme une réunion finie de composantes irréductibles: $\cH_\mu(s_0)=\cup_i F_i$.
On a alors $\mu=\sum_i \mu_i$, où $\mu_i=\mu|_{F_i}$.
Chacune des mesures $\mu_i$ est localement régulière en $s_0$, et $\cH_{\mu_i}(s_0)=F_i$ est irréductible.
On peut donc se ramèner facilement au cadre du théorème.
\end{remark}

Nous verrons au chapitre suivant que la mesure de Lebesgue sur une sous-variété analytique est localement régulière.
Admettant ce point pour l'instant, on retrouve comme cas particulier du théorème ci-dessus le résultat suivant, essentiellement dû à Pengyu Yang \cite{yang}.

\begin{corollary}[Variétés analytiques stables]
\label{semstab}
%Soit $G$, $(a_t)$, ...etc.
Soit $M$ une sous-variété analytique connexe de $G$ qui n'est incluse dans aucune sous-variété de Bruhat instable pour $(a_t)$.
Alors, pour presque tout $s\in M$, pour toute représentation rationnelle $V$,
\[
\lim_{t\to\infty} \frac{1}{t} \log \lambda_1(a_t s V(\Z)) = 0.
\]
\end{corollary}
\begin{proof}
Si $M$ n'est incluse dans aucune sous-variété de Bruhat instable pour $(a_t)$, alors $\tau_{\R}(M,a_t)=0$.
Le théorème~\ref{diagstab} montre donc que presque sûrement $\lim_{t\to\infty}\frac{1}{t}c(a_ts)=0$.
Ensuite, si $V$ est une représentation rationnelle, et $\chi$ le plus haut poids apparaissant dans $V$, le premier minimum de $a_tsV(\Z)$ est donné par
\[
\lambda_1(a_tsV(\Z)) \asymp e^{\chi(c(a_ts))} = e^{o(t)}.
\]
\end{proof}

Pour la démonstration du théorème~\ref{diagstab}, nous aurons besoin de la proposition suivante, très voisine d'un résultat plus général de Pengyu Yang \cite[Theorem~1.2]{yang}.
Toutefois, la démonstration dans notre cas particulier est sensiblement plus simple, car on ne s'intéresse qu'aux vecteurs dans l'orbite d'un vecteur de plus haut poids; en particulier nous n'aurons pas besoin des résultats de théorie géométrique des invariants dûs à Mumford \cite{mumford} ou Kempf \cite{kempf}.

\begin{proposition}[Stabilité linéaire]
\label{stablin}
Soit $S$ un ensemble algébrique irréductible dans $G$, et $\tau=\tau_{\R}(S,a_t)$.
Il existe $c>0$ tel que pour tout $i\in[1,r]$ et tout vecteur $v\in\tx_i$,
\[
\forall t>0,\quad \sup_{s\in S} \norm{a_t s v} \geq c e^{t\omega_i(\tau)}\norm{v}.
\]
\end{proposition}

La démonstration de cette proposition est basée sur l'observation suivante.

\begin{lemma}
\label{union}
Soit $V$ une représentation de $G$ engendrée par une unique droite rationnelle $\Q e_\chi$ de plus haut poids $\chi$. % et $(a_t)$ un sous-groupe à un paramètre dans le $\Q$-tore déployé maximal de $G$.
Pour $\lambda\in\R$, notons
\[
G(e_\chi,V^{\lambda}(a_t)) = \{g\in G\ |\ \lim_{t\to\infty} \frac{1}{t}\log\norm{a_tge_\chi} \leq \lambda\}.
\]
On peut écrire %Si $P$ est le sous-groupe parabolique associé à $(a_t)$ et $B$ un $\Q$-parabolique minimal de $G$, on peut écrire
\[
G(e_\chi,V^\lambda(a_t)) = \bigsqcup_{w\in W^\lambda(\chi,a_t)} PwB
\]
comme une réunion de cellules de Bruhat de $G$, où
\[
W^\lambda(\chi,a_t) = \{w\in W_P\ |\ \bracket{\chi,Y^w} \leq \lambda\}.
\]
\end{lemma}
\begin{proof}
Soit $g$ un élément quelconque de $G$, et $g=pwb$ sa décomposition de Bruhat, avec $p\in P$, $b\in B$ et $w\in W_P$.
Le sous-groupe parabolique minimal $B$ préserve la direction $e_\chi$, et lorsque $t$ tend vers l'infini, l'élément $a_tpa_t^{-1}$ converge dans $G$.
Donc $a_tge_\chi$ et $a_twe_\chi$ ont le même comportement asymptotique.
Comme
\[
a_t w e_\chi = w w^{-1}a_tw e_\chi = w e^{tY^w} e_\chi = e^{t\bracket{\chi,Y^w}}we_\chi,
\]
on trouve bien que $g=pwb$ est dans $G(e_\chi,V^\lambda(a_t))$ si et seulement si $w\in W^\lambda(\chi,a_t)$.
\end{proof}

La proposition~\ref{stablin} découle du lemme ci-dessus, par un argument élémentaire de compacité.

\begin{proof}[Démonstration de la proposition~\ref{stablin}]
Soit $i\in[1,r]$.
Notons $\pi^+:V_i\to V_i$ la projection sur la somme des espaces propres de $a_t$ associés à des valeurs propres supérieures ou égales à $\omega_i(\tau)$.
On a l'équivalence
\[
\pi^+(sv)=0 
\quad\Leftrightarrow\quad
\lim_{t\to\infty}\frac{1}{t}\log\norm{a_tsv} < \omega_i(\tau).
\]
Soit $v=g^{-1}e_i\in\tx_i$.
Avec l'équivalence ci-dessus, le lemme~\ref{union} montre que si $\pi^+(sv)=0$, alors $s$ appartient à une cellule de Bruhat $PwBg$ avec $w\in W^\lambda(\omega_i,a_t)$, avec $\lambda<\omega_i(\tau)$, i.e. $\omega_i(Y^w)\leq \lambda <\omega_i(\tau)$.
Mais par définition de $\tau=\tau_{\R}(S,a)$, 
\[
\omega_i(\tau)\leq\inf_{PwBg\supset S}\omega_i(Y^w),
\]
donc $S$ n'est inclus dans aucune cellule de Bruhat $PwBg$, avec $g\in G$ et $w\in W^\lambda(\omega_i,a_t)$.
Par irréductibilité, $S$ n'est pas inclus dans la réunion (finie) de ces cellules, et il existe donc $s\in S$ tel que $\pi^+(sv)\neq 0$.
Comme $v\mapsto\sup_{s\in S}\norm{\pi^+(sv)}$ est semi-continue inférieurement
\note{s.c.i. : voir Brézis \cite[page 8]{brezis}}
sur le compact $\tx_i'=\{v\in\tx_i\ |\ \norm{v}=1\}$, il existe $c>0$ tel que
\[
\forall v\in\tx_i,\quad \sup_{s\in S} \norm{\pi^+(sv)} \geq c\norm{v}.
\]
Cela implique, pour tout $v\in\tx_i$,
\[
\sup_{s\in S} \norm{a_tsv}
\geq \sup_{s\in S} e^{t\omega_i(\tau)}\norm{\pi^+(sv)}
\geq ce^{t\omega_i(\tau)}\norm{v}.
\]
\end{proof}

Nous pouvons enfin démontrer le théorème~\ref{diagstab}.

\begin{proof}[Démonstration du théorème~\ref{diagstab}]
Notons
\[
\tau =\tau_{\R}(\cH_\mu(s_0),a_t).
\]
La proposition~\ref{stablin} montre que pour chaque $i\in[1,r]$, il existe $c>0$ tel que
\[
\forall t>0,\
\forall v\in V_i(\Z)\cap\tx_i,\quad
\sup_{s\in\cH_\mu(s_0)} \norm{a_tsv} \geq c e^{t\omega_i(\tau)}\norm{v} \geq ce^{t\omega_i(\tau)}.
\]
Cela implique naturellement les inégalités
\[
\liminf_{t\to\infty}\frac{1}{t}\log\left(\min_{v\in V_i(\Z)\cap\tx_i}\sup_{s\in\cH_\mu(s_0)}\norm{a_ts v}\right) \geq \omega_i(\tau), \qquad{i=1,\dots,r}
\]
ce qui se réécrit
\[
\liminf_{t\to\infty} \frac{1}{t}c(a_t\cH_\mu(s_0)) \succ \tau_{\R}(\cH_\mu(s_0),a_t).
\]
Soit $B$ la boule centrée en $s_0$ donnée par le théorème~\ref{diagan}, et $S=B\cap\Supp\mu$.
D'après la proposition~\ref{heredite}, $c(a_tS)-c(a_t\cH_\mu(s_0))=O(1)$, et donc
\[
\liminf_{t\to\infty} \frac{1}{t}c(a_tS) \succ \tau_{\R}(\cH_\mu(s_0),a_t).
\]
Le théorème~\ref{diagan} permet d'en déduire que pour presque tout $s$ au voisinage de $s_0$,
\[
\liminf_{t\to\infty} \frac{1}{t}c(a_ts) \succ \tau_{\R}(\cH_\mu(s_0),a_t).
\]

Montrons maintenant l'inégalité concernant la limite supérieure.
\note{et donc $S\subset PwB\gamma$.}
Supposons $\cH_\mu(s_0)\subset PwB\gamma$, avec $\gamma\in G(\Q)$.
Alors, pour tout $i\in[1,r]$ et tout $t>0$, 
\[
\sup_{s\in S}\norm{a_t s \gamma^{-1}e_i}\ll e^{t\omega_i(Y^w)}
\]
et par conséquent
\[
\limsup_{t\to\infty}\frac{1}{t} c(a_ts) \prec p_{\ka^-}(Y^w).
\]
Comme ceci vaut pour tout $w$ tel qu'il existe $\gamma\in G(\Q)$ vérifiant $PwB\gamma\supset S$, on trouve bien
\[
\limsup_{t\to\infty}\frac{1}{t} c(a_ts) \prec \tau_{\Q}(S,a).
\]
\end{proof}

%
%\begin{remark}
%Le théorème~\ref{diagan} est valable plus généralement pour une mesure $\mu$ localement régulière en $s_0$.
%Si $S$ est l'adhérence de Zariski locale du support de $\mu$ au voisinage de $s_0$, définie par
%\[
%S = \bigcap_{r>0} \overline{S\cap B(s_0,r)}^{Z},
%\]
%l'argument utilisé pour la démonstration du théorème~\ref{diagan} montre qu'au voisinage de $s_0$, pour $\mu$-presque tout point $s$,
%\[
%\liminf_{t\to\infty} c(a_ts) \succ \tau_{\R}(S,a_t).
%\]
%\end{remark}

\section{Variétés algébriques définies sur $\QQ$}

Dans le cas où l'ensemble algébrique $\cH_\mu(s_0)$ est défini sur $\QQ$, on peut améliorer le théorème~\ref{diagstab} et déterminer la limite $\lim_{t\to\infty}\frac{1}{t}c(a_ts)$, pour $\mu$-presque tout $s$ au voisinage de $s_0$.
C'est ce que décrit le théorème ci-dessous.

\begin{theorem}[Orbites diagonales et sous-variétés algébriques]
\label{diaganalg}
Soit $\mu$ une mesure localement régulière en $s_0\in G$ telle que $M=\cH_\mu(s_0)$ soit irréductible et définie sur $\QQ$.
Soit $(a_t)_{t>0}$ un sous-groupe diagonal à un paramètre dans $G$ et $c_M=\tau_{\Q}(M,a)$.
Pour $s\in M$ et $t>0$, on note
\[
a_t s = k_{t,s} b_{t,s} n_{t,s} \gamma_{t,s},
\]
une décomposition de Siegel de $a_ts$.
\begin{enumerate}
\item Pour presque tout $s$ au voisinage de $s_0$, $\lim_{t\to\infty}\frac{1}{t}\log b_{t,s} = c_M$.
\note{$Q_M$ contient toutes les racines de $\theta_M$.}
\item Il existe $\gamma_M\in\Gamma$ tel que pour $\theta_M=\{i\in[1,r]\ |\ \alpha_i(c_M)=0\}$ et $Q_M=P_{\theta_M}$ le sous-groupe parabolique associé à $\theta_M$, alors pour presque tout $s$ au voisinage de $s_0$, pour tout $t>0$ suffisamment grand, $Q_M\gamma_{t,s}=Q_M\gamma_M$.
\note{En général, une intersection de variétés de Bruhat n'est pas une variété de Bruhat. Mais l'intersection de variétés de Bruhat \emph{standard} est une variété de Bruhat standard.}
\item Si $X_w\supset M\gamma_M^{-1}$ est la plus petite variété de Bruhat standard contenant $M\gamma_M^{-1}$, alors $c_M=p_{\ka^-}(Y^w)$.
%où $Y$ est tel que $a_t=e^Y$, et $p_{\ka^-}:\ka\to\ka^-$ désigne la projection sur le convexe fermé $\ka^-$. Soit $P=\{g\in G\ |\ \lim_{t\to\infty}a_tga_{-t}\ \mathrm{existe}\}$ et $B$ le parabolique minimal. 
\end{enumerate}
\end{theorem}

Pour la démonstration, nous appliquerons le théorème~\ref{diagalg} à un point $s_1\in M\cap G(\QQ)$ bien choisi, dont l'existence sera assurée par le lemme suivant.
Nous dirons que des ensembles algébriques $F_i$, $i\in I$ sont de degré borné s'il existe une constante $D\geq 0$ telle pour pour chaque $i$, $F_i$ est l'ensemble des zéros d'une famille de polynômes de degré au plus $D$.

\begin{lemma}\label{deg}
Soit $k$ un corps de nombres, et $V$ une variété algébrique affine irréductible définie sur $k$.
On suppose que $(F_i)_{i\in I}$ est une famille de sous-ensembles algébriques stricts définis sur $k$ et de degré borné.
Alors, il existe un point $s_1\in V(\QQ)\setminus\bigcup_{i\in I}F_i$.
\end{lemma}
\begin{proof}
Si $V=\mathbb{A}^d$ est l'espace affine tout entier, le résultat est clair: si le point $s_1=(x_1,\dots,x_d)$ à coordonnées dans $\QQ$ est choisi de sorte que pour chaque $i$, $[k(x_1,\dots,x_{i+1}):k(x_1,\dots,x_i)]>D$ alors $s_1$ ne satisfait aucune relation de degré au plus $D$ à coefficients dans $k$. 

D'après le lemme de normalisation de Noether \cite[Theorem~10, page 66]{shafarevich}, il existe toujours un morphisme fini de variétés algébriques $V\to\mathbb{A}^{\dim V}$ défini sur $k$, et le cas général découle donc du cas particulier ci-dessus.
\end{proof}

\begin{proof}[Démonstration du théorème~\ref{diaganalg}]
D'après le théorème~\ref{diagstab}, on sait déjà que pour presque tout $s\in M$,
\[
\limsup_{t\to\infty}\frac{1}{t}\log b_{t,s}\prec c_M.
\]
Comme les sous-variétés de Bruhat rationnelles sont définies sur $\Q$ et de degré borné, le lemme ci-dessus montre qu'il existe un point $s_1\in M\cap G(\QQ)$ qui n'est inclus dans aucune sous-variété de Bruhat rationnelle qui ne contient pas $M$.
D'après le théorème~\ref{diagalg},
\[
\lim_{t\to\infty}\frac{1}{t}\log b_{t,s_1}=c_M.
\]
Cela implique nécessairement
\[
\frac{1}{t}c(a_tM) \succ \frac{1}{t}c(a_ts_1) = c_M + o(1),
\]
et avec le théorème~\ref{diagan}, pour presque tout $s$ au voisinage de $s_0$,
\[
\liminf_{t\to\infty}\frac{1}{t}c(a_ts) = \liminf_{t\to\infty}\frac{1}{t}c(a_tM) \succ c_M.
\]
Cela montre le premier point du théorème.

Ensuite, on applique la proposition~\ref{drapartiel2} pour $t>0$, avec $\eps=e^{-\delta t}$, où $\delta>0$ est choisi tel que
\[
\forall i\not\in\theta_M,\qquad \delta<\alpha_i(c_M).
\]
Cela montre qu'il existe un élément $\gamma_t\in G(\Q)$ tel que dans une petite boule $B$ centrée en $s_0$,
\[
\mu(\{s\in B\ |\ Q_M\gamma_{s,t}\neq Q_M\gamma_t\}) \leq C'e^{-\delta\alpha't}\mu(B).
\]
Par le lemme de Borel-Cantelli, il s'ensuit que pour presque tout $s$ au voisinage de $s_0$, pour tout $t>0$ suffisamment grand,
\[
Q_M\gamma_{s,t}= Q_M\gamma_t.
\]
Mais d'après la proposition~\ref{drapartiel}, à $s$ fixé, l'application $t\mapsto Q_M\gamma_{s,t}$ est localement constante, et donc constante au voisinage de l'infini.
Ainsi, il existe un élément $\gamma_M\in G(\Q)$ tel que pour tout $t>0$ suffisamment grand, 
\[
Q_M\gamma_{s,t}= Q_M\gamma_M.
\]
Enfin, si $s_1\in G(\QQ)$ est le point déjà utilisé ci-dessus et $X_w\supset M\gamma_M^{-1}$, alors $s_1\in PwB$, et d'après le théorème~\ref{diagalg},
\[
c_M = \lim_{t\to\infty} \frac{1}{t} c(a_ts_1) = p_{\ka^-}(Y^w).
\]
\end{proof}

\begin{corollary}[Mesures algébriques semi-stables]
\label{semstabalg}
Soit $\mu$ une mesure localement régulière en $s_0\in G$ telle que $M=\cH_\mu(s_0)$ soit irréductible et définie sur $\QQ$.
%Soit $(a_t)_{t>0}$ un sous-groupe diagonal à un paramètre dans $G$.
Si $M$ n'est incluse dans aucune sous-variété de Bruhat \emph{rationnelle} instable, alors pour toute représentation rationnelle $V$, pour presque tout $s$ au voisinage de $s_0$,
\[
\lim_{t\to\infty}\frac{1}{t}\log\lambda_1(a_tsV(\Z)) = 0.
\]
\end{corollary}
\begin{proof}
Avec les notations du théorème~\ref{diaganalg}, soit
\[
X_w \supset M\gamma_M^{-1}
\]
la plus petite sous-variété de Bruhat standard contenant $M\gamma_M^{-1}$.
Comme $M$ n'est incluse dans aucune sous-variété de Bruhat instable, on doit avoir, pour chaque $i\in[1,r]$, $\omega_i(Y^w)\geq 0$.
Par suite, $p_{\ka^-}(Y^w)=0$, et pour presque tout $s$ au voisinage de $s_0$,
\[
\lim_{t\to\infty}\frac{1}{t}\log c(a_ts) = 0.
\]
Cela montre ce qu'on veut, car la décomposition de Siegel de $a_ts$ montre que dans une représentation rationnelle $V$ de plus haut poids $\chi$,
\[
\lambda_1(a_tsV(\Z)) \asymp e^{t \chi(c(a_ts))} = e^{o(t)}.
\]
\end{proof}

\chapter{Approximation dans les sous-variétés}
\label{chap:sousvariete}

Dans ce chapitre, nous nous donnons une sous-variété analytique $M$ dans la variété de drapeaux $X$, et étudions les propriétés diophantiennes d'un point $x$ choisi aléatoirement sur $M$.
Plus précisément, nous cherchons d'abord à déterminer sous quelles conditions la conclusion du théorème~\ref{exposantps} reste valable, et montrons un critère analogue à celui du théorème~\ref{extqq} obtenu pour les points algébriques.
Ensuite, nous étudierons le cas où la variété $M$ est algébrique et définie sur $\QQ$, où nous pouvons donner une formule pour l'exposant presque sûr d'un point choisi aléatoirement dans $M$.

\bigskip

\emph{Dans toute la suite, $X=P\bcs G$ désigne une variété de drapeaux obtenue comme quotient d'un $\Q$-groupe semi-simple $G$ par un sous-groupe parabolique $P$ défini sur $\Q$.
On munit $X$ de la métrique de Carnot-Carathéodory introduite au paragraphe~\ref{sec:cc}, et d'une hauteur $H_\chi$ provenant d'une représentation irréductible de $G$ engendrée par une unique droite de plus haut poids $\chi$.
Enfin, on suppose que $G$ est de rang rationnel $r$, et on note $V_i$, $i=1,\dots,r$ ses représentations fondamentales.}

\section{Variétés analytiques réelles}

Si $M$ est une sous-variété analytique de $\SL_d(\R)$ de dimension $m$, on note $\lambda_M$ la mesure de Lebesgue sur $M$, i.e. la mesure de Hausdorff de dimension $m$ restreinte à $M$.
Comme nous nous intéresserons seulement à des événements de mesure pleine ou nulle, seule la classe de $\lambda_M$ aura une importance pour nous, et l'on aurait aussi bien pu définir $\lambda_M$ localement comme l'image de la mesure de Lebesgue sur $\R^{\dim M}$ par un paramétrage analytique local de $M$.

La mesure $\lambda_M$ est localement régulière.
C'est ce résultat important, dû à Kleinbock et Margulis \cite{kleinbockmargulis}, qui fait l'intérêt principal du théorème~\ref{diagan}.
Nous rappelons donc ici les grandes lignes de sa démonstration.
L'argument est basé sur la proposition suivante \cite[Proposition~2.1]{kleinbock_dichotomy}.

\begin{proposition}
\label{analreg}
Soit $U$ un ouvert connexe de $\R^n$, et $\cF$ un sous-espace de dimension finie de fonctions analytiques sur $U$ à valeurs réelles.
Pour tout $x$ dans $U$, il existe des constantes $C,\alpha>0$ et un voisinage $W\ni x$ tels que toute fonction $f\in\cF$ soit $(C,\alpha)$-régulière.
\end{proposition}
\begin{proof}
Nous admettons cette proposition pour l'instant, en attendant d'inclure la démonstration donnée par Kleinbock et Margulis \cite[Proposition~3.4]{kleinbockmargulis}.
\end{proof}

\begin{corollary}
\label{regsec}
Soit $M$ une sous-variété analytique de $X$, $\lambda_M$ une mesure de Lebesgue sur $M$, $x_0\in M$ et $s:X\to G$ une section analytique locale au voisinage de $x_0$.
Il existe des constantes $C,\alpha>0$ et une boule ouverte $B=B(x_0,r)$ telles que pour tout $g\in G$, pour tout $i\in[1,r]$ et tout $v\in V_i$, l'application
\[
x \mapsto \norm{g s(x) v}
\]
est $(C,\alpha)$-régulière sur $B$ pour la mesure $\lambda_M$.
\end{corollary}
\begin{proof}
Soit $U$ un ouvert de $\R^m$ et $\varphi:U\to M$ un paramétrage local de $M$ au voisinage de $x_0$ tel que $\lambda_M=\varphi_*\lambda$ soit l'image par $\varphi$ de la mesure de Lebesgue $\lambda$ sur $U$.
Supposons en outre que $\varphi(0)=x_0$.
Soit $\cF_0$ l'espace vectoriel engendré par les applications coefficients:
\[
\cF_0 = \Vect \{ u\mapsto\bracket{v,s(\varphi(u))w};\ v,w\in V_i,\ i\in[1,r] \},
\]
et
\[
\cF = \Vect\{ f_1f_2;\ f_1,f_2\in\cF_0\}.
\]
Ces espaces de fonctions analytiques sur $U$ sont de dimension finie, donc d'après la proposition~\ref{analreg}, il existe un voisinage $W$ de $0$ dans $\R^m$ et $C,\alpha>0$ tels que toute fonction $f\in\cF$ soit $(C,\alpha)$-régulière sur $W$ pour la mesure de Lebesgue.
Les applications de la forme $u\mapsto \norm{g s(\varphi(u)) v}^2$ sont des éléments de $\cF$, donc satisfont l'égalité souhaitée.
\end{proof}

Une première application de ces propriétés de régularité des variétés analytiques et des résultats du chapitre précédent est que l'exposant diophantien est constant presque sûrement sur une sous-variété analytique.

\begin{theorem}[Exposant diophantien d'une sous-variété analytique]
Soit $M$ une sous-variété analytique connexe de $X$.
Il existe une constante $\beta_\chi(M)$ telle que pour presque tout $x\in M$,
\(
\beta_\chi(x)=\beta_\chi(M).
\)
\end{theorem}
\begin{proof}
Soit $x_0$ un point de $M$.
%\[
%\begin{array}{lll}
%X & \mapsto & \GL(V_\chi)\\
%x & \mapsto & \rho_\chi(s(x))
%\end{array}
%\]
%est localement régulière au voisinage de $\rho_\chi(s(x_0))$.
Pour $x\in M$ et $t>0$, notons
\[
c_t^x=\left(\log\mu_1(a_ts_x^{-1}V_\chi(\Z)),\dots,\log\mu_{\dim V_\chi}(a_ts_x^{-1}V_\chi(\Z))\right)
\]
la suite des logarithmes des covolumes successifs du réseau $a_ts_x^{-1}V_\chi(\Z)$.
Notons aussi $\cC$ l'ensemble des fonctions convexes sur le segment d'entiers $[0,\dim V_\chi]$ telles que $f(0)=f(V_\chi)=0$.
Le théorème~\ref{nd} appliqué dans le groupe $\GL(V_\chi)$ dans un voisinage adéquat $U$ de $x_0$ -- dont l'existence est assurée par le corollaire~\ref{regsec} -- montre qu'il existe une application
\[
\begin{array}{lll}
\R^+ & \to & \cC\\
t & \mapsto & c_t^M
\end{array}
\]
telle que pour tout $\eps>0$ et tout $t\in\R^+$,
\[
\lambda_M(\{x\in U\ |\ \norm{c_t^x-c_t^M}\geq t\eps\}) \leq e^{-t\alpha\eps}\lambda_M(U).
\]
En outre, d'après la remarque qui suit le théorème~\ref{nd}, si 
\[
I_t^M(\eps)=\{i_1(M,\eps,t),\dots,i_k(M,\eps,t)\} \subset[1,\dim V_\chi]
\]
désigne l'ensemble des points où la dérivée de $i\mapsto c_t^M(i)$ fait un saut de taille supérieure à $t\eps$, il existe un drapeau partiel 
\[
\{0\} < V_{t,\eps}^{i_1} < \dots < V_{t,\eps}^{i_k} < V_\chi(\Z)
\]
tel qu'avec probabilité supérieure à $1-Ce^{-t\alpha\eps}$, pour chaque $\ell=1,\dots,k$, les $i_\ell$ premiers minima successifs de $a_ts_x^{-1}V_\chi(\Z)$ sont atteints dans $V_{t,\eps}^{i_\ell}$.

Notons $E_\chi$ la somme des espaces de poids de $a_t$ différents de $e_\chi$.
Pour chaque $i\in I_t^M(\eps)$, il existe $f_i=f_i(t,\eps)\geq 0$ tel que
\[
\lambda_M(\{x\in U\ |\ e^{-t(f_i+\eps)} \leq d(a_ts_x^{-1}V_{t,\eps}^i,E_\chi)\leq e^{-t(f_i-\eps)}\})\geq (1- Ce^{\alpha\eps t})\lambda_M(U).
\]
\details{
En effet, si $\bv$ est un représentant de $V_{t,\eps}^i$ dans $\wedge^iV_\chi$ et $(\bu_j)_{j\in J}$ une base orthonormée de $\wedge^{i-1}E_\chi$, la distance $d(a_ts_x^{-1}V_{t,\eps}^i,E_\chi)$ est comparable au maximum $\max_{j\in J} \frac{\norm{\bu_j\wedge a_ts_x^{-1}\bv}}{\norm{a_ts_x^{-1}\bv}}$.
Or, par régularité des fonctions $x\mapsto\norm{\bu_j \wedge a_ts_x^{-1}\bv}$ et $x\mapsto\norm{a_ts_x^{-1}\bv}$, avec probabilité $1-e^{\alpha\eps t}$,
\[
\norm{\bu_j\wedge a_ts_x^{-1}\bv} \geq e^{-\eps t} \sup_{y\in U} \norm{\bu_j\wedge a_ts_y^{-1}\bv}
\]
et 
\[
\norm{a_ts_x^{-1}\bv} \geq e^{-\eps t}\sup_{y\in U} \norm{a_ts_y^{-1}\bv}.
\]
Par conséquent, si $f_i$ est choisi tel que $e^{-f_i t} = \max_j \frac{\sup_{y\in U} \norm{\bu_j\wedge a_ts_y^{-1}\bv}}{\sup_{y\in U} \norm{a_ts_y^{-1}\bv}}$, on a, avec probabilité supérieure à $1-Ce^{\alpha\eps t}$,
\[
d(a_ts_x^{-1}V_{t,\eps}^i,E_\chi) \leq \max_j \frac{\sup_{y\in U} \norm{\bu_j\wedge a_ts_y^{-1}\bv}}{\norm{a_ts_x^{-1}\bv}} \leq e^{-t(f_i-\eps)}
\]
et 
\[
d(a_ts_x^{-1}V_{t,\eps}^i,E_\chi) \geq \max_j \frac{\norm{\bu_j\wedge a_ts_x^{-1}\bv}}{\sup_y\norm{a_ts_y^{-1}\bv}} \geq e^{-t(f_i+\eps)}.
\]
}
Avec le lemme de Borel-Cantelli, ces inégalités montrent que pour presque tout $x$ au voisinage de $x_0$, pour tout $t>0$ suffisamment grand,
\[
\left\{\begin{array}{l}
\norm{c_t^x-c_t^M} \leq t\eps\\
\forall \ell=1,\dots,k,\quad c_t^x(i_\ell) = \log\norm{a_ts_x^{-1}V_{t,\eps}^{i_\ell}}\\
d(a_ts_x^{-1}V_{t,\eps}^{i_\ell},E_\chi) \in [e^{-t(f_i+\eps)},e^{-t(f_i-\eps)}]
\end{array}
\right.
\]

Notons que $f_{i_1}\geq f_{i_2} \geq\dots\geq f_{i_k}$.
Soit $j_{t,\eps}\in I_t^M(\eps)$ le plus petit indice tel que
\(
f_{j_{t,\eps}}\leq 2\eps.
\)
Montrons que pour presque tout $x$ au voisinage de $x_0$,
\[
\gamma_\chi(x) = \limsup_{t\to\infty} \frac{-1}{t}(c_t^M(j_{t,\eps})-c_t^M(j_{t,\eps}-1)) + O(\eps).
\]
Tout d'abord, l'inégalité $f_{j_{t,\eps}}\leq2\eps$ implique $d(a_ts_x^{-1}V_{t,\eps}^{j_{t,\eps}},E_\chi)\geq e^{-3t\eps}$, et comme $a_ts_x^{-1}V_{t,\eps}^{j_{t,\eps}}(\Z)$ admet une bonne%
\footnote{Comme cette base réalise les minima, elle est presque orthogonale.}
base constituée de vecteurs de norme au plus $e^{c_t^x(j_{t,\eps})-c_t^x(j_{t,\eps}-1)}=e^{c_t^M(j_{t,\eps})-c_t^M(j_{t,\eps}-1)+tO(\eps)}$, il existe un tel vecteur $v\in a_ts_x^{-1}V_{t,\eps}^{j_{t,\eps}}(\Z)$ tel que $\norm{\pi^+(v)}\geq e^{-tO(\eps)}\norm{v}$.
Cela implique
\[
r_\chi(a_{t(1-O(\eps))}s_x^{-1}V_\chi(\Z)) \leq e^{c_t^M(j_{t,\eps})-c_t^M(j_{t,\eps}-1)+O(\eps)t},
\]
puis
\[
\gamma_\chi(x) \geq \limsup_{t\to\infty} \frac{-1}{t}(c_t^M(j_{t,\eps})-c_t^M(j_{t,\eps}-1)) - O(\eps).
\]
Réciproquement, soit $V_{t,\eps}^{j_{t,\eps}'}$ le sous-espace qui précède $V_{t,\eps}^{j_{t,\eps}}$ dans le drapeau $V_{t,\eps}^{i_1}<\dots<V_{t,\eps}^{i_k}$.
Comme pour presque tout $x$, pour tout $t$ assez grand
\[
d(a_ts_x^{-1}V_{t,\eps}^{j_{t,\eps}'},E_\chi) \leq e^{-(f_{j-1}-\eps) t} 
\leq e^{-\eps t}
\]
aucun vecteur $v\in a_ts_x^{-1}V_{t,\eps}^{j_{t,\eps}'}$ ne saurait satisfaire $\norm{\pi^+(v)}\geq\frac{\norm{v}}{2}$.
Or tout vecteur $v\in a_ts_x^{-1}V_\chi(\Z)$ hors de $a_ts_x^{-1}V_{t,\eps}^{j_{t,\eps}'}$ vérifie
\[
\norm{v} \gg e^{c_t^M(j_{t,\eps}'+1)-c_t^M(j_{t,\eps}')}
\geq e^{c_t^M(j_{t,\eps})-c_t^M(j_{t,\eps}-1) + (j_{t,\eps}-j_{t,\eps}')t\eps},
\]
où la deuxième inégalité provient de ce que les sauts de la dérivée de $i\mapsto c_t^M(i)$ sont majorés par $t\eps$ sur tout l'intervalle $]j_{t,\eps}',j_{t,\eps}[$.
Par suite
\[
r_\chi(a_ts_x^{-1}V_\chi(\Z)) \geq e^{c_t^M(j_{t,\eps})-c_t^M(j_{t,\eps}-1)+tO(\eps)}
\]
puis
\[
\gamma_\chi(x) \leq \limsup_{t\to\infty}\frac{-1}{t}(c_t^M(j_{t,\eps})-c_t^M(j_{t,\eps}-1)) + O(\eps).
\]
Ainsi, il existe un voisinage $U$ de $x_0$ tel que pour tout $\eps>0$, il existe $\gamma_\eps$ tel que pour presque tout $x$ dans $U$, $\gamma_\chi(x)\in[\gamma_\eps-\eps,\gamma_\eps+\eps]$.
En faisant tendre $\eps$ vers $0$, cela montre que $\gamma_\chi(x)$ est constant presque partout au voisinage de tout point $x_0\in M$, et comme $M$ est connexe, $\gamma_\chi(x)$ est constant presque sûrement sur $M$.
Avec la proposition~\ref{dani}, on en déduit que $\beta_\chi(x)$ est constant presque sûrement sur $M$.
\end{proof}

\section{Un critère d'extrémalité}
\label{sec:extan}

La distance de Carnot-Carathéodory et la hauteur sur $X$ nous ont permis de définir au paragraphe~\ref{sec:ed} l'exposant diophantien $\beta_\chi(x)$ d'un point $x\in X(\R)$.
Nous avons vu en outre qu'il existe une constante $\beta_\chi(X)$ telle que pour presque tout $x\in X(\R)$, $\beta_\chi(x)=\beta_\chi(X)$.
Pour suivre la terminologie existant dans le cadre de l'espace projectif, nous posons la définition suivante.

\begin{definition}
Une mesure borélienne $\mu$ sur $X$ est dite \emph{extrémale} si pour $\mu$-presque tout $x$ dans $X$, $\beta_\chi(x)=\beta_\chi(X)$.
Dans le cas où $\mu$ est une mesure de Lebesgue sur une sous-variété analytique $M$, nous dirons aussi que $M$ est extrémale.
\end{definition}

Nous voulons énoncer une condition suffisante pour qu'une sous-variété analytique $M\subset X$ soit extrémale.
Pour cela, rappelons que si $\theta\subset\Pi$ est l'ensemble de racines simples associé sous-groupe parabolique $P$,
on définit alors un sous-groupe à un paramètre dans $G$ en posant
\begin{equation}\label{atf}
a_t=e^{tY}
\quad\mbox{où}\quad Y\in\ka\ \mbox{est défini par}\
\alpha(Y) = 
\left\{
\begin{array}{ll}
0 & \mbox{si}\ \alpha\in\theta\\
-1 & \mbox{si}\ \alpha\not\in\theta.
\end{array}
\right.
\end{equation}
Rappelons qu'une sous-variété de Schubert dans $X$ est une sous-variété de la forme 
\[
X_wg = \overline{PwB}g, \quad\mbox{avec}\quad w\in W_P\ \mbox{et}\ g\in G,
\]
et qu'une variété de Schubert $X_wg$ est dite instable pour le flot $a_t=e^{tY}$ s'il existe un poids dominant $\omega$ tel que $\omega(Y^w)<0$.
Les résultats de la partie précédente permettent de montrer le théorème suivant.

\begin{theorem}[Critère d'extrémalité pour les variétés analytiques]
\label{extan}
Soit $M$ une sous-variété analytique connexe de $X$.
Si $M$ n'est incluse dans aucune sous-variété de Schubert instable, alors $M$ est extrémale.
\end{theorem}
\begin{proof}
Soit $x_0\in M$ et $s:X\to G$ une section analytique locale au voisinage de $x_0$.
Notons $\lambda_M$ la mesure de Lebesgue sur $M$ et $\mu=s_*\lambda_M$ la mesure image de $\lambda_M$ par la section $s$.
D'après le corollaire~\ref{regsec}, la mesure $\mu$ est localement régulière au voisinage de $s_0=s(x_0)$.
Comme $M$ n'est incluse dans aucune sous-variété de Schubert instable, la variété $\cH_\mu(s_0)$ n'est incluse dans aucune sous-variété de Bruhat instable.
Le corollaire~\ref{semstab} s'applique donc: pour $\lambda_M$-presque tout $x$ au voisinage de $x_0$,
\[
\lim_{t\to\infty} \frac{1}{t} \lambda_1(a_ts(x)V_\chi(\Z)) = 0.
\]
Le lemme~\ref{daniextremal} permet d'en conclure que $\beta_\chi(x)=\beta_\chi(X)$.
Comme cela vaut pour $\lambda_M$-presque tout $x$ au voisinage d'un point $x_0\in M$ arbitraire, la variété $M$ est extrémale. 
\end{proof}

\begin{remark}
Une sous-variété $M\subset P\bcs G$ est dite \emph{dégénérée} si $M$ est contenue dans une variété de Schubert stricte.
\note{Une variété de Schubert $X_wg$ est stricte si et seulement si $w\neq w_0$, classe de l'élément de longueur maximale dans $W$.}
Naturellement, toute sous-variété incluse dans une variété de Schubert instable est dégénérée.
\note{Noter que la variété de Schubert maximale $X=P\bcs G$ n'est pas instable, puisque $w_0\cdot Y$ est dans la chambre de Weyl opposée à $\ka^-$, et donc $\omega(w_0\cdot Y)>0$ pour tout poids fondamental $\omega$.}
Le théorème ci-dessus implique donc que toute sous-variété analytique $M\subset X$ non dégénérée est extrémale.
Cependant, il peut exister des sous-variétés dégénérées extrémales; nous en verrons quelques exemples au chapitre~\ref{chap:exemples}.
\end{remark}

Avec une condition supplémentaire sur les coefficients qui définissent la sous-variété analytique $M$, on peut même améliorer ce critère.

\begin{theorem}[Critère d'extrémalité pour les variétés algébriques]
\label{extanalg}
Soit $M$ une sous-variété analytique connexe de $X$ dont l'adhérence de Zariski est définie sur $\QQ$.
Si $M$ n'est incluse dans aucune sous-variété de Schubert \emph{rationnelle} instable, alors $M$ est extrémale.
\end{theorem}
\begin{proof}
La démonstration est identique à celle du théorème~\ref{extan}, en appliquant le corollaire~\ref{semstabalg} au lieu du corollaire~\ref{semstab}.
\end{proof}

\begin{remark}
Ce critère n'est pas toujours optimal.
\note{On peut même avoir $\beta_\chi(x)<\beta_\chi(X)$.}
Pour certains choix de $X$, il existe des sous-variétés $M$ de dimension strictement positive telles que pour presque tout $x$ dans $M$, $\beta_\chi(x)=\beta_\chi(X)$.
Nous verrons toutefois plus loin quelques exemples où le critère est optimal: si $M$ est incluse dans une sous-variété de Schubert rationnelle instable, alors $M$ n'est pas extrémale.
C'est le cas par exemple lorsque $X=\PP^n$ est un espace projectif, ou lorsque $X=\Grass(\ell,d)$ est une variété grassmannienne.
\end{remark}

Plus généralement, lorsque l'adhérence de Zariski de $M$ est définie sur $\QQ$, on peut donner une formule pour l'exposant diophantien presque sûr d'un point de $M$.
C'est ce que nous détaillons au paragraphe suivant.

\section{Sous-variétés algébriques définies sur $\QQ$}

Ce paragraphe a pour but le théorème suivant, analogue du théorème~\ref{expalg}, qui permet de calculer l'exposant diophantien $\beta_\chi(x)$ pour un point $x$ arbitraire dans $X(\QQ)$.
Dans toute la suite, le groupe à un paramètre $(a_t)$ est celui défini ci-dessus en \eqref{atf}.

%\begin{theorem}[Exposant diophantien d'une variété définie sur $\QQ$]
%\label{expanalg}
%Soit $M$ une sous-variété analytique connexe de $X$ dont l'adhérence de Zariski est définie sur $\QQ$.
%Pour chaque $x\in X$, on note $s_x\in G$ un élément tel que $x=Ps_x$.
%Soit $c_M$, $Q_M$, $\gamma_M$ les éléments donnés par le théorème~\ref{diaganalg} pour décrire l'orbite $(a_ts_x\Gamma)$ dans $\Omega$ lorsque $x$ est choisi aléatoirement sur $M$, et 
%\[
%X_w \supset M\gamma_M^{-1}, \quad w\in W_P,
%\]
%la plus petite variété de Schubert standard contenant $M\gamma_M^{-1}$.
%
%Soit encore $\theta_M$ l'ensemble de racines simples associé à $P_M$, et $\pi_M:\ka^*\to\ka^*$ la projection orthogonale sur $\bigcap_{\alpha\in\theta_M}\alpha^\perp$.
%Alors pour presque tout $x\in M$,
%\[
%\gamma_\chi(x) = -\pi_M(w^{-1}\chi)(w^{-1}Y),
%\]
%et par conséquent,
%\[
%\beta_\chi(x) = \frac{1}{-\chi(Y) + \pi_M(w^{-1}\chi)(w^{-1}Y)}.
%\]
%\end{theorem}
%

\begin{theorem}[Exposant diophantien d'une variété définie sur $\QQ$]
\label{expanalg}
Soit $M$ une sous-variété analytique connexe de $X$ dont l'adhérence de Zariski est définie sur $\QQ$.
Pour chaque $x\in X$, on note $s_x\in G$ un élément tel que $x=Ps_x$.
Soit $c_M$, $Q_M$, $\gamma_M$ les éléments donnés par le théorème~\ref{diaganalg} pour décrire l'orbite $(a_ts_x\Gamma)$ dans $\Omega$ lorsque $x$ est choisi aléatoirement sur $M$, et 
\[
X_w \supset M\gamma_M^{-1}, \quad w\in W_P,
\]
la plus petite variété de Schubert standard contenant $M\gamma_M^{-1}$.

Alors pour presque tout $x\in M$,
\[
\gamma_\chi(x) = -\bracket{\chi^w,p_{\ka^-}(Y^w)},
\]
et par conséquent,
\[
\beta_\chi(x) = \frac{1}{-\bracket{\chi,Y} + \bracket{\chi^w,p_{\ka^-}(Y^w)}}.
\]
\end{theorem}

\begin{proof}
La démonstration est presque identique à celle du théorème~\ref{expalg}, avec les changements qui s'imposent: il faut appliquer le théorème~\ref{diaganalg} au lieu du théorème~\ref{diagalg}, et remplacer les éléments $c_\infty$, $\gamma_\infty$, $P_\infty$ par $c_M$, $\gamma_M$, $P_M$, ...etc.
Les détails sont laissés au lecteur.
\end{proof}

Nous concluons ce chapitre en résumant quelques propriétés importantes de l'exposant diophantien d'un point $x$ choisi aléatoirement sur une variété algébrique définie sur $\QQ$.

\begin{corollary}
Soit $X$ une variété de drapeaux, munie de la distance de Carnot-Carathéodory usuelle et d'une hauteur $H_\chi$ associée au poids dominant $\chi$.
Soit $M\subset X$ une sous-variété analytique connexe dont l'adhérence de Zariski est définie sur $\QQ$.
\begin{enumerate}
\item L'exposant $\beta_\chi(M)$ est déterminé par l'intersection des sous-variétés de Schubert rationnelles contenant $M$. Et même, il existe une sous-variété de Schubert $X_w\gamma\supset M$ avec $\gamma\in G(\Q)$ telle que $\beta_\chi(M)=\beta_\chi(X_w\gamma)$. 
\item Pour tout $x\in M\cap X(\QQ)$ hors de toute sous-variété de Schubert rationnelle $X_w'\gamma'\not\supset M$, $\beta_\chi(x)=\beta_\chi(M)$.
\end{enumerate}
\end{corollary}

\chapter{Quelques exemples}
\label{chap:exemples}

Pour illustrer les théorèmes généraux démontrés dans ce mémoire, nous en décrivons maintenant quelques cas particuliers.
C'est souvent après l'étude approfondie de ces exemples importants qu'ont pu être démontrés les résultats plus abstraits sur les variétés de drapeaux générales.

\section{Espace projectif}

L'espace projectif $\PP^{d-1}$ constitue le cadre de l'approximation diophantienne classique.
Dans ce cadre, tous les résultats présentés dans ce mémoire étaient déjà connus.
Nous les rappelons toutefois brièvement, puisque notre objectif était justement de comprendre ces théorèmes de façon plus générale, à partir des groupes arithmétiques.

\bigskip

Si $x,y\in\PP^{d-1}$ sont engendrés respectivement par les vecteurs $u,v\in\R^d$ leur distance est donnée par la formule
\[
d(x,y)=\frac{\norm{u\wedge v}}{\norm{u}\norm{v}}.
\]
La hauteur sur $\PP^{d-1}(\Q)$ est la hauteur usuelle: si $v\in\PP^{d-1}(\Q)$ s'écrit en coordonnées homogènes $v=[v_1:\dots:v_d]$, où les $v_i$ sont des entiers premiers entre eux dans leur ensemble, alors
\[
H(v) = \max_{1\leq i\leq d} \abs{v_i}.
\]
Nous commençons par le célèbre théorème de Dirichlet \cite{dirichlet}, bien que ce résultat ne semble pas se généraliser aisément dans une variété drapeau arbitraire.

\begin{theorem}[Théorème de Dirichlet]
Pour tout $x\in\PP^{d-1}(\R)$,
\[
\beta(x)\geq 1+\frac{1}{d-1}.
\]
\end{theorem}

Une simple application du lemme de Borel-Cantelli permet de montrer que presque tout $x$ dans $\PP^{d-1}(\R)$ vérifie l'égalité $\beta(x)=1+\frac{1}{d-1}$.
Le théorème de Khintchine \cite{khintchine} donne un critère simple sur une fonction $\psi:\R_+\to\R_+$ décroissante pour que l'inégalité
\begin{equation}\label{khinp}
d(x,v) \leq H(v)^{-1-\frac{1}{d-1}} \psi(H(v))
\end{equation}
ait une infinité de solutions lorsque $x$ est choisi suivant la mesure de Lebesgue sur $\PP^{d-1}$.

\begin{theorem}[Théorème de Khintchine]
Soit $\psi:\R_+\to\R_+$ une fonction décroissante.
\begin{itemize}
\item Si $\int_1^\infty\psi(u)^{d-1}\frac{\dd u}{u} <+\infty$, alors pour presque tout $x\in\PP^{d-1}(\R)$, l'inégalité \eqref{khinp} admet un nombre fini de solutions $v\in\PP^{d-1}(\Q)$.
\item Si $\int_1^\infty\psi(u)^{d-1}\frac{\dd u}{u} =+\infty$, alors pour presque tout $x\in\PP^{d-1}(\R)$, l'inégalité \eqref{khinp} admet une infinité de solutions $v\in\PP^{d-1}(\Q)$.
\end{itemize}
\end{theorem}

L'exposant diophantien d'un point $x\in\PP^{d-1}(\QQ)$ a été calculé par Schmidt \cite{schmidt_simultaneous} grâce à son théorème du sous-espace, qui généralise les résultats de Thue \cite{thue}, Siegel \cite{siegel} et Roth \cite{roth} pour $\PP^1(\QQ)$.

\begin{theorem}[Théorème de Thue-Siegel-Roth-Schmidt]
Si $x\in\PP^{d-1}(\QQ)$, alors $\beta(x)=1+\frac{1}{d_x}$, où $d_x$ est la dimension du plus petit sous-espace projectif rationnel contenant $x$.
\end{theorem}

Le problème de l'approximation diophantienne sur les sous-variétés a été posé en premier par Mahler \cite{mahler} pour la courbe $[1:x:x^2:\dots:x^{d-1}]$ dans $\PP^{d-1}$.
Ayant résolu le problème de Mahler, Sprindzuk a conjecturé dans \cite{sprindzuk} le résultat suivant, démontré finalement par Kleinbock et Margulis \cite{kleinbockmargulis} en 1998.
Rappelons qu'une sous-variété $M\subset\PP^{d-1}(\R)$ est dite non dégénérée si elle n'est incluse dans aucun sous-espace projectif strict.

\begin{theorem}[Théorème de Kleinbock-Margulis]
Toute sous-variété analytique connexe non dégénérée dans $\PP^{d-1}(\R)$ est extrémale.
\end{theorem}

Dans un travail en commun avec Emmanuel Breuillard \cite{subspace}, nous avons observé que les méthodes utilisées pour démontrer ces deux derniers théorèmes permettent de donner une formule pour l'exposant d'un point pris aléatoirement sur une sous-variété algébrique définie sur $\QQ$.

\begin{theorem}[Exposant d'une sous-variété définie sur $\QQ$]
Soit $M\subset\PP^{d-1}(\R)$ une sous-variété analytique connexe.
On suppose que le plus petit sous-espace projectif réel contenant $M$ est défini sur $\QQ$.
Alors, pour presque tout $x\in M$, $\beta(x)=1+\frac{1}{d_M}$, où $d_M$ est la dimension du plus petit sous-espace projectif rationnel contenant $M$.
En particulier, si $M$ est non dégénérée, alors $M$ est extrémale.
\end{theorem}

\begin{remark}
Dans l'espace projectif, les sous-variétés de Schubert ne sont autres que les sous-espaces projectifs.
Toute sous-variété de Schubert stricte est instable pour le flot $(a_t)$, et par conséquent, si une sous-variété $M$ n'est incluse dans aucune sous-variété de Schubert instable, elle est non dégénérée.
\end{remark}

\section{Quadriques}

À notre connaissance, ce sont Kleinbock et Merrill \cite{kleinbockmerrill} qui ont obtenu les premiers résultats remarquables pour l'approximation diophantienne intrinsèque sur les quadriques, en démontrant pour une sphère de dimension arbitraire les analogues des théorèmes de Dirichlet et de Khintchine.
Dans un article avec Fishman et Simmons \cite{fkms}, ils ont ensuite généralisé leurs résultats à une quadrique arbitraire.
Pour une introduction élémentaire à ces problèmes, on renvoie à l'article \cite{simplex}.

Dans ce cadre $X$ désigne une quadrique projective non singulière, i.e. l'ensemble des droites isotropes pour une forme quadratique rationnelle $Q$ non dégénérée sur $\R^d$.
La distance et la hauteur sur $X$ sont obtenues par restriction de la distance et de la hauteur usuelles sur $\PP^{d-1}$.
On suppose en outre que $X$ contient un point rationnel; par projection stéréographique, cela implique en fait que $X(\Q)$ est dense dans $X(\R)$.

\begin{theorem}[Fishman-Kleinbock-Merrill-Simmons]
Soit $X$ une quadrique rationnelle projective non singulière contenant un point rationnel.
Pour presque tout $x\in X(\R)$, $\beta(x)=1$.
\end{theorem}

Soit $X_0\subset\PP^3$ la quadrique définie par l'équation $x_1x_2-x_3x_4$.
La généralisation du théorème de Khintchine aux quadriques nécessite de distinguer deux cas, suivant que la quadrique $X$ est rationnellement isomorphe à $X_0$, ou non.

\begin{theorem}[Fishman-Kleinbock-Merrill-Simmons]
Soit $X$ une quadrique rationnelle projective non singulière de dimension $n$ contenant un point rationnel, et $\psi:\R^+\to\R^+$ une fonction décroissante.
Pour $x\in X(\R)$ on considère l'inégalité
\begin{equation}\label{khinq}
d(x,v) \leq H(v)^{-1} \psi(H(v)).
\end{equation}
Si $X$ n'est pas rationnellement isomorphe à $X_0$, alors:
\begin{itemize}
\item si $\int_1^\infty \psi(u)^n \frac{\dd u}{u} = +\infty$, l'inégalité \eqref{khinq} admet une infinité de solutions $v\in X(\Q)$ pour presque tout $x\in X(\R)$;
\item si $\int_1^\infty \psi(u)^n \frac{\dd u}{u} < +\infty$, l'inégalité \eqref{khinq} n'a qu'un nombre fini de solutions $v\in X(\Q)$ pour presque tout $x\in X(\R)$.
\end{itemize}
Si $X$ est rationnellement isomorphe à $X_0$, alors:
\begin{itemize}
\item si $\int_1^\infty \psi(u)^n (\log\log u) \frac{\dd u}{u} = +\infty$, l'inégalité \eqref{khinq} admet une infinité de solutions $v\in X(\Q)$ pour presque tout $x\in X(\R)$;
\item si $\int_1^\infty \psi(u)^n (\log\log u) \frac{\dd u}{u} < +\infty$, l'inégalité \eqref{khinq} n'a qu'un nombre fini de solutions $v\in X(\Q)$ pour presque tout $x\in X(\R)$.
\end{itemize}
\end{theorem}

\begin{remark}
Écrivons $X=P\bcs G$, où $G=\SO_Q$ est le groupe orthogonal associé à la forme quadratique $Q$, et $P$ le sous-groupe parabolique stabilisateur d'une droite rationnelle isotrope dans la représentation standard.
Si $X\not\sim X_0$, le groupe $G$ est $\Q$-simple et le sous-groupe parabolique $P$ est maximal: $\rang_{\Q} P = \rang_{\Q} G -1$.
\note{Faire le calcul explicite du sous-groupe parabolique $P$.}
En revanche, si $X\sim X_0$, on a un isomorphisme $G\simeq\SO(2,2)\simeq\SO(2,1)\times\SO(2,1)$, et $\rang_{\Q} P = \rang_{\Q} G - 2$.
Avec les résultats généraux du chapitre~\ref{chap:khintchine}, et en particulier les lemmes~\ref{maximal} et \ref{somrac}, cette différence explique la distinction de cas dans le théorème ci-dessus.
\end{remark}

Dans un article récent \cite{quadriques}, nous avons poursuivi les travaux de Fishman, Kleinbock, Merrill et Simmons en étudiant l'approximation diophantienne des points algébriques et des quantités dépendantes sur les quadriques.
Cela nous a permis en particulier d'établir le résultat suivant.

\begin{theorem}[Exposant diophantien d'une sous-variété algébrique]
Soit $X$ une quadrique rationnelle projective non singulière contenant un point rationnel, et $M$ une sous-variété analytique de $X$.
On suppose que le plus petit sous-espace totalement isotrope réel contenant $M$ est défini sur $\QQ$.
Alors, pour presque tout $x\in M$,
\[
\beta(x) = 1 + \frac{1}{d_M},
\]
où $d_M$ est la dimension du plus petit sous-espace totalement isotrope \emph{rationnel} contenant $M$.
(S'il n'existe pas de tel sous-espace, on pose $d_M=+\infty$.)
\end{theorem}

\begin{remark}
Le théorème ci-dessus s'applique en particulier dans les deux cas suivants:
\begin{itemize}
\item si $M=\{x\}$ est réduite à un seul point $x\in X(\QQ)$, on obtient un analogue du résultat de Thue-Siegel-Roth-Schmidt pour la quadrique $X$;
\item si $M$ est \emph{non dégénérée}, i.e. n'est incluse dans aucun sous-espace totalement isotrope, alors $d_M=+\infty$ et $M$ est extrémale.
\end{itemize}
\end{remark}

\begin{remark}
Dans une quadrique $X$, les sous-variétés de Schubert sont les sous-espaces totalement isotropes, toute sous-variété de Schubert stricte est instable.
Comme dans le cas de l'espace projectif, une sous-variété n'est incluse dans aucune sous-variété de Schubert instable si et seulement si elle est non dégénérée.
\end{remark}

\section{Grassmannienne}

Dans l'article \cite{schmidt_grass} écrit en 1967, Schmidt a proposé d'utiliser le plongement de Plücker pour définir la hauteur d'un sous-espace rationnel, et étudier l'approximation d'un sous-espace réel $x$ par des sous-espaces rationnels $v$ de diverses dimensions.
Cela s'inscrit bien dans le cadre de ce mémoire: si $G=\SL_d$ et $P$ le sous-groupe parabolique stabilisateur du sous-espace $\Vect\{e_1,\dots,e_\ell\}$ dans la représentation standard, on obtient la grassmannienne des $\ell$-plans dans un espace de dimension $d$ comme quotient
\[
X = \Grass(\ell,d) \simeq P\bcs G.
\]
La hauteur utilisée par Schmidt est alors celle associée à la représentation fondamentale $\wedge^\ell\R^d$, la distance est la distance riemannienne usuelle.
Cela permet de définir l'exposant diophantien d'un point $x\in\Grass(\ell,d)$.
Avec ces définitions, Schmidt obtient une minoration optimale de l'exposant diophantien d'un point choisi aléatoirement sur la grassmannienne: pour presque tout $x\in\Grass(\ell,d)$, $\beta(x) \geq \frac{1}{\ell}+\frac{1}{d-\ell}$.
Le théorème~\ref{exposantps} permet de préciser ce résultat en une égalité presque sûre.

\begin{theorem}
Pour presque tout $x\in\Grass(\ell,d)$, $\beta(x) = \frac{1}{\ell}+\frac{1}{d-\ell}$.
\end{theorem}

\comm{
Pour montrer l'inégalité $\beta(x)\leq\frac{1}{\ell}+\frac{1}{d-\ell}$, il suffit en fait d'utiliser le lemme de Borel-Cantelli et le fait que le nombre de points de hauteur au plus $T$ dans $X=\Grass(\ell,d)$ est majoré par $T^d$.
}

Pour une variété grassmannienne, le théorème~\ref{khintchine}, analogue du théorème de Khintchine, s'écrit plus simplement comme suit.

\begin{theorem}
Soit $X=\Grass(\ell,d)$ la variété grassmannienne des sous-espaces de dimension $\ell$ dans un espace de dimension $d$.
Étant donnée une fonction décroissante $\psi:\R^+\to\R^+$ on considère l'inégalité
\begin{equation}\label{khing}
d(x,v) \leq H(v)^{-\frac{1}{\ell}-\frac{1}{d-\ell}} \psi(H(v)).
\end{equation}
\begin{itemize}
\item Si $\int_1^\infty \psi(u)^{\ell(d-\ell)} \frac{\dd u}{u}=+\infty$, alors \eqref{khing} admet une infinité de solutions $v\in X(\Q)$ pour presque tout $x\in X(\R)$.
\item Si $\int_1^\infty \psi(u)^{\ell(d-\ell)} \frac{\dd u}{u}<+\infty$, alors \eqref{khing} n'admet qu'un nombre fini de solutions $v\in X(\Q)$ pour presque tout $x\in X(\R)$.
\end{itemize}
\end{theorem}

\begin{remark}
On remarque que l'exposant $\ell(d-\ell)$ qui apparaît dans la condition d'intégrabilité sur $\psi$ est égal à la dimension de $X$.
Pour calculer cet exposant, on peut utiliser le lemme~\ref{maximal}.
La somme des racines apparaissant dans le radical unipotent de $P$ est
\begin{align*}
\rho_\ell & = \sum_{\substack{1\leq i\leq \ell\\ \ell<j\leq d}} \eps_i-\eps_j\\
& = (d-\ell)(\eps_1+\dots+\eps_\ell)-\ell(\eps_{\ell+1}+\dots+\eps_d)\\
& = d(\eps_1+\dots+\eps_\ell) = d\omega_\ell.
\end{align*}
L'exposant recherché est donc
\[
\frac{a_\chi}{\beta_\chi} = \frac{d}{\frac{1}{\ell}+\frac{1}{d-\ell}} = \ell(d-\ell).
\]
\end{remark}

Toute sous-variété de Schubert $X_wg$ dans $X=\Grass(\ell,d)$ est de la forme
\[
X_wg = \{ x\in X\ |\ \forall i=1,\dots,d-1,\ \dim x\cap V_i \geq m_i\},
\]
où $\{0\}=V_0<V_1<\dots<V_d=\R^d$ est un drapeau total de $\R^d$ et $m_1\leq \dots\leq m_d$ une suite d'entiers naturels.
On peut montrer que toute sous-variété de Schubert distincte de $X=\Grass(\ell,d)$ est incluse dans un pinceau
\[
P_{W,r} = \{ x\in X\ |\ \dim W\cap x\geq r\} \qquad\mbox{avec}\ r>d-\ell-\dim W,
\]
et qu'une sous-variété de Schubert est instable si et seulement si elle est incluse dans un pinceau $P_{W,r}$ contraignant, i.e. satisfaisant
\[
\frac{r}{\dim W}>\frac{\ell}{d}.
\]
Les théorèmes généraux de ce mémoire impliquent donc les résultats suivants.

\begin{theorem}[Critère d'extrémalité dans la grassmannienne]
Soit $M$ une sous-variété analytique de $\Grass(\ell,d)$.
Si $M$ n'est incluse dans aucun pinceau contraignant, alors $M$ est extrémale.
\end{theorem}

\note{Pour voir que ce drapeau partiel est bien défini, on peut utiliser le lemme de sous-modularité; il faudrait donner quelques détails.}
À toute partie $M\subset X$, on associe un drapeau rationnel partiel $\{0\}=V_{d_0}<V_{d_1}<\dots< V_{d_{r-1}}< V_{d_r} =\Q^d$ avec pour chaque $k$, $\dim V_{d_k}=d_k$.
Pour cela, on définit $V_{d_1}$ comme l'unique sous-espace rationnel de dimension maximale qui maximise la quantité $\min_{x\in M}\frac{\dim x\cap V}{\dim V}$, parmi tous les sous-espaces vectoriels $V\leq\Q^d$; ensuite $V_{d_2}$ est l'unique sous-espace rationnel contenant $V_{d_1}$, de dimension maximale, et qui maximise $\min_{x\in M}\frac{\dim x\cap V -\dim x\cap V_{d_1}}{\dim V- d_1}$, ...etc.

\begin{theorem}[Exposant d'une sous-variété algébrique]
\note{Il suffit que l'adhérence linéaire dans le plongement de Plücker soit définie sur $\QQ$.}
Soit $M$ une sous-variété analytique connexe de $\Grass(\ell,d)$ dont l'adhérence de Zariski est définie sur $\QQ$, et $V_0<V_{d_1}<\dots<V_{d_r}$ le drapeau partiel associé à $M$.
Pour $k=1,\dots,r$, on note
\note{$i_0=d_0=0$, $i_r=\ell$, $d_r=d$, et $c_r=0$.}
\note{$c_k$ est le taux de dilatation minimal de $V_{d_k}$ suivant un flot $(a_ts_x)$, $x\in M$; noter que $c_k\leq 0$.}
\[
i_k=\min_{x\in M} \dim x\cap V_{d_k}
\quad\mbox{et}\quad
c_k=-\frac{i_k}{\ell}+\frac{d_k-i_k}{d-\ell}.
\]
Pour presque tout $x$ dans $M$,
\[
\beta(x) = (\frac{1}{\ell}+\frac{1}{d-\ell})\frac{1}{1-\gamma_M},
\]
où
\[
\gamma_M = \sum_{k=1}^r \frac{i_k-i_{k-1}}{d_k-d_{k-1}}(\frac{i_k-i_{k-1}}{\ell}-\frac{d_k-d_{k-1}-i_k+i_{k-1}}{d-\ell})
= \frac{\ell(d-\ell)}{d}\sum_{k=1}^r \frac{(c_k - c_{k-1})^2}{d_k-d_{k-1}}.
\]
En particulier, $M$ est extrémale si, et seulement si, $M$ n'est incluse dans aucun pinceau rationnel contraignant.
\end{theorem}
\begin{proof}
Nous donnerons une démonstration directe dans \cite{grass}.
Ici nous remarquons seulement que ce théorème est un cas particulier du théorème~\ref{expanalg}; il suffit de voir que le drapeau partiel $\{0\}<V_{d_1}<\dots<V_{d_r}$ s'identifie à la variété de Schubert $X_w\gamma_M$ contenant $M$, et de faire le calcul explicite de la quantité $\bracket{\chi^w,p_{\ka^-}(Y^w)}$.
Les détails sont laissés au lecteur.
Pour vérifier l'égalité entre les deux formules pour $\gamma_M$, il suffit d'observer que
\[
i_k = \frac{\ell}{d} d_k - \frac{\ell(d-\ell)}{d}c_k,
\]
et donc
\[
\frac{i_k-i_{k-1}}{d_k-d_{k-1}}(\frac{i_k-i_{k-1}}{\ell}-\frac{d_k-d_{k-1}-i_k+i_{k-1}}{d-\ell})
=
 \frac{\ell}{d}(1-\frac{c_k-c_{k-1}}{d_k-d_{k-1}}(d-\ell))(c_k-c_{k-1}).
\]
L'égalité souhaitée s'en déduit en sommant sur $k$, et en observant que $0=c_r=\sum_{k=1}^r c_k-c_{k-1}$.

Si $M$ n'est incluse dans aucun pinceau rationnel contraignant, alors le drapeau associé est le drapeau trivial $\{0\}=V_0 < V_d = \Q^d$, donc $\gamma_M=0$ et $\beta(x)=\frac{1}{\ell}+\frac{1}{d-\ell}$ pour presque tout $x$ dans $M$.
Réciproquement, s'il existe un pinceau rationnel contraignant contenant $M$, alors le drapeau associé à $M$ est non trivial, et la formule ci-dessus montre donc que $\gamma_M>0$.
Cela implique que $M$ n'est pas extrémale.
\end{proof}

Dans son article \cite{schmidt_grass}, après avoir obtenu plusieurs encadrements pour l'exposant diophantien d'un élément $x\in\Grass(\ell,d)$, Schmidt pose notamment le problème suivant: déterminer la valeur minimale de $\beta(x)$ lorsque $x$ varie dans $X(\R)$.
Il est possible que cette valeur soit égale à $\frac{1}{\ell}+\frac{1}{d-\ell}$, valeur presque sûre de $\beta(x)$ lorsque $x$ est choisi aléatoirement dans $X$.
Nous n'avons pas de preuve de cela pour l'instant.
Cependant, le théorème ci-dessus permet déjà de minorer convenablement l'exposant de tout point de $X(\QQ)$.
En fait, dans l'écriture $X=P\bcs G$, le sous-groupe parabolique est maximal, et le corollaire ci-dessous est donc un cas particulier du corollaire~\ref{rang1}.

\begin{corollary}
Soit $X=\Grass(\ell,d)$.
Pour tout $x\in X(\QQ)$, $\beta(x)\geq\frac{1}{\ell}+\frac{1}{d-\ell}$, avec égalité si et seulement si $x$ n'est inclus dans aucun pinceau rationnel contraignant.
\end{corollary}

\begin{example}
Soit $X=\Grass(2,d)$ la variété des $2$-plans dans $\R^d$.
Dans ce cas,
\[
a_t=e^{tY},\quad\mbox{avec}\quad Y=\frac{1}{\frac{1}{2}+\frac{1}{d-2}}\diag(-\frac{1}{2},-\frac{1}{2},\frac{1}{d-2},\dots,\frac{1}{d-2}).
\]
On considère la sous-variété
\[
M=\{ x\in X\ |\ \dim x\cap V_1\geq 1\},\quad\mbox{où}\ V_1=\Vect(e_1,\dots,e_k).
\]
Si $k<\frac{d}{2}$, $M$ est un pinceau contraignant.
Pour presque tout $x$ dans $M$,
\note{Inclure un dessin de $c_\infty$, en identifiant $\ka^-$ à l'ensemble des fonctions convexes sur $[0,d]$ telles que $f(0)=f(d)=0$.}
\begin{align*}
c_\infty & = \lim_{t\to\infty} \frac{1}{t}c(a_ts_x)\\
& = \diag(-\frac{1}{k}\frac{d-2k}{2(d-2)},\dots,-\frac{1}{k}\frac{d-2k}{2(d-2)},\frac{1}{d-k}\frac{d-2k}{2(d-2)},\dots,\frac{1}{d-k}\frac{d-2k}{2(d-2)}).
\end{align*}
En fait, si $w$ est une matrice de permutation telle que $w^{-1}\{1,2\}=\{k,d\}$ (cela détermine un élément de $W_P$), alors
\[
M = PwB.
\]
\note{Pour cette vérification, il est commode d'identifier $\ka^-$ au fonctions convexes sur $[0,d]$ telles que $f(0)=f(d)=0$. La projection $p_{\ka^-}$ associe à une fonction son plus grand minorant convexe.}
On vérifie sans peine que $c_\infty=p_{\ka^-}(Y^w)$.
D'ailleurs, $\theta_\infty=[1,d-1]\setminus\{k\}$ et si $\pi_\infty:\ka\to\ka$ est la projection sur $\bigcap_{i\in\theta_\infty}\alpha_i^\perp$ (l'espace des fonctions dont le seul point angulaire est en $k$), alors $c_\infty=p_{\ka^-}(Y^w) = \pi_\infty(Y^w)$.

La hauteur associée au plongement de Plücker correspond au plus haut poids $\chi=\eps_1+\eps_2$, et $\chi^w = \eps_{w^{-1}(1)}+\eps_{w^{-1}(2)} = \eps_k+\eps_d$.
Par conséquent, pour presque tout $x\in M$,
\[
\gamma_\chi(x) = \bracket{\chi^w,\pi_\infty(Y^w)} = - \frac{1}{k}\frac{d-2k}{2(d-2)} + \frac{1}{d-k}\frac{d-2k}{2(d-2)} = -\frac{(d-2k)^2}{2(d-2)k(d-k)},
\]
et 
\[
\beta(x) = (\frac{1}{2}+\frac{1}{d-2})\frac{1}{1-\gamma_\chi(x)}.
\]

Dans cet exemple, le sous-groupe $Q_M$ est égal au stabilisateur de $V_1$ dans la représentation standard.
Si $V_2=\Vect(e_{k+1},\dots,e_d)$, on peut identifier le facteur de Levi de $Q_M$ à $L_M=\GL(V_1)\times\GL(V_2)$.
La représentation $V_\chi$ se décompose en irréductibles pour $L_M$ de la façon suivante:
\[
V_\chi = \wedge^2V_1 \oplus V_1\otimes V_2 \oplus \wedge^2 V_2.
\]
La représentation irréductible contenant $we_\chi=e_k\wedge e_d$ est $V_1\otimes V_2$.
\note{Quelques détails pour ce calcul sont inscrits dans le fichier .tex}
Grâce à la base $(e_i\wedge e_j)_{1\leq i\leq k; k+1\leq j\leq d}$, on calcule facilement le taux de contraction de $V_1\wedge V_2$ par $a_t^w=e^{tY^w}$:
%w(1)=k, w(2)=d
%wY = (\frac{1}{d-2},\dots,\frac{1}{d-2},-\frac{1}{2},\frac{1}{d-2},\dots,\frac{1}{d-2},-\frac{1}{2})
\[
\frac{2}{d-2}(k-1)(d-k-1) %pour les $e_i\wedge e_j$, $i=1,\dots,k-1$, $j=k+1,\dots,d-1$
+ (\frac{1}{d-2}-\frac{1}{2})(k-1) %pour les $e_i\wedge e_d$, $i=1,\dots,k-1$
+ (-\frac{1}{2}+\frac{1}{d-2})(d-k-1) %pour les $e_k\wedge e_j$, $j=k+1,\dots,d-1$
- 1 %pour $e_k\wedge e_d
= - \frac{(d-2k)^2}{2(d-2)} %merci Xcas!
\]
Comme $\dim V_1\otimes V_2 = k(d-k)$, on trouve bien la même valeur que ci-dessus pour $\gamma_\chi(x)$.
\end{example}

\section{Variété des drapeaux dans $\R^d$}
\label{sec:drap}

\subsection*{Drapeaux dans $\R^3$.}

Ici $G=\SL_3$ et $P$ est le sous-groupe parabolique minimal constitué des matrices triangulaires supérieures.
La variété quotient $X=P\bcs G$ s'identifie à l'ensemble des drapeaux dans $\R^3$:
\[
X = \{ x=(x_1,x_2)\ ;\ x_1\in\Grass(1,3),\, x_2\in\Grass(2,3),\, x_1\leq x_2\}.
\]
Cette fois, la distance de Carnot-Carathéodory sur $X$ n'est pas riemannienne.
Géométriquement, pour se déplacer en partant d'un drapeau $x=(x_1,x_2)$, on s'autorise à déplacer infinitésimalement $x_1$ dans $x_2$, et $x_2$ contenant $x_1$; cela définit un champ de plans sur $X$, et les seuls chemins autorisés sont ceux qui sont tangents à ce champ de plans.

Le flot diagonal est donné par l'élément $Y=\diag(1,0,-1)$.
Les éléments $I$, $(1,2)$ et $(2,3)$ du groupe de Weyl $S_3$ donnent les variétés de Schubert instables.
Dire qu'une $M$ est incluse dans une cellule instable revient donc à dire que tous les éléments de $M$ ont la même droite, ou le même plan.

Les autres sous-variétés de Schubert sont stables.
Par exemple, la variété $M=\{(x_1,x_2)\ |\ x_2\ni e_1\}$ est extrémale quel que soit le choix de hauteur sur $X(\Q)$.

D'après le corollaire~\ref{deploy}, si $X$ est munie de la hauteur anti-canonique, alors
\[
\forall x\in X(\QQ),\quad \beta_\chi(x) \geq \beta_\chi(X),
\]
avec égalité si et seulement si $x$ n'est inclus dans aucune sous-variété de Schubert rationnelle instable.
\note{Refaire ce calcul, et l'inclure dans le fichier .tex.}
En fait, pour la variété $X$ des drapeaux dans $\R^3$, ce résultat est encore valable quelle que soit la hauteur $H_\chi$ sur $X$.

\subsection*{Drapeaux dans $\R^4$.}

Cette fois, $G=\SL_4$ et $P=B$ est le sous-groupe parabolique minimal constitué des matrices triangulaires supérieures.
La variété quotient $X=P\bcs G$ s'identifie à l'ensemble des drapeaux dans $\R^4$:
\[
X = \{ x=(x_1,x_2,x_3)\ ;\ x_1< x_2< x_3,\ x_i\in\Grass(i,4),\, i=1,2,3\}.
\]
C'est une variété de dimension $6$, et la distance de Carnot-Carathéodory sur $X$ est déterminée par un champ de $3$-plans: au voisinage de $(x_1,x_2,x_3)$, on s'autorise à déplacer infinitésimalement $x_1$ dans $x_2$, $x_2$ contenant $x_1$ et à l'intérieur de $x_3$, et $x_3$ contenant $x_2$.

Le flot diagonal $(a_t)$ est donné par l'élément
\[
Y=\frac{1}{2}\diag(3,1,-1,-3).
\]
En écrivant la liste des $Y^w$, pour $w\in S_4$ on vérifie que
\begin{enumerate}
\item Si $w^{-1}(4)=4$, tous les éléments de $M=BwB$ ont le même hyperplan, et $M$ est instable.
\item Si $w^{-1}(4)=3$, il existe un plan $p$ tel que pour tout $(x_1,x_2,x_3)$ dans $M=BwB$,  $p\subset x_3$.
La variété $M$ est instable sauf si $w^{-1}:(1,2,3,4)\mapsto (4,2,1,3)$.
\item Si $w^{-1}(4)=2$, il existe une droite $d$ tel que pour tout $(x_1,x_2,x_3)$ dans $M=BwB$, $d\subset x_3$.
La variété $M$ est instable sauf si $w^{-1}:(1,2,3,4)\mapsto (4,3,1,2)$ (aucune autre contrainte) ou $w^{-1}:(1,2,3,4)\mapsto (3,4,1,2)$ (on impose en outre qu'il existe un hyperplan $h$ tel que pour tout $x$ dans $M$, $x_1\subset h$.)
\item Si $w^{-1}(4)=1$, il n'y a pas de contrainte sur $x_3$, et $M=BwB$ n'est pas instable, sauf pour $w^{-1}:(1,2,3,4)\mapsto (2,3,4,1)$ et $w^{-1}:(1,2,3,4)\mapsto(3,2,4,1)$.
\end{enumerate}

Le corollaire~\ref{deploy} montre que si $X$ est munie de la hauteur anti-canonique, alors tout point $x\in X(\QQ)$ vérifie $\beta_\chi(x)\geq\beta_\chi(X)$.
Cela cependant n'est pas vrai en général, comme le montre l'exemple suivant.

\begin{example}[Des points \emph{très} mal approchables]
Pour certains choix de hauteur sur la variété $X$ des drapeaux dans $\R^4$, il peut exister un point algébrique moins bien approchable qu'un point générique.
On choisit un point $x$ tel que $\bracket{e_1,e_2,e_4}$ soit stable par $a_tx$, et générique pour le reste.
En d'autres termes, pour un élément $p$ algébrique générique de $P(\QQ)$,
\[
x =
\begin{pmatrix}
0 & 0 & 0 & 1 \\
0 & 1 & 0 & 0 \\
0 & 0 & 1 & 0 \\
1 & 0 & 0 & 0 
\end{pmatrix}
\cdot p
=
\begin{pmatrix}
* & * & * & * \\
* & * & * & * \\
0 & 0 & * & 0 \\
* & * & * & * 
\end{pmatrix}
\]
Le sous-espace $\bracket{e_1,e_2,e_4}$ est semi-stable, contracté globablement par $e^{-\frac{t}{2}}$, donc les trois premiers minima successifs de $a_tx\Z^4$ sont atteints dans $\bracket{e_1,e_2,e_4}$, de longueur approximativement $e^{-\frac{t}{6}}$.
Le dernier minimum est donc de longueur $e^{\frac{t}{2}}$.
On choisit alors le poids $\chi$ induit par la représentation de plus haut poids $(k+2)\eps_1+(k+1)\eps_2+k\eps_3$, avec $k$ suffisamment grand.
Cela correspond au diagramme de Young de longueurs $(k+2,k+1,k)$, ou en termes de poids fondamentaux: $\omega_1+\omega_2+k\omega_3$.
Le vecteur $e_\chi$ est donné par le tableau rempli avec des 1 sur la première ligne, des 2 sur la deuxième, et des 3 sur la troisième.

Les premiers minima successifs de $a_txV(\Z)$ sont atteints avec des tableaux contenant seulement $(1,2,4)$; il valent à peu près $e^{-\frac{(3k+3)t}{6}}$.
Ensuite viennent ceux qui contiennent une occurrence de $3$, puis deux occurrences de $3$, ...etc.
Mais pour avoir une projection positive sur $e_\chi$, il faut au moins $k$ occurrences de $3$, et alors la norme du vecteur est minorée par $e^{k\frac{t}{2}-(2k+3)\frac{t}{6}}=e^{(k-3)\frac{t}{6}}$.
Dès que $k>3$, on voit que $r_\chi(a_tx)$ tend vers l'infini à vitesse exponentielle, ce qui montre que $\gamma_\chi(x)<0$, et donc $\beta_\chi(x)<\beta_\chi(X)$.
\end{example}

\begin{remark}
On peut même donner un exemple pour la variété $X$ des drapeaux dans $\R^3$, si l'on autorise une quasi-distance différente de la distance de Carnot-Carthéodory.
À une petite perturbation près de $a_t$ et $\chi$, pour que $P$ soit bien le sous-groupe de Borel, l'exemple est le suivant: $a_t=\diag(e^{-2t},e^t,e^t)$, et $x$ est tel que $\bracket{e_1,e_3}$ est stable par $a_tx$.
Les deux premiers minima successifs de $a_tx\Z^3$ sont atteints dans $\bracket{e_1,e_3}$, et de longueur $e^{-\frac{t}{2}}$ ; le troisième est de longueur $e^t$.

On choisit le poids $\chi$ donné par $\eps_1+\eps_2$, ce qui correspond à la représentation $\wedge^2\R^3$.
On a alors $e_\chi=e_1\wedge e_2$.
Le premier minimum de $a_tx\wedge^2\Z^3$ est de longueur $e^{-t}$ mais sa projection sur $e_\chi$ est nulle.
Les deux suivants sont de longueur $e^{\frac{t}{2}}$, et cela montre bien que $r_\chi(a_tx)$ tend vers l'infini à vitesse exponentielle.
\end{remark}

\chapter{Conclusion}

Nous concluons ce mémoire par quelques problèmes ouverts qui nous semblent mériter d'être mentionnés.

\bigskip

\noindent\textbf{Valeur minimale de l'exposant.}
Soit $X=P\bcs G$ une variété de drapeaux munie de la distance de Carnot-Carathéodory usuelle et de la hauteur induite par un poids dominant $\chi$.
\begin{itemize}
\item Si $\chi$ est choisi de sorte que $X$ est munie de la hauteur anti-canonique, a-t-on, pour tout $x\in X(\QQ)$, $\beta(x)\geq\beta(X)$?
D'après le théorème~\ref{expalg}, cela reviendrait à dire que dans ce cas, on a toujours $\bracket{\chi^w,p_{\ka^-}(Y^w)}\leq 0$, quel que soit $w\in W_P$.
Nous avons vu aux corollaires~\ref{rang1} et \ref{deploy} que cette inégalité est valable lorsque $X$ est de rang 1 ou lorsque $G$ est déployé et $P$ minimal.
\item Déterminer $\inf_{x\in X(\R)}\beta_\chi(x)$. Y a-t-il une égalité
\[
\inf_{x\in X(\R)}\beta_\chi(x)=\inf_{x\in X(\QQ)}\beta_\chi(x)?
\]
Par exemple, pour la grassmannienne $X=\Grass(\ell,d)$, a-t-on, pour tout $x$, $\beta(x)\geq\frac{1}{\ell}+\frac{1}{d-\ell}$?
\comm{Dans le cas particulier où $\chi(Y)\geq -1$, il n'y a qu'une seule valeur propre négative dans $a_t=e^{tY}$. La condition $\norm{\pi^+(v)}\geq\frac{\norm{v}}{2}$ peut être omise dans la correspondance drapeau-réseau, et par conséquent $\inf_{x\in X(\R)} =\beta_\chi(X)$.
La condition $\chi(Y)\geq -1$ est très restrictive. Par exemple, si $X=\Grass(\ell,d)$, cela revient à dire que $\ell=1$ ou $\ell=d-1$ ou $\ell=2$ et $d=4$.
Elle est toutefois valable pour les quadriques non dégénérées, et il serait intéressant de faire une liste d'autres exemples, à l'aide de la classification des groupes semi-simples. On pourrait notamment avoir des exemples lorsque $G$ est symplectique ou exceptionnel.}
\end{itemize}

\bigskip

\noindent\textbf{Métrique riemannienne.}
Soit $X$ une variété de drapeaux rationnelle.
Lorsque $X$ est munie d'une distance riemannienne, peut-on montrer des résultats analogues à ceux établis dans ce mémoire pour la métrique de Carnot-Carathéodory?

\bigskip

\noindent\textbf{Autres variétés algébriques}
Soit $X$ une variété algébrique irréductible $X$ définie sur $\Q$, munie d'une distance et d'une hauteur, ce qui permet de définir l'exposant diophantien $\beta(x)$ d'un point $x$ dans $X(\R)$. 
L'exposant diophantien est-il constant presque sûrement sur $X(\R)$? Peut-on calculer sa valeur?
Plus généralement, en dehors des variétés de drapeaux, pour quelles variétés algébriques peut-on obtenir une théorie de l'approximation diophantienne intrinsèque satisfaisante?

\bigskip

\printbibliography

\end{document}